\documentclass[12pt, reqno]{amsart}
\title{Categorification and mirror symmetry for Grassmannians}
\date{15 Feb 2026}
\author{Bernt Tore Jensen}
\address{BTJ: Dept of Mathematical Sciences, Norwegian University of Science and Technology, 
Gj\o vik, Teknologvn. 22,2815 Gj\o vik, Norway}
\email{bernt.jensen@ntnu.no}

\author{Alastair King}
\address{AK: Dept of Mathematical Sciences, University of Bath, Bath BA2 7AY, U.K.}
\email{a.d.king@bath.ac.uk}

\author{Xiuping Su}
\address{XS: Dept of Mathematical Sciences, University of Bath, Bath BA2 7AY, U.K.}
\email{xs214@bath.ac.uk}

\usepackage[a4paper, hmargin=2.75cm, vmargin=2.25cm]{geometry}
\usepackage{amssymb,amsthm,amsfonts,amscd,array}
\usepackage{tikz}
\usetikzlibrary{calc}
\usetikzlibrary{decorations.markings, arrows.meta, arrows}
\tikzset{
  equals/.style={double=none, double distance=2pt}, 
  cdarr/.style={->},
  quivarr/.style={-latex,blue,thick},
}   
\usepackage{color}
\usepackage[linktocpage,pdfstartview=FitH,colorlinks=true,linkcolor=blue,citecolor=magenta]{hyperref}

\usepackage{cleveref}
\newcommand{\itmref}[1]{(\ref{#1})}

\theoremstyle{definition}
\newtheorem{theorem}{Theorem}[section]
\newtheorem{proposition}[theorem]{Proposition}
\newtheorem{lemma}[theorem]{Lemma}
\newtheorem{corollary}[theorem]{Corollary}
\newtheorem{remark}[theorem]{Remark}
\newtheorem{example}[theorem]{Example}

\theoremstyle{definition}
\newtheorem{definition}[theorem]{Definition}

\numberwithin{equation}{section}
\numberwithin{figure}{section}

\renewcommand{\setminus}{\smallsetminus}
\renewcommand{\subset}{\subseteq}
\renewcommand{\leq}{\leqslant}
\renewcommand{\geq}{\geqslant}
\newcommand{\subspc}{\leq}
\newcommand{\submod}{\leq}
\newcommand{\isom}{\cong}
\newcommand{\st}{:}
\renewcommand{\ker}{\operatorname{Ker}}
\newcommand{\cok}{\operatorname{Cok}}
\newcommand{\img}{\operatorname{Im}}
\newcommand{\dual}{^\vee}
\newcommand{\op}{^{\mathrm{op}}}

\newcommand{\cob}[1]{d}
\newcommand{\bdry}{\partial}
\newcommand{\Matr}[2]{\operatorname{M}_{#1}(#2)} % ring of nxn matrices over R
\newcommand{\compo}{\circ}
\newcommand{\idfun}{\operatorname{id}}
\newcommand{\idmap}{\operatorname{id}}
\newcommand{\moreq}{\simeq}
\newcommand{\euler}{\chi}
\newcommand{\Grot}{K}
\newcommand{\dyn}[1]{\mathbf{#1}} % Dynkin type
\newcommand{\motsum}{\sum^{\small\mathsf{mot}}} % motivic sum

\newcommand{\lra}{\longrightarrow}
\newcommand{\lraa}[1]{\stackrel{#1}{\longrightarrow}}
\newcommand{\from}{\leftarrow}
\newcommand{\ShExSeq}[3]{0 \lra {#1} \lra {#2}\lra {#3} \lra 0} 
\newcommand{\quadand}{\quad\text{and}\quad}

\newcommand{\CC}{\mathbb{C}} 
\newcommand{\RR}{\mathbb{R}}
 
\newcommand{\ZZ}{\mathbb{Z}} 
\newcommand{\NN}{\mathbb{N}}
\newcommand{\tensR}{\otimes_\ZZ \RR}
\newcommand{\Hom}{\operatorname{Hom}}
\newcommand{\sHom}{\operatorname{\underline{Hom}}}
\newcommand{\sEnd}{\operatorname{\underline{End}}}
\newcommand{\End}{\operatorname{End}}
\newcommand{\Ext}{\operatorname{Ext}}
\newcommand{\Tor}{\operatorname{Tor}}

\newcommand{\funP}{\mathsf{P}}
\newcommand{\funJ}{\mathsf{J}}
\newcommand{\funG}{\mathsf{G}} 
\newcommand{\funK}{\mathsf{K}}

\newcommand{\unitmap}[1]{\eta_{#1}}
\newcommand{\counitmap}[1]{\varepsilon_{#1}}
\newcommand{\epsemb}[1]{{\epsilon_{#1}}}
\newcommand{\CM}{\operatorname{CM}}
\newcommand{\GP}{\operatorname{GP}}
\newcommand{\add}{\operatorname{add}} 

\newcommand{\dimv}{\operatorname{\mathbf{dim}}}

\newcommand{\dg}{\operatorname{deg}}
\newcommand{\len}{\operatorname{len}}
\newcommand{\rnk}[1]{\operatorname{rk}_{#1}} % rk over an algebra
\newcommand{\rkk}{\operatorname{rk}} % rk map on K(CMA)

\newcommand{\labset}[1]{[#1]}
\newcommand{\labsubset}[2]{\binom{\labset{#1}}{#2}}
\newcommand{\val}[1]{\operatorname{val}_{#1}}
\newcommand{\Val}[1]{\operatorname{Val}_{#1}}
\newcommand{\cluschar}[1]{\Psi_{#1}} 
 
\newcommand{\maxNC}{\mathbf{S}}
\newcommand{\maxdiag}{\operatorname{MaxDiag}}
\newcommand{\minor}[1]{\Delta_{#1}}
\newcommand{\qminor}[1]{\Delta_{#1}}
\newcommand{\partn}[1]{\lambda_{#1}}
\newcommand{\Gr}{{\operatorname{Gr}}}
\newcommand{\Gro}{\Gr^{\circ}}
\newcommand{\qvGr}[1]{\operatorname{Gr}_{#1}}
\newcommand{\quotGr}[1]{\operatorname{Gr}^{#1}}
\newcommand{\qvGrsH}[3]{\qvGr{#1} \sHom(#2, #3)}
\newcommand{\perm}{\pi}
\newcommand{\Mlat}{\mathbb{M}}
\newcommand{\Mzero}{\mathsf{M}_0}
\newcommand{\vstar}{*}
\newcommand{\Nstar}{\mathsf{N}_\vstar}
\newcommand{\modJ}{\mathbf{J}_\vstar}
\newcommand{\modP}{\mathbf{P}_\vstar}

\newcommand{\Rspan}{\RR_{\geq 0} \operatorname{-span}}
\newcommand{\ConvHull}{\operatorname{ConvHull}}
\newcommand{\ConvHullBar}{\overline{\operatorname{ConvHull}}}

\newcommand{\Ladj}{\mathsf{L}}
\newcommand{\Radj}{\mathsf{R}}
\newcommand{\proj}{\operatorname{Proj}}
\newcommand{\Ffunct}{\operatorname{\mathsf{F}}}
\newcommand{\efunct}{\operatorname{\mathsf{e}}}
\newcommand{\efun}{\mathsf{e}}

\newcommand{\rrad}{\operatorname{rad}}
\newcommand{\mmod}{\operatorname{mod}}
\newcommand{\wtmod}{\operatorname{\mathsf{Wt}}}
\newcommand{\fd}{\operatorname{fd}}
\newcommand{\wtcom}{\operatorname{wt}}
\newcommand{\wt}{\operatorname{\mathsf{wt}}}
\newcommand{\wthat}{\operatorname{\widetilde{\mathsf{wt}}}}
\newcommand{\betahat}{\operatorname{\widetilde{\beta}}}
\newcommand{\betahatdual}{\betahat{}\dual}

\newcommand{\pbfun}[2]{{^{#1}{#2}}}

\newcommand{\ptbar}{\widetilde{\mathcal{P}}}
\newcommand{\ptfn}{\mathcal{P}}
\newcommand{\fp}{\mathcal{F}}
\newcommand{\flowp}{F}
\newcommand{\partfun}{P}
\newcommand{\Psitil}{\Psi^\dagger}

\newcommand{\olA}{\overline{A}}
\newcommand{\olT}{\overline{T}}
\newcommand{\olQ}{\overline{Q}}
\newcommand{\projdim}{\operatorname{proj.dim}}
\newcommand{\injdim}{\operatorname{inj.dim}}
\newcommand{\gldim}{\operatorname{gl.dim}}
\newcommand{\Irr}{\mathbf{IrrC}}
\newcommand{\Iso}{\operatorname{Iso}}
\newcommand{\Inj}{\operatorname{Inj}}
\renewcommand{\top}{\operatorname{top}}
\newcommand{\soc}{\operatorname{soc}}
\newcommand{\trop}{\operatorname{Trop}}
\newcommand{\GL}[1]{\operatorname{GL}(#1)}
\newcommand{\SL}[1]{\operatorname{SL}(#1)}

\newcommand{\cone}{\operatorname{\mathsf{Cone}}}
\newcommand{\monoid}{\operatorname{\mathsf{Mon}}}
\newcommand{\nob}[1]{\Delta_{\mathrm{NO}}(#1)}
\newcommand{\noc}[1]{\cone_{\mathrm{NO}}(#1)}
\newcommand{\nom}[1]{\monoid_{\mathrm{NO}}(#1)}
\newcommand{\gvb}[1]{\Delta_{\mathrm{GV}}(#1)}
\newcommand{\kab}[1]{\Delta_{\kapvec}(#1)}
\newcommand{\gvc}[1]{\cone_{\mathrm{GV}}(#1)}
\newcommand{\gvm}[1]{\monoid_{\mathrm{GV}}(#1)}

\newcommand{\gtc}{\cone_{\mathrm{GT}}}
\newcommand{\gtm}{\monoid_{\mathrm{GT}}}

\newcommand{\clualgA}{\mathcal{A}}
\newcommand{\clusalg}[2]{\mathcal{C}(#1,#2)}

\newcommand{\unirad}{\mathcal{N}}

\newcommand{\match}{\mathfrak{m}}
\newcommand{\matmod}[1]{N_{#1}}
\newcommand{\flow}{\mathfrak{f}}
\newcommand{\flowpath}{\mathfrak{p}}
\newcommand{\torus}{\mathbb{T}}
\newcommand{\cantor}[1]{\torus_{#1}}
\newcommand{\cluX}{\mathbb{X}}
\newcommand{\cluA}{\mathbb{A}}
\newcommand{\ring}{\mathcal{R}}
\newcommand{\fieldfract}{\mathcal{K}}
\newcommand{\sumrange}[1]{D_{#1}}

\newcommand{\tropmut}{\mathsf{L}}
\newcommand{\tropAmut}[2]{\tropmut_{#2}}
\newcommand{\tropAmuthat}[2]{\widetilde{\tropmut}_{#2}}
\newcommand{\kapvec}{\boldsymbol{\kappa}}

\newcommand{\algC}{C}
\newcommand{\algB}{B}
\newcommand{\algA}{A}
\newcommand{\BCmod}{{}_CB}
\newcommand{\ol}{\overline}
\newcommand{\hC}{\mathbf{c}}
\newcommand{\compC}{\mathbf{C}}

\newcommand{\Yhat}{\widehat{Y}} % see (6.6)

\newcommand{\resfun}{\Ffunct}

\newcommand{\xmut}[1]{\xi_{#1}}
\newcommand{\amut}[1]{\alpha_{#1}}

\newcommand{\MRden}[1]{J_{#1}}
\newcommand{\MRnum}[1]{\widehat{J}_{#1}}

\newcommand{\twi}{\mathrm{twi}}

\newcommand{\procov}[1]{P{#1}}

\newcommand{\superpot}{W}

\newcommand{\syzT}{\Syz T}
\newcommand{\Syz}{\Omega}
\newcommand{\coSyz}{\Sigma}
\newcommand{\CTOsyz}[1]{\mathbf{\Omega}{#1}}
\newcommand{\CTOcosyz}[1]{\mathbf{\Sigma}{#1}}

\newcommand{\preproj}{\Pi}
\newcommand{\pifun}{\pi}
\newcommand{\omfun}{\omega}
\newcommand{\rep}{\operatorname{Rep}}
\newcommand{\sub}{\operatorname{Sub}}

\newcommand{\IrrSub}{\operatorname{Irr}_{\sub}}
\newcommand{\IrrCM}{\operatorname{Irr}_{\CM}}

\newcommand{\preprojQ}{\mathbf{Q}}

\newcommand{\grid}{\Gamma}
\newcommand{\rectCTO}{T^{\square}}
\newcommand{\vempty}{\emptyset}

\newcommand{\yvar}{x}
\newcommand{\xvar}{x}
\newcommand{\cogvec}[1]{\gamma'(#1)}

\newcommand{\Fpolyd}[1]{\mathsf{F}'_{#1}}
\newcommand{\extd}[1]{#1^\diamond}
\newcommand{\opmod}{\op}
\newcommand{\nethat}{\operatorname{n\widehat{e}t}}
\newcommand{\net}{\operatorname{net}}

\newcommand{\neckI}{\mathcal{I}}
\newcommand{\neckJ}{\mathcal{J}}

\begin{document}
% ======================================================================
\begin{abstract}

The homogeneous coordinate ring $\mathbb{C}[\operatorname{Gr}(k,n)]$ of the Grassmannian 
is a cluster algebra, with an additive categorification $\operatorname{CM}C$. 
In particular, every $M$ in $\operatorname{CM}C$ has a cluster character 
$\Psi_M\in\mathbb{C}[\operatorname{Gr}(k,n)]$.
  
We work initially in a more general Frobenius 2-CY subcategory $\operatorname{GP}B$ of $\operatorname{CM}C$,
where $B$ is an algebra defined simply relative to $C$, but equivalent to a choice of Grassmann necklace.
For any cluster tilting object $T$ in $\operatorname{GP}B$, with $A=\operatorname{End}(T)^{\mathrm{op}}$, 
we define two new cluster characters, a generalised partition function 
$\mathcal{P}^T_M\in\mathbb{C}[K(\operatorname{CM}A)]$, 
whose leading exponent is a $g$-vector/index of $M$,
and a generalised flow polynomial $\mathcal{F}^T_M\in\mathbb{C}[K(\operatorname{fd}A)]$,
whose leading exponent is $\boldsymbol{\kappa}(T,M)$,
an invariant introduced in an earlier paper.
These (formal) polynomials are related by applying a map 
$\operatorname{\mathsf{wt}}\colon K(\operatorname{CM}A)\to K(\operatorname{fd}A)$
to their exponents.

When $B=C$, in the $\mathbb{X}$-cluster chart corresponding to $T$, 
we can show that the function $\Psi_M$ 
becomes $\mathcal{F}^T_M$. 
Furthermore when $T$ mutates, $\mathcal{F}^T_M$ undergoes $\mathbb{X}$-mutation
and $\boldsymbol{\kappa}(T,M)$ undergoes tropical $\mathbb{A}$-mutation.
We also show that the monoid of g-vectors can be described by inequalities obtained by tropicalising Marsh--Rietsch's superpotential for Grassmannians and give a module-theoretic interpretation of the inequalities. 
This provides a categorical incarnation of Grassmannian mirror symmetry, in the sense of Rietsch--Williams. 
In the process, we also prove that the Newton--Okounkov body constructed by Rietsch--Williams 
can be described using $\boldsymbol{\kappa}(T, M)$.

%Some of the machinery we develop works in a greater generality,
%which is relevant to positroid subvarieties of $\operatorname{Gr}(k,n)$.
\end{abstract}
% ======================================================================
\maketitle
\setcounter{tocdepth}{1}
\tableofcontents
% ======================================================================

\newpage
% ===========================================================
\section{Introduction and main results}
% ===========================================================

% ======================================================================
\subsection{Background}\label{subsec:background}
% ======================================================================

Let $\Gr(k,n)$  be the Grassmannian of subspaces of $\CC^n$ with $k$-dimensional quotient 
and let $\CC[\Gr(k, n)]$ be its homogeneous coordinate ring. 
As representations of $\GL{n}$ (for example, by the Borel--Weil Theorem), we have
\begin{equation}\label{eq:irrep-decomp}
  \CC[\Gr(k,n)]=\bigoplus_{r\geq 0} V_{r \omega_k},
\end{equation}
where  $\omega_k$ is the $k$th fundamental weight and $V_{r \omega_k}$ is the irreducible 
representation with highest weight $r \omega_k$. 
Fomin--Zelevinksy showed in \cite{FZ1} that $\CC[\Gr(2,n)]$ is 
a cluster algebra of finite type $\dyn{A}_{n-3}$. 
Scott \cite{Sc} extended this result, proving that all $\CC[\Gr(k, n)]$ have cluster structures,
in which the Pl\"ucker coordinates, i.e.~the minors $\minor{J}$ for each $k$-subset $J$ of $\{1,\ldots,n\}$, 
are cluster variables.
There is a rich combinatorics of clusters of minors controlled by plabic graphs (see e.g.~\cite{OPS}),
but there are usually also higher degree cluster variables (i.e.~when $2< k <n-2$).

Following e.g.~\cite{MS,RW}, 
we write $\labset{n}$ for $\{1,\ldots,n\}$
and $\labsubset{n}{k}$ for the collection of all $k$-subsets of $\labset{n}$.
Throughout this paper, 
\[
 \ring= \CC[[t]] \quadand \fieldfract=\ring[t^{-1}]. 
\]  
Note that $\ring$ is a complete local ring
and $\fieldfract$ is its field of fractions.

In \cite[\S3]{JKS1}, we introduced an algebra 
\begin{equation}\label{eq:algC}
  \algC=\algC(k,n),
\end{equation}
which can be described as the path $\ring$-algebra of the circular double quiver with $n$ vertices, 
with relations $xy=t=yx$ and $x^k=y^{n-k}$, 
where $x$ and~$y$ label clockwise and anti-clockwise arrows respectively,
as illustrated here for $n=5$.
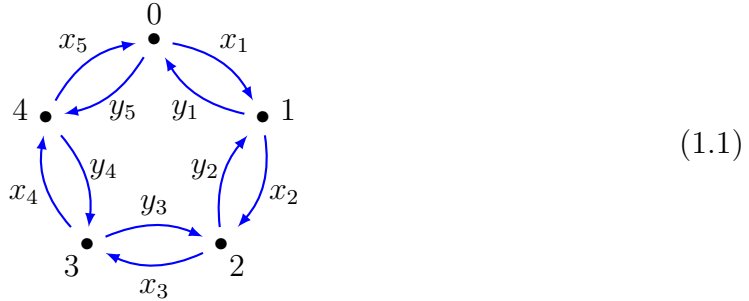
\begin{equation}\label{eq:circle-quiv}
\begin{tikzpicture}[scale=1.0,baseline=(bb.base)]
\draw (0,0) node (bb) {};
\pgfmathsetmacro{\stang}{90}
\pgfmathsetmacro{\incang}{72}
\newcommand{\vrad}{1.5}
\newcommand{\nrad}{1.85}
\newcommand{\xrad}{1.8}
\newcommand{\yrad}{0.7}
\foreach \j in {1,...,5}
{ \coordinate (v\j) at (\stang+\incang-\incang*\j:\vrad);
  \draw (v\j) node (v\j) {\small $\bullet$};  }  
\foreach \t/\h in {5/1,1/2, 2/3, 3/4, 4/5}
{ \draw[quivarr] (v\t) to [bend left=25] (v\h);
  \draw[quivarr] (v\h) to [bend left=23] (v\t);
  \draw (\stang+0.5*\incang-\incang*\h:\xrad) node {$x_{\h}$};
  \draw (\stang+0.5*\incang-\incang*\h:\yrad) node {$y_{\h}$};  }
\foreach \j in {0,...,4}
  \draw (\stang-\incang*\j:\nrad) node {$\j$};
\end{tikzpicture}
\end{equation}
More precisely, we label the vertices clockwise by $1,\dots, n=0\in \ZZ_n$,  
the arrow from $i-1$ to $i$ by $x_{i}$ and the arrow from $i$ to $i-1$ by $y_{i}$. 
Furthermore, the given relations start at every vertex; for example, when $(k,n)=(2,5)$,
the relations starting from 0 are $x_5y_5=t=y_1x_1$ and $x_2x_1=y_3y_4y_5$. 

In \cite[\S9]{JKS1}, building on work of Geiss--Leclerc--Schr\"oer~\cite{GLS08},
we showed how the cluster structure on $\CC[\Gr(k, n)]$ can be categorified
by the category $\CM\algC$ of Cohen--Macaulay $\algC$-modules,
that is, $\algC$-modules that are free over $\ring$.
More precisely, we defined a cluster character
\[
  \Psi\colon \CM\algC \to \CC[\Gr(k, n)]\colon M\mapsto \Psi_M
\]
by homogenising an inhomogeneous cluster character from \cite{GLS08}.
Now any $M\in \CM\algC$ has a well-defined rank $\rnk{\algC} M$ 
(see \cite[Def.~3.5]{JKS1} and \S\ref{subsec:ranks}),
which satisfies
\[
  \rnk{\algC} M=\rnk{\ring} e_i M
\] 
for any vertex idempotent $e_i\in\algC$. 
The rank 1 modules $M_J$ % in $\CM\algC$ 
are indexed by $J\in \labsubset{n}{k}$
(see \cite[\S5]{JKS1} or \S\ref{subsec:nota-conv}(1) for details of their construction)
and we have
\begin{equation}\label{eq:Psi-minor}
  \Psi_{M_J} = \minor{J}.
\end{equation}

In \cite{JKS2}, we found that $\CM\algC$ 
also categorifies the quantum cluster algebra  $\CC_q[\Gr(k, n)]$,
in the sense that it knows the quasi-commutation rules.
More precisely, we defined an invariant $\kappa(M,N)$ for modules $M,N$ in $\CM\algC$. 
When $I, J\in\labsubset{n}{k}$ are non-crossing, 
 the quantum minors $\qminor{I}$ and $\qminor{J}$ quasi-commute
\cite[Thm.~1.1]{LZ},
and  we have \cite[Thm.~6.5]{JKS2} that
\[
  q^{\kappa(M_I, M_J)} \qminor{I} \qminor{J}
  =q^{\kappa(M_J, M_I)} \qminor{J} \qminor{I}.
\]
We also observed \cite[Lem.~7.1]{JKS2} the apparently independent fact that 
\[
  \kappa(M_I, M_J)=\maxdiag(\partn{I}\setminus\partn{J}),
\] 
where the right-hand side is a combinatorial invariant that plays a key role in Rietsch--Williams'
work \cite{RW} on mirror symmetry for Grassmannians.
Here $\partn{I}, \partn{J}$ are certain Young diagrams labelled by $I,J$
and the invariant is the maximal length of the diagonals in $\partn{I}\setminus\partn{J}$.

In \cite{RW}, Rietsch--Williams uncovered an instance of mirror symmetry for Grassmannians 
by studying their Newton--Okounkov (NO) bodies,
which are certain convex sets (in this case, rational polytopes) that encode the leading exponents 
$\val{G}(f)$ of all functions $f\in \CC[\Gr(k, n)]$, when expressed in the network (or $\cluX$-cluster) chart
associated to a plabic graph $G$.
On the mirror side, there is a superpotential $\superpot$ 
and \cite[Thm~16.18]{RW} shows that,
by expressing $\superpot$ in the dual $\cluA$-cluster chart associated to $G$
and tropicalising, one obtains the linear inequalities that describe the NO-body.
Furthermore, in \cite[Thm~15.1]{RW}, the leading exponents of minors $\minor{J}$
are shown to be given by 
\begin{equation}\label{eq:valG=maxdiag}
  \val{G}(\minor{J})=(\maxdiag(\partn{I}\setminus\partn{J})\st I\in\maxNC),
\end{equation}
where $\maxNC\subset\labsubset{n}{k}$ is the maximal non-crossing 
set of face labels of $G$.

One key feature of categorification is that clusters are upgraded to 
cluster tilting objects $T$ in $\CM\algC$. 
In particular, a cluster of minors $\{\minor{I} \st I\in\maxNC\}$,
for a maximal non-crossing set $\maxNC$, is upgraded to
\begin{equation}\label{eq:CTO-maxNC}
 T_\maxNC = \textstyle\bigoplus_{I\in\maxNC} M_I.
\end{equation}
Thus the leading exponent of $\minor{J}$ can be expressed as a vector invariant
$\kapvec(T_\maxNC,M_J)$,
which is a dimension vector for the algebra $\algA=\End(T_\maxNC)\op$,
whose Gabriel quiver $Q$ is (opposite) dual to the plabic graph $G$
\cite[Thm.~10.3]{BKM}.
In other words, $\kapvec(T_\maxNC,M_J)$ is an element of the Grothendieck group 
$\Grot(\fd\algA)\isom\ZZ^{Q_0}$
of finite-dimensional $\algA$-modules.

This observation was the starting point for the current paper,
since a natural generalisation of \eqref{eq:valG=maxdiag} would be 
\begin{equation}\label{eq-int:valG=kappadvat}
  \val{G}(\Psi_M)=\kapvec(T,M),
\end{equation}
where $M$ is any module in $\CM\algC$, 
while $G$ stands for an arbitrary $\cluX$-cluster chart (as in \cite[Rem~6.17]{RW})
and $T$ is the corresponding (reachable) cluster tilting object in $\CM\algC$.
In fact, \eqref{eq-int:valG=kappadvat} yields a description of all the integral
points in the cone on the NO-body.

One of the main goals of this paper is to prove \eqref{eq-int:valG=kappadvat} for general $M$ and $T$.
We do this by defining a generalised flow polynomial $\fp^T_M$ whose leading exponent is $\kapvec(T,M)$
and then show that $\Psi_M$ becomes $\fp^T_M$ when written in a suitably defined network chart.
In fact, these flow polynomials will be defined in much greater generality,
in the context of general positroid subvarieties of $\Gr(k,n)$,
related to the cell decomposition of
the totally non-negative Grassmannian \cite{Pos}.
However, we can only show the second part for $\Gr(k,n)$ itself.
%\danger{
%Some of the machinery we develop to do this works in a greater generality,
%which is relevant in the context of positroid subvarieties of $\Gr(k,n)$,
%related to the cell decomposition of
%the totally non-negative Grassmannian \cite{Pos}.
%}%\comment{R: which parts generalise?}

We now explain the main topics and results of the paper.
Sections~\ref{Sec:3}--\ref{sec:clusalg} \&~\ref{Sec:11}
apply to the category $\GP\algB$, that we will introduce in Section~\ref{Sec:2} and
which is shown in separate work~\cite{JRS} to categorify general positroid varieties.

% ====================================================
\subsection{Positroids and necklace algebras} \label{subsec:pos-neck}
% ====================================================
%\subsection{Frobenius $2$-Calabi--Yau categories.} \label{subsec:pos-neck}
% ====================================================

Positroid varieties are characterised by various sorts of combinatorial data, including 
Grassmann necklaces~\cite[\S16]{Pos}.
One of our first results (\Cref{cor:neck-alg}) is that 
this necklace data is equivalent to specifying 
a `necklace algebra' $\algB$, 
which is defined in a quite simple way as
an $\ring$-order in 
$\algC[t^{-1}]\isom \Matr{n}{\fieldfract}$, lying over $\algC$, that is,
\begin{equation}\label{eq-intro:CinB}
  \algC\subset \algB\subset \algC[t^{-1}],
\end{equation}
with the additional property that $\algB$ is rigid as a left $\algC$-module.
There are two important subcategories associated to $B$,
\[
  \GP\algB \subset \CM\algB \subset \CM\algC .
\]
When $\algB$ is a necklace algebra, we show (\Cref{Thm:GPB3}) 
that $\algB$ is Iwanaga--Gorenstein of injective dimension at most~$2$, 
and that $\GP\algB$ is the category of Gorenstein projective $\algB$-modules 
and is Frobenius $2$-Calabi--Yau.
Furthermore, $\GP\algB$ has a cluster structure (\Cref{rem:cluster-structure})
and, in particular, a collection of cluster tilting objects $T$ related by mutation.

%This is the general context in which we define generalised flow polynomials $\fp^T_M$,
%which we further show give a cluster character on $\GP\algB$ .

% ======================================================================
\subsection{Generalised partition functions and flow polynomials} \label{subsec:flow-part}
% ======================================================================

Let $G$ be a plabic graph of `rank' $n$ and `helicity' $k$,
and let $Q$ be its dual quiver with faces (cf. \cite{BKM}, \cite{CKP}).
As described by Muller-Speyer~\cite{MS},
the homogeneous network chart associated to $G$ can be given by the partition function map 
\begin{equation}\label{eq:netchart}
  \nethat_G\colon \CC[\Gr(k,n)] \to \CC[\torus_G]
  \colon \minor{J}\mapsto \partfun_J = \sum_{\mu\, :\,\bdry\mu=J} \xvar^\mu,
\end{equation}
which corresponds geometrically to an embedding of a torus $\torus_G$ into the affine cone on the Grassmannian.
The image of the positive real part of $\torus_G$ is the positroid cell in the decomposition.

Note that the $\mu$ in the sum are perfect matchings on $G$, which can be viewed
as elements of the character lattice $\Mlat$ of $\torus_G$.
Indeed we prefer to write the coordinate ring $\CC[\torus_G]$ as
the formal Laurent polynomial ring $\CC[\Mlat]$, 
so that $\mu\in\Mlat$ becomes the formal monomial $\xvar^\mu\in \CC[\Mlat]$.

%Note also that the image of the torus $\torus_G$ is open in (the affine cone on) the closed positroid subvariety of 
%$\Gr(k,n)$ on which $\minor{J}=0$, for all $J$ which are not boundary values of matchings,
%that is, not in the positroid (a subset of $\labsubset{n}{k}$). 

The homogeneous chart $\nethat_G$ becomes the more usual network chart $\net_G$
expressed in terms of flow polynomials $F_J$ (cf.~\cite[\S6]{RW}) 
by making a dehomogenising monomial change of variables, given by a lattice map 
\begin{equation}\label{eq-int:wt-map}
\wtcom\colon \Mlat\to \Nstar\subset \ZZ^{Q_0},
\end{equation}

\noindent
where $\Nstar$ is the sublattice of vectors that vanish 
at a chosen boundary vertex $\vstar$
(see \Cref{sec:classical} for details).
More precisely, $\net_G=\CC[\wtcom] \compo\nethat_G$, 
that is, the flow polynomials are $F_J=\CC[\wtcom] \partfun_J \in \CC[\Nstar]$.
In particular, all exponents are in $\NN^{Q_0}$.

%\medskip
%We now explain the main topics and results of the paper.
%Sections~\ref{Sec:2}--\ref{sec:clusalg} \&~\ref{Sec:11}
%apply to the category $\GP\algB$, that we will introduce below and
%which categorifies general positroid varieties~\cite{JRS}.

In the general setting of a cluster tilting object $T$ in $\GP\algB$, 
a key part of the categorification process is to interpret the map 
$\wtcom\colon \Mlat\to \ZZ^{Q_0}$ in \eqref{eq-int:wt-map}
as a natural map between Grothendieck groups associated to $\algA=\End(T)\op$
\[
  \wt\colon \Grot(\CM\algA)\to \Grot(\fd\algA),
\] 
which is induced by an exact functor $\wtmod\colon \CM\algA\to\fd\algA$,
introduced in Section~\ref{Sec:4}.
In addition, we define an isomorphism $\nu\colon\Grot(\CM\algA)\to\Mlat$ 
(see \Cref{Prop:equivof2exactseq}).
%which is used together with the standard isomorphism $\Grot(\fd\algA)\isom\ZZ^{Q_0}$
%that realises classes as dimension vectors 

We observe (\Cref{Lem:kerwt}) that 
\begin{equation}\label{eq-int:kap-wt}
  \kapvec(T,M)=\wt [T, M],
\end{equation}
where $[T,M]=[\Hom_\algB(T,M)]\in \Grot(\CM\algA)$ is the \emph{$g$-vector} of $M\in \CM\algB$.
Note that what we call the $g$-vector here is often called the `index',
while `$g$-vector' is reserved for the mutable part (e.g.~\cite[Def~7.9]{FZ4} or \cite[\S4]{FK}).
The $g$-vector $[T,M]$ can be expressed in components in the basis of indecomposable projective $A$-modules,
and so can also be identified with the class of an $\add T$-presentation of $M$ in $\Grot(\add T)=\Grot(\proj\algA)$.
%but we will avoid doing this here as far as possible.

Furthermore, following \cite{CKP},
the boundary value map $\mu\mapsto\bdry\mu$ is replaced by
the restriction functor $\efunct\colon \CM\algA\to\CM\algB \colon X\mapsto eX$,
so the condition $\bdry\mu=J$ is replaced by $eX=M_J$.
Note that, to interpret $eX=M$ correctly, we must use the right adjoint $\Radj = \Hom_\algB(T,-)$ to $\efunct$
and only consider $X\submod\Radj M$, so that \emph{a priori} $eX\submod M$.

Thus one might try to generalise the partition function $\partfun_J\in\CC[\Mlat]$ in \eqref{eq:netchart}
to arbitrary $M\in\CM\algB$ by writing
\begin{equation}\label{eq-int:mot-pf}
  \ptfn^T_M \, =\, \sum_{X :\,eX=M} \xvar^{[X]},
\end{equation}
where the exponents are in $\Grot(\CM\algA)$.
One problem is that this sum can be infinite,
so must be interpreted motivically, that is,
counting families of submodules by their Euler characteristic~$\euler$.

The condition $eX=M$ turns out to be equivalent to requiring that $\Radj M/X$
is a quotient of the stable Hom space $\sHom_\algB(T,M)$,
so \eqref{eq-int:mot-pf} can be written as 
\begin{equation}\label{eq-int:gen-pf}
  \ptfn^T_M =\xvar^{[T, M]}\sum_{d} \euler \bigl( \quotGr{d}\sHom_\algB(T,M) \bigr)\xvar^{-\beta(d)},
\end{equation}
where $\quotGr{d}H$ is the (quiver) Grassmannian of quotients of $H$ 
of class $d\in\Grot(\fd\algA)$ %and $\euler$ is the Euler characteristic, 
and $\beta\colon \Grot(\fd\algA)\to \Grot(\CM\algA)$ is induced by CM-approximation or, 
equivalently, projective resolution
(see \S\ref{sec:gen-pt-fn} for more details,
including \Cref{rem:ptfn-F-poly}). 
% for why \eqref{eq-int:gen-pf} is equivalent to the initial definition \eqref{eq:part-fun}).
In particular, we can write \[[X]=[T,M] - \beta[\Radj M/X].\]

We also obtain a generalised flow polynomial, by a dehomogenising monomial change of variables, as before.
\begin{equation}\label{eq:int-flowp}
\fp^T_M=\CC[\wt] \ptfn^T_M = \xvar^{\kapvec(T, M)} \Fpolyd{\sHom_\algB(T,M)}(\yvar), 
\end{equation}
where $\Fpolyd{H}(\yvar)$ is the F-polynomial for quotients of $H$.
Note that \eqref{eq-int:gen-pf} bears a strong resemblance to the more 
familiar cluster character formula of Fu--Keller~\cite{FK}, that is, 
\begin{equation}\label{eq:int-FK-CC}
  \Phi^T_{M}
 =\xvar^{[T, M]}\sum_{d} \euler \bigl( \qvGr{d}\Ext_\algB^1(T, M) \bigr)\xvar^{-\beta(d)},
% = \xvar^{[T, M]} \Fpoly{\Ext^1(T, M)}(\xvar^{-\beta}),
\end{equation}
where $\qvGr{d}E$ is the Grassmannian of submodules of $E$ 
of class $d\in\Grot(\fd\algA)$. 
%Thus $\Fpoly{E}(\yvar)$ is the usual F-polynomial for submodules of~$E$.

In fact, we prove two formulae that relate $\ptfn^T_{M}$ and $\Phi^T_{M}$, namely
\begin{equation}\label{eq:PF-formulae}
\ptfn^T_M = \CC[-1]\frac{ \Phi^T_{\Syz M} }{ \Phi^T_{\procov{M}} }
  = \CC[\zeta]\, \Phi^{\CTOcosyz{T}}_M,
\end{equation}
where $0 \to \Syz M \to \procov{M} \to M \to 0$ is a syzygy sequence for $M$ (i.e.~$PM$ is projective)
and $\CC[-1]$ and $\CC[\zeta]$ are monomial changes of variables
(see \Cref{Thm:char-parfunA} and \Cref{thm:PF-zeta-Phi} for details). 
In particular, this shows that $\ptfn^T_{M}$ is a cluster character and hence $\fp^T_M$ is also.

% ======================================================================
\subsection{Network and cluster charts} \label{subsec:net-chart}
% ======================================================================

To go further, we need to restrict to the case $B=C$,
in which case $\GP\algB=\CM\algC$.
Our first major result (\Cref{cor:partfun-map}) is that, 
for any cluster tilting object $T$ in $\CM C$, there is a map
\begin{equation}\label{eq-int:hom-net-chart}
  \Xi^T\colon \CC[\Gr(k,n)] \to \CC[\Grot(\CM A)],
%  \colon \cluschar{M}\mapsto \ptfn_{M}
\end{equation}
for $\algA=\End(T)\op$, with the property that, for all $M\in\CM\algC$,
\[
  \Xi^T\bigl(\cluschar{M}\bigr)=\ptfn^T_{M}.
\]
Thus $\Xi^T$ generalises the homogeneous network chart $\nethat_G$ from \eqref{eq:netchart}
and consequently
\begin{equation}\label{eq-int:net-chart}
\pbfun{\wt}{\Xi^T} = \CC[\wt]\compo\Xi^T\colon \CC[\Gr(k,n)]  \to \CC[\Grot(\fd\algA)]
\colon \cluschar{M}\mapsto \fp^T_{M}
\end{equation}
generalises the inhomogeneous network chart $\net_G$.
The main stepping stone to this result is (\Cref{Thm:PsiandTheta}) that the $\cluA$-cluster chart associated to $T$ is a map
\begin{equation} \label{eq:int-UpsT}
  \Upsilon^{T} \colon \CC[\Gr(k, n)]\to \CC[\Grot(\CM\algA)]
  %\colon \Psi_{M} \mapsto \Phi^{T}_{M} 
\end{equation}
with the property that, for all $M\in\CM\algC$,
\[ \Upsilon^{T}\bigl(\cluschar{M}\bigr)=\Phi^{T}_{M}.\]
Note that \emph{a priori} this formula only holds for reachable rigid $M$,
since $\Phi^{T}$ is a cluster character.
This result is proved by lifting a similar result of Geiss--Leclerc--Schr\"oer~\cite{GLS12}.
Then \Cref{cor:partfun-map} follows by setting $\Xi^T=\CC[\zeta]\compo \Upsilon^{\CTOcosyz{T}}$
and using \Cref{thm:PF-zeta-Phi}.

We can also use these results to study the twist map
\[ \twi\colon \Gro(k, n) \to \Gro(k, n) \]
defined by Muller--Speyer \cite{MS} more generally for open positroid varieties,
that is, where the frozen variables are non-zero.
We show (\Cref{Thm:twistchar}) that, for any $M\in\CM\algC$, 
\[
\twi(\Psi_M)=\frac{\Psi_{\Syz M}}{\Psi_{\procov{M}}}, % =: \Psitil_M,
\]
generalising \cite[Thm~12.2]{CKP}, that is, the case when $M$ has rank~1.

% ======================================================================
\subsection{$g$-vectors and bases} \label{subsec:gen-bas-gvec}
% ======================================================================

The fact that $g$-vectors are the leading exponents of cluster characters, 
in both cluster charts and homogeneous network charts, 
leads us to look more closely at the set of all $g$-vectors in $\Grot(\CM A)$,
which is \emph{a priori} a monoid (i.e.~closed under addition and with a zero element)
\[
  \gvm{T} = \{ [T, M] \st M\in\CM C \} 
\]
and to investigate its relationship with $\nom{T}$,
the monoid of leading exponents of all functions $f\in \CC[\Gr(k,n)]$ in the network chart $\Xi^T$.

We use the irreducible components 
of representation varieties of preprojective algebras
to say what it means for a $g$-vector $[T, M]$ to be generic.
It is known, from Lusztig~\cite{Lus00} and Geiss--Leclerc--Schroer \cite{GLS12},
that choosing modules $M$ whose cluster characters $\cluschar{M}$ are generic in a component
gives a basis of $\CC[\Gr(k,n)]$, called the dual semicanonical basis,
indexed by the components.
We show more generally (\Cref{Thm:GrassGeneric}) that,
it is sufficient that each $[T, M]$ is generic in a component for the $\cluschar{M}$
to give a basis.
A key ingredient (\Cref{prop:ghgen-inj}) is that generic $g$-vectors
for different components are distinct, so that the $\cluschar{M}$ have distinct leading exponents.

\newcommand{\gvec}{\gamma}

Having this basis enables us to prove
(\Cref{thm:subQComp-too}) that, in fact, every $g$-vector is generic.
Hence (\Cref{cor:bottom-line}) it is sufficient to choose, for each $\gvec\in \gvm{T}$, a representative~$M$ 
with $[T,M]=\gvec$ to ensure that the cluster characters $\cluschar{M}$ form a basis.
This both realises and extends the objective of Fomin-Zelevinsky \cite[\S7]{FZ4}
in introducing $g$-vectors to index cluster monomials (i.e. part of a basis) in a cluster algebra.

Then we can also deduce (\Cref{Thm:NOcone}) that
\begin{equation*}\label{eq:int-gvecs-r-generic}
  \nom{T} % = \{[T, M] \st \text{$M$ generic} \}
 = \gvm{T}.
\end{equation*}

% ======================================================================
\subsection{Tropical $\cluA$-mutation} \label{subsec:mut-kappa}
% ======================================================================
We can get a better handle on $g$-vectors by using the lattice isomorphism
\begin{equation}\label{eq-int:wthat}
  \wthat\colon \Grot(\CM\algA) \to \ZZ\oplus \Nstar\colon [Z]\mapsto (\rkk[Z],\wt[Z]),
\end{equation}
where $\Nstar\subset \Grot(\fd\algA)$ is as in \eqref{eq-int:wt-map}.
In particular, $\wthat[T,M]=(\rnk{\algC} M,\kapvec(T, M))$.

When $T$ mutates to $T'=\mu_i(T)$, we can interpret tropical $\cluA$-mutation
(\Cref{def:tropAmut}) as a piecewise linear map 
\[
\tropAmut{Q}{i}\colon \Grot(\fd\algA) \to \Grot(\fd\algA'),
\]
where $A'=\End(T')\op$, and show (\Cref{Cor:mutkappa}) that
\[ \tropAmut{Q}{i} \bigl( \kapvec(T, M) \bigr) = \kapvec(T', M), \]
when $M$ is suitably generic.
Then $\tropAmut{Q}{i}$ and $\wthat$ induce a piecewise linear map 
\[
  \tropAmuthat{Q}{i}\colon K(\CM A)\to K(\CM A') \colon [T,M]\mapsto [T',M].
\]
As a consequence, we can prove (\Cref{rem:rat-pol-cone2}) that $\gvm{T}$ is 
(the integral points of) a rational polyhedral cone,
following the inductive strategy of Rietsch--Williams~\cite{RW}.
The induction starts from the rectangles cluster tilting object $\rectCTO$ (\Cref{def:rect-clus}) 
and we show (\Cref{Prop:kappaGT}) that 
\begin{equation} \label{eq:gzmonoid}
\wthat\gvm{\rectCTO} = \gtm, 
\end{equation}
the monoid of (cumulative) Gelfand--Tsetlin patterns,
which is defined by finitely many integral inequalities, 
so is, by construction, a rational polyhedral
cone $\gtc$.

% ======================================================================
\subsection{Valuations and bodies} \label{subsec:val-bodies}
% ======================================================================

Suppose $G_\maxNC$ is a plabic graph %(of type $\perm_{k,n}$), 
whose face labels are a maximal non-crossing collection $\maxNC$.
Then $T_\maxNC=\bigoplus_{J\in \maxNC} M_J$ is 
the corresponding cluster tilting object in $\CM\algC$.
As in \cite{RW}, a general $\cluX$-seed will also be denoted $G$
and is related to a plabic seed $G_\maxNC$ by a sequence of $\cluX$-mutations,
which inductively determine an $\cluX$-chart $\net_G$.
Applying the same sequence of mutations to $T_\maxNC$ associates a cluster tilting object $T$ to $G$.

We know (\Cref{prop:class-flow-poly}) that the generalised network chart $\pbfun{\wt}{\Xi^T}$
in \eqref{eq-int:net-chart}
coincides with the classical network chart $\net_G$
when $G$ is a plabic seed.
On the other hand, we can show (\Cref{thm:Xflow}) that $\pbfun{\wt}{\Xi^T}$ undergoes $\cluX$-mutation 
when $T$ mutates and so, 
for all associated $G$ and $T$, we have
\[ 
  \net_G=\pbfun{\wt}{\Xi^T}, 
\]
that is, \eqref{eq-int:net-chart} provides an explicit formula for the general $\cluX$-chart.

In \cite{RW}, Rietsch--Williams define the valuation $\val{G}(f)$, for non-zero $f\in \CC[\Gr(k, n)]$,
to be the minimal exponent (in $\Nstar$) of $\net_G(f)$ and 
we know (\Cref{Thm:char-flopol}) that the minimal exponent of 
$\fp^{T}_M$ is $\kapvec (T, M)$.
Hence we conclude (\Cref{Thm:KVal}) that,
for any $M\in \CM\algC$,
\begin{equation*}%\label{eq:KVal1}
  \val{G}(\cluschar{M}) = \kapvec (T, M).
\end{equation*}
In particular, since $\cluschar{M_I}=\minor{I}$,
\begin{equation*}%\label{eq:KVal2}
  \val{G}(\minor{I}) = \kapvec(T, M_I),
\end{equation*}
which recovers \eqref{eq:valG=maxdiag} when $G$ is a plabic seed.
Thus we have proved \eqref{eq-int:valG=kappadvat} as intended, 
and as conjectured in  \cite[Remark~7.3]{JKS2}.
Using $\val{G}$, 
Rietsch--Williams \cite[(8.2)]{RW} also define the \emph{Newton--Okounkov body} 
in $\Nstar\tensR$,
\[
  \nob{G} = \ConvHullBar\, \bigcup_r\frac{1}{r} \bigl\{ \val{G}(f)  \st  f\in \CC[\Gr(k, n)]_r \setminus 0 \bigr\} .
\]
On the other hand, using the $\kapvec$-invariant, we can likewise define a body
\[
  \kab{T} = \ConvHullBar\, \bigcup_r\frac{1}{r} 
  \bigl\{ \kapvec(T,M) \st M\in\CM\algC,\, \rnk{\algC} M = r \bigr\}
  \subset \Nstar\tensR,
\]
which is the degree 1 slice of the cone $\wthat\gvc{T}$.
Clearly \eqref{eq-int:valG=kappadvat} indicates a relationship between these cones/bodies 
and indeed we can show (\Cref{prop:nob=kab}) that 
\begin{equation*}%\label{eq:int-nob=kab}
  \nob{G}=\kab{T}.
\end{equation*}  
The proof uses the existence of bases of $\CC[\Gr(k, n)]$ 
consisting of cluster characters~$\cluschar{M}$ with distinct leading exponents~$[T,M]$,
as described in \S\ref{subsec:gen-bas-gvec}.

% ======================================================================
\subsection{Mirror symmetry} \label{subsec:mir-sym}
% ======================================================================
Knowing that $\gvm{T}$ is rational polyhedral, the natural question is: what
inequalities define it?
It is known \cite[Lem~16.2]{RW} that 
\begin{equation}\label{eq-int:GT-SP}
\gtc=\cone_W(\rectCTO),
\end{equation}
where, in general, the \emph{superpotential cone} $\cone_W(T)$
is obtained from the Marsh--Rietsch \cite{MR} superpotential $W$ 
by writing it in the $\cluA$-cluster chart associated to $T$ and then tropicalising, 
that is, using the exponents to determine the defining inequalities.

In \Cref{Prop:superpot}, we give a new proof of  \eqref{eq-int:GT-SP},
by writing the terms in $W$ as ratios of cluster characters.
Given \eqref{eq-int:GT-SP} and \Cref{Cor:mutkappa}, we immediately deduce (\Cref{thm:trop-GT}) that 
\begin{equation*}%\label{eq-int:RW-coneSP}
 \wthat\gvc{T} = \cone_W(T).
\end{equation*}
for all reachable cluster tilting objects $T$,
because both sides mutate in the same way.

Hence we can use the dual of the isomorphism $\wthat$ 
to rewrite $W$ (\Cref{thm:Wformula}) 
and thereby find the inequalities defining $\gvc{T}$ 
inside $\Grot(\CM\algA)$ (\Cref{thm:coneGV-SP}).
More precisely, these inequalities are obtained by tropicalising the 
following superpotential, written in $\cluX$-cluster coordinates, that is, in $\CC[\Grot(\fd\algA)]$.
\begin{equation}\label{eq:int-UpsW2}
  W_\cluX(T) = \sum_{i=1}^n \yvar^{[S_i]} \Fpolyd{\Ext^1(T,\extd{P_i})} (\yvar) % \in \CC[\Grot(\fd\algA)].
\end{equation}
Here $\extd{P_i}$ is an extension of~$P_i$ by a simple $\algC$-module~$\extd{S_i}=S_{i+k}$, 
that is, we have a short exact sequence
\begin{equation*}%\label{eq:int-PPS}
  0\lra P_i \lra \extd{P_i} \lra \extd{S_i} \lra 0
\end{equation*}
and $\Fpolyd{E} (\yvar)$ is the quotient F-polynomial of $E$
(see \S\ref{subsec:cat-superpot} for more details).

Explicitly, the inequalities defining $\gvc{T}$ are
\begin{equation}\label{eq:int-gv-faces1}
  ([S_i]+[V])(x) \geq 0, \quad\text{for all quotients $V$ of $\Ext^1(T,\extd{P_i})$, for $i=1,\ldots,n$.} 
\end{equation}
Alternatively, we can write the inequalities  as
\begin{equation}\label{eq:int-gv-faces2}
  [U](x) \geq 0, \quad\text{for all $U\subspc\Ext^1(\extd{S_i},T)$, for $i=1,\ldots,n$,} 
\end{equation}
where $[U]\in\Grot(\fd\algA\op)$, which we can also view as dual to $\Grot(\CM\algA)$.

Note that the proof very much follows the strategy of \cite{RW} and it would be very interesting to
have a more categorical proof of the characterisation of $g$-vectors
given by \eqref{eq:int-gv-faces1} or \eqref{eq:int-gv-faces2}.
Note also that the superpotential in \eqref{eq:int-UpsW2} should
coincide with the Gross--Hacking--Keel--Kontsevich superpotential~\cite{GHKK}, 
as described in~\cite[Thm~6.5]{LFS} (see also \cite[App ~D]{SW}).

% ======================================================================
\subsection{Notation and conventions}\label{subsec:nota-conv}
% ======================================================================

As in \cite{JKS1}, the Pl\"ucker label $I$ on a rank 1 $\algC$-module $M_I$
specifies that the \emph{clockwise} arrows $x_i$ in \eqref{eq:circle-quiv} act as $1$, for $i\in I$.
The opposite convention is used in \cite{CKP}.

A plabic graph $G$ is assumed to come from a (consistent) Postnikov diagram,
associated to a positroid variety. 
The \emph{type} of $G$ is the (decorated) strand permutation $\perm$.
When the variety is $\Gr(k,n)$, this is the uniform permutation 
$\perm_{k,n}\colon j\mapsto j+k$ mod $n$.

We obtain a cluster $\maxNC$ of Pl\"ucker labels from a Postnikov diagram 
using \emph{left target} labelling (cf.~\Cref{fig:rs-labelling}),
so that the boundary labels are a necklace, in the sense of \cite[\S16]{Pos} or \cite[\S2.1]{MS}.
Consequently, we take the \emph{opposite} dual quiver to obtain the Gabriel quiver of $\End(T_\maxNC)\op$
(cf.~\eqref{eq:CTO-maxNC}).
Other conventions are used in~\cite{BKM, CKP}.

% ======================================================================
\subsection*{Acknowledgements} 
% ======================================================================

We wish to thank Timothy Magee, Matthew Pressland and Konstanze Rietsch for 
many enlightening discussions, especially about mirror symmetry.
We are also grateful to the Heilbronn Institute for Mathematical Research for 
a Focused Research Grant (2021) that enabled these discussions to take place.

%\goodbreak
\newpage
%%%% =========================================================== %%%%
\section{Frobenius subcategories $\GP\algB$ in $\CM\algC$}\label{Sec:2}
%%%% =========================================================== %%%%

We know from \cite{JKS1} that $\CM\algC$ is a (stably) 2-Calabi-Yau (2-CY) Frobenius category.
In this section, we find many other such categories $\GP\algB$ as full subcategories of $\CM\algC$.

% ======================================================================
\subsection{Adjoint functors}
% ======================================================================
\newcommand{\catC}{\mathcal{C}}
\newcommand{\catD}{\mathcal{D}}

We list some basic facts about adjoint functors, for later use.

\begin{lemma}\label{Lem:adjprops} 
\cite[Prop 3.4.1]{Bo}
Let $\Ffunct\colon \catC\to \catD$ be a functor with a left adjoint $\Ladj\colon \catD\to \catC$ 
and a right adjoint $\Radj\colon \catD\to \catC$. 
\begin{enumerate}
\item\label{itm:adjp1} % 1
The following are equivalent.
\begin{enumerate}
\item The functor $\Ffunct$ is faithful.
\item The unit $\unitmap{}\colon \idfun_{\catC} \to \Radj \Ffunct$ 
 is a pointwise monomorphism, 
 i.e.~$X\to \Radj \Ffunct X$ is mono for any $X\in  \catC$. 
\item The counit $\counitmap{}\colon \Ladj \Ffunct \to  \idfun_{\catC} $ 
is a pointwise epimorphism, 
i.e.~$\Ladj \Ffunct Y \to  Y$ is epi for any $Y\in  \catC$. 
\end{enumerate}
\item\label{itm:adjp3} % 2
The following are equivalent.
\begin{enumerate}
\item The functor $\Ffunct$ is fully faithful.
\item The unit $\unitmap{}\colon \idfun_{\catC} \to \Radj \Ffunct$ is an isomorphism.
\item The counit $\counitmap{}\colon \Ladj \Ffunct \to  \idfun_{\catC}$ is an isomorphism.
\end{enumerate}
\end{enumerate}
\end{lemma}
Recall that, in an additive category, a morphism $f\colon X\to Y$ is an epimorphism if 
$gf=0$ implies that $g=0$, for any $g\colon Y\to Z$. 
For example, in the category of torsion-free modules for any $\ring$-algebra, 
an epimorphism is a map with torsion cokernel, not necessarily a surjection.

% ======================================================================
\subsection{The category $\CM\algB$}\label{subsec:CMB}
% ======================================================================

Note that $\algC[t^{-1}]$ is isomorphic to the matrix algebra $\Matr{n}{\fieldfract}$, by \cite[(3.6)]{JKS1}.
Let $\algB$ be an $\ring$-order with 
\begin{equation}\label{eq:CinB}
  \algC\subset \algB\subset \algC[t^{-1}]. %\isom \Matr{n}{\fieldfract}.
\end{equation}
Consequently, $e_iBe_i=e_iCe_i=\ring$, for any vertex idempotent $e_i$ of $C$.

Recall that a $\algC$-module is Cohen--Macaulay if and only if $\algC$ is a free $\ring$-module. 
We also denote by $\CM\algB$  the category of  $\algB$-modules that are free over $\ring$.

\begin{lemma}\label{Lem:restriFF} \cite[Chap.~23, Exer.~2]{CR}  
The restriction functor $\resfun\colon \CM\algB \to \CM\algC$ is fully faithful. 
In particular, all $C$-summands of $M\in\CM\algB$ are $B$-summands.
\end{lemma}
 
 \begin{proof}
 If $M, N\in \CM\algB$, then  $M\leq M\otimes_\ring \fieldfract$, $N\leq N\otimes_\ring \fieldfract$
 and 
 \[
 \Hom_\algB(M, N)=\{f\colon M\otimes_\ring \fieldfract\to N\otimes_\ring \fieldfract \st f(M)\subset N\}=\Hom_\algC(M, N).
 \]
Thus $\resfun$ is fully faithful. 
For the last part, note that, since $\Hom_C(M,M)=\Hom_B(M,M)$, any $C$-idempotent is a $B$-idempotent.
\end{proof}

Since the restriction functor $\Ffunct$ is equal to $\BCmod\otimes_\algB -$, 
it has right adjoint $\Radj$ equal to $\Hom_\algC(\algB, -)$.

\begin{lemma}\label{Lem:HomFaith}
The counit $\counitmap{M}\colon \Ffunct\Radj M\to M$ is injective with finite-dimensional cokernel.
Thus $\Radj\colon \CM \algC\to \CM\algB$ is faithful. 
\end{lemma}

\begin{proof}
Note that $\Radj$ does map $\CM \algC$ to $\CM\algB$,
because any $M\in \CM \algC$ is free as an $\ring$-module
and $\Hom_\algC(\algB, M)\leq\Hom_\ring(\algB, M)$, which is thus free as an $\ring$-module.

The counit $\counitmap{M}\colon \Ffunct\Radj M\to M$ can be identified with the map
\[
\Hom_\algC(\algB, M)\to \Hom_\algC(\algC, M)\cong M \colon f\mapsto f(1)
\]
induced by the inclusion $\algC\subset \algB$. 
This map is injective, because $\rnk{\ring}\algC=\rnk{\ring}\algB$, so $\algB/\algC$ is torsion,
while $M$ is torsion-free, so $\Hom_\algC(\algB/\algC, M)=0$.
Furthermore, as an immediate consequence of \cite[Lemma~5.5]{JKS2}, we have
\[ \rnk{\ring}\Hom_\algC(\algB, M)
= (\rnk{\algC} \algB) (\rnk{\algC} M)
= (\rnk{\algC} \algC) (\rnk{\algC} M) 
= \rnk{\ring}\Hom_\algC(\algC, M).
\]
Hence  $\Ffunct\Radj M$ is effectively a submodule  of $M$ 
with the same rank as $M$, so the quotient is finite-dimensional.
Thus $\counitmap{M}$ is an epimorphism in $\CM\algC$,
so $\Radj$ is faithful, by \Cref{Lem:adjprops}\itmref{itm:adjp1}.
\end{proof}

\begin{lemma}\label{Prop:CMFactoo}
If $M\in \CM\algC$ is a quotient of some $N\in \CM\algB$, then $M\in \CM\algB$.
\end{lemma}

\begin{proof}
For a module $M$ in $\CM\algC$ to be in $\CM\algB$, it is necessary and sufficient that the counit 
$\counitmap{M}\colon  \resfun\Radj M \to M$ is an isomorphism.
This is necessary, because, by \Cref{Lem:restriFF}, when $M\in \CM\algB$, 
\begin{equation}\label{eq:FRM=M}
  \Hom_\algC(\algB, M)
  = \Hom_\algB(\algB, M)
   \isom  M.
\end{equation}

So suppose that $M\in\CM\algC$ and $q\colon N\to M$ is a surjection with $N\in \CM\algB$.
The counits give a commutative square
\[
\begin{tikzpicture}[xscale=2,yscale=1.5] 
\draw (2,2) node (b2) {$\Ffunct \Radj N$};
\draw (3,2) node (b3) {$\Ffunct \Radj M$};
\draw (2,1) node (c2) {$N$};
\draw (3,1) node (c3) {$M$.};
\foreach \t/\h in {b2/b3, c2/c3} 
   \draw[cdarr] (\t) to (\h);
\draw[cdarr] (b3) to node [right]{$\counitmap{M}$} (c3);
\draw[cdarr] (b2) to node [left]{$\counitmap{N}$} (c2);
\draw[cdarr] (c2) to node [above]{$q$} (c3);
\end{tikzpicture}
\]
Now $\counitmap{N}$ is an isomorphism, because $N\in \CM\algB$,
so $\counitmap{M}$ is surjective, because $q$ is surjective.
But $\counitmap{M}$ is also injective, by \Cref{Lem:HomFaith},
so it is an isomorphism, that is, $M\in \CM\algB$, as required.
\end{proof}

\begin{remark}\label{rem:addBfactor}
Note that any $\algB$-module, considered as a $\algC$-module, 
is a quotient of some $\algB'\in \add \BCmod$,
namely its $\algB$-module projective cover, considered as a $\algC$-module.
\Cref{Prop:CMFactoo} shows that the converse is 
true for modules in $\CM\algC$, although it is not true for general $\algC$-modules
(cf. \cite[Ch.~37, Ex.~10]{CR}).
\end{remark}

Although restriction $\resfun\colon \CM\algB \to \CM\algC$ is necessarily exact,
so induces a map $\Ext^1_\algB(M,N)\to \Ext^1_\algC(M, N)$,
for $M, N\in \CM\algB$,
this map may not be an isomorphism,
i.e.~$\resfun$ may not be `full exact'.
However, we do have the following.

\begin{lemma} \label{lem:ExtBinExtC}
For any $M, N\in \CM\algB$, we have
$\Ext^1_\algB(M,N)\subset \Ext^1_\algC(M, N)$.
Thus $\Ext^1_\algB(M,N)$ is finite-dimensional.
\end{lemma}

\begin{proof}
Any extension in $\CM\algB$ is an extension in $\CM\algC$. 
However, as restriction is fully faithful, 
if an extension splits in $\CM\algC$, then it also splits in $\CM\algB$. 
Hence we have the claimed inclusion.
On the other hand, $\Ext^1_\algC(M, N)$ is always finite-dimensional, by \cite[Cor.~4.6]{JKS1},
so the second part follows from the first.
\end{proof}

To see when the inclusion in \Cref{lem:ExtBinExtC} is an equality, we have the following.

\begin{lemma} \label{lem:ExtB=ExtC}
For $M, N\in \CM\algB$, if $\Ext^1_\algC(\algB, N)=0$, 
then $\Ext^1_\algB(M,N)= \Ext^1_\algC(M, N)$.
\end{lemma}

\begin{proof}
We need to prove that, if we have an extension in $\CM\algC$ of the form
\begin{equation}\label{eq:NXMseq}
 \ShExSeq {N} {X} {M},
% 0 \lra N \lra \algB' \lra M \lra 0
\end{equation}
then $X\in\CM\algB$.
Applying the counit $\counitmap{}\colon  \resfun\Radj  \to \idfun$ to this extension
gives the following commutative diagram 
\[
\begin{tikzpicture}[xscale=2,yscale=1.5] 
\pgfmathsetmacro{\eps}{0.2}
\draw (1,2) node(b1) {$\Ffunct \Radj N$};
\draw (2,2) node(b2) {$\Ffunct \Radj X$};
\draw (3,2) node(b3) {$\Ffunct \Radj M$};
\draw (1,1) node (c1) {$N$};
\draw (2,1) node (c2) {$X$};
\draw (3,1) node (c3) {$M$};
\draw (0+\eps,1) node (c0) {$0$};
\draw (4-\eps,1) node (c4) {$0$,};
\draw (0+\eps,2) node (b0) {$0$};
\draw (4-\eps,2) node (b4) {$0$};
\foreach \t/\h in {b1/b2, b2/b3, c1/c2, c2/c3} 
   \draw[cdarr] (\t) to (\h);
\draw[cdarr] (b3) to node [auto]{$\counitmap{M}$}(c3);
\draw[cdarr] (b2) to node [auto]{$\counitmap{X}$} (c2);
\draw[cdarr] (b1) to node [auto]{$\counitmap{N}$} (c1);
\foreach \t/\h in {b0/b1, b3/b4, c0/c1, c3/c4} 
   \draw[cdarr] (\t) to (\h);
\end{tikzpicture}
\]
where the top sequence is exact, because $\Ext^1_\algC(B, N)=0$.
But $\counitmap{M}$ and $\counitmap{N}$ are isomorphisms, 
because $M, N\in \CM\algB$, and so $\counitmap{X}$ is an isomorphism,
that is, $X\in\CM\algB$, as required.
\end{proof}

\begin{remark}\label{rem:no-converse}
Conversely, if $\Ext^1_\algC(\algB, N)\neq 0$, then $M=\algB$ provides an example where
the inclusion in \Cref{lem:ExtBinExtC} is strict.
More concretely, for $(k,n)=(2,5)$, there is a short exact sequence  
\begin{equation}\label{eq:extinCnotB}
  0\to M_{25}\to M_{12}\oplus M_{45}\to M_{14}\to 0
\end{equation} 
in $\CM\algC$ and an algebra $\algB$, as in \eqref{eq:algB}, for which $M_{45}$ is not a $\algB$-module,
but all other rank 1 modules are. In this case, $M_{14}$ is a projective $\algB$-module.
\end{remark}

% ======================================================================
\subsection{The category $\GP\algB$} \label{subsec:GPB}
% ======================================================================

We will now assume that $\algB$ is rigid as a left $\algC$-module. 
By a temporary abuse of notation (since we don't yet know that $B$ is Iwanaga--Gorenstein), 
we define the full subcategory
\begin{equation}\label{def:GPB}
 \GP\algB = \{ M \in \CM\algB  \st \Ext^i_\algB(M, \algB)=0,\,  \forall i>0\}. 
\end{equation}
Note that \emph{a priori} a Gorenstein projective $\algB$-module 
satisfies this Ext-vanishing condition, but is not necessarily in $\CM\algB$.
However, we will show (in \Cref{Thm:GPB3}) that actually it is.
Observe that $\GP\algB$ is non-trivial, since $\add \algB\subset \GP\algB$,
and it is an extension-closed subcategory of $\CM\algB$. 
While $\CM\algB$ may not be an extension-closed subcategory of $\CM\algC$,
we do have the following. 

\begin{proposition} \label{Thm:GPB1}
If $M\in\CM\algB$ and $N\in \GP\algB$, then $\Ext^1_\algB(M, N)=\Ext^1_\algC(M, N)$. 
In particular, $\GP\algB$ is an extension-closed subcategory of $\CM\algC$.
\end{proposition}

\begin{proof}
Since we are assuming that $\algB$ is rigid as a $\algC$-module, we can apply \Cref{lem:ExtB=ExtC}
in the case $N=B$, to deduce that, for any $M\in\CM\algB$,
\begin{equation}\label{eq:Exts=}
  \Ext^1_\algB(M, \algB) = \Ext^1_\algC(M, \algB) = \Ext^1_\algC(\algB, M),
\end{equation}
where the second equality follows by the 2-CY property of $\CM\algC$.

But, if $N\in \GP\algB$, then \eqref{eq:Exts=} implies that $\Ext^1_\algC(\algB, N)=0$,
so we can apply \Cref{lem:ExtB=ExtC} again
to get the required equality.

In particular, if both $M$ and $N$ are in $\GP\algB$, 
then the fact that $X$ in \eqref{eq:NXMseq} is in $\CM\algB$ means that it is actually in $\GP\algB$, 
using the long exact sequence for $\Ext$.
\end{proof}

\Cref{Thm:GPB1} shows that, in the example in \eqref{eq:extinCnotB}, the module $M_{25}$ is not in $\GP\algB$,
although it is in $\CM\algB$.
The following result provides a simpler description of $\GP\algB$. 

\begin{proposition}\label{Thm:GPB2} 
$\GP\algB = \{ M\in \CM\algB\st \Ext^1_\algB(M, \algB)=0\}$.
\end{proposition}

Note that, the condition here could equally be $\Ext^1_\algC(M, \algB)=0$, by \Cref{Thm:GPB1}, 
or $\Ext^1_\algC(\algB,M)=0$, by the 2-CY property of $\CM\algC$.

\begin{proof} 
Clearly any $M\in \GP\algB$, as defined in \eqref{def:GPB},
satisfies this weaker condition.

For the converse, if $M\in \CM\algB$ satisfies $\Ext^1_\algB(M, \algB)=0$,
then also $\Ext^1_\algC(M, \algB)=0$, by \Cref{Thm:GPB1}, since $\algB\in\GP\algB$.
Consider a partial presentation of $M$ in $\CM\algB$, 
\[
  0 \lra \Syz M \lra \algB' \lraa{f} M \lra 0,
\]
which is also an $\add\BCmod$-approximation in $\CM\algC$,
as restriction is fully faithful.
Applying $\Hom_\algC(\algB, -)$ gives
\[
\Hom_\algC(\algB, \algB') \lraa{f_*} \Hom_\algC(\algB, M)
 \lra \Ext^1_\algC(\algB, \Syz M) \lra \Ext^1_\algC(\algB, \algB'),
\]
where $f_*$ is surjective, because $f$ is an $\add\BCmod$-approximation,
and $\Ext^1_\algC(\algB, \algB')=0$, because $\algB$ is rigid by assumption. 
Hence $\Ext^1_\algC(\algB, \Syz M)=0$, and so,  as $\CM\algC$ is 2-CY,
$\Ext^1_\algC(\Syz M, \algB)=0$.  
Thus, e.g.~by \Cref{Thm:GPB1},
\[
  \Ext^2_\algB(M, \algB)= \Ext^1_\algB(\Syz M, \algB)=0.
\]
In particular, $\Syz M$ also satisfies the right-hand side condition,
so we can proceed inductively to show that 
$\Ext^i_\algB(M, \algB)=\Ext^{i-1}_\algB(\Syz M, \algB)=0$, for all $i>0$,
and so $M\in \GP\algB$, as required.
\end{proof}

\begin{remark}\label{rem:CMQCB}
Combining \Cref{rem:addBfactor} %\Cref{Prop:CMFac} 
and \Cref{Thm:GPB2} yields a characterisation of $\GP\algB$ within $\CM\algC$, using just $\BCmod$. 
More precisely $M\in \CM\algC$ is in $\GP\algB$
if and only if $M$ is a quotient of some $\algB'\in\add\BCmod$ and $\Ext^1_C(B,M)=0$.
\end{remark}

\begin{lemma} \label{lem:GPBBapp} 
For $M\in \CM\algC$, let $K$ be the kernel of an $\add\BCmod$-approximation of $M$ in $\CM\algC$.  
Then $K\in\GP\algB$.
In particular, if $M\in \CM\algB$, then any syzygy $\Syz_\algB M\in\GP\algB$.
\end{lemma}

\begin{proof}
Let $M\in \CM\algC$ have an $\add\BCmod$ approximation
\[
 0 \lra K \lra \algB' \lraa{f} M,
\]
which is also an $\add\BCmod$ approximation of $\img f$ in $\CM\algC$. 
By \Cref{Prop:CMFactoo}, we have $\img f\in \CM\algB$, so also $K\in \CM\algB$.

As in the proof of \Cref{Thm:GPB2}, 
we deduce that $\Ext^1_\algC(\algB, K)=0$, because $\BCmod$ is rigid, and so 
$\Ext^1_\algB(K, \algB)=0$, because $\CM\algC$ is 2-CY.
Hence $K\in \GP\algB$, by \Cref{Thm:GPB2}.

For the last part, note again that a partial presentation of $M$ in $\CM\algB$
is also an $\add\BCmod$-approximation in $\CM\algC$, as restriction is fully faithful.
\end{proof}

\begin{lemma} \label{Lem:injdim}
%Assume that $\algB$ is a rigid $\algC$-module.
If $M$ is a finitely-generated $\algB$-module, then $\Ext^i_\algB(M,\algB)=0$ for all $i> 2$.
Thus $\injdim {}_\algB\algB\leq 2$.
\end{lemma}

\begin{proof}
Note that the first syzygy $\Syz M\in \CM\algB$, 
so, by \Cref{lem:GPBBapp}, the second syzygy $\Syz^2 M\in\GP\algB$.
Hence, for $i>2$, we have 
$\Ext^i_\algB(M, \algB)=\Ext^{i-2}_\algB(\Syz^2 M, \algB)=0$,
as required.
\end{proof}

Now we can see that $\GP\algB$, as defined in \eqref{def:GPB}, 
is indeed the category of Gorenstein-projective $\algB$-modules,
by temporarily making an \emph{a priori} stronger rigidity assumption.
However, we will see in \Cref{cor:LRrigid} that it is actually 
sufficient to  just assume that $B$ is rigid as a left $C$-module.

\begin{theorem}\label{Thm:GPB3}
Let $\algB$ be an $\ring$-order, as in \eqref{eq:CinB},
which is rigid as both a left and right $\algC$-module.
Then the following hold.
%Assume further that $\algB$ be rigid as a right $\algC$-module.
\begin{enumerate}
\item\label{itm:GP3-1}
 $\algB$ is Iwanaga--Gorenstein of injective dimension at most $2$. 
\item\label{itm:GP3-2}
 $\GP\algB$ is the category of Gorenstein projective $\algB$-modules.
\item\label{itm:GP3-3}
 $\GP\algB$ is a Frobenius $2$-CY category.
\end{enumerate}
\end{theorem}

\begin{proof}
(1) By \Cref{Lem:injdim}, we have $\injdim {}_\algB\algB\leq 2$.
Since $B$ is also rigid as a right $C$-module, the arguments in \Cref{lem:GPBBapp}
and \Cref{Lem:injdim} apply on the right and show that also $\injdim \algB_\algB\leq 2$. 
Thus $\algB$ is Iwanaga--Gorenstein. 

(2) Let $M$ be a Gorenstein projective $\algB$-module, that is, $\Ext^i_\algB(M, \algB)=0$ for all $i>0$. 
As $\algB$ is Iwanaga--Gorenstein, the category of Gorenstein projective $\algB$-modules is Frobenius
\cite{Bu} (cf.~\cite[Cor~3.7]{JKS1}).
Hence there is an embedding $M\hookrightarrow \algB'$ with $\algB'\in \add\algB$,
so $M\in \CM \algB$.
Thus $M$ satisfies \eqref{def:GPB}, so $M\in \GP\algB$. 
Thus $\GP\algB$ is after all the category of Gorenstein projective $\algB$-modules
and so it is Frobenius. 

(3) By \Cref{Thm:GPB1} and the fact $\CM\algC$ is 2-CY, 
we have for any $M, N\in \GP\algB$, 
\[ \Ext^1_B(M, N)\isom  \Ext^1_B(N, M).\]
Therefore $\GP \algB$ is also 2-CY. 
\end{proof}
 
\begin{remark}\label{rem:cluster-structure}
It is known (see \cite[Prop 4.2]{JKS2}, based on \cite[Prop 8.2]{GLS08})
that $\CM\algC$ has a cluster structure in the sense of \cite[\S II.1]{BIRS}.
In particular, mutation of cluster tilting objects reflects Fomin-Zelevinsky mutation of quivers.
Alternatively, the Gabriel quiver of any cluster tilting object has no loops or 2-cycles
at mutable vertices.

In fact, it can be shown that $\GP\algB$ is always functorially finite and admits a cluster structure
(see \cite{JRS}). 
\end{remark}

% ======================================================================
\subsection{Algebras from Grassmann necklaces}\label{subsec:necklace}
% ======================================================================

%\newcommand{\cyc}{\mathrm{cyc}}
%\newcommand{\cycint}[2]{\hlite{[#1, #2]^{\cyc}}}
\newcommand{\cycint}[2]{[#1, #2]}
\newcommand{\cycCOint}[2]{[#1, #2)}
\newcommand{\cycOCint}[2]{(#1, #2]}

We shall now see that all examples of the algebras~$\algB$, as in \S\ref{subsec:GPB}, 
are constructed from Grassmann necklaces.

Recall, from \cite[Def. 4.1]{OPS} or \cite[Def.~16.1]{Pos}, 
that a sequence $I_1, \dots, I_n$ 
in $\labsubset{n}{k}$ is a \emph{Grassmann necklace} 
if, for all $i\in\labset{n}$,
\begin{equation}\label{eq:neck-cond}
I_i\setminus \{i\}\subset I_{i+1},
\end{equation} 
%$I_i\setminus \{i\}\subset I_{i+1}$, for all $i$, 
which, in particular, implies that $I_{i+1}=I_i$ when $i\not\in I_i$.
Here $\labset{n}$ is cyclically ordered and
thus identified with $\ZZ_n$ as an additive group.
Furthermore, intervals such as $\cycint{i}{j}$ %and $\cycCOint{i}{j}$
are interpreted cyclically, that is, $\cycint{i}{j}=\{i, i+1, \dots, j\}$,
even if $j<i$.

For any sequence $\neckI=(I_i)_{i\in\labset{n}}$,
there is a unique $\algC$-module $\algB_\neckI$ with $\algC\subset \algB_\neckI\subset \algC[t^{-1}]$
such that $e_{\ell}\in\algC$ is an $\ring$-module generator of $e_{\ell}(\algB_\neckI)e_{\ell}$ and 
\begin{equation}\label{eq:def-B_I}
 (\algB_\neckI)e_{\ell} \isom M_{I_{\ell+1}}.
%\algB_\neckI\isom \textstyle\bigoplus_i M_{I_i},
\end{equation}
Then $\algB_\neckI$ is a subalgebra of $\algC[t^{-1}]$ precisely when
\begin{equation}\label{eq:alg-cond}
 I_i \leq_i I_j \quad\text{for all $i,j\in\labset{n}$,}
\end{equation}
where $\leq_i$ is lex-order with $\labset{n}$ reordered to start at $i$,
as in \cite[Def.~4.2]{OPS}.

\begin{remark}\label{rem:neck-lab}
In \eqref{eq:neck-cond}, the terms in a necklace are indexed by the set $\labset{n}$
that indexes the arrow pairs in the quiver \eqref{eq:circle-quiv}.
However, the terms are more naturally indexed by the vertices of the quiver 
and (unfortunately) it is the term $I_{\ell+1}$ that is associated to the vertex $\ell$ 
in the labelling of \eqref{eq:circle-quiv}.
This explains the apparent mismatch in \eqref{eq:def-B_I}.
\end{remark}

By \cite[Lemma~4.4]{OPS}, the algebra condition \eqref{eq:alg-cond}
holds when $\neckI=(I_i)$ is a necklace,
so, in that case, we will call $\algB_\neckI$ a \emph{necklace algebra}.
However, \cite[Lemma~4.4]{OPS} also shows that, 
when $\neckI$ is a necklace, the collection $\{I_i\}$
is `weakly separated', as in \cite{LZ} or \cite[Def~3.1]{OPS}, 
or `non-crossing', as in \cite[Def~3]{Sc} or \cite[Def~5.5]{JKS1}.
In other words, by \cite[Prop.~5.6]{JKS1},
$\algB_\neckI$ is rigid as a left $\algC$-module.

Note that the necklace condition \eqref{eq:neck-cond} is stronger than the algebra condition \eqref{eq:alg-cond}.
For example, for $(k,n)=(2,5)$, the sequence $(I_i) = (13,24,34,45,51)$
satisfies \eqref{eq:alg-cond}, but is not weakly separated.
On the other hand, as we will shortly see, for 
$\algB_\neckI$ to be an algebra that is rigid as a left $\algC$-module
amounts to precisely the necklace condition.

\begin{lemma}\label{lem:necklacecond}
Suppose $I, J\in \labsubset{n}{k}$ and $i\neq j\in \labset{n}$. 
If $I$, $J$ are weakly separated, with $I\leq_i J$ and $J\leq_j I$, 
then $J\setminus I\subset \cycCOint{j}{i}$ and $I\setminus J\subset \cycCOint{i}{j}$.
\end{lemma}

Note that the converse is straightforward.

\begin{proof} 
If $I=J$, then $I\setminus J$ and $J\setminus I$ are both empty and so the result is immediate. 
Now assume that $I\neq J$. 
As $I$ and $J$ are weakly separated, there exist (disjoint) cyclic intervals 
$\cycint{a}{b}$ and $\cycint{c}{d}$ such that 
\begin{enumerate}
\item $I\setminus J\subset \cycint{a}{b}$ with $a, b\in I\setminus J$,
\item  $J\setminus I\subset \cycint{c}{d}$ with $c, d\in J\setminus I$.
\end{enumerate}
As $I <_i J$, we must have $d<i\leq a$ in the cyclic order, otherwise either
\begin{itemize}
\item[(a)] $a \in I\setminus J$, but $y<_i a$, for any $y\in J\setminus I$, or
\item[(b)] $d \in J\setminus I$, but $d<_i x$, for any $x\in  I\setminus J$.
\end{itemize}
Either case  contradicts $I<_iJ$. 
Similarly, $b<j\leq c$ in the cyclic order. 
Therefore, $I\setminus J\subset \cycCOint{i}{j}$ and 
$J\setminus I\subset \cycCOint{j}{i}$, as claimed.
\end{proof}

\begin{lemma}\label{prop:necklacecond1}
The following are equivalent.
\begin{enumerate}
\item\label{itm:neck1}
The sequence $(I_i)$ is a necklace.
\item\label{itm:neck2}
For all $i, j$, we have $I_i\setminus I_j\subset \cycCOint{i}{j}$.
\item\label{itm:neck3}
For all $i, j$, we have $I_i\leq_i I_j$ 
and $I_i$ and $I_j$ are weakly separated.
\end{enumerate}
\end{lemma}

\begin{proof}
$\itmref{itm:neck1} \Rightarrow \itmref{itm:neck2}$
is the main content of the proof of \cite[Lem 4.4]{OPS}, 
which is the statement that $\itmref{itm:neck1} \Rightarrow \itmref{itm:neck3}$, while
$\itmref{itm:neck2} \Rightarrow \itmref{itm:neck3}$ 
is a straightforward conclusion, since the assumption also gives
$I_j\setminus I_i\subset \cycCOint{j}{i}$.

$\itmref{itm:neck2} \Rightarrow \itmref{itm:neck1}$ 
is a specialisation, as $I_{i}\setminus I_{i+1}\subset \cycCOint{i}{i+1}=\{i\}$
if and only if $I_{i}\setminus \{i\}\subset I_{i+1}$,
while $\itmref{itm:neck3} \Rightarrow \itmref{itm:neck2}$ is a direct application of \Cref{lem:necklacecond}.
\end{proof}

\begin{proposition}\label{cor:neck-alg}
Let $B$ be a (finitely-generated) $\ring$-order with $\algC\subset \algB\subset \algC[t^{-1}]$.
Then $\algB$ is rigid as a left $\algC$-module if and only if $\algB$ is a necklace algebra. 
\end{proposition}

\begin{proof}
This is essentially the content of $\itmref{itm:neck1} \Leftrightarrow \itmref{itm:neck3}$
in \Cref{prop:necklacecond1}.
\end{proof}

Let $\sigma$ be a decorated permutation \cite[Def.~13.3]{Pos}.
Denote by $\sigma^{-1}$ the inverse decorated permutation,
which has the colours of the fixed points swapped. 
Decorated permutations can be represented graphically by chord diagrams \cite[\S16]{Pos},
with chords $(i, \sigma(i))$, whose orientation corresponds to the colour when $\sigma(i)=i$.
There is a bijection between decorated permutations
and necklaces \cite[Lem.~16,2]{Pos},
so we can talk about `the necklace defined by $\sigma$'.

Recall from \cite[\S17]{Pos} that an ordered pair of chords $(i, \sigma(i))$, $(j, \sigma(j))$ with $i\neq j$
is a \emph{crossing} if $j\in \cycint{\sigma(i)}{i}$ and $\sigma(j)\in \cycint{i}{\sigma(i)}$
(see \Cref{fig:crossing}).  
The crossing is said to be \emph{simple} if there is no other chord $(l, \sigma(l))$ 
such that $l\in \cycint{j}{i}$ and $\sigma(l)\in \cycint{\sigma(j)}{\sigma(i)}$. 
In this case, the pair of chords $(i, \sigma(j))$ and $(j, \sigma(i))$ 
is called a \emph{simple alignment}.
Note that, if $\sigma(i)=j$ and $\sigma(j)=i$,
then this can be interpreted as a crossing in two ways.

%==================================
\begin{figure}[h]
%==================================
\begin{tikzpicture} [scale=1, 
 chord/.style={thick, teal, -latex},
 bdry/.style={thick, blue}]
\pgfmathsetmacro{\cRad}{1}
\pgfmathsetmacro{\shRad}{0.45*\cRad}
\pgfmathsetmacro{\labRad}{1.2}
\pgfmathsetmacro{\BlabRad}{1.4}
\pgfmathsetmacro{\Vstep}{-0.5}
\pgfmathsetmacro{\Hstep}{2.3}
\pgfmathsetmacro{\Hextra}{0.7}
%\pgfmathsetmacro{\Pstep}{0.35}
%===============
\begin{scope} [shift={(0,0)}]
\draw [chord] (-135:\cRad) to (45:\cRad);
\draw [chord] (-45:\cRad) to (135:\cRad);
\draw [bdry] (0,0) circle (\cRad);
\draw (-45:\labRad) node {\small $j$};
\draw (-135:\labRad) node {\small $i$};
\draw (135:\BlabRad) node {\small $\sigma(j)$};
\draw (45:\BlabRad) node {\small $\sigma(i)$};
\end{scope}
%===============
\begin{scope} [shift={(1*\Hstep,\Vstep)}]
\draw [chord] (-135:\cRad) to [out=45,in=-45] (135:\cRad);
\draw [chord] (-45:\cRad) to [out=135,in=-135] (45:\cRad);
\draw [bdry] (0,0) circle (\cRad);
\end{scope}
%===============
\begin{scope} [shift={(2*\Hstep+\Hextra,0)}]
\draw [chord] (-135:\cRad) to [out=45,in=135] (0:\cRad);
\draw [chord] (0:\cRad) to [out=-135,in=-45] (135:\cRad);
\draw [bdry] (0,0) circle (\cRad);
\end{scope}
%===============
\begin{scope} [shift={(3*\Hstep+\Hextra,\Vstep)}]
\draw [chord] (-135:\cRad) to [out=45,in=-45] (135:\cRad);
\draw [chord] (0:\cRad) to [out=-135,in=-90] (0:\shRad) to [out=90,in=135] (0:\cRad);
\draw [bdry] (0,0) circle (\cRad);
\end{scope}
%===============
\begin{scope} [shift={(4*\Hstep+2*\Hextra,0)}]
\draw [chord] (180:\cRad) to [out=-45,in=135] (0:\cRad);
\draw [chord] (0:\cRad) to [out=-135,in=45] (180:\cRad);
\draw [bdry] (0,0) circle (\cRad);
\end{scope}
%===============
\begin{scope} [shift={(5*\Hstep+2*\Hextra,\Vstep)}]
\draw [chord] (180:\cRad) to [out=-45,in=-90] (180:\shRad) to [out=90,in=45] (180:\cRad);
\draw [chord] (0:\cRad) to [out=-135,in=-90] (0:\shRad) to [out=90,in=135] (0:\cRad);
\draw [bdry] (0,0) circle (\cRad);
\end{scope}
%===============
\end{tikzpicture}
%==================================
\caption{Crossings and their corresponding alignments}
\label{fig:crossing}
\end{figure}
%==================================

Recall that, for the algebra $C=C(k, n)$ as in \eqref{eq:algC}, 
a rank one (left) $C$-module $M_I$ is indexed by the $k$-set $I$ consisting 
of the indices of the $x$-arrows that act as isomorphisms.
From the definition, we immediately see that 
%Following the definition of the algebra, we have 
\[
C(k, n)\op\cong C(n-k, n),
\]
where the isomorphism fixes the idempotents $e_i$ and swaps the $x$- and $y$-arrows. 
Note that a right $C$-module is a left $C\op$-module. 
So for right $C(k,n)$-modules we can use the indexing convention of left
$C(n-k,n)$-modules.  
That is, we denote a rank one right $C$-module by $M^J$, 
where $J$ is the set of the indices of $y$-arrows that act as isomorphisms. 
For example, we have 
\begin{equation} \label{eq:projs}
Ce_i=M_{\cycint{i+1}{i+k}} \quadand e_iC=M^{\cycint{i+1}{i+n-k}}.
\end{equation}

\begin{example}\label{ex:Corder}
Recall from \cite[\S3]{JKS1} that $C[t^{-1}]\cong M_{n}(\ring[t^{-1}])$,
the $n\times n$ matrix algebra over $\ring[t^{-1}]$, in such a way that 
the idempotent $e_{i-1}$ maps to the elementary matrix at $(i,i)$.
However, this identification requires some choices
(effectively the choice of $x$ and $y$).
For example, for $(k,n)=(2,5)$, 
\[
C=\begin{pmatrix}
R& tR  & t^2R & t^2R & tR \\
R& R  &  tR    & tR & tR \\
R& R  & R     & R &  R\\
R& tR & tR    & R &  R \\
R& tR & t^2R & tR  & R
\end{pmatrix}
\text{with }
x=\begin{pmatrix}
0&0 &0 &0 &t\\
1&0 &0 &0 &0\\
0&1 &0 &0 &0\\
0&0 &t &0 &0\\
0&0 &0 &t &0\\
\end{pmatrix}
\text{and }
y=\begin{pmatrix}
0&t &0 &0 &0\\
0&0 &t &0 &0\\
0&0 &0 &1 &0\\
0&0 &0 &0 &1\\
1&0 &0 &0 &0\\
\end{pmatrix}.
\]
By considering the left action of $x$ and the right action of $y$ on $C$, 
we can check directly that the index sets of the column modules %$Ce_i$ 
are $12, 23, 34, 45, 51$ and the index sets of the row modules %$e_iC$ 
are $123, 234, 345, 451, 512$, confirming \eqref{eq:projs}.
For example, to see that the third column module $Ce_2=M_{34}$,
observe that multiplication by $x$ is an isomorphism
from $e_2Ce_2=R$ to $e_3Ce_2=tR$ 
and from $e_3Ce_2=tR$ to $e_4Ce_2=t^2R$.
 %\comment{R: more detail for one of the two cases}
\end{example}

Let $B$ be an $\ring$-order with $\algC\subset \algB\subset \algC[t^{-1}]$.  
If $C_i$ is the index $k$-set of the left $C$-module $Be_{i-1}$,
then $\BCmod=B_\neckI$ with $\neckI=\{C_1, \dots, C_n\}$. 
Similarly, if $R_i$ is the index $(n-k)$-set of the right $C$-module $e_{i-1}B$
and $\neckJ=\{R_1, \dots, R_n\}$,
then we will write $B_C=B^\neckJ$.
When $\neckJ$ is a necklace, we call $B$ a \emph{right necklace algebra}.

\begin{proposition}\label{prop:LRrigid}
Let $B$ be a (left) necklace algebra with $\BCmod=B_\neckI$, 
for $\neckI$ defined by a decorated permutation $\sigma$. 
Then $B$ is also a right necklace algebra with $B_C=B^\neckJ$, 
for $\neckJ$ defined by $\sigma^{-1}$. 
\end{proposition}

\begin{proof}
If $B=C$, then, by \eqref{eq:projs},
$\neckI$ is the necklace defined by $\pi_{k, n}: i\mapsto i+k$
and $\neckJ$ is the necklace defined by $\pi_{n-k, n}=\pi^{-1}_{k, n}$,
so the result holds in this case. 
For ease of exposition, we choose an identification $C[t^{-1}]\cong M_{n}(\ring[t^{-1}])$, as in \Cref{ex:Corder},
so that $C$ and $B$ are explicit matrix algebras.
In particular, the natural indexing of rows and columns of a matrix by $1,\ldots,n$ 
matches the traditional labelling of the terms in a necklace,
avoiding the issue from \Cref{rem:neck-lab}.

By \cite[Thm 17.8]{Pos}, 
any necklace of type $(k, n)$ can be obtained from 
the necklace defined by $\pi_{k,n}$ by a sequence of operations replacing a 
simple crossing in a decorated permutation $\sigma$ by the corresponding simple alignment,
with suitable orientations of loops when they appear (see \Cref{fig:crossing}).
So, to complete the proof, it suffices to show that statement in the proposition remains true after 
each such replacement. 

Suppose the statement holds for some $B$,
that is, $\BCmod=B_\neckI$ for $\neckI=\{C_1, \dots, C_n\}$, the necklace given by a decorated permutation $\sigma$, 
and $B_C=B^\neckJ$ for $\neckJ=\{R_1, \dots, R_n\}$, the necklace given by $\sigma^{-1}$.
Suppose that the ordered pair of chords $(i, \sigma(i))$, $(j, \sigma(j))$ is any crossing,
although it is sufficient for the induction to only consider simple ones.  
From the properties of necklaces \cite[\S16]{Pos},
%and simple crossings \cite[\S16,17]{Pos}, 
we have 
\begin{enumerate}
\item 
  $\sigma(j)\in C_s$, for $s\in \cycOCint{j}{\sigma(j)}$,
  and $\sigma(i)\not\in C_s$, for $s\in \cycOCint{\sigma(i)}{i}$,
\item 
 $j\in R_s$, for $s\in \cycOCint{\sigma(j)}{j}$,
 and $i\not\in R_s$, for $s\in \cycOCint{i}{\sigma(i)}$.
\end{enumerate}
Let $X$ be the submatrix with rows $\cycOCint{\sigma(j)}{\sigma(i)}$ and columns $\cycOCint{j}{i}$. 
Let $B'$ be obtained from $B$ by replacing $X$ by $t^{-1}X$
and let $R_s'$ and $C_s'$ be the index sets of the row and column modules of $B'$. 
Then 
\begin{enumerate}
\item 
$C'_s=(C_s\setminus\{\sigma(j)\})\cup \{\sigma(i)\}$, for $s\in \cycOCint{j}{i}$,
\item
$R'_s=(R_s\setminus\{j\})\cup\{i\}$, for $s\in \cycOCint{\sigma(j)}{\sigma(i)}$,
\end{enumerate}
while the other column/row modules remain unchanged. 
Thus $\BCmod'$ comes from the necklace defined by the (decorated) permutation $\sigma'$ obtained from $\sigma$ 
with the images of $i$ and $j$ swapped, that is, 
with the crossing replaced by the corresponding alignment,
%$\sigma'(j)=\sigma(i)$, $\sigma'(i)=\sigma(j)$ and $\sigma'(s)=\sigma(s)$ for $s\neqi, j$; 
and $B'_C$ comes from the necklace defined by $(\sigma')^{-1}$. 
This completes the proof.
\end{proof}

\begin{remark}\label{rem:indexing}
Despite appearances, the construction in the proof of \Cref{prop:LRrigid}
%is geometrically natural, i.e.~symmetric under the dihedral group not just the cyclic group.
has full dihedral symmetry, not just cyclic symmetry,
when the terms in the necklace are indexed by the quiver vertices, as in \Cref{rem:neck-lab}.
More precisely, $B'$ is obtained from $B\subset C[t^{-1}]$ by replacing 
$e_\ell B e_m$ by $t^{-1}(e_\ell B e_m)$
where $\ell$ is any vertex between the arrow pairs indexed by $\sigma(j)$ and $\sigma(i)$, 
in the quiver \eqref{eq:circle-quiv}
and $m$ is any vertex between the arrow pairs indexed by $j$ and $i$.

Note that the construction also gives an inductive way to write $e_\ell \algB e_m=t^{-b_{\ell m}}(e_\ell \algC e_m)$,
for some $b_{\ell m}\geq 0$, with $b_{\ell m}=0$ when $\ell=m$.
It would be interesting to find precisely what constraints the $b_{\ell m}$ satisfy 
and explicitly how they are determined by $\sigma$.
In particular, \Cref{prop:LRrigid} should then be transparent.
\end{remark}

\begin{corollary}\label{cor:LRrigid}
Let $B$ be an $\ring$-order with $C\subseteq B\subseteq C[t^{-1}]$.
Then $B$ is rigid as a left $C$-module
if and only if it is rigid as a right $C$-module.
\end{corollary}

\begin{proof}
By \Cref{cor:neck-alg}, if $B$ is rigid as a left $C$-module, then $B$ is a (left) necklace algebra. 
By \Cref{prop:LRrigid}, $B$ is also a right necklace algebra and hence rigid as a right $C$-module,
again by \Cref{cor:neck-alg}.
The converse holds similarly.
\end{proof}

As a consequence,  the assumption of \Cref{Thm:GPB3}
can be weakened to $B$ being rigid on just one side.

% ======================================================================
\subsection{Syzygies, cosyzygies and stable Hom}\label{subsec:syz-cosyz}
% ======================================================================

For any $M\in \CM\algB$, we will denote by $\sHom_\algB(T, M)$ 
the space of stable homomorphisms, that is,
the space of homomorphisms modulo those that factor through $\add\algB$.
As earlier, the syzygy $\Syz M$ is the kernel in a partial presentation of $M$,
that is, a short exact sequence
\[
 \ShExSeq{\Syz M}{\procov{M}}{M},
\]
where $\procov{M}\in \add \algB$ is projective.
This presentation may not be minimal, so $\Syz M$ is only defined up to projective summands.

On the other hand, for $M\in\GP\algB$, we denote by $\coSyz M\in\GP\algB$,
the cokernel in a `partial copresentation' of $M$, that is, a short exact sequence
\[
 \ShExSeq{M}{\procov{\coSyz M}}{\coSyz M},
\]
where $\procov{\coSyz M}$ is projective, and hence also injective, in $\GP\algB$.
Thus, when $M\in\GP\algB$, we can take this as a partial presentation of $\coSyz M$,
so that $\Syz\coSyz M=M$.
Note that $\coSyz M$ may not be the coszygy of $M$ in $\CM\algB$, because $\procov{\coSyz M}$
may not be injective in $\CM\algB$.

\begin{lemma} \label{Lem:ExtSHom}
Let $X\in \GP\algB$ and let $M\in \CM\algB$.
\begin{enumerate}
\item\label{itm:ESH1}
There is a commutative diagram
\begin{equation}\label{eq:syzcosyz}
\begin{tikzpicture} [xscale=2.0, yscale=1.5, baseline=(bb.base)]
\pgfmathsetmacro{\eps}{0.2}
\coordinate (bb) at (0,0.5);
\draw (-1+\eps,1) node (b0) {$0$};
\draw (0,1) node (b1) {$\Syz M$};
\draw (1,1) node (b2) {$\algB'$};
\draw (2,1) node (b3) {$ \coSyz \Syz M$};
\draw (3-\eps,1) node (b4) {$0$};
\draw (-1+\eps,0) node (c0) {$0$};
\draw (0,0) node (c1) {$\Syz M$};
\draw (1,0) node (c2) {$\procov{M}$};
\draw (2,0) node (c3) {$M$};
\draw (3-\eps,0) node (c4) {$0$};
\foreach \t/\h in {b0/b1, b1/b2, b2/b3, b3/b4, c0/c1, c1/c2, c2/c3, c3/c4, b2/c2} 
   \draw[cdarr] (\t) to (\h);
\draw[equals] (b1) to (c1);
\draw[cdarr] (b3) to node [right]{ $f$} (c3);
\end{tikzpicture}
\end{equation}
of partial presentations such that $\coSyz\Syz M \in \GP\algB$.
\item\label{itm:ESH2}
There are isomorphisms
\begin{align*}
  \Ext^i_\algB(X,f) & \colon \Ext^i_\algB(X,\coSyz \Syz M)\to \Ext^i_\algB(X,M),\quad\text{for $i>0$, and} \\
  \sHom_\algB(X,f) &\colon \sHom_\algB(X,\coSyz \Syz M)\to \sHom_\algB(X,M).
\end{align*}
\item\label{itm:ESH3}
Moreover, $\coSyz$ induces an isomorphism
$\sEnd_\algB X\isom \sEnd_\algB (\coSyz X)$, such that 
\[
  \Ext^1_\algB(X, \Syz M) \isom \sHom_\algB(X, M)\isom \Ext^1_\algB(\coSyz X, M),
\]
as $\sEnd_\algB X$-modules. 
Thus $\sHom_\algB(X, M)$ is finite dimensional.
\end{enumerate}
\end{lemma}

\begin{proof}
(1) The lower row is the partial presentation of $M$, 
and so $\Syz M\in \GP\algB$, by \Cref{lem:GPBBapp}. 
The extension on top can be constructed in $\GP\algB$, since it is a Frobenius category. 
Then there is a map $B'\to \procov{M}$ making the left square commute, 
because $B$ is injective in $\GP\algB$. 
So we obtain the map $f$, and the diagram exists. 

(2) That $\Ext^i_\algB(X,f)$ are isomorphisms follows using the diagram of
long exact sequences, obtained by applying $\Hom_\algB(X,-)$, and using that
$\Ext^i_\algB(X,B)=0$ for $i>0$, since $X\in \GP\algB$.
It remains to prove that $\sHom_\algB(X,f)$ is an isomorphism. 
If a map $X\to M$ factors through $\add B$, then it factors through 
the surjection $g\colon \procov{M}\to M$, since $B$ is projective in $\CM \algB$. 
So the connecting homomorphism induces a natural isomorphism 
\begin{equation} \label{eq:helpisom}
\sHom_\algB(X,M)\isom \Ext^1_\algB(X,\Omega M).
\end{equation}
Consequently, $\sHom_\algB(X,f)$ is an isomorphism, as required. 

(3) The first isomorphism is \eqref{eq:helpisom}.
Applying $\Hom_\algB(\coSyz X,-)$, 
we obtain a diagram of long exact sequences, and since $\Ext^i_\algB(\coSyz X,B)=0$ for all 
$i\geq 1$, we see that 
\[
  \Ext^1_\algB(\coSyz X, M)\isom \Ext^1_\algB(\coSyz X,\coSyz\Syz M).
\]
Now $\coSyz$ is an equivalence on the stable category, and so 
\[
  \Ext^1_\algB(\coSyz X, \coSyz\Syz M)=\Ext^1_\algB(X, \Syz M)=\sHom_\algB(X,M),
\]
by the first isomorphism.
\end{proof}

\begin{remark}
The lemma is standard and straightforward when $M$ is in 
the Frobenius category $\GP \algB$, but we emphasise that the lemma is proved for 
$M$ in $\CM B$. 
\end{remark}

\begin{corollary}\label{cor:syz-inj-proj}
If $J$ is injective in $\CM\algB$, then $\Syz J$ is projective.
\end{corollary}

\begin{proof}
We know that $\Syz J\in\GP\algB$, by \Cref{lem:GPBBapp}, and that $\GP\algB$
is Frobenius, by \Cref{Thm:GPB3}(\ref{itm:GP3-3}).
Hence it is sufficient to show that $\Syz J$ is injective in $\GP\algB$.
On the other hand, by \Cref{Lem:ExtSHom}(\ref{itm:ESH3}), for any $X\in\GP\algB$,
\[
 \Ext^1_\algB(X, \Syz J) \isom \Ext^1_\algB(\coSyz X, J)=0,
\]
as $J$ is injective in $\CM\algB$, and so $\Syz J$ is injective in $\GP\algB$, as required.
\end{proof}

%%%% =========================================================== %%%%
\section{Endomorphism algebras of cluster tilting objects in $\GP\algB$} \label{Sec:3}
%%%% =========================================================== %%%%

In this and following sections, until stated otherwise, 
we work in the generality of \S\ref{subsec:GPB}, that is, 
the $\ring$-order $\algB$ %, as in \eqref{eq:CinB},
is rigid as a left $\algC$-module.
Hence the category $\GP\algB$ is Frobenius 2-CY and
admits cluster tilting objects as part of a cluster structure.
We are also interested in the larger category $\CM\algB$.

Recall, from \Cref{Lem:restriFF}, that $\CM\algB$ is a full subcategory of $\CM\algC$, 
so the functors $\Hom_\algB(-, -)$ and $\Hom_\algC(-, -)$ agree on $\CM\algB$. 
We will often write them both as just $\Hom(-, -)$, 
although we may sometimes keep the subscripts for emphasis. We might also 
drop the subscripts on $\Ext^1(-, -)$ when \Cref{Thm:GPB1} applies.

When $T$ is a cluster tilting object in $\GP\algB$,
a significant role is played by the endomorphism algebra
%of a cluster tilitng object $T$ in $\GP\algB$
 \[
  \algA=\End (T)\op
\]
and the category $\CM\algA$ of Cohen--Macaulay $\algA$-modules, 
that is, as before, $\algA$-modules that are free over $\ring$.

Note that all indecomposable projective $\algB$-modules
must be (isomorphic to) summands of any cluster tilting object.
Thus, without loss of generality, we can (and will) assume that 
$\algB$ is a summand of $T$ 
and that only the remaining mutable summands are multiplicity-free.
This differs from the usual assumption that $T$ itself is basic,
but has no effect on the Morita equivalence class of $\algA$, 
that is, on the category $\CM\algA$.

By this assumption, $\algA$ has a special idempotent 
\[
  e\colon T\to T
\] 
given simply by projection onto $\algB$.
Furthermore we have canonical identifications
\begin{equation}\label{eq:ABT}
e\algA e  = \End (B)\op= \algB 
\quadand
e\algA = \Hom(B,T) = T.
\end{equation}

\begin{example}\label{ex:plabic}
A special case of the above is when $\algB$ is the boundary algebra of
a dimer model on a disc (as in \cite{CKP, Pr} or \cite{BKM}, when $B=C$).
In other words, $\algB$ is the necklace algebra associated to a 
\emph{connected} Postnikov diagram or plabic graph.
Here the connectedness implies that $\algB$ is basic.

In this case, the \emph{dimer algebra} $\algA$ is constructed first 
and $\algB$ and $T$ are defined by \eqref{eq:ABT}.
The fact that $\algB$ is the corresponding necklace algebra
follows from \cite[Prop.~8.2]{CKP}.
The fact that $T$ is a cluster tilting object in $\GP\algB$, with $\End(T)\op=A$,
is proved in \cite[Thm.~4.5]{Pr}.
\end{example}

\begin{remark}\label{rem:gab-quiv}
As part of the categorical cluster structure, one associates to $T$ 
a quiver $Q=(Q_0,Q_1)$, whose vertices $i$ (in $Q_0$) index 
a complete set of (non-isomorphic) indecomposable summands $T_i$ of $T$
and whose arrows $i\to j$ (in $Q_1$) correspond to irreducible maps $T_j\to T_i$ in $\add T$.
In this paper, we will indicate this relationship by writing
\begin{equation}\label{eq:T-moreq}
    T\moreq  \textstyle\bigoplus_{i\in Q_0} T_i.
\end{equation}
%In particular, then $\add T=\add \{ T_i \}$.
When $T$ mutates to $\mu_k(T)=(T/T_k)\oplus T_k^*$, 
the quiver~$Q$ undergoes Fomin--Zelevinsky mutation at $k\in Q_0$.

Alternatively, $Q$ is the Gabriel quiver of $\algA$.
Thus $\algA$ contains an idempotent $e_i$, for each $i\in Q_0$,
given by projection $T\to T$ onto the summand $T_i$
and $Ae_i=\Hom(T,T_i)$ is the corresponding indecomposable projective $\algA$-module,
whose top is the simple $\algA$-module $S_i$. 

Note that, if $T$ (and thus $\algA$) is not basic, then 
$\sum_{i\in Q_0} e_i\neq 1$ and $\algA$ is not a quotient of the path algebra of $Q$.
However, any $A$-module $M$ still determines a representation of $Q$,
for which the space at $i\in Q_0$ is $e_i M=\Hom_\algA(Ae_i,M)$.
Furthermore, the Grothendieck group $\Grot(\fd\algA)$ of finite-dimensional $\algA$-modules
is still canonically identified with $\ZZ^{Q_0}$ 
via the basis $\{[S_i]\st i\in Q_0\}$ or, equivalently, by the map 
\begin{equation}\label{eq:dimvec}
  [M]\mapsto \dimv[M]:=(\dim e_iM)_{i\in Q_0}.
\end{equation}
We may still refer to $\dimv[M]$ as the \emph{dimension vector} of $M$.
\end{remark}

Following the terminology of \cite{BKM} and \cite{CKP}, but in our more general context,
the vertices of~$Q$ indexing summands of~$\algB$ will be called \emph{boundary} vertices,
while those indexing the other summands will be called \emph{interior} vertices.
It is precisely the interior vertices that are `mutable' in the sense that $T_k^*$ and
thus $\mu_k$ are defined, while the boundary vertices are `frozen' and $\mu_k$ is not defined for them.

Let $\efunct$ be the restriction functor induced by $e$ 
\[
  \efunct\colon \mmod\algA\to \mmod\algB
  \colon X\mapsto e X.
\]
As $\efunct$ can be interpreted as $e\algA \otimes_\algA  -$,
it has right adjoint
\[
  \Radj\colon \mmod\algB\to\mmod\algA
  \colon M\mapsto \Hom_\algB(T , M),
\]
which restricts to a standard equivalence
\begin{equation}\label{eq:add-proj}
  \Radj= \Hom_\algB(T,-)\colon \add T \to \proj \algA.
\end{equation}
As $\efunct$ can be also be interpreted as $\Hom_\algA (\algA e, -)$, it has a left adjoint
\[
  \Ladj\colon \mmod\algB\to\mmod\algA
  \colon M \mapsto \algA e\otimes_\algB M.
\]
Note that $\efunct$ and $\Radj$ restrict to $\CM$-modules,
that is, to functors  $\efunct\colon\CM\algA\to\CM\algB$ and $\Radj\colon\CM\algB\to\CM\algA$.
However, $\Ladj$ needs to be modified, by dividing out by torsion, to restrict similarly (cf. \cite[\S5]{CKP}).

% ==========================================
\subsection{The right adjoint}
% ==========================================
We observe first that the counit $\counitmap{M}\colon \efunct \Radj M \to M$ is always an isomorphism,
because $e\Hom_\algB(T , M)=\Hom_\algB(B , M)=M$. 
Thus, by \Cref{Lem:adjprops}\itmref{itm:adjp3}, the functor $\Radj$ is fully faithful.
In fact, we can identify the essential image of $\Radj$,
starting as follows.

\begin{lemma} \label{Lem:appseq}
Every $M\in \CM\algB$ has an exact $\add T$-presentation, that is,
there is a short exact sequence
\begin{equation}\label{eq:addT->M}
0 \lra T'' \lra T' \lra M \lra 0
\end{equation}
with $T', T'' \in \add T$.
\end{lemma}

\begin{proof}
Let $g\colon T'\to M$ be an $\add T$-approximation, 
which exists, because $\Hom_B(T,M)$ is a free $\ring$-module of finite rank,
and is surjective, because $B\in \add T$. 
Denote $\ker g$ by $T''$
and note that $T''\in \CM\algB$ because $T''\leq T'\in \CM\algB$.
We want to show that $T''\in \add T$. 

By \Cref{Thm:GPB1}, $\GP\algB$ is a full exact subcategory of $\CM\algC$, so
\[ 
  \Ext^1_\algC(T, T)=\Ext^1_\algB(T, T)=0.
\]
As $\CM\algB$ is a full subcategory of $\CM\algC$, 
applying $\Hom_\algC(T, -)$ to \eqref{eq:addT->M}
and using that $g_*\colon \Hom_\algC(T,T')\to \Hom_\algC(T, M)$ is surjective, we get
\begin{equation}\label{eq:extT}
  \Ext^1_\algC(T, T'')=0.
\end{equation} 
In particular, since $\CM\algC$ is 2-CY and $B\in \add T$, we get
\[
  \Ext^1_\algC(T'',B)=\Ext^1_\algC(\algB, T'')=0.
\] 
Hence \Cref{Thm:GPB2} implies that $T''\in \GP\algB$
and so \Cref{Thm:GPB1} and \eqref{eq:extT} imply that $\Ext^1_\algB(T, T'')=0$.
But $T$ is cluster tilting in~$\GP\algB$,
so this implies that $T''\in\add T$, as required.
\end{proof}

\begin{remark}
This result is familiar when $M$ is in~$\GP\algB$.
To get the same for~$\CM\algB$, we need to use the more delicate relationships with $\CM\algC$ and $\GP\algB$ 
from \Cref{Sec:2}.  
Note that, a similar argument without this delicacy shows that every $M\in \CM\algB$ also
has an exact $\add T$-presentation on the other side, that is, a short exact sequence
\begin{equation}\label{eq:addT->M-other}
0 \lra M \lra T' \lra T'' \lra 0,
\end{equation}
with $T', T'' \in \add T$.
\end{remark}

\begin{remark}
Since the restriction functor from $\CM\algB\to\CM\algC$ is fully faithful, 
by \Cref{Lem:restriFF}, we also have $\algA= \End_\algC(T)\op$ and 
so have a well-defined functor 
\[\Hom_\algC(T, -)\colon \CM\algC \to \CM\algA,\]
which coincides with $\Radj$ on $\CM\algB$.
However, the functor $\Hom_\algC(T, -)$ is typically only faithful, because the  
counit map is only pointwise epi.
Indeed $e\Hom_C(T, M)=\Hom_C(B, M)$ is necessarily in $\CM\algB$ and 
so only \emph{a priori} a submodule of $\Hom_C(C, M)=M$, for $M\in\CM\algC$.

Actually, this is a rather general fact. That is, for any $M\in\CM\algC$,  
we can show that the counit $M\otimes \Hom_\algC(M, -) \to \idfun$ is pointwise epi, 
so that 
\[
  \Hom_\algC(M, -)\colon \CM\algC \to \CM \End_\algC (M)\op
\]
is faithful, by \Cref{Lem:adjprops}\itmref{itm:adjp1}.
\end{remark}

\begin{proposition} \label{prop:globdim}
The global dimension of $\algA$ is at most $3$.
\end{proposition}

\begin{proof}
Let $X\in \mmod\algA$. By the equivalence \eqref{eq:add-proj},
we can choose $g\colon T''\to T'$ in $\add T$ to get a projective presentation of $X$
\begin{equation}\label{eq:p-pres-X}
 \Radj T'' \lraa{\Radj g} \Radj T' \lra X \lra 0.
\end{equation}
Since $\counitmap{}\colon\efunct\Radj \to\idfun$ is a natural isomorphism,
applying $\efunct$ to \eqref{eq:p-pres-X} gives a exact sequence
\begin{equation}\label{eq:p-pres-eX}
 T'' \lraa{g} T' \lra \efun X \lra 0.
\end{equation}
If $K''=\ker g$, then also $\Radj K''=\ker \Radj g$, since $\Radj$ is left exact.
But $K''\in\CM\algB$, so, by \Cref{Lem:appseq}, 
$K''$ has an exact $\add T$-presentation of the form \eqref{eq:addT->M}. 
Applying $\Radj$ to the presentation gives a projective resolution of $\Radj K''$.
Hence
\[ \projdim \Radj K'' \leq 1\] 
and so $\projdim X\leq 3$.
As $X\in \mmod\algA$ was arbitrary, this means that $\gldim\algA\leq 3$, as required.
\end{proof}

\begin{lemma} \label{lem:eta-ker-coker}
For any $X\in \mmod\algA$, there is an exact sequence
\begin{equation}\label{eq:eta-ker-coker}
  0\lra \Ext^1_\algB(T,\efunct\Syz^2 X) 
  \lra X \lraa{\unitmap{X}} \Radj\efunct X 
  \lra \Ext^1_\algB(T,\efunct\Syz X) \to 0.
\end{equation}
\end{lemma}

\begin{proof}
Consider the presentations \eqref{eq:p-pres-X} and \eqref{eq:p-pres-eX}.
As before, let $K''=\ker g$, so that $\Radj K''=\ker \Radj g$, 
and also let $K'=\img g=\ker (T' \to\efun X)$.
Notice that, taking syzygies from \eqref{eq:p-pres-X}, 
we have $K'=\efunct\Syz X$ and $K''=\efunct\Syz^2 X$, because $\efunct$ is exact.
Then we can construct 
the following commutative diagram of exact sequences,
in which the middle equality is really $\unitmap{\Radj T'}$, 
after we have used $\counitmap{T'}$ to identify $\efunct \Radj T'$ with $T'$.
\[
\begin{tikzpicture}[xscale=2.5, yscale=1.6,
 equals/.style={double=none, double distance=2pt}]
\draw (0.2,2) node (b0) {$0$};
\draw (1,2) node(b1) {$\Radj T''/\Radj K''$};
\draw (2,2) node(b2) {$\Radj T'$};
\draw (3,2) node(b3) {$X$};
\draw (3.8,2) node (b4) {$0$};
\draw (0.2,1) node (c0) {$0$};
\draw (1,1) node(c1) {$\Radj K'$};
\draw (2,1) node(c2) {$\Radj T'$};
\draw (3,1) node(c3) {$\Radj \efunct X$};
\draw (4,1) node(c4) {$\Ext^1_\algB(T, K')$};
\draw (4.9,1) node (c5) {$0$.};
\foreach \t/\h in {b0/b1, b1/b2, b2/b3, b3/b4, c0/c1, c1/c2, c2/c3, c3/c4, c4/c5, b1/c1} 
  \draw[cdarr] (\t) to (\h);
 \draw[cdarr] (b3) to node [right] {\small $\unitmap{X}$} (c3);
 \draw[equals] (b2) to (c2);
 \end{tikzpicture}
\]
The first vertical map has cokernel $\Ext^1_\algB(T, K'')$ and, by the Snake Lemma, 
this cokernel is isomorphic to $\ker\unitmap{X}$.
On the other hand, the image of $\unitmap{X}$ coincides with the image of 
the map $\Radj T'\to \Radj \efunct X$, so $\cok\unitmap{X}\isom\Ext^1_\algB(T, K')$.
\end{proof}

\begin{remark}\label{rem:eta-ker-coker}
\Cref{lem:eta-ker-coker}, with \Cref{lem:ExtBinExtC},
implies that $\ker\unitmap{X}$ and $\cok\unitmap{X}$
are finite dimensional (i.e.~torsion) modules.
Furthermore, they vanish on applying~$\efunct$,
because $\efunct \Ext^1_\algB(T,-)=\Ext^1_\algB(B,-)=0$.
This is consistent with the fact that $\efunct\unitmap{X}\colon \efun X\to \efun X$
is the identity map, after using $\counitmap{\efun X}$ to identify 
$\efunct \Radj \efunct X$ with $\efun X$.
Note also that \Cref{lem:eta-ker-coker} immediately implies that, if $\efun X=0$, then $X$ is torsion.
\end{remark}

\begin{proposition} \label{prop:projdim}
If $X\in \CM \algA$, then $\projdim X\leq 1$ if and only if 
$X\isom \Hom_\algB(T, M)$, for some $M\in \CM\algB$
(and then necessarily $M\isom eX$).
\end{proposition}

\begin{proof}
As already used in the proof of \Cref{prop:globdim}, 
any $M\in \CM\algB$ has a presentation~\eqref{eq:addT->M}, by \Cref{Lem:appseq},
and applying $\Hom_\algB(T, -)$ gives a projective resolution of $\Hom_\algB(T, M)$, 
since $\Ext^1(T,T'')=0$, so $\projdim \Hom_\algB(T, M)\leq 1$. 

Conversely, if $X$ has $\projdim X\leq 1$,
then $\Syz X$ is projective, so $\efunct \Syz X\in\add T$, and $\Syz^2 X=0$.
Hence \Cref{lem:eta-ker-coker} implies that $\unitmap{X}\colon X \to \Radj \efunct X$ 
is an isomorphism.
If $X\in \CM\algA$, then $eX\in\CM\algB$, which proves the claim.
\end{proof}

Thus $\Radj\colon\CM\algB\to\CM\algA$ is an equivalence of $\CM\algB$ with the full subcategory of
$\CM\algA$ consisting of modules with $\projdim \leq 1$.
In the other direction, we have

\begin{proposition}\label{prop:rest-fful}
For $X\in\CM\algA$, the unit map 
$\unitmap{X}\colon X\to \Hom_\algB(T, eX)$
is injective. 
Thus the restriction functor $\efunct\colon \CM\algA\to \CM\algB$ is faithful. 
\end{proposition}

\begin{proof}
By \Cref{rem:eta-ker-coker}, we know that $\ker\unitmap{X}$ is torsion,
but, if $X\in\CM A$, then $X$ is torsion-free, so $\ker\unitmap{X}=0$, as required.
That $\efunct$ is faithful follows by \Cref{Lem:adjprops}\itmref{itm:adjp1}. 
\end{proof}

% ==========================================
\subsection{The left adjoint}
% ==========================================
\newcommand{\incmap}{\mathsf{j}}

Define $\algA e X \submod X$ to be the image of the $(\Ladj,\efunct)$ counit map
\[
  \mu_X\colon \algA e\otimes _\algB eX\to X,
\]
given by multiplication.
As the notation suggests, we treat $e X$ as an $\ring$-submodule of~$X$, 
so that $\algA e X$ is the $\algA$-submodule generated by it.
In the case $X=\Hom(T, M)$, we can identify $e X=\Hom(B, M)=M$ 
and we will then write $\algA e X$ as 
\[
  \algA\cdot M\subspc \Hom(T, M).
\]
\begin{remark}\label{rem:AeX=A.eX}
Thus, in general, we have $\algA e X \submod X$ and $\algA\cdot eX \subspc \Hom(T, eX)=\Radj\efunct X$,
but we can readily see that the unit map $\unitmap{X}\colon X \to \Radj \efunct X$
identifies $\algA e X$ with $\algA\cdot eX$,
using the following natural diagram and noting that $\efunct \unitmap{X}$ is an isomorphism 
(cf. \Cref{rem:eta-ker-coker}).
\[
\begin{tikzpicture}[xscale=2.5, yscale=1.6]
\draw (1,2) node(b1) {$\Ladj \efunct X$};
\draw (2,2) node(b2) {$X$};
\draw (1,1) node(c1) {$\Ladj \efunct \Radj \efunct X$};
\draw (2,1) node(c2) {$\Radj \efunct X$.};
\foreach \t/\h/\lab in {b1/b2/$\mu_X$, c1/c2/$\mu_{\Radj \efunct X}$} 
  \draw[cdarr] (\t) to node [above] {\small \lab}(\h);
\foreach \t/\h/\lab/\side in {b1/c1/$\Ladj \efunct \unitmap{X}$/left, b2/c2/$\unitmap{X}$/right} 
 \draw[cdarr] (\t) to node [\side] {\small \lab} (\h);
 \end{tikzpicture}
\]
In other words, if we use $\unitmap{X}\colon X \to \Radj \efunct X$
to identify $X$ with a submodule of $\Radj \efunct X$, 
then we can consider that
\[
  \algA\cdot \efun X \submod X \submod  \Hom(T,\efun X).
\]
\end{remark}

\begin{remark}\label{rem:eta=inc}
Conversely, if $X\subspc \Hom(T,M)$,
then we can consider that $e X\subspc M$.
If $e X= M$, then the inclusion map $\incmap\colon X\to \Radj M$ 
is identified with the unit map $\unitmap{X}\colon X \to \Radj \efunct X$.
More precisely, $e X= M$ means that
$\efunct\incmap\colon \efunct X\to \efunct\Radj M$ is an isomorphism or, equivalently
(as $\counitmap{M}$ is always an isomorphism), that 
$\theta_X= \counitmap{M}\compo \efunct\incmap\colon \efunct X\to \efunct\Radj M\to M$
is an isomorphism.
Then there is a natural diagram
\[
\begin{tikzpicture}[xscale=2.5, yscale=1.6,
 equals/.style={double=none, double distance=2pt}]
\draw (1,2) node(b1) {$X$};
\draw (2,2) node(b2) {$\Radj \efunct X$};
\draw (1,1) node(c1) {$\Radj M$};
\draw (2,1) node(c2) {$\Radj \efunct \Radj M$};
\draw (3,1) node(c3) {$\Radj M$};
\foreach \t/\h/\lab in {b1/b2/$\unitmap{X}$, c1/c2/$\unitmap{\Radj M}$, c2/c3/$\Radj \counitmap{M}$} 
  \draw[cdarr] (\t) to node [above] {\small \lab}(\h);
\foreach \t/\h/\lab/\side in {b1/c1/$\incmap$/left, b2/c2/$\Radj\efunct\incmap$/right,
   b2/c3/$\Radj\theta_X$/above right} 
 \draw[cdarr] (\t) to node [\side] {\small \lab} (\h);
 \end{tikzpicture}
\]
in which the bottom composite is the identity map, as for a general adjunction.
\end{remark}

\begin{proposition}\label{Thm:EndA}
The following hold. 
\begin{enumerate}
\item\label{itm:EA5}
$A\cdot M$ is the kernel of the canonical quotient $\Hom_\algB(T, M)\to \sHom_\algB(T, M)$.
\item\label{itm:EA6}
For $X \submod \Hom_\algB(T, M)$, we have $eX=M$ if and only if $\algA\cdot M \submod X$.
\end{enumerate}
\end{proposition}

\begin{proof}
\itmref{itm:EA5} % ===== (3) =====
$\algA\cdot M$ is the image of the counit map $\algA e\otimes _B e\Hom(T, M)\to \Hom(T, M)$,
which more explicitly is the composition map 
\[
  \Hom(T, B) \otimes_B \Hom(B, M) \to \Hom(T, M).
\]
Thus $\algA\cdot M$ consists of sums of maps that factor through $\algB$.
Generally, this is the same as maps that factor through $\add \algB$,
that is, in this case, that factor through projectives. 
By definition, this is the kernel of the quotient $\Hom_\algB(T, M)\to \sHom_\algB(T, M)$.

\itmref{itm:EA6} % ===== (4) =====
First note that $e\algA\cdot M = \algB\otimes_\algB M = M$.
Hence, if $\algA\cdot M \submod X \submod \Hom_\algB(T, M)$, 
then $M\submod eX\submod M$, that is, $eX=M$.

Conversely, suppose $X \submod \Hom_\algB(T, M)$ and $eX=M$, 
that is, we have an inclusion $\incmap\colon X\to \Radj M$ such that
$\efunct\incmap\colon \efunct X\to \efunct\Radj M$ is an isomorphism.
Then we have a natural diagram
\[
\begin{tikzpicture}[xscale=2.5, yscale=1.6]
\draw (1,2) node(b1) {$\Ladj \efunct X$};
\draw (2,2) node(b2) {$X$};
\draw (1,1) node(c1) {$\Ladj \efunct \Radj M$};
\draw (2,1) node(c2) {$\Radj M$};
\foreach \t/\h/\lab in {b1/b2/$\mu_X$, c1/c2/$\mu_{\Radj M}$} 
  \draw[cdarr] (\t) to node [above] {\small \lab}(\h);
\foreach \t/\h/\lab/\side in {b1/c1/$\Ladj \efunct \incmap$/left, b2/c2/$\incmap$/right} 
 \draw[cdarr] (\t) to node [\side] {\small \lab} (\h);
\end{tikzpicture}
\]
showing that $\incmap$ identifies $AeX$ with $\algA\cdot M$
and thus $\algA\cdot M\subspc X$, as required.
\end{proof}

\begin{remark}
In the dimer model context of \Cref{ex:plabic},
some of the results here are already proved in \cite[\S5]{CKP}.
For example, \Cref{Thm:EndA} is \cite[Prop.~5.4, Cor.~5.11]{CKP},
while \Cref{prop:rest-fful} is \cite[Lemma~5.3]{CKP}.
\end{remark}

% ======================================================================
\subsection{Ranks of Cohen--Macaulay modules}\label{subsec:ranks}
% ======================================================================

By the definition of $\algB$, we have 
\begin{equation}\label{eq:BCoK}
\algB \otimes_\ring \fieldfract=\algC \otimes_\ring \fieldfract \isom \Matr{n}{\fieldfract},
\end{equation}
the $n\times n$ matrix algebra over $\fieldfract$.
So we can define the rank of a $\algB$-module $M$ in the same way as that of a $\algC$-module, 
that is, 
\[
\rnk{\algB} M=\len_{\algB \otimes_\ring \fieldfract }(M\otimes_\ring \fieldfract),
\]
the length of $M\otimes_\ring \fieldfract$ over $\algB \otimes_\ring \fieldfract$.
Indeed, by \eqref{eq:BCoK} and \cite[Def.~3.5 \emph{et seq.}]{JKS1}, we have
\begin{equation}\label{eq:rkBCR}
  \rnk{\algB} M = \rnk{\algC} M = \rnk{\ring}(e_i M),
\end{equation}
for any vertex idempotent $e_i$.

\begin{lemma}
The algebra $\algA\otimes_\ring\fieldfract$ is simple, isomorphic to $\Matr{r}{\fieldfract}$,
where $r=\rnk{\algB} T$. 
\end{lemma}

\begin{proof}  
$\algA\otimes_\ring\fieldfract=\End_\algB(T)\op \otimes_\ring \fieldfract\isom 
\End_{\Matr{n}{\fieldfract}} (T \otimes_\ring \fieldfract)\op \isom \Matr{r}{\fieldfract}$.
\end{proof}

Thus $\algA$ is an $\ring$-order and any $\algA$-module is an $\ring$-module,
so we can define the rank of an $\algA$-module $X$ in the same way as earlier:
\[
  \rnk{\algA} X=\len_{A\otimes_\ring\fieldfract} (X\otimes_\ring\fieldfract), 
\]
The functor $-\otimes_\ring\fieldfract$ is exact, 
so, over any of the algebras $\ring$, $\algA$, $\algB$ or $\algC$, 
we see that $\rnk{}$ is additive on short exact sequences of modules.
Furthermore, over any of these algebras, $\rnk{}M=0$ 
if and only if $M$ is torsion, i.e.~finite dimensional.

\begin{lemma} \label{Lem:ranks}
Suppose $M\in \CM\algB$ and $X\in \CM\algA$. 
Then 
\begin{enumerate} 
\item\label{itm:rks1}
  $\rnk{\algA} \Hom(T, M) = \rnk{\algB} M$. 
\item\label{itm:rks3} 
  $\rnk{\algB} \efun X = \rnk{\algA} X$.
\item\label{itm:rks4} 
 $\rnk{\ring} e_i X = (\rnk{\algA} X)(\rnk{\algB} T_i)$. 
\end{enumerate}
\end{lemma}

\begin{proof} 
\itmref{itm:rks1} % ===== (1) =====
By definition, 
\[
  \rnk{\algA} \Hom(T, M) = \len_{\Matr{r}{\fieldfract}}  (\Hom(T, M)\otimes_\ring\fieldfract),
\]
where $r=\rnk{\algB} T$. But
\[
  \Hom_\algB(T, M)\otimes_\ring \fieldfract\isom 
  \Hom_{\algB\otimes_\ring \fieldfract}(T\otimes_\ring \fieldfract, M\otimes_\ring \fieldfract),
\]
which has $\dim_\fieldfract$ equal to $(\rnk{\algB} T)(\rnk{\algB} M)$
and so $\len_{\Matr{r}{\fieldfract}}$ equal to $\rnk{\algB} M$, as required.

\smallskip
\itmref{itm:rks3} % ===== (2) =====
By \itmref{itm:rks1}, $\rnk{\algB} \efun X=\rnk{\algA} \Hom(T,  \efun X)$.
On the other hand, \Cref{lem:eta-ker-coker} implies that $\unitmap{X}\colon X\to \Hom(T,  \efun X)$
becomes an isomorphism after applying $-\otimes_\ring\fieldfract$,
because $Y\otimes_\ring \fieldfract=0$, when $Y$ is torsion.
Hence $\rnk{\algA} \Hom(T, \efun X)=\rnk{\algA}  X$, giving the result.

\smallskip
\itmref{itm:rks4} % ===== (3) =====
By \Cref{prop:rest-fful} and \Cref{rem:eta-ker-coker}, %\Cref{lem:eta-ker-coker}, 
we know that $X \submod  \Hom(T,\efun X)$, with finite-dimensional quotient. 
Hence, as a consequence of \cite[Lemma~5.5]{JKS2},
\[
  \rnk{\ring} e_i X =  \rnk{\ring} \Hom(T_i,\efun X) = (\rnk{\algB} T_i)(\rnk{\algB} \efun X),
\] 
so the result follows from \itmref{itm:rks3}.
\end{proof}

%%%% =========================================================== %%%%
\section{Weight modules}\label{Sec:4}
%%%% =========================================================== %%%%

We continue with the setting in \Cref{Sec:3}
and continue to write $\Hom(-, -)$ for $\Hom_\algB(-, -)=\Hom_\algC(-, -)$
on $\CM\algB\subset\CM\algC$.

We choose a vertex $\vstar$ (e.g.~$0$) of the circular double quiver \eqref{eq:circle-quiv} 
and let $e_\vstar\in\algC$ be the associated idempotent, 
which is also an idempotent of $\algB$. 
Since $\algB e_\vstar$ is an indecomposable summand of $\algB$,
it is also an indecomposable summand of $T$, 
and so there is a corresponding vertex, also denoted $\vstar$, in the Gabriel quiver $Q$ of $\algA$.
We can and will assume that $\algB e_\vstar$ is the summand $T_\vstar$ chosen
in \Cref{rem:gab-quiv}.
Note that, consequently, $\algA$ also has an idempotent denoted by $e_\vstar$,
but it should be clear which is meant in any situation.

Define functors 
\[ \begin{aligned}
  \funJ & \colon \CM\algB\to \CM\algB \colon M\mapsto \Hom_\ring(e_\vstar\algB, e_\vstar M) \\
  \funP &\colon \CM\algB\to \CM\algB \colon M\mapsto \algB e_\vstar\otimes_\ring e_\vstar M.
\end{aligned} \] 
Then the natural maps $\funP M\to M$ and $M\to \funJ M$, become isomorphisms under $e_\vstar$,
so their kernels have rank 0 and hence are 0 (cf. \cite[Lem~5.1]{JKS2}).
In other words, we have natural embeddings
\begin{equation}\label{eq:PtoJ}
  \funP M\subspc M \subspc \funJ M 
  \quad\text{with}\quad 
  e_\vstar\funP M=e_\vstar M =e_\vstar \funJ M.
\end{equation}
If we define
\begin{equation}\label{eq:defPJ}
 \modP = Be_\vstar
\quadand
 \modJ=\Hom_\ring(e_\vstar\algB, \ring),
\end{equation}
then, using \eqref{eq:rkBCR} and \eqref{eq:CinB}, 
we have $\rnk{\algB} \modP = \rnk{\ring} e_\vstar Be_\vstar = 1$,
and similarly $\rnk{\algB} \modJ=1$.
Furthermore, 
\begin{equation}\label{eq:isomPJ-rkB}
 \funP M = \modP \otimes_\ring e_\vstar M\isom \modP^{\rnk{\algB} M}
\quadand
 \funJ M = \modJ \otimes_\ring e_\vstar M \isom \modJ^{\rnk{\algB} M}.
\end{equation}
Note that, while $\modP$ is a projective $\algB$-module,
$\modJ$ is not an injective $\algB$-module, 
but it is injective in $\CM\algB$.

\begin{definition}\label{def:KTM}
For $M\in\CM\algB$, let $\epsemb{M}\colon M\to \funJ M$ be the natural embedding. 
Applying $\Hom(T, -)$ gives a short exact sequence, as in \cite[\S5]{JKS2},
\begin{equation}\label{eq:coker-epsM}
 0 \lra \Hom(T, M) \lraa{\epsemb{M}_*} \Hom(T, \funJ M) 
   \lra \funK(T, M) \lra 0,
\end{equation}
which defines the finite dimensional $\algA$-module $\funK(T, M)$.
The class of $\funK(T, M)$ in the Grothendieck group $\Grot(\fd\algA)$
is denoted $\kapvec(T, M)$.
\end{definition}

Note that $\funK(T, M)$ is finite dimensional, because it is a submodule of $\Hom(T, \funJ M/M)$
and $\funJ M/M$ is a finite dimensional $\algB$-module,
since $\rnk{\algB}\funJ M = \rnk{\algB}M$,
by \eqref{eq:isomPJ-rkB}.

More generally, as in \cite[\S5]{JKS2}, we can replace $T$ by $N$ in \eqref{eq:coker-epsM} 
to define $\funK(N, M)$ for any $N\in\CM\algB$.
Then $\funK(N, M)$ is just a finite dimensional vector space
and we define $\kappa(N,M)=\dim\funK(N, M)$.
Thus, making the identification $\Grot(\fd\algA)\isom\ZZ^{Q_0}$,
as in \Cref{rem:gab-quiv}, we have
\[
  \kapvec(T, M) = \bigl( \kappa(T_i,M) \bigr)_{i\in Q_0}.
\]
Furthermore,
\begin{equation}\label{eq:KP*M}
 \funK(T_\vstar, M)=0,
\end{equation}
because $\funJ M/M$ vanishes at $\vstar$. 
More precisely, as $T_\vstar=\algB e_\vstar$, for $e_\vstar\in\algB$,
\[
  \funK(T_\vstar, M)\submod \Hom(T_\vstar, \funJ M/M) 
  = e_\vstar(\funJ M/M) = e_\vstar\funJ M/ e_\vstar M = 0.
\]
In other words, the $\algA$-module $\funK(T, M)$ vanishes at $\vstar$,
i.e.~$e_\vstar\funK(T, M)=0$, for $e_\vstar\in\algA$.

%\begin{definition}\label{def:wt-mod}
For $X\in \CM\algA$, the unit map 
$\unitmap{X}\colon X\to \Hom(T, \efun X)$
is injective, by \Cref{prop:rest-fful}.
We define $\funG X$ to be its cokernel,
giving a short exact sequence
\begin{equation}\label{eq:coker-eta}
 0 \lra X \lraa{\unitmap{X}} \Hom(T, \efun X) \lra \funG X \lra 0 .
\end{equation}
In fact, by \Cref{rem:AeX=A.eX}, $\funG X$ is a quotient of $\sHom(T, \efun X)$.

Define the \emph{weight module} $\wtmod X$ to be 
the cokernel of the composite $\epsemb{\efun X}_*\compo \unitmap{X}$, 
so that the Third Isomorphism Theorem yields
the following commutative diagram
in which the two rows and two columns
are short exact sequences.
\begin{equation}\label{eq:wtmod-diag}
\begin{tikzpicture}[xscale=2.8, yscale=1.6,baseline=(bb.base)]
\coordinate (bb) at (0,2);
\pgfmathsetmacro{\xoff}{0.4}
\pgfmathsetmacro{\xofff}{0.5}
\pgfmathsetmacro{\yoff}{0.7}
\draw (1,3) node (a3) {$X$}; 
\draw (1,2) node (a2) {$X$}; 
\draw (2,3) node(b3) {$\Hom(T, \efun X)$};
\draw (2,2) node(b2) {$\Hom(T, \funJ \efun X)$};
\draw (2,1) node(b1) {$\funK(T, \efun X)$};
\draw (3,1) node(c1) {$\funK(T, \efun X)$};
\draw (3,2) node(c2) {$\wtmod X$};
\draw (3,3) node(c3) {$\funG X$};
\foreach \t/\h in {b3/b2, b2/b1, c3/c2, c2/c1, a2/b2, b2/c2, b3/c3} 
   \draw[cdarr] (\t) to (\h);
\foreach \t/\h in {a3/a2, b1/c1} 
  \draw[equals] (\t) to (\h);
\draw[cdarr] (a3) to node [above] {$\unitmap{X}$} (b3);
\draw[cdarr] (b3) to node [right] {$\epsemb{\efun X}_*$} (b2);
\draw (1-\xoff,3) node (z3) {$0$}; 
\draw (1-\xoff,2) node (z2) {$0$}; 
\draw (3+\xofff,3) node (zz3) {$0$}; 
\draw (3+\xofff,2) node (zz2) {$0$}; 
\draw (2,3+\yoff) node (b4) {$0$}; 
\draw (3,3+\yoff) node (c4) {$0$}; 
\draw (2,1-\yoff) node (b0) {$0$}; 
\draw (3,1-\yoff) node (c0) {$0$}; 
\foreach \t/\h in {z2/a2, z3/a3, c4/c3, b4/b3, c1/c0, b1/b0, c3/zz3, c2/zz2} 
   \draw[cdarr] (\t) to (\h);
\end{tikzpicture}
\end{equation}
%\end{definition}

In \Cref{sec:classical}, we will make the connection to dimer models and perfect matchings
and show that the dimension vector of $\wtmod X$ is the weight of a corresponding flow,
when $X$ is a matching module. 
This explains the terminology ``weight module''. 
Both $\funK(T, \efun X)$ and $\wtmod X$ depend on the choice of the boundary vertex $\vstar$
in the definition of the functor~$\funJ$. 
In practice, this vertex can be any in the circular double quiver \eqref{eq:circle-quiv}. 

\begin{remark}\label{rem:overBorC}
For any $M\in\CM\algB$, we have, by Hom-tensor adjunction, natural isomorphisms
\[ 
  \Hom_\algB(T, \funJ M) \isom  \Hom_\ring(e_\vstar T, e_\vstar M) \isom  \Hom_\algC(T, \funJ_\algC M), 
\]
where $\funJ_\algC M= \Hom_\ring(e_\vstar\algC, e_\vstar M)$.
Hence (cf. \cite[Lemma~5.4]{JKS2}), $\funK(T, M)$ is also the cokernel of the restriction map 
$ %\[ 
  \Hom_\algB(T, M) \to  \Hom_\ring(e_\vstar T, e_\vstar M)
$ %\]
and of the natural map 
$ %\[
  \Hom_\algC(T, M) \to \Hom_\algC(T, \funJ_\algC M)
$. %\]
In particular, $\funK(T, M)$ can be computed from just the $\algC$-module structure of $M$ (and $T$)
and $\wtmod X$ can be computed from just the $\algC$-module structure of $\efun X$.
\end{remark}

\begin{lemma} \label{Lem:somefacts}
Let $X, Y\in \CM\algA$.
\begin{enumerate}
\item\label{itm:facts1}
$\funG X$, $\wtmod X$ and $\funK(T, \efun X)$ are (left) $\algA$-modules,
which are functorial in $X$.
\item\label{itm:facts2}
$\funG X$ is finite dimensional and  $\efun\funG X=0$. 
\item\label{itm:facts3}
$\wtmod X$ is finite dimensional and
$\efun \wtmod X =\efun \funK(T, \efun X) = \funJ \efun X/\efun X$. 
Hence $\wtmod X$ is not supported at vertex $\vstar$. 
\item\label{itm:facts5} 
If $Y\subspc X$ and $\efun Y=\efun X$, then $\funG X$ is a quotient of  $\funG Y$ 
and $\wtmod X$ is a quotient of $\wtmod Y$.
Both kernels are isomorphic to $X/Y$.
\end{enumerate}
\end{lemma}

\begin{proof} 
\itmref{itm:facts1} % ===== (1) =====
The right action of $A=\End(T)\op$ on $T$ makes $\Hom(T,M)$ 
a left $A$-module for any $B$-module $M$.
Furthermore, both $\unitmap{X}$ and $\epsemb{\efun X}_*$ are $A$-module maps.
Thus $\funG X$, $\wtmod X$ and $\funK(T, \efun X)$ are all defined as cokernels of 
$A$-module maps, hence they are all $A$-modules.
All three are functorial because they are cokernels of natural transformations.

\medskip\itmref{itm:facts2} % ===== (2) ===== 
Since $\funG X=\cok\unitmap{X}$, both follow from \Cref{lem:eta-ker-coker}
(see \Cref{rem:eta-ker-coker}).

\medskip\itmref{itm:facts3} % ===== (3) ===== 
First $\wtmod X$ is an extension of $\funK(T, \efun X)$ by $\funG X$,
and these are both finite-dimensional, by \Cref{def:KTM} and \itmref{itm:facts2}.

Next, since $\efunct$ is exact, applying it to 
the right-hand vertical sequence in \eqref{eq:wtmod-diag}
gives $\efunct \wtmod X\isom \efunct \funK(T, \efun X)$
under the induced canonical map,
because $\efunct \funG X=0$ by \itmref{itm:facts2}.
Then applying $\efunct$ to the middle vertical sequence in \eqref{eq:wtmod-diag},
gives that $\efun \funK(T, \efun X) = \funJ \efun X/\efun X$.

Finally recall, from \eqref{eq:PtoJ}, that $\funJ \efun X$ and $\efun X$ agree at vertex $\vstar$,
so their quotient vanishes there.

\medskip\itmref{itm:facts5} % ===== (4) ===== 
This follows from the Third Isomorphism Theorem, 
using the top and middle rows of \eqref{eq:wtmod-diag}, respectively.
\end{proof}

\begin{lemma} \label{Lem:kerwt} 
For any $M\in\CM\algB$, we have $\wtmod \Hom(T, M)=\funK(T, M)$.
In particular, $\wtmod \Hom(T, \modJ)=\funK(T, \modJ)=0$. 
\end{lemma}

\begin{proof} 
If $X= \Hom(T, M)$, then $\efun X=M$ and 
$\unitmap{X}\colon X\to \Hom(T, \efun X)$ is an isomorphism. 
Hence $\funG X=0$ and so the map $\wtmod X\to \funK(T, \efun X)$
in \eqref{eq:wtmod-diag} is an isomorphism.

When $M=\modJ$, the map $\epsemb{M}\colon M\to \funJ M$ in \Cref{def:KTM}
is an isomorphism and so $\funK(T, M)=0$, by \eqref{eq:coker-epsM}.
\end{proof}

It follows immediately from \eqref{eq:isomPJ-rkB} 
that $\wtmod \Hom(T, \funJ M)=0$, for any $M\in\CM\algB$.
 
\begin{proposition}\label{Lem:exactWt}
The functor $\wtmod\colon \CM\algA\to \fd \algA$ is exact.
\end{proposition}

\begin{proof}
For any short exact sequence in $\CM\algA$
\[
  \ShExSeq{X}{Y}{Z},
%  0 \lra X \lra Y \lra Z \lra 0,
\]
applying $\funJ\efunct$ gives an exact sequence
\[
  \ShExSeq{\funJ \efunct X}{\funJ \efunct Y}{\funJ \efunct Z},
% 0\lra \funJ \efunct X \lra \funJ \efunct Y \lra \funJ \efunct Z \lra 0.
\]
which is split, because $\funJ \efunct X$ is injective in $\CM\algB$.
Hence the sequence remains exact on applying $\Hom(T, -)$.
Since the functoriality of $\wtmod$ 
comes from a natural transformation $\idfun\to \Hom(T, \funJ \efunct-)$,
we have a commutative diagram
\[
\begin{tikzpicture}[xscale=2.75, yscale=1.5]
\pgfmathsetmacro{\xoff}{0.45}
\pgfmathsetmacro{\xofff}{0.55}
\pgfmathsetmacro{\yoff}{0.75}
\draw (1,3) node (a3) {$X$}; 
\draw (1,2) node (a2) {$Y$}; 
\draw (1,1) node (a1) {$Z$}; 
\draw (2,3) node(b3) {$\Hom(T, \funJ \efunct X)$};
\draw (2,2) node(b2) {$\Hom(T, \funJ \efunct Y)$};
\draw (2,1) node(b1) {$\Hom(T, \funJ \efunct Z)$};
\draw (3,3) node(c3) {$\wtmod X$};
\draw (3,2) node(c2) {$\wtmod Y$};
\draw (3,1) node(c1) {$\wtmod Z$};
\draw (1-\xoff,3) node (z3) {$0$}; 
\draw (1-\xoff,2) node (z2) {$0$}; 
\draw (1-\xoff,1) node (z1) {$0$}; 
\draw (3+\xofff,3) node (zz3) {$0$}; 
\draw (3+\xofff,2) node (zz2) {$0$}; 
\draw (3+\xofff,1) node (zz1) {$0$}; 
\draw (1,3+\yoff) node (a4) {$0$}; 
\draw (2,3+\yoff) node (b4) {$0$}; 
\draw (1,1-\yoff) node (a0) {$0$}; 
\draw (2,1-\yoff) node (b0) {$0$}; 
\foreach \t/\h in {a3/a2, a2/a1, b3/b2, b2/b1, c3/c2, c2/c1, a1/b1, b1/c1, a2/b2, b2/c2, a3/b3, b3/c3} 
   \draw[cdarr] (\t) to (\h);
\foreach \t/\h in {z1/a1, z2/a2, z3/a3, a4/a3, b4/b3, a1/a0, b1/b0, c3/zz3, c2/zz2, c1/zz1} 
   \draw[cdarr] (\t) to (\h);
\end{tikzpicture}
\]
The Snake Lemma then implies that
\[
  \ShExSeq{\wtmod X}{\wtmod Y}{\wtmod Z}
% 0 \lra  \wtmod X \lra  \wtmod Y \lra  \wtmod Z \lra 0
\]
is exact, as required.
\end{proof} 

\begin{remark} \label{rem:WeightClass}
Note that the functors $\funG$ and $\funK(T, \efunct -)$ are not exact and
indeed a proof following the strategy of \Cref{Lem:exactWt}
would fail because the functor $\Hom(T, \efunct -)$ is not exact.
In practical terms, the exactness of $\wtmod$ means that the class $[\wtmod X]$ 
in the Grothendieck group $\Grot(\fd\algA)$ depends only on the class $[X]$
in the Grothendieck group $\Grot(\CM\algA)$ and there is an explicit formula:
see \eqref{eq:beta-wt}.
\end{remark}

\begin{remark}
More generally, the (bi)functor $\funK(-, -)$ is additive in both slots, 
but neither $\funK(M, -)$ nor $\funK(-, M)$ is exact, for most $M\in\CM\algB$.
Indeed, applying one of these functors to a short exact sequence leads to 
a complex, which typically has homology at the middle term. 
In \Cref{Sec:11}, we will prove precise statements on the exactness 
of these functors, which will help us to understand mutations of $\kapvec(T, M)$.
\end{remark}

%%%% =========================================================== %%%%
\section{Grothendieck groups $\Grot(\fd \algA)$ and $\Grot(\CM\algA)$} \label{Sec:5}
%%%% =========================================================== %%%%

%Denote by $\Grot(\fd\algA)$ the Grothendieck group of finite dimensional $\algA$-modules
%and by $\Grot(\CM\algA) $ the Grothendieck group of Cohen--Macaulay $\algA$-modules.
In this section, we define (non-inverse) maps in both directions between the Grothendieck groups,
$\Grot(\fd\algA)$ and $\Grot(\CM\algA)$,
and discuss properties of the maps. 
We also denote by $\Grot(\proj\algA)$ the Grothendieck group of projective $\algA$-modules.

\begin{remark}\label{rem:lat-ranks}
Recall from \Cref{rem:gab-quiv}, that $\Grot(\fd\algA)\isom\ZZ^{Q_0}$ via the standard basis 
$\{ [S_i] \st i\in Q_0\}$ given by the classes of simple $\algA$-modules,
where $Q$ is the Gabriel quiver of $\algA$.
We will thus (often implicitly) make the identification $\Grot(\fd\algA)\isom\ZZ^{Q_0}$.

Note also that $\Grot(\proj\algA) \isom\ZZ^{Q_0}$ via the standard basis 
$\{ [Ae_i] \st i\in Q_0\}$ given by the classes of indecomposable projective $\algA$-modules,
but $\Grot(\proj\algA)$ is more naturally dual to $\Grot(\fd\algA)$,
with the basis of projectives being dual to the basis of simples.

Since projective $\algA$-modules are Cohen--Macaulay  %(i.e.~free over $\ring$) defined in sect3.
and by \Cref{prop:globdim},  $\algA$ has finite global dimension,
the inclusion $\proj\algA \subset \CM\algA$ induces
an isomorphism $\Grot(\proj\algA) \isom \Grot(\CM\algA)$.
Hence $\Grot(\CM\algA)$ also has a basis $\{ [Ae_i] \st i\in Q_0\}$
and so has the same rank as $\Grot(\fd\algA)$,
but again is more naturally dual to it. % namely $\size{Q_0}$.
We will mostly work explicitly with $\Grot(\CM\algA)$ and only identify it
with $\Grot(\proj\algA)$ when necessary.
\end{remark}

Note that, we denote by $[X]$ the class of a module $X$ in a Grothendieck group,
such as $\Grot(\CM\algA)$ and $\Grot(\fd\algA)$, without distinguishing the group.
In the special case when $X=\Hom_\algB(T, M)\in \CM\algA$, 
for $M\in\CM\algB$, we will write $[X]$ as $[T, M]$ for short. 

% ======================================================================
\subsection{Maps between the Grothendieck groups} \label{Sec:seq}
% ======================================================================

Recall, from \Cref{Sec:4} (in particular, \eqref{eq:wtmod-diag}),
the functor $\wtmod\colon \CM\algA\to \fd \algA$ constructed through the short exact sequence
\begin{equation}\label{eq:defwtmod}
 \ShExSeq {X} {\Hom(T,\funJ\efunct X)}{\wtmod X}.
\end{equation}
Since $\wtmod$ is exact, by \Cref{Lem:exactWt}, it induces a map
\begin{equation}\label{eq:wt-class}
  \wt\colon \Grot(\CM\algA) \to \Grot(\fd\algA) \colon [X] \mapsto [\wtmod X].
\end{equation}
In particular, by \Cref{Lem:kerwt}, we have
\begin{equation}\label{eq:wtTM}
  \kapvec(T,M)= \wt[T, M].
\end{equation}
In the other direction, there is a natural map
\begin{equation}\label{eq:def-beta}
  \beta\colon \Grot(\fd\algA)\to \Grot(\CM\algA)
  \colon [X] \mapsto [\procov{X}]-[\Syz X]
\end{equation}
defined from a minimal partial presentation
\[ 
  \ShExSeq{\Syz X}{\procov{X}}{X}.
\]
We choose a minimal partial presentation just to have an initially well-defined map,
but in fact $\beta$ can defined equally well by any $\CM$-approximation.

\begin{lemma}\label{Lem:classvsresol}
Let $X$ be a finite dimensional $\algA$-module. 
For any short exact sequence 
%$\ShExSeq{Z}{Y}{X}$, 
$0\to Z\to Y\to X\to 0$, with $Z,Y\in \CM\algA$, 
we have $\beta[X]=[Y]-[Z]$. 
\end{lemma}

\begin{proof}
As $\procov{X}$ is projective, the map $\procov{X}\to X$ in the partial presentation 
lifts to a map $\procov{X}\to Y$ which then induces a map of short exact sequences
\[
\begin{tikzpicture}[xscale=1.4, yscale=1.4] 
\draw (1,2) node(b1) {$\Syz X$};
\draw (2,2) node(b2) {$\procov{X}$};
\draw (3,2) node(b3) {$X$};
\draw (1,1) node (c1) {$Z$};
\draw (2,1) node (c2) {$Y$};
\draw (3,1) node (c3) {$X$};
\draw (0,1) node (c0) {$0$};
\draw (4,1) node (c4) {$0$.};
\draw (0,2) node (b0) {$0$};
\draw (4,2) node (b4) {$0$};
\foreach \t/\h in {b1/b2, b2/b3, c1/c2, c2/c3, b2/c2, b1/c1} 
   \draw[cdarr] (\t) to (\h);
\draw[equals] (b3) to (c3);
\foreach \t/\h in {b0/b1, b3/b4, c0/c1, c3/c4} 
   \draw[cdarr] (\t) to (\h);
\end{tikzpicture}
\]
Since the right-hand vertical map is an isomorphism, the left-hand square is Cartesian,
that is, it induces a short exact sequence in $\CM\algA$
\[
  \ShExSeq {\Syz X} {Z\oplus \procov{X}} {Y}
\]
so that $[\Syz X] + [Y] =  [Z] + [\procov{X}]$,
i.e. $[Y] - [Z] = [\procov{X}] - [\Syz X] = \beta[X]$,
as required.
\end{proof}

\newcommand{\Epair}[2]{\langle {#1},{#2}\rangle_{\mathsf E}}
\newcommand{\Tpair}[2]{\langle {#1},{#2}\rangle_{\mathsf T}}

\begin{remark}\label{rem:beta-stuff}
If we identify $\Grot(\CM\algA)$ with $\Grot(\proj\algA)$,
then $\beta\colon \Grot(\fd\algA)\to \Grot(\proj\algA)$ is induced by projective resolution.
If we write $\beta$ as a matrix $(\beta_{ij})$, 
using the (dual) bases of simples $[S_i]$ and projectives $[Ae_i]$,
then $\beta_{ij}=\Epair{S_i}{S_j}$ for the pairing on $\Grot(\fd\algA)$ given by
\[
 \Epair{X}{Y} = \sum_{\ell\geq 0} \dim\Ext^{\ell}(X,Y).
\]
Alternatively, $\beta_{ij}=\Tpair{S_i\op}{S_j}$, for the pairing on $\Grot(\fd\algA\op)\times\Grot(\fd\algA)$ given by
\[
 \Tpair{X}{Y} = \sum_{\ell\geq 0} \dim\Tor_{\ell}(X,Y).
\]
But this pairing can also be computed using the map $\beta\op\colon \Grot(\fd\algA\op)\to \Grot(\proj\algA\op)$
induced by projective resolution for the opposite algebra, showing that $\beta_{ij}=\beta_{ji}\op$.
Thus $\beta\op$ can be identified with $\beta\dual$, if we also identify 
$\Grot(\proj\algA\op)$ with $\Grot(\fd\algA)\dual$ and $\Grot(\fd\algA\op)$ with $\Grot(\proj\algA)\dual$,
using tensor product. 
\end{remark}

Define
\begin{equation} \label{eq:defrank}
  \rkk\colon \Grot(\CM\algA)\to  \ZZ 
  \colon[X]\mapsto \rnk{\algA} X
\end{equation}
which is well-defined because, in particular, 
$\rnk{\algA}$ is additive on short exact sequences in $\CM\algA$.
Let $\Mzero=\ker \rkk$.
Since $\rnk{\algA} \Hom(T, \modJ)=\rnk{\algB} \modJ =1$, we have
\begin{equation}\label{eq:decKCMA}
\Grot(\CM\algA)= \ZZ [T, \modJ] \oplus \Mzero.
\end{equation}

Also define
\begin{equation}\label{eq:def-delta}
  \delta = (\rnk{\algB} T_j)_{j\in Q_0} \in \ZZ^{Q_0}\isom \Grot(\fd\algA) 
\end{equation}
and
\begin{equation}\label{eq:def-Nstar}
  \Nstar = \bigl\{ x \st x_\vstar=0 \bigr\} \subspc \ZZ^{Q_0}\isom \Grot(\fd\algA),
\end{equation}
where $\vstar\in Q_0$ is the chosen boundary vertex, as in \Cref{Sec:4}.
Note that $\img\wt\subspc \Nstar$,
by \Cref{Lem:somefacts}\itmref{itm:facts3}.
%and $\Grot(\fd \algA)\isom \ZZ^{Q_0}$, using the basis of simples $\{[S_j] \st j\in Q_0\}$.
Since $\delta_\vstar=\rnk{\algB} T_\vstar=1$, we have
\begin{equation}\label{eq:decKfdA}
\Grot(\fd\algA) =  \ZZ \delta \oplus \Nstar.
\end{equation}

\begin{remark}\label{rem:dual}
Using the identification $\Grot(\CM\algA)=\Grot(\proj\algA)$
and the natural duality between $\Grot(\proj\algA)$ and $\Grot(\fd \algA)$, 
as in \Cref{rem:lat-ranks},
%for which $\{[S_j] \st j\in Q_0\}$ and $\{[\algA e_j] = [T,T_j] \st j\in Q_0\}$ are dual bases.
we can see that the two maps $\delta\colon \ZZ\to \Grot(\fd\algA)\colon n\mapsto n\delta$ 
and $\rkk\colon \Grot(\CM\algA)\to \ZZ$ are dual, 
because $\rkk [T,T_j] =  \rnk{\algB} T_j$.

We also have an isomorphism $\Mzero\to\Nstar\dual$ by restriction of linear functions.
In practice this means you write an element of $\Mzero$ in the basis 
$\{[\algA e_j] \st j\in Q_0\}$ of $\Grot(\CM\algA)$
and drop the coefficient of $[\algA e_\vstar]$. 
\end{remark}

\begin{proposition} \label{Thm:2exactseq}
The following sequence is exact and, in particular, $\img\beta=\Mzero$.
\begin{equation}\label{eq:nat-ex-seq}
\begin{tikzpicture}[scale=2.4, baseline=(bb.base)]
\coordinate (bb) at (0,-0.05);
\pgfmathsetmacro{\zeps}{0.5}
\pgfmathsetmacro{\mideps}{0.2}
\draw (-1-\zeps, 0) node (b0){$0$};
\draw (-1, 0) node(b1) {$\ZZ$};
\draw (0, 0) node(b2) {$\Grot(\fd\algA) $};
\draw (1+\mideps, 0) node(b3) {$\Grot(\CM\algA)$};
\draw (2+\mideps, 0) node (b4) {$\ZZ$};
\draw (2+\mideps+\zeps, 0) node (b5) {$0$.};
\draw[cdarr] (b0) to (b1);
\draw[cdarr] (b1) to node [above] {\small $\delta$} (b2);
\draw[cdarr] (b2) to node [above] {\small $\beta$} (b3);
\draw[cdarr] (b3) to node [above] {\small $\rkk$} (b4);
\draw[cdarr] (b4) to (b5);
\end{tikzpicture}
\end{equation}
Moreover, the restriction $\beta\colon \Nstar\to\Mzero$ is an isomorphism,
with inverse given by the restriction $-\!\wt\colon \Mzero\to\Nstar$.
Hence $\img\wt=\Nstar$.
\end{proposition}

\begin{proof}
First note that the map $\delta$ is injective, because the element $\delta$ is non-zero
and $\Grot(\fd\algA)$ is torsion-free.
Also $\rkk$ is surjective, because there are elements of $\Grot(\CM\algA)$ with $\rkk=1$,
such as $[T,\modJ]$.

To see that $\rkk\compo \beta=0$, consider any CM-approximation of $X\in \fd\algA$
\begin{equation}\label{eq:partialpres}
  \ShExSeq{Z}{Y}{X},
%0 \to \Syz X \lra \procov{X} \lra X \to 0
\end{equation}
such as a partial presentation. 
Then $\beta [X]=[Y]-[Z]$, by \Cref{Lem:classvsresol},
and $\rnk{\algA} X=0$, so $\rkk[Y]=\rkk[Z]$,
that is, $\rkk\beta [X]=0$, as required.

Furthermore, from \eqref{eq:defwtmod} and \Cref{Lem:classvsresol}, 
we see that, for any $Z\in\CM\algA$,
\begin{equation}\label{eq:beta-wt}
 \beta \wt [Z] = \rkk [Z] [T,\modJ] - [Z]
\end{equation}
because $\rnk{\algA} Z = \rnk{\algB} \efunct Z$,
by \Cref{Lem:ranks}\itmref{itm:rks3},
and so, by \eqref{eq:isomPJ-rkB},
\[
  \Hom(T,\funJ\efunct Z)\isom \Hom(T,\modJ)^{\rnk{\algA} Z}.
\]
When we restrict $\wt$ to $\Mzero=\ker \rkk$, we deduce from \eqref{eq:beta-wt} 
that $\beta\compo \wt = -\idmap_{\Mzero}$,
and thus that $\img\beta= \Mzero$, 
so the sequence \eqref{eq:nat-ex-seq} is exact at $\Grot(\CM\algA)$.

To see that $\beta\compo \delta=0$, that is, the element $\delta\in\ker\beta$,
%To verify the claim that $\beta(\delta)=0$,
observe that multiplication by $t \in \ring$ is injective on any $Z\in\CM\algA$,
so $Z\isom tZ$.
Hence $\beta[Z/tZ]=[Z]-[tZ]=0$ in $\Grot(\CM\algA)$.
On the other hand, $[Z/tZ]=(\rnk{\algA} Z)\,\delta$,
because
\[
 \dim e_iZ / t e_iZ
  = \rnk{\ring} e_iZ 
 = (\rnk{\algA} Z) (\rnk{\algB} T_i),
\]
by \Cref{Lem:ranks}\itmref{itm:rks4}.
Thus $\beta(\delta)=0$, as required.

Furthermore, since $\Grot(\CM\algA)$ and $\Grot(\fd\algA)$ have the same rank (\Cref{rem:lat-ranks}),
we see that $\ker\beta$ has rank~$1$.
But $\delta$ is primitive, since $\delta_\vstar=1$,
so $\delta$ must generate $\ker\beta$, 
that is, the sequence \eqref{eq:nat-ex-seq} is exact at $\Grot(\fd\algA)$.

In particular, this means that the restriction of $\beta$ to $\Nstar$ is injective, by \eqref{eq:decKfdA},
and so, since we already know that $\beta\compo \wt = -\idmap_{\Mzero}$,
the restrictions of $\beta$ to $\Nstar$ and $-\!\wt$ to $\Mzero$
must be mutual inverses, thereby completing the proof of the second part.
\end{proof}

\begin{remark}\label{rem:alt-proof}
In the proof of \Cref{Thm:2exactseq}, we deduced indirectly that,
after restriction to $\Nstar$, we have $\wt\compo \beta= -\idmap_{\Nstar}$.
It is also possible to prove this directly.

Consider any $X\in\fd\algA$ with $e_\vstar X=0$, so that $[X]\in\Nstar$.
Given a CM-approximation
\[
 0\lra Z \lraa{\lambda} Y\lra X\lra 0,
\]
the induced map $e_\vstar Z\to e_\vstar Y$ is an isomorphism,
inducing an isomorphism 
\[
\Hom(T, \funJ\efunct Z) \isom \Hom(T, \funJ\efunct Y).
\]
Thus, following \eqref{eq:defwtmod}, we have the following commutative diagram, 
\[
\begin{tikzpicture}[xscale=3.2,yscale=1.6]
\draw (-0.5, 2) node (b0){0};
\draw (0,2) node(b1) {$Z$};
\draw (1,2) node(b2) {$\Hom(T, \funJ\efunct Z)$};
\draw (2,2) node(b3) {$\wtmod Z$};
\draw (2.5,2) node (b4) {$0$};
\draw (-0.5,1) node (c0){0};
\draw (0,1) node(c1) {$Y$};
\draw (1,1) node(c2) {$\Hom(T, \funJ\efunct Y)$};
\draw (2,1) node(c3) {$\wtmod Y$};
\draw (2.5,1) node (c4) {$0$.};
\draw[cdarr] (b1) to node[auto]{$\lambda$}  (c1);
\draw[cdarr] (b2) to node[auto]{$\isom$} (c2);
\draw[cdarr] (b3) to node[auto]{$\wtmod \lambda$}  (c3);
\foreach \t/\h in {b0/b1, b1/b2, b2/b3, b3/b4, c0/c1, c1/c2, c2/c3, c3/c4, b1/c1}
  \draw[cdarr] (\t) to (\h);
\end{tikzpicture}
\]  
By the Snake Lemma, $\ker\wtmod\lambda \isom \cok \lambda =X$ and so
\[
 \wt \beta [X] 
 = \wt [Y] - \wt [Z] 
 = [\wtmod Y] - [\wtmod Z]
 = -[\ker\wtmod\lambda]
 =-[X].
\]
\end{remark}

\begin{remark}\label{rem:unnat-ex-seq}
As well as the natural (i.e.~choice independent) exact sequence \eqref{eq:nat-ex-seq},
we have an exact sequence in the reverse direction,
which does depend on the choice of boundary vertex $\vstar\in Q_0$,
\begin{equation}\label{eq:unnat-ex-seq}
\begin{tikzpicture}[scale=2.5, baseline=(bb.base)]
\coordinate (bb) at (0,-0.05);
\pgfmathsetmacro{\zeps}{0.5}
\pgfmathsetmacro{\mideps}{0.2}
\draw (-1-\zeps, 0) node (b0){$0$};
\draw (-1, 0) node(b1) {$\ZZ$};
\draw (0, 0) node(b2) {$\Grot(\CM\algA)$};
\draw (1+\mideps, 0) node(b3) {$\Grot(\fd\algA)$};
\draw (2+\mideps, 0) node (b4) {$\ZZ$};
\draw (2+\mideps+\zeps, 0) node (b5) {$0$.};
\draw[cdarr] (b0) to (b1);
\draw[cdarr] (b1) to node [above] {\small $[T,\modJ]$} (b2);
\draw[cdarr] (b2) to node [above] {\small $\wt$} (b3);
\draw[cdarr] (b3) to node [above] {\small $[T,\modP]$} (b4);
\draw[cdarr] (b4) to (b5);
\end{tikzpicture}
\end{equation}
Here, as before, we just write $[T,\modJ]$ for the map 
$\ZZ\to\Grot(\CM\algA)\colon n \mapsto n[T,\modJ]$.
We also identify the element $[T,\modP]=[\algA e_\vstar]\in \Grot(\proj\algA)$
with the corresponding (dual) map $\Grot(\fd\algA)\to\ZZ\colon x\mapsto x_\vstar$,
whose kernel is $\Nstar$.
Thus we have already seen that $\img\wt=\Nstar$, i.e.~ the sequence \eqref{eq:unnat-ex-seq}
is exact at $\Grot(\fd\algA)$.
On the other hand, the sequence is exact at $\Grot(\CM\algA)$ by \Cref{Lem:kerwt},
since we know that $\ker\wt$ has rank~$1$ and $[T,\modJ]$ is primitive,
because $\rkk[T,\modJ]=1$, or equally from the splitting~\eqref{eq:decKCMA}.
Note also that, as a map, $[T,\modP]$ is surjective, because 
$[T,\modP]\delta = \delta_\vstar =1$,
or equally from the splitting~\eqref{eq:decKfdA}.
\end{remark}

% ======================================================================
%\subsection{A partial order on $\Grot(\CM\algA)$}\label{sec:part-ord}
% ======================================================================

From \Cref{Thm:2exactseq} and \eqref{eq:beta-wt},
we have an isomorphism
\begin{equation}\label{eq:wt-hat}
 \wthat\colon \Grot(\CM\algA) \to \ZZ\oplus \Nstar\colon [Z]\mapsto (\rkk[Z],\wt[Z])
\end{equation}
with inverse
\begin{equation}\label{eq:wt-hat-inv}
 \betahat\colon \ZZ\oplus \Nstar \to \Grot(\CM\algA) \colon (r, [X]) \mapsto r[T,\modJ]-\beta[X].
\end{equation}
Thus we can enhance \eqref{eq:wtTM}, using \Cref{Lem:ranks}\itmref{itm:rks1}, to
\begin{equation}\label{eq:wthatTM}
  \wthat\,[T,M]=(\rnk{B} M,\kapvec(T,M)).
\end{equation} 
%by \Cref{Lem:ranks}\itmref{itm:rks1} and \Cref{Lem:kerwt}. %(with \Cref{def:KTM}).

Note that $\Grot(\fd\algA)\isom \ZZ^{Q_0}$. For any $d\in \Grot(\fd\algA)$, 
writing $d=(d_i)_{i\in Q_0}$,
we can define a partial order on $\Grot(\fd\algA)$ by 
\[
  d \leq d' 
  \quad\text{if $d_i\leq d_i'$, for all $i\in Q_0$}, 
\] 
and can also restrict this partial order to $\Nstar\subset\Grot(\fd\algA)$.

We can then use the isomorphism  $\wthat\colon\Grot(\CM\algA) \to \ZZ\oplus \Nstar$ 
from \eqref{eq:wt-hat} to induce (lexicographically) a partial order on $\Grot(\CM\algA)$.
Explicitly, this gives the following.

\begin{definition}\label{def:part-ord}
For $\lambda, \lambda'\in \Grot(\CM\algA)$, we say $\lambda\leq \lambda'$ 
if $\rkk\lambda <\rkk\lambda'$ or $\rkk\lambda =\rkk\lambda' $ and $\wt\lambda \leq \wt\lambda'$.
\end{definition}

An alternative would be to use the isomorphism $\beta\colon\Nstar\to\Mzero$ 
to induce a partial order on $\Mzero$ and then use the splitting $\Grot(\CM\algA)=\ZZ\oplus \Mzero$
to give a lexicographic order on $\Grot(\CM\algA)$.
This is not quite the same, because the inverse of $\wt$ is $-\beta$,
so the partial order on elements of fixed rank is reversed.

% ======================================================================
\subsection{Projective resolutions of simples}\label{subsec:projresol}
% ======================================================================

Let  $T\moreq  \textstyle\bigoplus_{j\in Q_0} T_j$ be 
a (not necessarily basic) cluster tilting object in $\GP\algB$.
For any vertex $i$ of the Gabriel quiver $Q$ of $A=\End(T)\op$
(see \Cref{rem:gab-quiv}), let 
\begin{equation}\label{eq:EiFi}
  E_i=\bigoplus_{j\from i} T_j 
  \quadand
  F_i=\bigoplus_{j\to i} T_j .
\end{equation}
When $i$ is an interior vertex, we know that $E_i, F_i\in \add (T/ T_i)$ 
and that they have no common summands, 
because $Q$ has no loops or 2-cycles
at mutable vertices,
since $\GP\algB$ has a cluster stucture, 
by \Cref{rem:cluster-structure}.
In particular, we have the following two mutation sequences for $T_i$,
\begin{equation}\label{eq:mut-two}
\begin{aligned}
 &\ShExSeq {T_i^*} {E_i} {T_i^{\phantom*}}, \\
 &\ShExSeq {T_i^{\phantom*}} {F_i } {T_i^*}.
\end{aligned}
\end{equation}
Recall, from \Cref{rem:beta-stuff}, 
that the map $\beta\colon \Grot(\fd\algA) \to \Grot(\CM\algA)$ of \eqref{eq:def-beta}
can also be given by projective resolution.

\begin{proposition} \label{Prop:projres}
Write $\Radj=\Hom(T, -)$.
From $E_i,F_i$ as in \eqref{eq:EiFi}, 
we obtain projective resolutions of the simple $\algA$-module $S_i$ as follows.
\begin{enumerate}
\item\label{itm:pr1}
If $i$ is an interior vertex, then the resolution is
\begin{equation}\label{eq:projres1}
 0 \lra \Radj T_i \lra \Radj F_i \lra \Radj E_i \lra \Radj T_i \lra S_i \lra 0.
 \end{equation}
Consequently, $\projdim S_i=3$ and 
\[
  \beta [S_i] = [T, F_i] - [T, E_i]
 = \sum_{j\to i}\, [T, T_j] - \sum_{i\to j}\, [T, T_j].
\]
\item\label{itm:pr2}
If $i$ is a boundary vertex, then the resolution is 
 \begin{equation}\label{eq:projres2}
 0\lra \Radj K_i \lra \Radj E_i \lra \Radj T_i \lra S_i \lra 0,
\end{equation}
for some $K_i\in \add T$.
Consequently, $\projdim S_i\leq 2$ and 
\[\beta [S_i]=[T, T_i] -[T, E_i]+[T, K_i].\]
\end{enumerate}
\end{proposition}

\begin{proof}
\itmref{itm:pr1} 
Applying $\Hom(T, -)$ to the two mutation sequences,
and splicing the resulting sequences at $\Radj T_i^*$, 
gives the projective resolution \eqref{eq:projres1},
because $\Ext^1(T, T_i^*)=S_i$.
The resolution is minimal because adjacent terms have no common summands
and thus $\projdim S_i=3$.

\itmref{itm:pr2}  
When $i$ is a boundary vertex, $T_i$ is a projective $\algB$-module and 
the minimal $\add T$-approximation of $\rrad T_i$  has the following form
\begin{equation}\label{Eq:projres3}
0\lra K_i \lra E_i \lraa{f} \rrad T_i\lra 0.
\end{equation}
Note that $f$ is surjective, because $_{\algB}\algB\in \add T$.
Applying $\Hom(T, -)$ leads to 
\[
0\lra \Hom(T, K_i) \lra \Hom(T,  E_i) \lraa{f_*} \Hom(T, \rrad T_i)\lra \Ext^1(T, K_i)\lra 0.
\]
As $f$ is an $\add T$-approximation, $f_*$ is surjective. Hence $\Ext^1(T, K_i)=0$ and so $K_i\in \add T$. 
Applying $\Hom(T, -)$ to the sequence \eqref{Eq:projres3} 
and to the embedding $\rrad T_i  \lra T_i$,
and splicing the results at $\Radj(\rrad T_i)$, gives the projective resolution \eqref{eq:projres2},
because the cokernel of the (injective) map 
$\Hom(T,  \rrad T_i )\to \Hom(T, T_i)$ is isomorphic to $ S_i$. 
\end{proof}

\Cref{Prop:projres}\itmref{itm:pr1} shows that, 
in the bases of simples and projectives, 
the map $\beta$ is given on interior vertices
by the cluster exchange matrix (up to a sign).
In particular, this means that it is a candidate for Fock-Goncharov's $p^*$ map, as in \cite[(2.1)]{GHK}.

%%%% =========================================================== %%%%
\section{Cluster characters, partition functions and flow polynomials}\label{Sec:6}
%%%% =========================================================== %%%%
\newcommand{\ddel}[1]{{d_{#1}}}

In this section, we construct various cluster characters on $\GP B$,
with natural extensions to $\CM\algB$.  
In particular, we define a generalised partition function $\ptfn$ and 
a generalised flow polynomial $\fp$. 
In \Cref{sec:classical}, we will prove that $\ptfn_M$ and $\fp_M$ 
coincide with combinatorially defined partition functions and flow 
polynomials, when $M$ and all the indecomposable summands of $T$ have rank $1$.
This justifies the names of these cluster characters.  

% ======================================================================
\subsection{Fu--Keller's cluster character}\label{sec:FKclucha}
% ======================================================================

Following Fu--Keller~\cite[Def.~3.1]{FK} and Palu~\cite{Palu}, 
what we mean by a `cluster character' is as follows. 

\begin{definition}\label{def:clucha}
 A \emph{cluster character} on $\GP\algB$ with values in a commutative ring $H$
 is map $\Phi\colon \GP\algB\to H\colon M\mapsto \Phi_M$ such that 
 \begin{enumerate}
 \item\label{itm:cc1}
 if $M$ and $N$ are isomorphic, then $\Phi_M=\Phi_N$,
 \item\label{itm:cc2}
 $\Phi_{M\oplus N}=\Phi_M\Phi_N$, for any $M, N$,
 \item\label{itm:cc3}
 if $\dim\Ext^1(M, N)=1$ and the short exact sequences
\begin{equation}\label{eq:two-ses}
 \ShExSeq {M} {E} {N}
%0\lra M\lra E\lra N\lra 0
\quadand
 \ShExSeq {N} {F} {M}
%0\lra N\lra E'\lra M\lra 0,
\end{equation} 
are non-split, then $\Phi_M\Phi_N=\Phi_E +\Phi_F$.
\end{enumerate}
\end{definition}

\begin{remark}\label{rem:new-clucha}
If $\Phi\colon \GP\algB\to H$ is a cluster character 
and $\alpha\colon H\to H'$ is any ring homomorphism,
then $\alpha\compo \Phi$ is also a cluster character.
We will often consider the case when $H=\CC[L]$, 
the ring of formal Laurent polynomials with exponents in a lattice~$L$. 
In this case, any lattice map $\lambda\colon L\to L'$
induces a `monomial change of variables'  
$\CC[\lambda]\colon \CC[L]\to\CC[L']\colon \xvar^\ell\mapsto \xvar^{\lambda(\ell)}$.
We will write $\pbfun{\lambda}{\Phi}$ for $\CC[\lambda]\compo \Phi$
and for the map $(-1)\colon L\to L$, i.e. multiplication by $-1$,
we will write $\widetilde\Phi$ for $\pbfun{(-1)}{\Phi}$.

Note that $\CC[L]$ is the coordinate ring of an algebraic torus $\cantor{L}$,
whose character lattice is naturally $L$. 
We can thus regard the formal monomial $\xvar^\ell\in \CC[L]$ 
as the function $\cantor{L}\to \CC$ given by the character $\ell\in L$,
so that the formal multiplication rule $\xvar^\ell \xvar^m=\xvar^{\ell+m}$ becomes a true formula.
We can then see that $\widetilde\Phi$ is the pull-back of $\Phi$
along the involution $\cantor{L}\to \cantor{L}\colon \xvar\mapsto \xvar^{-1}$.
\end{remark}

We adapt Fu--Keller's cluster character \cite[Rem.~3.5]{FK} to $\GP\algB$,
noting that $\GP\algB$ is not Hom-finite, so is strictly outside their declared context,
but it is stably Hom-finite, which is all that is needed
to lift Palu's construction from the stable category. 
Thus, for any cluster tilting object $T$ in $\GP\algB$, we define 
\begin{equation}\label{eq:FK-CC}
  \Phi^T\colon \CM\algB\to \CC[\Grot(\CM\algA)] \colon M\mapsto \Phi^T_{M}
 =\xvar^{[T, M]}\sum_{d} \euler \bigl( \qvGr{d}\Ext^1(T, M) \bigr)\xvar^{-\beta(d)}, 
\end{equation}
where $\qvGr{d}\Ext^1(T, M)$ is the Grassmannian of $\algA$-submodules of $\Ext^1(T, M)$ 
of dimension vector $d\in\ZZ^{Q_0}$
%(or, equivalently, of class $d\in\Grot(\fd\algA)$)
and $\euler$ is the Euler characteristic. 
By \Cref{lem:ExtBinExtC}, extension groups in $\CM\algB$ are finite
dimensional, and so $\Phi$ is well defined and the sum is finite.
Although the map $\Phi$ is defined on $\CM\algB$,
it can only meaningfully be a cluster character when restricted to $\GP \algB$.  

\begin{remark}\label{rem:FKstuff}
In \cite{FK}, the exponents $[T, M]$ and $\beta(d)$ are written
in the basis of projectives $[Ae_i]=[T, T_i]$
(cf. \Cref{rem:beta-stuff}), 
that is, using $\Grot(\CM\algA)\isom\Grot(\proj\algA)$.
Thus \eqref{eq:FK-CC} is written explicitly in the variables $\xvar_i=\xvar^{[Ae_i]}$.
In addition, using $\Grot(\proj\algA) \isom\Grot(\add T)$,
the leading exponent $[T,M]$ is identified with the `index' of $M$ \cite[\S4]{FK}, 
that is, the class of an $\add T$-presentation of~$M$.
Note that \cite[\S4]{FK} uses the term `$g$-vector' to mean just the mutable
part of the index, following \cite{FZ4}.
\end{remark}

The proof of \cite[Thm.~3.3]{FK} shows the following.

\begin{proposition} \label{prop:Phi-clu-char}
$\Phi^T$ is a cluster character, when restricted to $\GP(B)$.
\end{proposition}

The following lemma is standard,
but we include a proof for completeness.

\begin{lemma}\label{lem:Phi-monom}
If $M\in\GP\algB$, then 
$\Phi^T_M$ is a monomial if and only if $M\in \add T$, 
in which case $\Phi^T_M= \xvar^{[T, M]}$.
In particular, this holds when $M$ is projective.
\end{lemma}

\begin{proof}
In \eqref{eq:FK-CC}, the submodules $0$ and $\Ext^1(T, M)$
both give terms with coefficient~1.
Thus $\Phi^T_M$ is a monomial if and only $\Ext^1(T, M)=0$, 
which is equivalent to $M\in\add T$,
since $T$ is a cluster tilting object.
The formula is then immediate from \eqref{eq:FK-CC}.
Finally $B$ is a summand of $T$, so $\add\algB\subset\add T$, 
and so the condition holds for all projectives.
\end{proof}

% ======================================================================
\subsection{Generalised partition functions}\label{sec:gen-pt-fn}
% ======================================================================

For any $M\in \CM\algB$, define 
\begin{equation}\label{eq:sumrange}
\begin{aligned}
  \sumrange{M} 
  & = \{ [X] \in \Grot(\CM\algA) \st X \leq \Hom(T, M)\;\text{and}\; eX=M \} \\
  & = \{ [X] \in \Grot(\CM\algA) \st \algA\cdot M \leq X \leq \Hom(T, M) \},
\end{aligned}
\end{equation}
where the equality of the two sets follows from \Cref{Thm:EndA}\itmref{itm:EA6}.
For any $[X]\in \sumrange{M}$, 
by which we implicitly mean $X$ satisfies the conditions in \eqref{eq:sumrange},
the class, in $\Grot(\fd\algA)$,
\begin{equation}\label{eq:Dd-bij}
\ddel{[X]}= [X/(\algA\cdot M)] = \wt[\algA\cdot M] - \wt[X],
\end{equation}
depends only on the class $[X]$.
% by Remark \ref{rem:WeightClass}.
Furthermore, $\ddel{[X]}=\ddel{[Y]}$ if and only if $[X]=[Y]$. 
So there is a bijection $[X]\mapsto \ddel{[X]}$ between $\sumrange{M}$ 
and the set of classes of submodules of $\Hom(T,M)/\algA\cdot M$,
which equals $\sHom(T,M)$, by \Cref{Thm:EndA}\itmref{itm:EA5}.

For any cluster tilting object $T$ in $\GP\algB$, define
\begin{equation}\label{eq:part-fun}
\begin{aligned}
  \ptfn^T &\colon \CM\algB\to \CC[\Grot(\CM\algA)]\colon M\mapsto \ptfn^T_M, \\
  \ptfn^T_M &= \sum_{[X]\in \sumrange{M}} \euler\bigl( \qvGrsH{\ddel{[X]}}{T}{M} \bigr) \xvar^{[X]}, 
\end{aligned}
\end{equation}
where $\euler$ and $\qvGr{d}$ are as before in \eqref{eq:FK-CC}.  

\begin{remark}\label{rem:motivic-sum}
We can rewrite the sum in \eqref{eq:part-fun} as a `motivic sum' 
\begin{equation}\label{eq:motivic-sum}
  \ptfn^T_M = \motsum_{\substack{X \subspc \Hom(T, M)\\ eX=M} }\xvar^{[X]}
  = \motsum_{Y \subspc \sHom(T, M)} \xvar^{[\Yhat]}, 
\end{equation}
where $\Yhat$ is the inverse image of $Y$ under the quotient $\Hom(T,M)\to\sHom(T,M)$.
In the sums in \eqref{eq:motivic-sum}, infinite families are counted by the Euler characteristic of the family.
Indeed, it is precisely $\qvGrsH{\ddel{[X]}}{T}{M}$ that parametrises 
those~$Y$ with $[Y]=\ddel{[X]}$, or equivalently with $[\Yhat]=[X]$ .

To make \eqref{eq:motivic-sum} look even more like a classical partition function (cf. \eqref{eq:class-pf}),
we could write the sum range as ``$X\st eX=M$'', with the implicit understanding
that we have a prescribed isomorphism $eX\to M$ and use the unit map 
$\unitmap{X}\colon X\to \Radj\efunct X$ to make $X \subspc \Hom(T, M)$. 
Note that we can not just write ``$X\st eX \isom M$''.
\end{remark}

\begin{remark}\label{rem:ptfn-F-poly}
For any $[X]\in \sumrange{M}$, as in \eqref{eq:sumrange},
we know, from \eqref{eq:wtmod-diag} and \Cref{rem:eta=inc}, that
$ %\[
 [\funG X] = [ \Hom(T, M)/X]
$, %\]
which we can also write in $\Grot(\CM\algA)$ as
\begin{equation}\label{eq:X-GX}
  [X] = [T,M] - \beta[\funG X].
\end{equation}
Since $\ddel{[X]}= [X/(\algA\cdot M)]$, as in \eqref{eq:Dd-bij},
and $\Hom(T, M)/\algA\cdot M = \sHom(T, M)$, by \Cref{Thm:EndA}\itmref{itm:EA5},
we also have, in $\Grot(\fd\algA)$,
\begin{equation}\label{eq:GX-dX}
  [\funG X] = [\sHom(T, M)] - \ddel{[X]}.
\end{equation}
Hence the generalised partition function can be rewritten as follows, 
\begin{equation}\label{eq:ptfn-as-quotFpoly}
\begin{aligned}
  \ptfn^T_M &=\xvar^{[T, M]}\sum_{[X]\in \sumrange{M}} 
 \euler \bigl( \Gr_{\ddel{[X]}} \sHom(T, M) \bigr) \xvar^{-\beta[\funG X]}, \\
&=\xvar^{[T, M]} \sum_{d} 
 \euler \bigl( \Gr^{d} \sHom(T, M) \bigr) \xvar^{-\beta(d)}, 
\end{aligned}
\end{equation}
where $\Gr^{d} V$ is the Grassmannian of $d$-dimensional quotients of $V\in\fd\algA$,
so that the second sum in \eqref{eq:ptfn-as-quotFpoly} is 
the `quotient F-polynomial' of $\sHom(T, M)$, written in $\widehat y$ variables,
i.e.~the monomials $\xvar^{-\beta(d)}$, as in \eqref{eq:FK-CC}.

We could also write the partition function from the other end, that is,
\begin{equation}\label{eq:ptfn-as-subFpoly}
  \ptfn^T_M =\xvar^{[A\cdot M]} \sum_{d} 
 \euler \bigl( \Gr_{d} \sHom(T, M) \bigr) \xvar^{\beta(d)}, 
\end{equation}
so that the sum comes from the more usual (submodule) F-polynomial of $\sHom(T, M)$.
\end{remark}

\goodbreak
First, we show that $\ptfn^T$ is a cluster character, by relating it to $\Phi^T$.

\begin{proposition}\label{Thm:char-parfunA} 
Let $\ptbar^T=\CC[-1]\compo\ptfn^T$, in the notation of \Cref{rem:new-clucha}.
%For $\ptfn=\ptfn^T$, as in \eqref{eq:part-fun}, we have
%\begin{enumerate}
%\item\label{itm:pfun1}
For any partial presentation $0 \to \Syz M \to \procov{M} \to M \to 0$, we have
\begin{equation}\label{eq:pt-Phi}
   \ptbar^T_M=\frac{\Phi^T_{\Syz M}}{\Phi^T_{\procov{M}}}.
\end{equation}
Consequently, when restricted to $\GP\algB$, both $\ptbar^T$ and $\ptfn^T$ are cluster characters.
\end{proposition}

\begin{proof}
By \Cref{Lem:ExtSHom}\itmref{itm:ESH3}, we have $\sHom(T, M)\isom \Ext^1(T, \Syz M)$ and,
by \Cref{lem:Phi-monom}, we have $\Phi^T_{\procov{M}}=\xvar^{[T,\procov{M}]}$.
%and, by \Cref{Lem:ExtSHom}\itmref{itm:ESH3}, we have $\sHom(T, M)\isom \Ext^1(T, \Syz M)$.
Thus, modifying \eqref{eq:part-fun} and \eqref{eq:FK-CC}, we want to compare
\begin{align}
\ptbar^T_M  &= \sum_{[X]\in \sumrange{M}} \euler\bigl( \Gr_\ddel{[X]} \Ext^1(T, \Syz M) \bigr) \xvar^{-[X]},
\label{eq:def-ptbar}\\
\frac{\Phi^T_{\Syz M}}{\Phi^T_{\procov{M}}}
& = \xvar^{[T, \Syz M]-[T, \procov{M}]}\sum_{d} \euler\bigl( \Gr_d \Ext^1(T, \Syz M) \bigr) \xvar^{-\beta(d)}
\label{eq:def-Phi/Phi}.
\end{align}
The sums in \eqref{eq:def-ptbar} and \eqref{eq:def-Phi/Phi} are effectively over the same range,
by the discussion following \eqref{eq:Dd-bij}, and have the same coefficients, 
so it remains to prove that the exponents are equal, that is,
\begin{equation}\label{eq:pt-Phi-RTP}
  -[X] = [T,\Syz M] - [T,\procov{M}] -\beta(\ddel{[X]}),
\end{equation}
for any $[X]\in\sumrange{M}$.
By \eqref{eq:beta-wt} and \eqref{eq:Dd-bij},
we have $\beta(\ddel{[X]})=[X]-[\algA\cdot M]$.
But now, applying $\Hom(T, -)$ to the partial presentation of $M$ gives an exact sequence
\[
0 \lra \Hom(T, \Syz M) \lra \Hom(T, \procov{M}) \lra \Hom(T, M). 
%\lra \Ext^1(T, \Syz M)\lra 0.
\]
Since $\procov{M}$ is a projective cover, the image of $\Hom(T, \procov{M})$ here is 
the kernel of the natural map $\Hom(T, M)\to\sHom(T, M)$,
which is $\algA\cdot M$ by \Cref{Thm:EndA}\itmref{itm:EA5}. 
Thus
\[  [\algA\cdot  M]=[T, \procov{M}]- [T, \Syz M], \]
which completes the proof of \eqref{eq:pt-Phi},
which we can now use to show that $\ptbar^T$ is a cluster character
and \Cref{rem:new-clucha} immediately implies that $\ptfn^T$ is also.

First, \eqref{eq:pt-Phi} implies that $\ptbar^T$ satisfies \Cref{def:clucha}\itmref{itm:cc1}, 
because $\Phi^T$ does.
Second, the direct sum of partial presentations of $M$ and $N$
is a partial presentation of $M\oplus N$. 
Thus, since $\Phi^T$ satisfies \Cref{def:clucha}\itmref{itm:cc2}, we have
\[
\ptbar^T_{M\oplus N}
= \frac{ \Phi^T_{\Syz M\oplus \Syz N} }{ \Phi^T_{\procov{M}\oplus \procov{N}} }
= \frac{ \Phi^T_{\Syz M} }{ \Phi^T_{\procov{M}} }\;  \frac{ \Phi^T_{ \Syz N} }{ \Phi^T_{\procov{N}} }
=\ptbar^T_{M} \ptbar^T_{N}.
\]
Thus $\ptbar^T$ satisfies \Cref{def:clucha}\itmref{itm:cc2}.

Now suppose that we have $M, N \in \GP\algB$ with $\dim\Ext^1(M, N)=1$ 
and the two non-split short exact sequences with middle terms $E$ and $F$, 
as in \Cref{def:clucha}\itmref{itm:cc3}.
Choose partial presentations of $M$ and $N$
and use the Horseshoe Lemma to get partial presentations of $E$ and $F$ with 
\begin{equation}\label{eq:procovEF}
\procov{E} = \procov{M}\oplus \procov{N} = \procov{F},
\end{equation}
giving the following commutative diagram of short exact sequences for $E$
\[
\begin{tikzpicture}[xscale=1.8, yscale=1.2, baseline=(bb.base)] 
\coordinate (bb) at (2,2);
\pgfmathsetmacro{\xeps}{0.7}
\pgfmathsetmacro{\yeps}{0.7}
\draw (1,3+\yeps) node (a1) {$0$};
\draw (2,3+\yeps) node (a2) {$0$};
\draw (3,3+\yeps) node (a3) {$0$};
\draw (1-\xeps,3) node (b0) {$0$};
\draw (1,3) node (b1) {$\Syz{M}$};
\draw (2,3) node (b2) {$ \Syz{E}$};
\draw (3,3) node (b3) {$ \Syz{N}$};
\draw (3+\xeps,3) node (b4) {$0$};
\draw (1-\xeps, 2) node (c0) {$0$};
\draw (1,2) node (c1) {$\procov{M}$};
\draw (2,2) node (c2) {$\procov{E}$};
\draw (3,2) node (c3) {$\procov{N}$};
\draw (3+\xeps,2) node (c4) {$0$};
\draw (1-\xeps,1) node (d0) {$0$};
\draw (1,1) node (d1) {$M$};
\draw (2,1) node (d2) {$E$};
\draw (3,1) node (d3) {$N$};
\draw (3+\xeps,1) node (d4) {$0$};
\draw (1,1-\yeps) node (e1) {$0$};
\draw (2,1-\yeps) node (e2) {$0$};
\draw (3,1-\yeps) node (e3) {$0$};
\foreach \t/\h in {b0/b1, b1/b2, b2/b3, b3/b4, c0/c1, c1/c2, c2/c3, c3/c4, d0/d1, d1/d2, d2/d3, d3/d4,
  a1/b1, a2/b2, a3/b3, b1/c1, b2/c2, b3/c3, c1/d1, c2/d2, c3/d3, d1/e1, d2/e2, d3/e3}
\draw[cdarr] (\t) to (\h);
 \end{tikzpicture}
\]  
and similiarly for $F$, with $M$ and $N$ swapped.

As $\Syz$ is an equivalence on the stable category, 
we know that $\Ext^1(\Syz M,\Syz N)=1$ and the top rows
in the diagrams for both $E$ and $F$ are non-split.
Therefore, since $\Phi^T$ satisfies \Cref{def:clucha}\itmref{itm:cc3},
\[
\ptbar^T_{M} \ptbar^T_{N}
= \frac{ \Phi^T_{\Syz M} }{ \Phi^T_{\procov{M}} } \; \frac{ \Phi^T_{\Syz N} }{ \Phi^T_{\procov{N}} }
=\frac{\Phi^T_{\Syz {E}} +\Phi^T_{\Syz {F}}}{ \Phi^T_{\procov{M}\oplus \procov{N}} } 
= \ptbar^T_{E} +\ptbar^T_{F},
\]
where the last equality uses \eqref{eq:procovEF}.
Thus $\ptbar^T$ satisfies \Cref{def:clucha}\itmref{itm:cc3}
and hence it is a cluster character. 
\end{proof}

Some elementary properties of $\ptfn^T$ are as follows.

\begin{proposition}\label{Thm:char-parfunB} 
For any $M\in\CM\algB$, let $\ptfn^T_M$ be as defined in \eqref{eq:part-fun}. Then
\begin{enumerate}
\item\label{itm:pfun3}
$\ptfn^T_M$  is homogeneous of degree $\rnk{\algB} M$,
\item\label{itm:pfun4}
%with respect to the partial order on exponents from Sec~\ref{sec:part-ord},
$\ptfn^T_M$ has minimum term $\xvar^{[T, M]}$ 
and maximum term $\xvar^{[\algA\cdot M]}$ 
and both terms have coefficient~$1$,
\item\label{itm:pfun5}
$\ptfn^T_M$ is a monomial if and only if $\Syz M\in \add T$, 
in which case $\ptfn^T_M= \xvar^{[T, M]}$.
In particular, this holds when $M$ is projective or injective in $\CM\algB$.
\item\label{itm:pfun6} 
If $M\not\in\GP\algB$, then there is an $N\in \GP\algB$ and a  $P\in \add B$ such that 
\[ \ptfn^T_M =\xvar^{-[T,P]} \ptfn^T_N. \]
\end{enumerate}
\end{proposition}

In \itmref{itm:pfun4}, the (partial) order on exponents is from \Cref{def:part-ord}
and so depends on the choice of boundary vertex $\vstar$.
However, the result clearly doesn't,
that is, these are the minimum and maximum terms for all choices of $\vstar$.

\begin{proof}
\medskip\itmref{itm:pfun3} % ===== (3) =====
All $X$ that contribute to the sum \eqref{eq:part-fun} have $\efunct X=M$ 
and so, by \Cref{Lem:ranks}\itmref{itm:rks3},
\[
  \rkk [X] = \rnk{\algA} X = \rnk{\algB} \efunct X = \rnk{\algB} M.
\]
Thus $\ptfn^T_M$ is a homogeneous polynomial of that degree.

\medskip\itmref{itm:pfun4} % ===== (4) =====
The (partial) order on terms comes from \Cref{def:part-ord}
applied to their exponents.
Since $\ptfn^T_M$ is homogeneous, we just have to compare 
the weights of the exponents in the sum \eqref{eq:part-fun}.
All $X$ that contribute have $\algA\cdot M \leq X \leq \Hom(T, M)$
and applying $\efunct$ to both these inclusions gives equality. 
Hence \Cref{Lem:somefacts}\itmref{itm:facts5} implies that 
\[
  \wt[T, M] \leq \wt [X] \leq \wt [\algA\cdot M]. 
\]
Thus the minimum exponent is $[T, M]$, the maximum exponent is $[\algA\cdot M]$
and all other exponents lie between these two in the partial order. 
Furthermore, both classes only contain a single module, namely,
$X = \Hom(T, M)$ and $X=\algA\cdot M$.
Hence the corresponding Euler characteristic is $1$. 

\medskip\itmref{itm:pfun5} % ===== (5) =====
As in the proof of \Cref{lem:Phi-monom}, 
the sum in \eqref{eq:part-fun} has only one term if and only if $\sHom(T, M)=0$,
which is equivalent to $\Ext^1(T, \Syz M)=0$, by \Cref{Lem:ExtSHom}\itmref{itm:ESH3}.
But $\Syz M\in \GP\algB$, by \Cref{lem:GPBBapp},
so this is equivalent to $\Syz M\in \add T$,
as $T$ is a cluster tilting object.
The one term in this case is from $X=\Hom(T,M)$ (cf. \itmref{itm:pfun4}).

Note that the property $\Syz M\in \add T$ does not depend on the choice of 
$\Syz M$, because two choices only differ by projectives and $\add\algB \subset \add T$.
If $M$ is projective, we can take $\Syz M=0$.
If $M$ is injective, then $\Syz M$ is projective, by \Cref{cor:syz-inj-proj},
and thus in $\add T$.

\medskip\itmref{itm:pfun6} % ===== (6) =====
Let $N=\coSyz \Syz M$.
By \Cref{Lem:ExtSHom}\itmref{itm:ESH2}, we know
$\sHom(T, N)\isom \sHom(T, M)$, as $\sEnd T$-modules.
Hence, from \eqref{eq:part-fun}, $\ptfn^T_M$ and $\ptfn^T_N$ 
are sums over the same dimension vectors with the same coefficients,
and so differ at most by their leading exponents. 
Applying $\Hom(T,-)$ to the diagram \eqref{eq:syzcosyz},
we see that the leading exponents are equal up to terms coming from $\add \Hom(T,B)$,
which completes the proof.
\end{proof}

\begin{remark}\label{rem:SigmaT} 
The condition $\Syz M\in \add T$ in \Cref{Thm:char-parfunB}\itmref{itm:pfun5},
that makes $\ptfn^T_M$ a monomial, holds in particular for all $M\in \add\CTOcosyz{T}$,
where $\CTOcosyz{T}$ is another cluster tilting object in $\GP B$, defined as
\begin{equation}\label{eq:SigmaT}
 \CTOcosyz{T} = \algB \oplus \bigoplus_{\text{interior $i$}} \Sigma T_i
\end{equation}
and $\Sigma T_i$ is a minimal coszygy of $T_i$ in $\GP\algB$, as in \S\ref{subsec:syz-cosyz}.
%(not necessarily the same as in $\CM\algB$).
In particular, $\Sigma T_i$ is indecomposable and has a partial presentation
\[
 \ShExSeq {T_i} {\procov{\Sigma T_i}} {\Sigma T_i},
\]
so we can suppose that $\Syz\Sigma T_i = T_i$. 

Note that the summands of $\CTOcosyz{T}$ are indexed by the same vertex set $Q_0$ as the summands of $T$,
but the arrows of the Gabriel quiver of $\End (\CTOcosyz{T})\op$ may be different.
Note also that the summands of $\CTOcosyz{T}$ provide an alternative basis 
$\{ [T,\CTOcosyz{T}_i] \st i\in Q_0 \}$ of $\Grot(\CM\algA)$ 
to the standard basis $\{ [T,T_i]\st i\in Q_0 \}$.
More precisely, 
\begin{equation}\label{eq:TSigmaT}
  [T, \CTOcosyz{T}_i] 
  = \begin{cases} %\left\{ \begin{array}{ll} 
  [T, \procov{\Sigma T_i}]-[T, T_i],  & \text{for interior $i$}, \\
  [ T,T_i ],  & \text{for boundary $i$,}
 \end{cases} %  \end{array} \right.
\end{equation}
giving an invertible change of basis matrix as, for interior~$i$, 
all summands of $\procov{\Sigma T_i}$ are projective, that is, of the form $T_j$ for some boundary~$j$.
\end{remark}

By \Cref{Lem:ExtSHom}\itmref{itm:ESH3}, 
we can see that $\sHom(T,M)\isom \Ext^1(\CTOcosyz{T},M)$,
so comparing \eqref{eq:part-fun} and \eqref{eq:FK-CC}
suggests another relationship between the partition function $\ptfn^T$
and the cluster character $\Phi^T$, as follows.

\begin{proposition}\label{thm:PF-zeta-Phi}
Let $T$ be a cluster tilting object in $\GP\algB$, with $\algA=\End (T)\op$,
and $\CTOcosyz{T}$ be as in \eqref{eq:SigmaT}, %\Cref{rem:SigmaT},
with $\algA'=\End (\CTOcosyz{T})\op$.
Then $\ptfn^T$ and  $\Phi^{\CTOcosyz{T}}$ are related by a monomial change of variables,
specifically,
\begin{equation}\label{eq:PF-zeta-Phi-0}
\ptfn^T = \CC[\zeta] \compo \Phi^{\CTOcosyz{T}},
\end{equation}
where, for $X\in\add\CTOcosyz{T}$,
\begin{equation}\label{eq:zeta-def-0}
\zeta\colon \Grot(\CM\algA') \to \Grot(\CM\algA)
\colon [\CTOcosyz{T}, X] \mapsto [T, X].
\end{equation}
\end{proposition}

Note that the ($\ZZ$-linear) map $\zeta$ is well-defined by \eqref{eq:zeta-def-0}, 
because, as $X$ runs over the indecomposable summands of $\CTOcosyz{T}$, 
the classes $[\CTOcosyz{T}, X]$ form a basis of $\Grot(\CM\algA')$.

\begin{proof}
We will actually prove that $\CC[-\zeta] \compo \Phi^{\CTOcosyz{T}}=\ptbar^T$,
as we can then use \eqref{eq:pt-Phi} from \Cref{Thm:char-parfunA}.
Indeed, by \eqref{eq:FK-CC} and \eqref{eq:pt-Phi},
\begin{align}
\Phi^{\CTOcosyz{T}}_M
& = \xvar^{[\CTOcosyz{T}, M]}\sum_{d} \euler\bigl( \Gr_d\Ext^1(\CTOcosyz{T}, M) \bigr) \xvar^{-\beta_{A'}(d)},
\label{eq:def-Phi}\\
\ptbar^{T}_M 
= \frac{\Phi^T_{\Syz M}}{\Phi^T_{\procov{M}}}
& = \xvar^{[T, \Syz M]-[T, \procov{M}]}\sum_{d} \euler\bigl( \Gr_d \Ext^1(T, \Syz M) \bigr) \xvar^{-\beta_A(d)},
\label{eq:def-pt'}
\end{align}
where $\beta_A\colon \Grot(\fd\algA)\to \Grot(\CM\algA)$ and $\beta_{A'}$ 
are defined as in \eqref{eq:def-beta} and \Cref{Lem:classvsresol}.
We need to show that the right-hand sides of \eqref{eq:def-Phi} and \eqref{eq:def-pt'}
are identified by $\CC[-\zeta]$.

By \Cref{Lem:ExtSHom}\itmref{itm:ESH3}, the stable equivalence $\coSyz$ induces an 
isomorphism 
\begin{equation}\label{eq:end-iso}
\sEnd (\CTOcosyz{T}) \isom \sEnd (T)
\end{equation}
so that
$\Ext^1(\CTOcosyz{T}, M)\isom \sHom(T,M) \isom \Ext^1(T, \Syz M)$ 
as $\sEnd(T)$-modules.
Thus the sums in \eqref{eq:def-Phi} and \eqref{eq:def-pt'} are over the same range 
and the corresponding terms have the same coefficients. 
So it remains to show that
\begin{equation}\label{eq:lead-term}
 -\zeta [\CTOcosyz{T}, M] = [T, \Syz M]-[T, \procov{M}]
\end{equation}
and, for all $d$,
\begin{equation}\label{eq:other-term}
 -\zeta \beta_{A'}(d) = \beta_A(d).
\end{equation}
Note that we need to prove these two separately, 
because \eqref{eq:lead-term} is the case $d=0$.

To prove \eqref{eq:lead-term}, we use the Horseshoe Lemma to construct
the following commutative diagram of short exact sequences,
\newcommand{\sS}{U} % temporary notation within proof 
\begin{equation}\label{eq:HS-SSM}
\begin{tikzpicture}[xscale=2.0, yscale=1.4, baseline=(bb.base)]
\coordinate (bb) at (2,2);
\pgfmathsetmacro{\xeps}{0.7}
\pgfmathsetmacro{\yeps}{0.7}
\draw (1,3+\yeps) node (a1) {$0$};
\draw (2,3+\yeps) node (a2) {$0$};
\draw (3,3+\yeps) node (a3) {$0$};
\draw (1-\xeps,3) node (b0) {$0$};
\draw (1,3) node (b1) {$\Syz \sS''$};
\draw (2,3) node (b2) {$\Syz \sS'$};
\draw (3,3) node (b3) {$ \Syz M$};
\draw (3+\xeps,3) node (b4) {$0$};
\draw (1-\xeps, 2) node (c0) {$0$};
\draw (1,2) node (c1) {$\procov{\sS''}$};
\draw (2,2) node (c2) {$\procov{\sS'}$};
\draw (3,2) node (c3) {$\procov{M}$};
\draw (3+\xeps,2) node (c4) {$0$};
\draw (1-\xeps,1) node (d0) {$0$};
\draw (1,1) node (d1) {$\sS''$};
\draw (2,1) node (d2) {$\sS'$};
\draw (3,1) node (d3) {$M$};
\draw (3+\xeps,1) node (d4) {$0$};
\draw (1,1-\yeps) node (e1) {$0$};
\draw (2,1-\yeps) node (e2) {$0$};
\draw (3,1-\yeps) node (e3) {$0$};
\foreach \t/\h in {b0/b1, b1/b2, b2/b3, b3/b4, c0/c1, c1/c2, c2/c3, c3/c4, d0/d1, d1/d2, d2/d3, d3/d4,
  a1/b1, a2/b2, a3/b3, b1/c1, b2/c2, b3/c3, c1/d1, c2/d2, c3/d3, d1/e1, d2/e2, d3/e3}
\draw[cdarr] (\t) to (\h);
\end{tikzpicture}
\end{equation} 
where the bottom row is any $\add\CTOcosyz{T}$ approximation of $M$,
the outer columns are any syzygy sequences and
the middle column is the syzygy sequence with $\procov{\sS'}=\procov{\sS''}\oplus\procov{M}$.
Then the top row of \eqref{eq:HS-SSM} is an $\add T$ approximation of $\Syz M$ 
(see \Cref{rem:SigmaT}).
Since the bottom row of \eqref{eq:HS-SSM} remains exact under $\Hom(\CTOcosyz{T},-)$
and $\sS',\sS''\in \add\CTOcosyz{T}$, we can use \eqref{eq:zeta-def-0} to write 
\begin{align*}
 -\zeta [\CTOcosyz{T}, M] = -\zeta [\CTOcosyz{T},\sS'] + \zeta [\CTOcosyz{T}, \sS''] 
 & = -[T,\sS'] + [T,\sS''] \\
 & = [T,\Syz \sS'] - [T, \procov{\sS'}] - [T, \Syz \sS''] + [T,\procov{\sS''}]\\
 & = [T,\Syz M]  - [T, \procov{M}] ,
\end{align*}
where the last two equalities hold because the left two sequences and the top two sequences 
in \eqref{eq:HS-SSM} remain exact under $\Hom(T,-)$.
This proves \eqref{eq:lead-term}.

To prove \eqref{eq:other-term},  
since $\beta$ is additive and the modules
$\Ext^1(\CTOcosyz{T}, M)\isom\Ext^1(T, \Syz M)$ %$\sHom(T, M)$ 
are supported on the interior vertices, 
we just need to do the case $d=[S_i]$ for any interior simple,
where because of the isomorphism \eqref{eq:end-iso},
we can use $S_i$ to denote both the simple top of $\Hom(T,T_i)$ 
and of $\Hom(\CTOcosyz T, \Sigma T_i)$. 

Consider the following two commutative diagrams of short exact sequences,
constructed using the Horseshoe Lemma
from the mutation sequences \eqref{eq:mut-two}.
\begin{equation}\label{eq:two-squares}
\begin{tikzpicture}[xscale=1.8, yscale=1.2, baseline=(bb.base)] 
\coordinate (bb) at (2,2);
\pgfmathsetmacro{\xeps}{0.7}
\pgfmathsetmacro{\yeps}{0.7}
\draw (1,3+\yeps) node (a1) {$0$};
\draw (2,3+\yeps) node (a2) {$0$};
\draw (3,3+\yeps) node (a3) {$0$};
\draw (1-\xeps,3) node (b0) {$0$};
\draw (1,3) node (b1) {$T_i$};
\draw (2,3) node (b2) {$F_i$};
\draw (3,3) node (b3) {${T_i}^*$};
\draw (3+\xeps,3) node (b4) {$0$};
\draw (1-\xeps, 2) node (c0) {$0$};
\draw (1,2) node (c1) {$\procov{\Sigma T_i}$};
\draw (2,2) node (c2) {$\procov{\Sigma F_i}$};
\draw (3,2) node (c3) {$\procov{\Sigma {T_i}^*}$};
\draw (3+\xeps,2) node (c4) {$0$};
\draw (1-\xeps,1) node (d0) {$0$};
\draw (1,1) node (d1) {$\Sigma T_i$};
\draw (2,1) node (d2) {$\Sigma F_i$};
\draw (3,1) node (d3) {$\Sigma {T_i}^*$};
\draw (3+\xeps,1) node (d4) {$0$};
\draw (1,1-\yeps) node (e1) {$0$};
\draw (2,1-\yeps) node (e2) {$0$};
\draw (3,1-\yeps) node (e3) {$0$};
\foreach \t/\h in {b0/b1, b1/b2, b2/b3, b3/b4, c0/c1, c1/c2, c2/c3, c3/c4, d0/d1, d1/d2, d2/d3, d3/d4,
  a1/b1, a2/b2, a3/b3, b1/c1, b2/c2, b3/c3, c1/d1, c2/d2, c3/d3, d1/e1, d2/e2, d3/e3}
\draw[cdarr] (\t) to (\h);
 \end{tikzpicture}
\quad
\begin{tikzpicture}[xscale=1.8, yscale=1.2, baseline=(bb.base)] 
\coordinate (bb) at (2,2);
\pgfmathsetmacro{\xeps}{0.7}
\pgfmathsetmacro{\yeps}{0.7}
\draw (1,3+\yeps) node (a1) {$0$};
\draw (2,3+\yeps) node (a2) {$0$};
\draw (3,3+\yeps) node (a3) {$0$};
\draw (1-\xeps,3) node (b0) {$0$};
\draw (1,3) node (b1) {${T_i}^*$};
\draw (2,3) node (b2) {$E_i$};
\draw (3,3) node (b3) {$T_i$};
\draw (3+\xeps,3) node (b4) {$0$};
\draw (1-\xeps, 2) node (c0) {$0$};
\draw (1,2) node (c1) {$\procov{\Sigma {T_i}^*}$};
\draw (2,2) node (c2) {$\procov{\Sigma E_i}$};
\draw (3,2) node (c3) {$\procov{\Sigma T_i}$};
\draw (3+\xeps,2) node (c4) {$0$};
\draw (1-\xeps,1) node (d0) {$0$};
\draw (1,1) node (d1) {$\Sigma {T_i}^*$};
\draw (2,1) node (d2) {$\Sigma E_i$};
\draw (3,1) node (d3) {$\Sigma T_i$};
\draw (3+\xeps,1) node (d4) {$0$};
\draw (1,1-\yeps) node (e1) {$0$};
\draw (2,1-\yeps) node (e2) {$0$};
\draw (3,1-\yeps) node (e3) {$0$};
\foreach \t/\h in {b0/b1, b1/b2, b2/b3, b3/b4, c0/c1, c1/c2, c2/c3, c3/c4, d0/d1, d1/d2, d2/d3, d3/d4,
  a1/b1, a2/b2, a3/b3, b1/c1, b2/c2, b3/c3, c1/d1, c2/d2, c3/d3, d1/e1, d2/e2, d3/e3}
\draw[cdarr] (\t) to (\h);
 \end{tikzpicture}
\end{equation}
In particular, the outer columns are minimal cosyzygy sequences, 
while the middle column is not necessarily minimal and constructed with
\begin{equation}\label{eq:procovSigEF}
\procov{\Sigma F_i}=\procov{\Sigma T_i}\oplus \procov{\Sigma T^*_i}=\procov{\Sigma E_i}.
\end{equation}
Given that $T$ is a cluster tilting object, the fact that the top rows are mutation sequences
is equivalent to $ F_i, E_i \in \add T/T_i$ and $\Ext^1(T/T_i,{T_i}^*)=0=\Ext^1({T_i}^*,T/T_i)$,
so that the maps to and from $T_i$ are $\add T/T_i$-approximations.

But then we know that $\Sigma F_i, \Sigma E_i \in \add\CTOcosyz{T}/\Sigma T_i$ 
(cf. \Cref{rem:SigmaT}) 
and, since $\Sigma$ is an equivalence on the stable category, 
it preserves the Ext vanishing condition, 
so that the bottom rows are mutation sequences for $\Sigma T_i$.
In particular, 
\begin{equation}\label{eq:Sigma-mut}
 (\Sigma{T_i})^*=\Sigma{T_i}^*.
\end{equation}
Hence, by \Cref{Prop:projres},  
\[
\beta_{\algA'}[S_i]=[\CTOcosyz{T}, \Sigma F_i]-[\CTOcosyz{T}, \Sigma E_i]
\quadand
\beta_{\algA}[S_i]=[T, F_i]-[T, E_i].
\]
Therefore, since $\Sigma F_i, \Sigma E_i \in \add\CTOcosyz{T}$ 
and the middle columns remain exact under $\Hom(T,-)$, 
we can use \eqref{eq:zeta-def-0} to write 
\begin{align*}
 -\zeta \beta_{\algA'}[S_i]
  = -\zeta[\CTOcosyz{T}, \Sigma F_i] +\zeta [\CTOcosyz{T}, \Sigma E_i] 
 & = - [T, \Sigma F_i] + [T, \Sigma E_i] \\
 & = [T, F_i] - [T, \procov{\Sigma F_i}] - [T, E_i] + [T, \procov{\Sigma E_i}] \\
 & = \beta_{\algA}[S_i],
\end{align*}
where the last equality uses \eqref{eq:procovSigEF}.
This completes the proof of \eqref{eq:other-term}
and thus of the whole result.
\end{proof} 

\begin{remark}\label{rem:T-vs-SigmaT}
The bottom and right-hand sequences in \eqref{eq:HS-SSM} 
usually do not remain exact under $\Hom(T,-)$ 
and the equal quantities in \eqref{eq:lead-term} are usually not $-[T,M]$.
Indeed the `leading'  term $\xvar^{[T,M]}$ in $\ptfn^{T}_M$ comes from the other end of the sum,
not from the  `leading'  term $\xvar^{[\CTOcosyz{T},M]}$ in $\Phi^{\CTOcosyz{T}}_M$.
\end{remark}

% ======================================================================
\subsection{Generalised flow polynomials}\label{subsec:gen-flo-pol}
% ======================================================================

Recall, from \eqref{eq:wt-class} and \eqref{eq:def-Nstar}, the lattice homomorphism 
\[
  \wt\colon \Grot(\CM\algA) \to \Nstar \subspc \Grot(\fd\algA). % \isom \ZZ^{Q_0}.
\]
For any cluster tilting object $T\in\GP\algB$, we can define $\fp^T = \pbfun{\wt}{\ptfn^T}$,
in the notation of \Cref{rem:new-clucha}, that is, 
\begin{equation}\label{eq:gen-flow-poly}
\begin{aligned}
  \fp^T &\colon \CM\algB\to \CC[\Nstar]\colon M\mapsto \fp^T_M, \\
  \fp^T_M &= \sum_{ [X]\in \sumrange{M} } \euler\bigl( \qvGrsH{\ddel{[X]}}{T}{M} \bigr) \yvar^{\wt[X]}. 
\end{aligned}
\end{equation}
Also recall, from \Cref{def:KTM} and \Cref{Lem:kerwt}, the invariant
\begin{equation}\label{eq:kap-def}
  \kapvec(T,M) = [\funK(T,M)] = \wt [T,M].
\end{equation}
More generally, since  $[\funG X]\in\Nstar$, by \Cref{Lem:somefacts}\itmref{itm:facts2}, 
we can apply $\wt$ to \eqref{eq:X-GX}, recalling that $\wt\compo\beta=-\idmap{}$ on  $\Nstar$,
by \Cref{Thm:2exactseq}, to get
\begin{equation}\label{eq:wtX-formG}
 \wt[X]=\kapvec(T, M) + [\funG X].
\end{equation}
In particular, as $\funG (A\cdot M) = \Hom(T, M)/\algA\cdot M = \sHom(T, M)$,
by \Cref{Thm:EndA}\itmref{itm:EA5},
\begin{equation}\label{eq:wtX-AM}
  \wt[A\cdot M] = \kapvec(T, M) +[\sHom(T, M)].
\end{equation} 
Properties of $\ptfn^T$ 
from \Cref{Thm:char-parfunA} and \Cref{Thm:char-parfunB}
give the following.

\begin{proposition} \label{Thm:char-flopol}
For $\fp^T$ as in \eqref{eq:gen-flow-poly}, we have
\begin{enumerate}
\item\label{itm:flo1}
  $\fp^T$ is a cluster character, when restricted to $\GP\algB$.
\item\label{itm:flo3}
  $\fp^T_M$ has minimum term $\yvar^{\kapvec(T, M)}$
%\item\label{itm:flo4}
  and maximum term $\yvar^{\kapvec(T, M)+[\sHom(T, M)]}$
  and both terms have coefficient~1. 
\item\label{itm:flo5} 
  $\fp^T_M$ is a monomial if and only if $\Syz M\in \add T$,
in which case $\fp^T_M=\yvar^{\kapvec(T, M)}$.
In particular, this holds when $M$ is projective or injective in $\CM\algB$.
\item\label{itm:flo6} 
$\fp^T_{\modJ}=1$. 
\end{enumerate}
\end{proposition}

\begin{proof}
\itmref{itm:flo1} % ===== (1) =====
As $\fp^T_M = \pbfun{\wt}{\ptfn^T_M}$, \Cref{rem:new-clucha} implies that
$\fp^T$ is a cluster character because $\ptfn^T$ is, by \Cref{Thm:char-parfunA}.

\medskip\itmref{itm:flo3}  % and \itmref{itm:flo4}% ===== (2) =====
follows from \Cref{Thm:char-parfunB}\itmref{itm:pfun4},\
using \eqref{eq:kap-def} and \eqref{eq:wtX-AM}.
%%and \eqref{eq:wtX-formula} and , 
%since $\wt[T,M]=\kapvec(T, M)$
%% and $\ddel{[A\cdot M]}=0$,
%and $\wt[A\cdot M] = \kapvec(T, M) +[\sHom(T, M)]$.

\medskip\itmref{itm:flo5} % ===== (3) =====
Note that $\Syz M\in \add T$ if and only if $\Ext^1(T, \Syz M)\isom\sHom(T, M)=0$.
If this happens, then the sum in \eqref{eq:gen-flow-poly} has just one term,
as given in \itmref{itm:flo3}.
Otherwise, \itmref{itm:flo3} shows that the sum has at least two non-zero terms., 
The condition holds when $M$ is projective, because $\Syz M=0$, 
and when $M$ is injective, by \Cref{cor:syz-inj-proj}.

\medskip\itmref{itm:flo6} % ===== (4) =====
$\modJ$ is injective, so $\fp^T_{\modJ}=\yvar^{\kapvec(T, \modJ)}$ by \itmref{itm:flo5}.
But $\kapvec(T, \modJ)=\wt\Hom(T,\modJ)=0$, by \Cref{Lem:kerwt}.
\end{proof}

\begin{remark}\label{rem:fp-F-poly}
Similarly to \Cref{rem:ptfn-F-poly}, 
but here using \eqref{eq:wtX-formG} (or, equivalently, applying $\wt$ to \eqref{eq:ptfn-as-quotFpoly}),
the generalised flow polynomial can be rewritten as follows, 
\begin{equation}\label{eq:fp-as-Fpoly}
\begin{aligned}
  \fp^T_M &=\yvar^{\kapvec (T, M)}\sum_{[X]\in \sumrange{M}} 
 \euler \bigl( \Gr_{\ddel{[X]}} \sHom(T, M) \bigr) \yvar^{[\funG X]}, \\
&=\yvar^{\kapvec (T, M)}\sum_{d} 
 \euler \bigl( \Gr^{d} \sHom(T, M) \bigr) \yvar^{d}, 
\end{aligned}
\end{equation}
where again $\Gr^{d} V$ is the Grassmannian of $d$-dimensional quotients of $V\in\fd\algA$,
so that the second sum in \eqref{eq:fp-as-Fpoly} is now precisely the quotient F-polynomial of $\sHom(T, M)$.
\end{remark}

%%%% =========================================================== %%%%
\section{Classical partition functions and flow polynomials}\label{sec:classical}
%%%% =========================================================== %%%%

\newcommand{\strarrow}{\arrow{angle 60}} 
\tikzset{strand/.style={teal}}
\newcommand{\drawsharkstrands}{
\pgfmathsetmacro{\ax}{20}
\pgfmathsetmacro{\ay}{5}
\pgfmathsetmacro{\az}{20}
% [out= ,in= ] 
\draw [strand] (V1) to [out=(-135-\ax, in=40] (A12) to [out=-140,in=\ay] (A13) to [out= -180+\ay,in= -\ay] (A34) to [out= 180-\ay,in= -30]  (V4)
  [postaction=decorate, decoration={markings, mark= at position 0.25 with \strarrow,
   mark= at position 0.45 with \strarrow, mark= at position 0.65 with \strarrow, mark= at position 0.85 with \strarrow}];
\draw [strand] (V4) to [out= 30,in= -180+\ay] (A45) to [out= \ay,in= 180-\ay] (A15) to [out= -\ay,in= 140] (A12) to [out= -40,in= 135+\ax] (V2)
  [postaction=decorate, decoration={markings, mark= at position 0.2 with \strarrow,
   mark= at position 0.4 with \strarrow, mark= at position 0.6 with \strarrow, mark= at position 0.8 with \strarrow}];
\draw [strand] (V2) to [out= 115,in= -90] (V12) to [out= 90,in= -115] (V1)
  [postaction=decorate, decoration={markings, mark= at position 0.55 with \strarrow}];
\draw [strand] (V3) to [out= 90-\ax,in= -90-\ay] (A13) to [out= 90-\ay,in= -90+\ay] (A15) to [out= 90+\ay,in= -90+\ax] (V5)
  [postaction=decorate, decoration={markings, mark= at position 0.25 with \strarrow,
   mark= at position 0.525 with \strarrow, mark= at position 0.8 with \strarrow}];
\draw [strand] (V5) to [out= -90-\ax, in= 90-\ay] (A45) to [out= -90-\ay, in= 90+\ay] (A34) to [out= -90+\ay, in= 90+\ax] (V3)
  [postaction=decorate, decoration={markings, mark= at position 0.25 with \strarrow, 
  mark= at position 0.525 with \strarrow, mark= at position 0.8 with \strarrow}];
}

In this section, we explain how the generalised partition functions and flow polynomials,
introduced in the previous section, are related to their classical counterparts,
as in e.g.~\cite{MS,RW,Tal}.

% ======================================================================
\subsection{Dimer models and perfect matchings}\label{subsec:dim-mod}
% ======================================================================

The context in which we can do this is that of \Cref{ex:plabic},
that is, for a dimer model on a disc, in the sense of \cite{BKM,CKP}.
Such a dimer model is formulated in terms of a quiver 
with faces $Q=(Q_0,Q_1,Q_2)$ dual to a plabic graph $G$ 
that arises from a connected Postnikov alternating strand diagram.
The connectedness is required for the topological space associated to $Q$, 
as a 2-dimensional cell complex, to be a disc.
The alternating strand properties ensure that the dimer model is `consistent'.
See \Cref{fig:shark} for an example, associated to the codimension 1 positroid stratum in $\Gr(2,5)$
where $\minor{45}=0$.

%==================================
\begin{figure}[h]
%==================================
\begin{tikzpicture} [scale=1.8, 
% strand/.style={teal},
 bdry/.style={thick, blue, densely dotted},
 quivvert/.style={red, fill},
 quivarr/.style={red, thick, -latex},
 plabedge/.style={blue, thick},
 blackplabvert/.style={blue, fill},
 whiteplabvert/.style={blue, fill=white}]
\newcommand{\dotRad}{1pt}
\pgfmathsetmacro{\Pstep}{0.35}
%===============
\begin{scope} [xscale=-1]
\foreach \n/\x/\y in {1/0/0, 5/-1/1, 3/-1/-1, 4/-2/0, 2/1.25/0}
  \coordinate (B\n) at (\x*\Pstep,\y*\Pstep);
\coordinate (V1) at ($(B2)+0.357*(1,1)$); 
\coordinate (V2) at ($(B2)+0.357*(1,-1)$); 
\coordinate (V12) at ($(B2)+0.22*(1,0)$); 
\coordinate (V3) at ($(B3)+0.422*(0,-1)$); 
\coordinate (V4) at ($(B4)+0.4*(-1,0)$); 
\coordinate (V5) at ($(B5)+0.422*(0,1)$); 
\foreach \a/\b in {1/2, 1/3, 4/5, 3/4, 1/5}
  \coordinate (A\a\b) at ($0.5*(B\a)+0.5*(B\b)$);
\drawsharkstrands % macro at section head
\foreach \n/\N in {2/1, 2/2, 3/3, 4/4, 5/5}
 \draw [plabedge] (B\n)--(V\N);
\foreach \n/\m in {1/2, 1/5, 1/3, 4/5, 3/4}
 \draw [plabedge] (B\n)--(B\m);
\foreach \n in {2,3,5}
  \draw [whiteplabvert] (B\n) circle (\dotRad);
\foreach \n in {1,4}
 \draw [blackplabvert] (B\n) circle (\dotRad);
\draw [bdry] (-0.1,0) ellipse (1 and 0.8);
\foreach \N/\where/\M in {1/above left/5, 2/below left/4, 3/below/3, 4/right/2, 5/above/1}
 \draw (V\N) node [\where] {\small $\M$};
\end{scope}
%===============
\begin{scope} [shift={(3,0)},xscale=-1]
\pgfmathsetmacro{\Qstep}{0.5}
\foreach \n/\x/\y in {1/0/0, 2/2/0, 3/1/1, 4/1/-1, 5/-1/-1, 6/-1/1}
  \coordinate (Q\n) at (\x*\Qstep,\y*\Qstep);
\draw [quivvert] (Q3) node (Q3) {\small$*$};
\foreach \N in {1,2,4,5,6}
 \draw [quivvert] (Q\N) node (Q\N) {\tiny$\bullet$};
\foreach \t/\h in {3/1, 1/4, 4/3, 3/2, 2/4, 4/5, 5/1,1/6, 6/3, 6/5}
 \draw [quivarr] (Q\t) -- (Q\h);
\foreach \t/\h/\where/\N in {3/2/above left/5, 2/4/below left/4, 4/5/below/3, 5/6/right/2, 6/3/above/1}
 \draw ($0.5*(Q\t)+0.5*(Q\h)$) node [\where] {\small $\N$};
\end{scope}
\end{tikzpicture}
%==================================
%\caption{Plabic graph, with alternating strand diagram, and dual quiver}
\caption{Plabic graph, with strand diagram, and (opposite) dual quiver}
\label{fig:shark}
\end{figure}
%==================================

From such a dimer model $Q$, we construct the dimer algebra $\algA$,
as in \cite[Def~2.11]{CKP},
and then $T=e\algA$ is a CM-module over $\algB=e\algA e$,
where $e$ is the sum of the idempotents at boundary vertices.
By \cite[Prop~3.6]{CKP}, 
with the convention differences noted in \S\ref{subsec:nota-conv},
such a $\algB$ is an $\ring$-order containing $\algC=\algC(k,n)$,
where $n$ is the number of boundary labels 
and $k$ is the `helicity', that is, 
the average clockwise increment of the strand permutation
or, equivalently, the size of the label on each alternating region in
the left target labelling of the Postnikov diagram
(cf.~\Cref{fig:rs-labelling}, where $(k,n)=(2,5)$).
We then know \cite[Thm~4.5]{Pr} that $T$ is a cluster tilting object in the cluster category $\GP\algB$
and $\algA=\End_\algB(T)\op$.
Note that this result also requires the Postnikov diagram to be connected.

The left target labelling of the Postnikov diagram also computes 
$\algB$ and $T$ as $\algC$-modules, by \cite[Thm~8.2]{CKP},
noting again the convention differences from \S\ref{subsec:nota-conv}.
Thus, in the example of~\Cref{fig:rs-labelling}, we have
\begin{equation}\label{eq:algB}
 {}_\algC\algB = M_{12}\oplus M_{23}\oplus M_{34}\oplus M_{14}\oplus M_{15},
 \qquad
 {}_\algC T={}_\algC\algB \oplus M_{24}.
\end{equation}
In particular, $\algB$ differs from $\algC$ by having $M_{14}$ as a projective in place of $M_{45}$,
which is not a $\algB$-module.
The positroid, in this case, consists of all 
$I\in\labsubset{5}{2}$ %$2$-sets in $\{1,\ldots,5\}$ 
except $45$.
%==================================
\begin{figure}[h]
%==================================
\begin{tikzpicture}[scale=1.7,
% strand/.style={teal},
 bdry/.style={thick, blue, densely dotted},
 quivvert/.style={red, fill},
 quivarr/.style={red, -latex},
 flowarr/.style={very thick,-latex},
 plabedge/.style={blue, thick},
 blackplabvert/.style={blue, fill},
 whiteplabvert/.style={blue, fill=white}]
\newcommand{\dotRad}{1pt}
\pgfmathsetmacro{\Qstep}{0.5}
\pgfmathsetmacro{\Qoff}{-0.25}
\pgfmathsetmacro{\Pstep}{0.35}
\pgfmathsetmacro{\Hstep}{2.9}
%===============
\begin{scope} [xscale=-1]
\foreach \n/\x/\y in {1/0/0, 5/-1/1, 3/-1/-1, 4/-2/0, 2/1.25/0}
  \coordinate (B\n) at (\x*\Pstep,\y*\Pstep);
\coordinate (V1) at ($(B2)+0.357*(1,1)$); 
\coordinate (V2) at ($(B2)+0.357*(1,-1)$); 
\coordinate (V12) at ($(B2)+0.22*(1,0)$); 
\coordinate (V3) at ($(B3)+0.422*(0,-1)$); 
\coordinate (V4) at ($(B4)+0.4*(-1,0)$); 
\coordinate (V5) at ($(B5)+0.422*(0,1)$); 
\foreach \a/\b in {1/2, 1/3, 4/5, 3/4, 1/5}
  \coordinate (A\a\b) at ($0.5*(B\a)+0.5*(B\b)$);
\foreach \n/\x/\y in {1/-2/1, 2/0.5/1, 3/-1/0, 4/2.3/0, 5/-2/-1, 6/0.5/-1}
  \coordinate (Q\n) at (\x*\Pstep,\y*\Pstep);
\foreach \n/\lab in {1/23, 2/12, 3/24, 4/15, 5/34, 6/14}
 \draw [quivvert] (Q\n) node (Q\n) {\tiny \lab};
\drawsharkstrands % macro at section head
\draw [bdry] (-0.1,0) ellipse (1 and 0.8);
\foreach \N/\where/\M in {1/above left/5, 2/below left/4, 3/below/3, 4/right/2, 5/above/1}
 \draw (V\N) node [\where] {\small $\M$};
\end{scope}
%===============
\begin{scope} [shift={(\Hstep,0)}, xscale=-1]
\foreach \n/\x/\y in {1/-1/1, 2/1/1, 3/0/0, 4/2/0, 5/-1/-1, 6/1/-1}
  \coordinate (Q\n) at (\x*\Qstep,\y*\Qstep);
\foreach \N/\lab in {1/23, 2/12, 3/24, 4/15, 5/34, 6/14}
 \draw (Q\N) node (Q\N) {\small \lab};
\foreach \t/\h in {1/2, 1/5, 2/4, 3/1, 4/6, 3/6, 2/3, 5/3, 6/2, 6/5}
 \draw [quivarr] (Q\t) -- (Q\h);
\end{scope}
\end{tikzpicture}
%==================================
\caption{Left target labelling and the Gabriel quiver of $\End_\algB(T)\op$}
\label{fig:rs-labelling}
\end{figure}
%==================================

\begin{definition}\label{def:match}
A \emph{(perfect) matching} $\match$ on a dimer model $Q=(Q_0,Q_1,Q_2)$ 
is a set of arrows $a\in Q_1$ containing exactly one arrow in the 
boundary of each face $f\in Q_2$.
\end{definition}

This definition is combinatorially equivalent to a perfect matching on the dual plabic graph $G$,
since, through the duality, the arrows of $Q$ are in bijection with the edges of $G$ 
and the faces of $Q$ are in bijection with the (internal) nodes of $G$.
Note that the points at which edges of $G$ meet the boundary are not considered to be nodes of $G$,
so are not required to be covered in a matching on $G$.

The edges of $G$ that meet the boundary carry the labels $1,\ldots,n$, 
which then also label the boundary arrows of $Q$, i.e. those that are in only one face.
Mirroring the definition in \cite[\S3.1]{MS}, the \emph{boundary value} $\bdry\match$ of a matching $\match$
consists of the labels of the anti-clockwise boundary arrows in $\match$ together with the 
labels of the clockwise arrows not in $\match$.

%==================================
\begin{figure}[h]
%==================================
\begin{tikzpicture}[scale=1.5,
 quivvert/.style={red, fill},
 quivarr/.style={red, -latex},
 dimerarr/.style={ultra thick, teal,-latex}]
\pgfmathsetmacro{\Qstep}{0.5}
\pgfmathsetmacro{\Qoff}{-1.75}
\pgfmathsetmacro{\Hstep}{2.4}
%=============== 
\begin{scope} [shift={(0,0)}, xscale=-1]
\foreach \n/\x/\y in {1/-1/1, 2/1/1, 3/0/0, 4/2/0, 5/-1/-1, 6/1/-1}
  \coordinate (Q\n) at (\x*\Qstep,\y*\Qstep);
\foreach \N in {1,2,3,4,5,6}
 \draw [quivvert] (Q\N) node (Q\N) {\tiny$\bullet$};
\foreach \t/\h in {1/2,4/6, 6/2, 3/1, 3/6, 5/3}
 \draw [quivarr] (Q\t) -- (Q\h);
\foreach \t/\h in {1/5, 6/5, 2/4, 2/3}
 \draw [dimerarr] (Q\t) -- (Q\h);
\draw (0,\Qoff*\Qstep) node {\small 35};
\end{scope}
%===============
\begin{scope} [shift={(\Hstep,0)}, xscale=-1]
\foreach \n/\x/\y in {1/-1/1, 2/1/1, 3/0/0, 4/2/0, 5/-1/-1, 6/1/-1}
  \coordinate (Q\n) at (\x*\Qstep,\y*\Qstep);
\foreach \N in {1,2,3,4,5,6}
 \draw [quivvert] (Q\N) node (Q\N) {\tiny$\bullet$};
\foreach \t/\h in {1/2,1/5, 4/6, 6/5, 6/2, 3/1, 3/6}
 \draw [quivarr] (Q\t) -- (Q\h);
\foreach \t/\h in {2/3, 5/3, 2/4}
 \draw [dimerarr] (Q\t) -- (Q\h);
\draw (0,\Qoff*\Qstep) node {\small 25};
\end{scope}
%===============
\begin{scope} [shift={(2*\Hstep,0)}, xscale=-1]
\foreach \n/\x/\y in {1/-1/1, 2/1/1, 3/0/0, 4/2/0, 5/-1/-1, 6/1/-1}
  \coordinate (Q\n) at (\x*\Qstep,\y*\Qstep);
\foreach \N in {1,2,3,4,5,6}
 \draw [quivvert] (Q\N) node (Q\N) {\tiny$\bullet$};
\foreach \t/\h in {1/2,1/5, 4/6, 6/5, 6/2, 2/3, 5/3}
 \draw [quivarr] (Q\t) -- (Q\h);
\foreach \t/\h in {3/1, 3/6, 2/4}
 \draw [dimerarr] (Q\t) -- (Q\h);
\draw (0,\Qoff*\Qstep) node {\small 25};
\end{scope}
%===============
\begin{scope} [shift={(3*\Hstep,0)}, xscale=-1]
\foreach \n/\x/\y in {1/-1/1, 2/1/1, 3/0/0, 4/2/0, 5/-1/-1, 6/1/-1}
  \coordinate (Q\n) at (\x*\Qstep,\y*\Qstep);
\foreach \N in {1,2,3,4,5,6}
 \draw [quivvert] (Q\N) node (Q\N) {\tiny$\bullet$};
\foreach \t/\h in {1/5, 2/4, 4/6, 6/5, 3/1, 3/6, 2/3}
 \draw [quivarr] (Q\t) -- (Q\h);
\foreach \t/\h in {1/2, 5/3, 6/2}
 \draw [dimerarr] (Q\t) -- (Q\h);
\draw (0,\Qoff*\Qstep) node {\small 12};
\end{scope}
\end{tikzpicture}
%==================================
\caption{Some matchings and their boundary values}
\label{fig:matchings}
\end{figure}
%==================================

Following \cite[\S4]{CKP}, a perfect matching $\match$ determines a rank 1 \emph{matching module} $\matmod{\match}$ as follows.
As a representation of $Q$, 
we put $\ring$ at each vertex, $t$ on every arrow in $\match$ and~$1$ on the other arrows.
From the construction, we see that, as a $\algC$-module
\begin{equation}\label{eq:bdry-mod}
 e \matmod{\match}\isom M_{\bdry\match}
\end{equation}
because $M_I$ can also be seen as a `matching module',
now on the circular double quiver \eqref{eq:circle-quiv},
where the indices in $I$ give the anticlockwise arrows which are $t$.

As observed in \cite[Prop.~8.6]{CKP}, $M_I$ is a $\algB$-module precisely 
when $I=\bdry\match$ for some matching $\match$ on $Q$.
One direction follows immediately from \eqref{eq:bdry-mod}.

% ======================================================================
\subsection{Grothendieck groups and quiver lattices}
% ======================================================================

In this subsection, when $\algA$ is a dimer algebra as in \S\ref{subsec:dim-mod},
we prove that the exact sequence in \Cref{Thm:2exactseq},
\begin{equation}\label{eq:grot-seq}
% (0\to)
 \ZZ \lraa{\delta} \Grot(\fd \algA)\lraa{\beta} \Grot(\CM\algA) \lraa{\rkk} \ZZ,
% (\to 0)
\end{equation}
is isomorphic to a subsequence of the extended cochain complex of 
the dimer model $Q=(Q_0, Q_1, Q_2)$,
that is,
\begin{equation}\label{eq:cochain}
  \ZZ \lraa{c} \ZZ^{Q_0} \lraa{\cob{0}} \ZZ^{Q_1} \lraa{\cob{1}} \ZZ^{Q_2},
\end{equation}
where $c(n)$ is the constant map $Q_0\to \ZZ$ with value $n$ and the coboundary maps are
\begin{equation*}
\begin{aligned}
\cob{0}\phi(a) &= \phi(ha)-\phi(ta), & &
\quad\text{for $\phi\in\ZZ^{Q_0}$, i.e.~$\phi\colon Q_0\to \ZZ$, and $a\in Q_1$,}\\
\cob{1}\theta(f) &= \sum_{a\in\bdry f} \theta(a), & &
\quad\text{for $\theta\in\ZZ^{Q_1}$, i.e.~$\theta\colon Q_1\to \ZZ$, and $f\in Q_2$.}
\end{aligned}
\end{equation*}
Note that \eqref{eq:cochain} is exact because the topological space associated to the 
cell complex~$Q$ is a disc and thus contractible.
The faces in $Q_2$ are implicitly oriented so that their boundary cycles are positive,
i.e.~not consistently with either orientation of the disc.

In terms of \eqref{eq:cochain}, 
we may say that a \emph{perfect matching} on $Q$ is a cochain $\mu\in\ZZ^{Q_1}$
which takes values in $\{0,1\}$ and for which $\cob{1}\mu$ is constant with value $1$.
Note that, with the second condition, it is enough to require that $\mu$ takes non-negative values.
One can then further define a \emph{multi-matching} to be a non-negative cochain $\mu\in\ZZ^{Q_1}$ 
such that $\cob{1}\mu$ is constant, so that these form a cone in the \emph{matching lattice}
\begin{equation}\label{eq:matlat}
 \Mlat = \bigl\{ \mu\in\ZZ^{Q_1} \st \text{$\cob{1}\mu$ is constant} \bigr\}.
\end{equation}
Thus we have the following subcomplex of \eqref{eq:cochain}
\begin{equation}\label{eq:mat-seq}
  \ZZ \lraa{c} \ZZ^{Q_0} \lraa{\cob{0}} \Mlat \lraa{\dg} \ZZ,
\end{equation}
where $\dg$ is defined so that, if $\mu\in\Mlat$, then $\cob{1}\mu$ is constant with value $\dg(\mu)$.
Strictly speaking, to view \eqref{eq:mat-seq} as a subcomplex of \eqref{eq:cochain}, 
the final $\ZZ$ must be embedded as the constant functions on faces.
Our aim is to show that \eqref{eq:mat-seq} is isomorphic to \eqref{eq:grot-seq}.

We already know that dimension vectors of finite dimensional modules 
give an isomorphism
$\dimv\colon \Grot(\fd \algA) \to \ZZ^{Q_0}$.
For the next terms, define 
\begin{equation}\label{eq:def-nu}
  \nu\colon \Grot(\CM\algA)\to \ZZ^{Q_1}
  \quad\text{by}\quad   \nu[X](a) = \dim \cok X_a.
\end{equation}
Note that $\cok X_a$ is finite dimensional, for all $a\in Q_1$, 
because $\rnk{\ring} X_{ta}=\rnk{\ring} X_{ha}$ and $X_a$ is injective.
This injectivity also implies that $\nu$ is additive on short exact sequences in $\CM\algA$,
by applying the Snake Lemma 
to the following diagram
\[
\begin{tikzpicture}[xscale=1.8,yscale=1.4]
\pgfmathsetmacro{\eps}{0.3}
\draw (0+\eps,1) node (A0) {$0$};
\draw (1,1) node (A1) {$X_{ta}$};
\draw (2,1) node (A2) {$Y_{ta}$};
\draw (3,1) node (A3) {$Z_{ta}$};
\draw (4-\eps,1) node (A4) {$0$};
\draw (0+\eps,0) node (B0) {$0$};
\draw (1,0) node (B1) {$X_{ha}$};
\draw (2,0) node (B2) {$Y_{ha}$};
\draw (3,0) node (B3) {$Z_{ha}$};
\draw (4-\eps,0) node (B4) {$0$};
\foreach \T/\H in {A0/A1, A1/A2, A2/A3, A3/A4, B0/B1, B1/B2, B2/B3, B3/B4}
  \draw[cdarr]  (\T) to (\H);
\foreach \T/\H/\lab in {A1/B1/$X_a$, A2/B2/$Y_a$, A3/B3/$Z_a$}
  \draw[cdarr] (\T) to node[right] {\small \lab} (\H);
\end{tikzpicture}
\]
and thus $\nu$ is a well-defined function on $\Grot(\CM\algA)$.
Furthermore, $\nu[X]\in\Mlat$, because, for any face $f\in Q_2$ with boundary cycle $a_1\cdots a_m$
(to and from $i\in Q_0$, say),
we have $X_{a_m}\cdots X_{a_1}=t\idmap$, 
and the cokernel dimensions of a composite of injective maps add up, 
so that 
\[
  \cob{1}\nu[X](f)=\dim \cok (t\idmap) =  \rnk{\ring} X_i = \rkk[X]
\]
independent of $f$.
Thus $\nu[X]\in\Mlat$ and $\dg\nu[X]=\rkk[X]$.

\begin{proposition}\label{Prop:equivof2exactseq} 
If $\algA$ is a dimer algebra as in \S\ref{subsec:dim-mod}, then
we have the following commutative diagram of (exact) sequences, 
where the horizontal maps are from \eqref{eq:grot-seq} and \eqref{eq:mat-seq} 
and all the vertical maps are isomorphisms.
\[
\begin{tikzpicture}[xscale=2,yscale=1.5]
\pgfmathsetmacro{\extra}{0.5}
\draw (0,0) node (A-) {$\ZZ$};
\draw (1,0) node (A0) {$\Grot(\fd\algA)$};
\draw (2+\extra,0) node (A1) {$\Grot(\CM\algA)$};
\draw (3+\extra,0) node (A2) {$\ZZ$};
\draw (0,-1) node (B-) {$\ZZ$};
\draw (1,-1) node (B0) {$\ZZ^{Q_0}$};
\draw (2+\extra,-1) node (B1) {$\Mlat$};
\draw (3+\extra,-1) node (B2) {$\ZZ$};
\foreach \T/\H in {A-/B-, A2/B2}
  \draw[equals]  (\T) to (\H);
\foreach \T/\H/\lab in {A-/A0/\delta, A0/A1/\beta, A1/A2/\rkk, B-/B0/c, B0/B1/\cob{0}, B1/B2/\dg}
  \draw[cdarr] (\T) to node[above] {\small $\lab$} (\H);
\foreach \T/\H/\lab in {A0/B0/\dimv, A1/B1/\nu}
  \draw[cdarr] (\T) to node[right] {\small $\lab$} (\H);
\end{tikzpicture}
\]
\end{proposition}

\begin{proof}
The fact that the third square commutes was just proved in the process of showing that 
$\nu$ can have codomain $\Mlat$.
The first square commutes because in this case $\rnk{\algB} T_i=1$, for all $i\in Q_0$.

For the second square, suppose that $M\in \fd\algA$ has a CM-presentation $X\to Y\to M$.
Then $\beta[M] = [Y]-[X]$ and, by a simple application of the Snake Lemma,
\begin{equation}\label{eq:nu-beta}
 \nu([Y]-[X]) (a) = \dim \cok M_a - \dim \ker M_a.
\end{equation}
On the other hand, 
\begin{equation}\label{eq:cob-dim}
 (\cob{1}\dimv[M])(a) = \dim M_{ha} - \dim M_{ta}
\end{equation}
and the right-hand sides of \eqref{eq:cob-dim} and \eqref{eq:nu-beta} agree 
(e.g.~using the Rank-Nullity Theorem).
Thus $\nu\beta[M]=\cob{1}\dimv[M]$, as required.

Finally, we deduce that $\nu$ is an isomorphism by the Five Lemma.
\end{proof}

\begin{remark}\label{rem:rank1=match}
For any $X\in \CM\algA$, note that $\nu[X]$ is non-negative on all $a\in Q_1$, so it is a multi-matching.
In particular, if $\rkk[X]=1$ then $\nu[X]$ is a (perfect) matching.

In \cite[Cor.~4.6]{CKP}, it is shown that every rank 1 module $X\in \CM\algA$ is isomorphic to some $\matmod{\match}$,
and in fact the proof shows that $X\isom \matmod{\match}$ for $\match=\nu[X]$.
In particular,
\begin{equation}\label{eq:nu=match}
  \nu [\matmod{\match}] = \match.
\end{equation}
The isomorphism of sequences in \Cref{Prop:equivof2exactseq} is 
essentially the inverse of the one in \cite[Prop.~6.17]{CKP},
after making the identification $\Grot(\CM\algA)=\Grot(\proj\algA)$. 
In particular $\nu$ is the inverse of the map 
\[
  \eta\colon \Mlat\to \Grot(\proj\algA)
\]
from \cite{CKP}, which is defined so that $[\matmod{\match}]=\eta(\match)$, for any matching $\match$.
However $\nu$ and~$\eta$ are not transparently inverse to each other,
because both the identification and $\eta$ depend (at least, implicitly) on projective resolution, while $\nu$ does not.
\end{remark}

As a special case of \Cref{Prop:equivof2exactseq}, we have

\begin{lemma}\label{lem:dWt=mat-diff}
For a matching module $\matmod{\match}$,
\begin{equation}\label{eq:dWt=mat-diff}
  \cob{0} \dimv \wtmod \matmod{\match} = \match_\vstar - \match,
\end{equation}
where $\match_\vstar$ is the matching such that $\Hom(T,\modJ) \isom \matmod{\match_\vstar}$.
\end{lemma}

\begin{proof}
Since $\rnk{\algA}\matmod{\match}=1$, we know $\funJ \efunct \matmod{\match}\isom \modJ$.
Hence, by definition of $\wtmod \matmod{\match}$, as in \eqref{eq:wtmod-diag}, and of $\match_\vstar$,
there is a short exact sequence 
\[
 \ShExSeq{ \matmod{\match} }{ \matmod{\match_\vstar} }{ \wtmod \matmod{\match} },
\]
so that $\beta [\wtmod \matmod{\match}] = [\matmod{\match_\vstar}] - [\matmod{\match}]$.
The result follows by applying $\nu$, and using \eqref{eq:nu=match} 
and the middle square of \Cref{Prop:equivof2exactseq}.
\end{proof}

By \Cref{Lem:somefacts}\itmref{itm:facts3}, 
we know that $\wtmod \matmod{\match}$ vanishes at the vertex $\vstar$,
so \eqref{eq:dWt=mat-diff} determines $\dimv \wtmod \matmod{\match}$,
and thus $\wt[\matmod{\match}]$, uniquely.

More generally, the map $\wt\colon \Grot(\CM\algA) \to \Grot(\fd\algA)$ in \eqref{eq:wt-class}
is identified with a map $\wtcom\colon \Mlat\to \ZZ^{Q_0}$,
taking values in $\Nstar$, as in \eqref{eq:def-Nstar},
and determined by
\[ 
\cob{0} \wtcom(\mu) = (\dg\mu)\match_\vstar - \mu.
\]
% $\wtcom(\mu)=w$, where $dw=(\dg\mu)\match_\vstar - \mu$.

% ======================================================================
\subsection{The classical partition function}\label{subsec:class-PF}
% ======================================================================

\newcommand{\Gm}{\CC^*}
\newcommand{\Qtor}{\widehat{\mathbb{T}}}
\newcommand{\Btor}{\mathbb{T}}
\newcommand{\charpi}{p}

The (formal) partition function associated to a $k$-set $I\subset \labset{n}$ is 
\begin{equation}\label{eq:class-pf}
  \partfun_I=\sum_{\match\st\bdry \match=I} \xvar^\match,
% \quad \in \CC[\Mlat]
\end{equation}
taking values in $\CC[\Mlat]$. 

We can interpret \eqref{eq:class-pf} 
as in \cite[\S1.5]{MS} by choosing edge weights $\{\xvar_{a} \st a\in Q_1\}$ 
on the quiver (or the dual plabic graph $G$) and writing $\xvar^\match=\prod_{a\in \match} \xvar_a$.
Then $\xvar^\match$ defines a function $(\Gm)^{Q_1}\to \Gm$, 
which is semi-invariant for the action of a `gauge torus' $(\Gm)^{Q_2}$
with respect to the character
\[
  \charpi\colon (\Gm)^{Q_2}\to\Gm\colon (t_f)\mapsto \prod_{f\in Q_2} t_f.
\]
Now $\Mlat$ is naturally the character lattice of the
quotient torus $\Qtor=(\Gm)^{Q_1}/\ker\charpi$
and so $\CC[\Mlat]$ is naturally the ($\ZZ$-graded) coordinate ring $\CC[\Qtor]$,
which we can think of as the homogeneous coordinate ring of $\Btor=(\Gm)^{Q_1}/(\Gm)^{Q_2}$.

Recall from \eqref{eq:part-fun} that the generalised partition function for any $M\in\CM\algB$ is
\[
\ptfn_{M}=
\sum_{[X]\in \sumrange{M}} \euler\bigl( \qvGrsH{\ddel{[X]}}{T}{M} \bigr)  \xvar^{[X]},
%\quad \in \CC[\Grot(\CM\algA)]
\] 
taking values in $\CC[\Grot(\CM\algA)]$.
In particular, we can consider this formula for a rank~1 module $M_I$.

\begin{proposition} \label{prop:partitionchar} 
Under the isomorphism $\nu\colon \Grot(\CM\algA)\to\Mlat$ 
from \Cref{Prop:equivof2exactseq}, 
$\partfun_I$ is identified with $\ptfn_{M_I}$,
that is, $\partfun_I= \pbfun{\nu}{\ptfn_{M_I}}$,
in the notation of \Cref{rem:new-clucha}.
\end{proposition}

\begin{proof}
By \cite[Prop.~5.5 \& Rem.~5.6]{CKP}, there is a bijection
\[
 \theta\colon \{X \leq \Hom(T, M_I) \st eX= M_I\} \to \{ \match\st \bdry\match=I\}
 \colon X\mapsto \nu[X]
\]
Note that $\theta$ is well-defined, because every $X$ in the domain has 
$\rnk{\algA} X=\rnk{\algB} eX=1$, by  \Cref{Lem:ranks}\itmref{itm:rks3}, and so,
by \Cref{rem:rank1=match}, $X$ is a matching module, indeed $X\isom \matmod{\match}$ for $\match=\nu[X]$.
Furthermore, since $e \matmod{\match}\isom M_{\bdry\match}$, by \cite[Prop.~4.9]{CKP}, $\bdry\match=I$.

Then $\theta$ is surjective, because, if $\bdry\match=I$, then an isomorphism $e \matmod{\match}\to M_{I}$
induces an injective map $\matmod{\match}\to \Hom(T,M_{I})$ by adjunction.
Thus the ranges in the sums for $\partfun_I$ and $\ptfn_{M_I}$ are identified by $\nu$.

The fact that $\theta$ is injective is the delicate part of the proof in \cite{CKP}, 
but this then means that there is a single submodule of $\Hom(T, M_I)$
in each of the relevant classes $[X]\in \sumrange{M_I}$,
so that $\qvGrsH{\ddel{[X]}}{T}{M_I}$ is a single point and thus has $\euler=1$.
In other words, every non-zero coefficient in the sum for $\ptfn_{M_I}$ is 1 
and so the sums agree under the identification induced by $\nu$.
\end{proof}

Note that this result and its proof follow closely \cite[Theorem~9.3]{CKP},
where an analogous formula is given for a slightly different partition function due to Marsh--Scott.

% ======================================================================
\subsection{Poincar\'e duality and flows}
% ======================================================================
\newcommand{\Or}{\mathbb{O}}
\newcommand{\pdiso}{\phi}
\newcommand{\fundcls}{f}
\newcommand{\perfor}{\mathcal{O}}

Recall from \eqref{eq:gen-flow-poly} that we also have a generalised flow polynomial
$\fp^T_M=\pbfun{\wt}{\ptfn^T_{M}}$, taking values in $\CC[\Grot(\fd\algA)]$.
From \Cref{prop:partitionchar} and \Cref{lem:dWt=mat-diff}
it follows that
\begin{equation}\label{eq:fp-rank1}
 \fp^T_{M_I} %= \sum_{\match\st\bdry \match=I} \yvar^{\wt[\matmod{\match}]}.
 = \sum_{\match\st\bdry \match=I} \yvar^{\wtcom(\match)}.
\end{equation}
Note that \Cref{lem:dWt=mat-diff} says that $\wt[\matmod{\match}]=\wtcom(\match)$. 
The quantity $\match_\vstar - \match$ in \eqref{eq:dWt=mat-diff} 
can be regarded as a `cohomological flow',
that is, a cocycle in $\ZZ^{Q_1}$ which takes values just  in $\{0,\pm1\}$.
Then $\match_\vstar - \match=\cob{0} \wtcom(\match)$.
See \Cref{fig:coh-flow} for some examples, where the cocycle value $0$ is omitted and $\pm$ denote $\pm1$.
Note that $\match_\vstar$ is the first matching in \Cref{fig:matchings} 
and we chose $\match$ to be each of the other three.

%==================================
\begin{figure}[h]
%==================================
\begin{tikzpicture}[scale=1.9,
% quivvert/.style={red, fill},
 wtvert/.style={black, fill},
 quivarr/.style={red, thick, -latex},
 blankit/.style={white, fill=white},
 arrlab/.style={teal}]
\pgfmathsetmacro{\Qstep}{0.5}
\pgfmathsetmacro{\Hstep}{2.6}
\pgfmathsetmacro{\blankrad}{0.08}
\newcommand{\qdot}{1.4pt}
%=============== 
\begin{scope} [shift={(0,0)}, xscale=-1]
\foreach \n/\x/\y in {1/-1/1, 2/1/1, 3/0/0, 4/2/0, 5/-1/-1, 6/1/-1}
  \coordinate (Q\n) at (\x*\Qstep,\y*\Qstep);
\foreach \N/\W in {1/0, 2/*, 3/0, 4/0, 5/1, 6/0}
 \draw [wtvert] (Q\N) node (Q\N) {\small $\W$};
\foreach \t/\h in {1/2,1/5, 4/6, 6/5, 6/2, 3/1, 3/6, 2/3, 5/3, 2/4}
 \draw [quivarr] (Q\t) -- (Q\h);
\foreach \t/\h in {5/3}
{\coordinate (M) at ($0.5*(Q\t) + 0.5*(Q\h)$);
 \draw [blankit] (M) circle (\blankrad);
 \draw [arrlab] (M) node {\small $-$};}
\foreach \t/\h in {1/5, 6/5}
{\coordinate (M) at ($0.5*(Q\t) + 0.5*(Q\h)$);
 \draw [blankit] (M) circle (\blankrad);
 \draw [arrlab] (M) node {\small $+$};}
\end{scope}
%===============
\begin{scope} [shift={(\Hstep,0)}, xscale=-1]
\foreach \n/\x/\y in {1/-1/1, 2/1/1, 3/0/0, 4/2/0, 5/-1/-1, 6/1/-1}
  \coordinate (Q\n) at (\x*\Qstep,\y*\Qstep);
\foreach \N/\W in {1/0, 2/*, 3/1, 4/0, 5/1, 6/0}
 \draw [wtvert] (Q\N) node (Q\N) {\small $\W$};
\foreach \t/\h in {1/2,1/5, 4/6, 6/5, 6/2, 2/3, 5/3, 3/1, 3/6, 2/4}
 \draw [quivarr] (Q\t) -- (Q\h);
\foreach \t/\h in {3/1, 3/6}
{\coordinate (M) at ($0.5*(Q\t) + 0.5*(Q\h)$);
 \draw [blankit] (M) circle (\blankrad);
 \draw [arrlab] (M) node {\small $-$};}
\foreach \t/\h in {1/5, 6/5, 2/3}
{\coordinate (M) at ($0.5*(Q\t) + 0.5*(Q\h)$);
 \draw [blankit] (M) circle (\blankrad);
 \draw [arrlab] (M) node {\small $+$};}
\end{scope}
%===============
\begin{scope} [shift={(2*\Hstep,0)}, xscale=-1]
\foreach \n/\x/\y in {1/-1/1, 2/1/1, 3/0/0, 4/2/0, 5/-1/-1, 6/1/-1}
  \coordinate (Q\n) at (\x*\Qstep,\y*\Qstep);
\foreach \N/\W in {1/1, 2/*, 3/1, 4/1, 5/2, 6/1}
 \draw [wtvert] (Q\N) node (Q\N) {\small $\W$};
\foreach \t/\h in {1/2,1/5, 4/6, 6/5, 6/2, 2/3, 5/3, 3/1, 3/6, 2/4}
 \draw [quivarr] (Q\t) -- (Q\h);
\foreach \t/\h in {1/2, 5/3, 6/2}
{\coordinate (M) at ($0.5*(Q\t) + 0.5*(Q\h)$);
 \draw [blankit] (M) circle (\blankrad);
 \draw [arrlab] (M) node {\small $-$};}
\foreach \t/\h in {1/5, 6/5, 2/4, 2/3}
{\coordinate (M) at ($0.5*(Q\t) + 0.5*(Q\h)$);
 \draw [blankit] (M) circle (\blankrad);
 \draw [arrlab] (M) node {\small $+$};}
\end{scope}
\end{tikzpicture}
%==================================
\caption{Cohomological flows $\match_*-\match$ and their weights $\wtcom(\match)$}
\label{fig:coh-flow}
\end{figure}
%==================================

To view \eqref{eq:fp-rank1} as a classical flow polynomial, 
we need to interpret matchings and weights on the plabic graph $G$.
We do this using Poincar\'e duality, which is an isomorphism~$\pdiso_{\bullet}$ 
between the (extended) cochain complex for $Q$ and the chain complex for $G$,
extended by a choice of fundamental class $\fundcls\in\Or G_2$,
i.e.~a map $\ZZ\to\Or G_2\colon 1\mapsto \fundcls$.
\begin{equation}\label{eq:pdiso}
\begin{tikzpicture}[xscale=2,yscale=1.5,baseline=(bb.base)]
% equals/.style={double=none, double distance=2pt}]
\pgfmathsetmacro{\extra}{0}
\coordinate (bb) at (0,-0.5);
\draw (0,0) node (A-) {$\ZZ$};
\draw (1,0) node (A0) {$\ZZ^{Q_0}$};
\draw (2+\extra,0) node (A1) {$\ZZ^{Q_1}$};
\draw (3+\extra,0) node (A2) {$\ZZ^{Q_2}$};
\draw (0,-1) node (B-) {$\ZZ$};
\draw (1,-1) node (B0) {$\Or G_2$};
\draw (2+\extra,-1) node (B1) {$\Or G_1$};
\draw (3+\extra,-1) node (B2) {$\Or G_0$};
\foreach \T/\H in {A-/B-}
  \draw[equals]  (\T) to (\H);
\foreach \T/\H/\lab in {A-/A0/c, A0/A1/\cob{0}, A1/A2/\cob{1}, B-/B0/\fundcls, B0/B1/\bdry, B1/B2/\bdry}
  \draw[cdarr] (\T) to node[above] {\small $\lab$}(\H);
\foreach \n in {0,1,2} 
  \draw[cdarr] (A\n) to node[right] {\small $\pdiso_\n$} (B\n);
\end{tikzpicture}
\end{equation}
Note that in \eqref{eq:pdiso} we write the canonical chain complex of the \emph{open} disc 
divided up by the unoriented graph $G$, where 
\[
  \Or G_i = \bigoplus_{g\in G_i} \Or_g
\]
and $\Or_g$ is the rank~one $\ZZ$-module 
whose two primitive elements
are the two orientations of the unoriented $i$-cell $g$.
We don't need to do this for the cochain complex of $Q$ because all those cells are already oriented.
The isomorphism $\pdiso_{\bullet}$ is uniquely determined 
by the choice of fundamental class $\fundcls$, or equivalently an orientation of the disc.
Notice that the chain complex here must be for the open disc.
In particular, some edges in $G_1$ have only one (internal) node in $G_0$ as their boundary, 
to match the fact that some arrows in $Q_1$ are in the boundary of only one face.

Under the isomorphism $\pdiso_1$ a cohomological flow becomes a `homological flow',
that is, a cycle in $\Or G_1$ made up just of oriented edges.
Such a flow $\flow$ must be an edge-disjoint union of oriented paths beginning and ending on the boundary
and we write $\flow\colon I \rightsquigarrow  J$ if the incoming boundary edges are labelled by $I\setminus J$
and the outgoing boundary edges are labelled by $J\setminus I$;
see \Cref{fig:hom-flow} for examples corresponding to the cohomological flows in \Cref{fig:coh-flow}.

%==================================
\begin{figure}[h]
%==================================
\begin{tikzpicture} [scale=1.7, 
 bdry/.style={thick, blue, densely dotted},
 quivvert/.style={red, fill},
 flowarr/.style={very thick,-latex},
 plabedge/.style={blue, thick},
 blackplabvert/.style={blue, fill},
 whiteplabvert/.style={blue, fill=white}]
\newcommand{\dotRad}{1pt}
\pgfmathsetmacro{\Rad}{1}
\pgfmathsetmacro{\Qstep}{0.5}
\pgfmathsetmacro{\Qoff}{-0.25}
\pgfmathsetmacro{\LXoff}{0.4}
\pgfmathsetmacro{\LYoff}{-1.1}
\pgfmathsetmacro{\Pstep}{0.35}
\pgfmathsetmacro{\Hstep}{2.9}
%===============
\begin{scope} [xscale=-1]
\foreach \n/\x/\y in {1/0/0, 5/-1/1, 3/-1/-1, 4/-2/0, 2/1.25/0}
  \coordinate (B\n) at (\x*\Pstep,\y*\Pstep);
\coordinate (V1) at ($(B2)+0.357*(1,1)$); 
\coordinate (V2) at ($(B2)+0.357*(1,-1)$); 
\coordinate (V12) at ($(B2)+0.22*(1,0)$); 
\coordinate (V3) at ($(B3)+0.422*(0,-1)$); 
\coordinate (V4) at ($(B4)+0.4*(-1,0)$); 
\coordinate (V5) at ($(B5)+0.422*(0,1)$); 
\foreach \a/\b in {1/2, 1/3, 4/5, 3/4, 1/5}
  \coordinate (A\a\b) at ($0.5*(B\a)+0.5*(B\b)$);
\foreach \n/\x/\y in {1/-2/1, 2/0.5/1, 3/-1/0, 4/2/0, 5/-2/-1, 6/0.5/-1}
  \coordinate (Q\n) at (\x*\Pstep,\y*\Pstep);
\foreach \N/\W in {1/0, 2/*, 3/0, 4/0, 5/1, 6/0}
 \draw [quivvert] (Q\N) node (Q\N) {\small $\W$};
\foreach \n/\N in {2/1, 2/2, 3/3, 4/4, 5/5}
 \draw [plabedge] (B\n)--(V\N);
\foreach \n/\m in {1/2, 1/5, 1/3, 4/5, 3/4}
 \draw [plabedge] (B\n)--(B\m);
\foreach \n in {2,3,5}
  \draw [blackplabvert] (B\n) circle (\dotRad);
\foreach \n in {1,4}
 \draw [whiteplabvert] (B\n) circle (\dotRad);
\draw [bdry] (-0.1,0) ellipse (1 and 0.8);
\foreach \N/\where/\M in {1/above left/5, 2/below left/4, 3/below/3, 4/right/2, 5/above/1}
 \draw (V\N) node [\where] {\small $\M$};
\foreach \T/\H in {V4/B4,B4/B3, B3/V3} 
{ \coordinate (M1) at ($0.57*(\T) + 0.43*(\H)$);
  \coordinate (M2) at ($0.43*(\T) + 0.57*(\H)$);
  \draw [flowarr] (M1)--(M2);}
\draw (\LXoff,\LYoff) node {$25\rightsquigarrow 35$};
\end{scope}
%===============
\begin{scope} [shift={(\Hstep,0)},xscale=-1]
\foreach \n/\x/\y in {1/0/0, 5/-1/1, 3/-1/-1, 4/-2/0, 2/1.25/0}
  \coordinate (B\n) at (\x*\Pstep,\y*\Pstep);
\coordinate (V1) at ($(B2)+0.357*(1,1)$); 
\coordinate (V2) at ($(B2)+0.357*(1,-1)$); 
\coordinate (V12) at ($(B2)+0.22*(1,0)$); 
\coordinate (V3) at ($(B3)+0.422*(0,-1)$); 
\coordinate (V4) at ($(B4)+0.4*(-1,0)$); 
\coordinate (V5) at ($(B5)+0.422*(0,1)$); 
\foreach \a/\b in {1/2, 1/3, 4/5, 3/4, 1/5}
  \coordinate (A\a\b) at ($0.5*(B\a)+0.5*(B\b)$);
\foreach \n/\x/\y in {1/-2/1, 2/0.5/1, 3/-1/0, 4/2/0, 5/-2/-1, 6/0.5/-1}
  \coordinate (Q\n) at (\x*\Pstep,\y*\Pstep);
\foreach \N/\W in {1/0, 2/*, 3/1, 4/0, 5/1, 6/0}
 \draw [quivvert] (Q\N) node (Q\N) {\small $\W$};
\foreach \n/\N in {2/1, 2/2, 3/3, 4/4, 5/5}
 \draw [plabedge] (B\n)--(V\N);
\foreach \n/\m in {1/2, 1/5, 1/3, 4/5, 3/4}
 \draw [plabedge] (B\n)--(B\m);
\foreach \n in {2,3,5}
  \draw [blackplabvert] (B\n) circle (\dotRad);
\foreach \n in {1,4}
 \draw [whiteplabvert] (B\n) circle (\dotRad);
\draw [bdry] (-0.1,0) ellipse (1 and 0.8);
\foreach \N/\where/\M in {1/above left/5, 2/below left/4, 3/below/3, 4/right/2, 5/above/1}
 \draw (V\N) node [\where] {\small $\M$};
\foreach \T/\H in {V4/B4,B4/B5, B5/B1, B1/B3, B3/V3} 
{ \coordinate (M1) at ($0.57*(\T) + 0.43*(\H)$);
  \coordinate (M2) at ($0.43*(\T) + 0.57*(\H)$);
  \draw [flowarr] (M1)--(M2);}
\draw (\LXoff,\LYoff) node {$25\rightsquigarrow 35$};
\end{scope}
%===============
\begin{scope} [shift={(2*\Hstep,0)},xscale=-1]
\foreach \n/\x/\y in {1/0/0, 5/-1/1, 3/-1/-1, 4/-2/0, 2/1.25/0}
  \coordinate (B\n) at (\x*\Pstep,\y*\Pstep);
\coordinate (V1) at ($(B2)+0.357*(1,1)$); 
\coordinate (V2) at ($(B2)+0.357*(1,-1)$); 
\coordinate (V12) at ($(B2)+0.22*(1,0)$); 
\coordinate (V3) at ($(B3)+0.422*(0,-1)$); 
\coordinate (V4) at ($(B4)+0.4*(-1,0)$); 
\coordinate (V5) at ($(B5)+0.422*(0,1)$); 
\foreach \a/\b in {1/2, 1/3, 4/5, 3/4, 1/5}
  \coordinate (A\a\b) at ($0.5*(B\a)+0.5*(B\b)$);
\foreach \n/\x/\y in {1/-2/1, 2/0.5/1, 3/-1/0, 4/2/0, 5/-2/-1, 6/0.5/-1}
  \coordinate (Q\n) at (\x*\Pstep,\y*\Pstep);
\foreach \N/\W in {1/1, 2/*, 3/1, 4/1, 5/2, 6/1}
 \draw [quivvert] (Q\N) node (Q\N) {\small $\W$};
\foreach \n/\N in {2/1, 2/2, 3/3, 4/4, 5/5}
 \draw [plabedge] (B\n)--(V\N);
\foreach \n/\m in {1/2, 1/5, 1/3, 4/5, 3/4}
 \draw [plabedge] (B\n)--(B\m);
\foreach \n in {2,3,5}
  \draw [blackplabvert] (B\n) circle (\dotRad);
\foreach \n in {1,4}
 \draw [whiteplabvert] (B\n) circle (\dotRad);
\draw [bdry] (-0.1,0) ellipse (1 and 0.8);
\foreach \N/\where/\M in {1/above left/5, 2/below left/4, 3/below/3, 4/right/2, 5/above/1}
 \draw (V\N) node [\where] {\small $\M$};
\foreach \T/\H in {V4/B4,B4/B3, B3/V3, V5/B5, B5/B1, B1/B2, B2/V1} 
{ \coordinate (M1) at ($0.57*(\T) + 0.43*(\H)$);
  \coordinate (M2) at ($0.43*(\T) + 0.57*(\H)$);
  \draw [flowarr] (M1)--(M2);}
\draw (\LXoff,\LYoff) node {$12\rightsquigarrow 35$};
\end{scope}
\end{tikzpicture}
%==================================
\caption{Dual homological flows $\flow$ and their weights $\wtcom(\flow)$}
\label{fig:hom-flow}
\end{figure}
%==================================

The matching $\match_\vstar$ determines a `perfect orientation' $\perfor_\vstar$,
as illustrated in \Cref{fig:perf-orient}.
More precisely, $\perfor_\vstar$ is an orientation for each edge in $G_1$, 
from black and/or to white if the edge is dual to an arrow in $\match_\vstar$
and the other way otherwise.
The outgoing boundary edges of $\perfor_\vstar$ are precisely those in 
$I_\vstar=\bdry\match_\vstar$,
that is, $M_{I_\vstar}=\modJ$.

%==================================
\begin{figure}[h]
%==================================
\begin{tikzpicture} [scale=1.7, 
 bdry/.style={thick, blue, densely dotted},
 quivvert/.style={red, fill},
 quivarr/.style={red, -latex},
 dimerarr/.style={ultra thick, teal,-latex},
 flowarr/.style={very thick,-latex},
 plabedge/.style={blue, thick},
 blackplabvert/.style={blue, fill},
 whiteplabvert/.style={blue, fill=white}]
\newcommand{\dotRad}{1pt}
\pgfmathsetmacro{\Rad}{1}
\pgfmathsetmacro{\Qstep}{0.5}
%\pgfmathsetmacro{\Qoff}{-0.25}
\pgfmathsetmacro{\LXoff}{0.4}
\pgfmathsetmacro{\LYoff}{-1.1}
\pgfmathsetmacro{\Pstep}{0.35}
\pgfmathsetmacro{\Hstep}{2.5}
%===============
\begin{scope} [shift={(-\Hstep,0)}, xscale=-1]
\foreach \n/\x/\y in {1/-1/1, 2/1/1, 3/0/0, 4/2/0, 5/-1/-1, 6/1/-1}
  \coordinate (Q\n) at (\x*\Qstep,\y*\Qstep);
\foreach \N in {1,3,4,5,6}
 \draw [quivvert] (Q\N) node (Q\N) {\tiny$\bullet$};
\draw [quivvert] (Q2) node (Q2) {\small$*$};
\foreach \t/\h in {1/2,4/6, 6/2, 3/1, 3/6, 5/3}
 \draw [quivarr] (Q\t) -- (Q\h);
\foreach \t/\h in {1/5, 6/5, 2/4, 2/3}
 \draw [dimerarr] (Q\t) -- (Q\h);
\end{scope}
%=============== 
\begin{scope} [xscale=-1]
\foreach \n/\x/\y in {1/0/0, 5/-1/1, 3/-1/-1, 4/-2/0, 2/1.25/0}
  \coordinate (B\n) at (\x*\Pstep,\y*\Pstep);
\coordinate (V1) at ($(B2)+0.357*(1,1)$); 
\coordinate (V2) at ($(B2)+0.357*(1,-1)$); 
\coordinate (V12) at ($(B2)+0.22*(1,0)$); 
\coordinate (V3) at ($(B3)+0.422*(0,-1)$); 
\coordinate (V4) at ($(B4)+0.4*(-1,0)$); 
\coordinate (V5) at ($(B5)+0.422*(0,1)$); 
\foreach \a/\b in {1/2, 1/3, 4/5, 3/4, 1/5}
  \coordinate (A\a\b) at ($0.5*(B\a)+0.5*(B\b)$);
 \coordinate (Q2) at (0.5*\Pstep,1*\Pstep);
 \draw [quivvert] (Q2) node (Q2) {\small $*$};
\foreach \n/\N in {2/1, 2/2, 3/3, 4/4, 5/5}
 \draw [plabedge] (B\n)--(V\N);
\foreach \n/\m in {1/2, 1/5, 1/3, 4/5, 3/4}
 \draw [plabedge] (B\n)--(B\m);
\foreach \n in {2,3,5}
  \draw [blackplabvert] (B\n) circle (\dotRad);
\foreach \n in {1,4}
 \draw [whiteplabvert] (B\n) circle (\dotRad);
\draw [bdry] (-0.1,0) ellipse (1 and 0.8);
\foreach \N/\where/\M in {1/above left/5, 2/below left/4, 3/below/3, 4/right/2, 5/above/1}
 \draw (V\N) node [\where] {\small $\M$};
\foreach \T/\H in {V4/B4,B3/V3, B2/V1, B5/B1} 
{ \coordinate (M1) at ($0.57*(\T) + 0.43*(\H)$);
  \coordinate (M2) at ($0.43*(\T) + 0.57*(\H)$);
  \draw [teal, flowarr] (M1)--(M2);}
\foreach \T/\H in {B4/B3, V2/B2, B1/B2, B1/B3, V5/B5, B4/B5} 
{ \coordinate (M1) at ($0.57*(\T) + 0.43*(\H)$);
  \coordinate (M2) at ($0.43*(\T) + 0.57*(\H)$);
  \draw [flowarr] (M1)--(M2);}
\end{scope}
\end{tikzpicture}
%==================================
\caption{The matching $\match_\vstar$ and perfect orientation $\perfor_\vstar$ ($\bdry\match_\vstar=35$).}
\label{fig:perf-orient}
\end{figure}
%==================================

The following result is essentially \cite[Lemma 12.3]{RW}, 
but with a slight change of language and conventions
(see \Cref{rem:compare-RW} for more explanation).

\begin{lemma}\label{lem:match-flow}
The map $\match\mapsto\pdiso_1(\match_\vstar-\match)$ is a bijection between matchings 
$\match$ with $\bdry\match=I$ and flows $\flow\colon I \rightsquigarrow I_{\vstar}$ 
that are contained in the perfect orientation $\perfor_\vstar$.
\end{lemma}

\begin{proof}
To see this, treat a matching $\match$ as a set of arrows in $Q_1$ and a flow $\flow$ as a set of edges in $G_1$, 
implicitly oriented in the direction given by $\perfor_\vstar$.
Identifying $Q_1$ with $G_1$ by duality, 
the equation $\flow=\pdiso_1(\match_\vstar-\match)$ can be written
\[
  \flow = (\match \setminus\match_\vstar) \cup (\match_\vstar \setminus \match),
\]
which can be inverted as
\[
  \match= (\flow\setminus \match_\vstar) \cup (\match_\vstar \setminus \flow).
\]
The fact that $\match$ is a subset with one arrow in each face in $Q_2$ corresponds to the fact that
$\flow$ is a subset with one incoming and one outgoing edge at each node in $G_0$ it passes through.
Since $\flow$ is in $\perfor_\vstar$ precisely one of those two edges must be in $\match_\vstar$.

Finally, to have $\flow\colon \bdry\match \rightsquigarrow \bdry\match_{\vstar}$,
we choose the fundamental class so that $\pdiso_1$ maps anti-clockwise boundary arrows in $Q_1$ 
to outgoing boundary edges in $G_1$.
\end{proof}

Flows in~$\perfor_\vstar$ are vertex-disjoint unions of paths in~$\perfor_\vstar$.
We can compute $\wtcom(\match)$ 
%the (unique) weight $w\in\Nstar\cap \NN^{Q_0}$ with $\cob{0} w=\match_\vstar-\match$ 
from the flow $\flow=\pdiso_1(\match_\vstar-\match)$ as in \cite[\S6]{RW}.
That is, the weight $\wtcom(\flow)\in \NN^{G_2}\isom \NN^{Q_0}$ is given by
summing functions~$w_\flowpath$, over the paths~$\flowpath$ in~$\flow$,
where $w_\flowpath$ is 1 on all faces to the left of~$\flowpath$ and 0 on all faces to the right of~$\flowpath$.
See \Cref{fig:hom-flow} for examples.
(Strictly, $\wtcom(\flow)$ here is the exponent of the `weight' in \cite{RW}).

\begin{remark}\label{rem:flow-wt}
Note that, because any path~$\flowpath$ in $\perfor_\vstar$ is itself a flow 
and so corresponds, as in \Cref{lem:match-flow}, to a matching $\match$, 
we know that $w_\flowpath(\vstar)=0$, 
that is (the face corresponding to) $\vstar$ is on the right of~$\flowpath$.
Thus $\wtcom(\flow)\in\Nstar\cap \NN^{Q_0}$, as required.
This further means that $I_\vstar=\bdry\match_\vstar$ is lexicographically maximal in 
the positroid (i.e.~the set of all boundary values of matchings),
since otherwise there would be a flow $I \rightsquigarrow I_{\vstar}$ with $\vstar$ on the left of some path.
It also follows from \Cref{Thm:char-parfunB}\itmref{itm:pfun5} and \Cref{prop:partitionchar} 
that $\match_\vstar$ is the unique matching with boundary value $I_\vstar$,
recovering the more combinatorial way to characterise $\match_\vstar$ (cf. \cite[Remark~6.4]{RW}).
\end{remark}

The observations above lead to the following conclusion.

\begin{proposition}\label{prop:class-flow-poly}
We can rewrite \eqref{eq:fp-rank1} as a classical flow polynomial (cf. \cite[(6.3)]{RW}).
\begin{equation}\label{eq:class-flow-poly}
  \fp_{M_I} = \flowp_I
  = \sum_{\substack{\flow\colon I \rightsquigarrow  I _{\vstar}\\\text{in $\perfor_\vstar$}}} \yvar^{ \wtcom(\flow)},
\end{equation}
which is a (regular) polynomial in $\CC[\Nstar]$, that is, the exponents are all in $\NN^{Q_0}$.
\end{proposition}

In our running example $I_\vstar=35$ and the first two flows in \Cref{fig:hom-flow}
are the only two flows $25\rightsquigarrow 35$, so we conclude that
\[
  \flowp_{25} = y_{34}(1+y_{24})
\]
where $y_i=\yvar^{[S_i]}$ and the vertices are labelled by the Pl\"ucker labels from \Cref{fig:rs-labelling}.

\begin{remark}\label{rem:RWflow-terms}
The fact that the classical flow polynomials $\flowp_I$ in \eqref{eq:class-flow-poly}
have a combinatorial form as implied by \Cref{rem:fp-F-poly} is familiar from \cite[Cor.~12.4]{RW}.
More precisely, the existence of a minimal (and maximal) term amounts to the fact that
 the matchings in the partition function \eqref{eq:class-pf} form a distributive lattice
(see \cite[Thm.~12.1]{RW} and references therein).
We can see this from the categorical viewpoint of this paper,
by interpreting this lattice as (the opposite of) the submodule lattice of $\sHom(T,M_I)$.
\end{remark}

\begin{remark}\label{rem:compare-RW}
There are various orientation choices that determine the conventions used above,
and some conventions may differ from \cite{RW} or elsewhere.

First, the choice of the fundamental class $\fundcls\in\Or G_2$ 
determines the direction of the homological flow 
$\flow=\pdiso_1(\match_\vstar-\match)$ associated to a matching $\match$ and
thus whether $w_\flowpath$ should be~1 on the left or right of the path $\flowpath$. 
%(and thus 0 on the other side). 
Note that $\wtcom(\match)$ is defined cohomologically by 
$\cob{0} \wtcom(\match)= \match_\vstar-\match$
and will not depend on this choice.

Second, the boundary value %$\bdry\match$ 
of a matching $\match$ is most naturally 
a `matching' on the circular double quiver \eqref{eq:circle-quiv}, that is, 
a choice of one of the two arrows between each pair of adjacent vertices.
The module $e N_\match$ is canonically determined by this boundary matching,
as it specifies which arrows should be $t$.

However (cf.~\S\ref{subsec:nota-conv}),
to describe such a module as $M_I$, for a Pl\"ucker label~$I$, requires a choice of 
whether $I$ gives the clockwise or anti-clockwise arrows in the boundary matching.
The choice made here and in \cite{JKS1} is anti-clockwise,
while in \cite{CKP} it is clockwise.
Changing this choice changes $I$ to the complementary label $I^c$
and swaps $k$ and $n-k$.

Rietsch--Williams \cite[\S2.3]{RW} use Young diagrams as indices, 
so that both Pl\"ucker labels~$I$ and $I^c$ can be read off as the south or west steps in the 
boundary (NE to SW) of the diagram.
Since they use south steps to label Pl\"ucker coordinates on the original Grassmannian,
they are effectively using the clockwise convention from our point of view.
On the other hand, a flow $I \rightsquigarrow I_{\vstar}$ is also a flow $I_{\vstar}^c \rightsquigarrow I^c$
and $I_{\vstar}^c$ is lexicographically minimal precisely when $I_{\vstar}$ is lexicographically maximal,
which explains the differences between e.g.~\cite[(6.3)]{RW} and \eqref{eq:class-flow-poly}.
\end{remark}

%%%% =========================================================== %%%%
\section{Cluster algebras, monoids and cones}\label{sec:clusalg}
%%%% =========================================================== %%%%

In this section, we use the cluster character $\ptfn^T$ to introduce a cluster algebra $\clualgA$  
and two cones (Newton--Okounkov and $g$-vector), associated to a cluster tilting object~$T$.
We say that $M$ is \emph{reachable} from $T$
if $M\in \add U$ for a cluster tilting object $U$ obtained from $T$ by a finite sequence of mutations.  
%Recall that we have assumed that $\GP \algB$ has a cluster structure. 
A \emph{cluster isomorphism} between cluster algebras is an algebra isomorphism 
which preserves clusters and mutation relations. 

\begin{proposition} \label{lem:clus-alg}
Let $T\moreq \bigoplus_{i\in Q_0} T_i$
be a cluster tilting object in $\GP \algB$ (cf.~\eqref{eq:T-moreq}) and 
%let 
\[ \gamma\colon \GP B \to H \] 
be a cluster character for which the $\gamma(T_i)\in H$ are algebraically independent.
\begin{enumerate}
\item\label{itm:clus1}
  There is a cluster algebra $\clusalg{T}{\gamma}\subset H$ with cluster variables 
$\gamma(M)$ for all rigid indecomposable $M$ reachable from $T$.
\item\label{itm:clus2} 
  If $\gamma':\GP\algB \to H'$ is another cluster character with 
$\gamma'(T_i)$ algebraically independent, then there is a cluster isomorphism
\[ f\colon \clusalg{T}{\gamma}\to \clusalg{T}{\gamma'}, \]
where $f(\gamma(M))=\gamma'(M)$ for all rigid $M$ reachable from $T$.
\end{enumerate}
\end{proposition}

\begin{proof}
\itmref{itm:clus1}  % ===== (1) =====
Let $\mathcal{F}_\gamma$ be the field of rational functions in the algebraically independent
generators $\gamma(T_i)$, which is a subfield of the field of fractions of $H$. 
Let
\[
  \clusalg{T}{\gamma}\subset \mathcal{F}_\gamma 
\]
be the cluster algebra constructed using $\{\gamma(T_i)\colon i\}$ and the quiver of 
$\End_C(T)\op$ for the initial seed. 

Since $\GP\algB$ has a cluster structure, by \Cref{rem:cluster-structure},
for any cluster tilting object~$U$, the Gabriel quiver of $\End_C(\mu_kU)\op$ is the Fomin--Zelevinsky 
mutation of the quiver of $\End_C (U)\op$. Also, using the mutation sequences 
\[
  0 \to U_k^* \to  E \to U_k \to 0 
  \quadand
  0\to U_k \to F \to U_k^* \to 0, 
\]
we have
\[
  \gamma(U_k)\gamma(U_k^*) = \gamma(E) + \gamma(F),
\]
which is precisely the Fomin--Zelevinsky mutation relation between cluster variables. 
So the cluster variables coincide with the set of $\gamma(M)\in H$ for all $M$
that are rigid indecomposable and reachable from $T$, by induction.
 
\itmref{itm:clus2} % ===== (2) =====
There is an isomorphism of fields
\[
 f\colon\mathcal{F}_\gamma \to \mathcal{F}_{\gamma'}
  \colon \gamma(T_i)\mapsto \gamma'(T_i).
\]
We need to show that $f(\gamma(M))=\gamma'(M)$ for all $M$ indecomposable rigid and reachable from $T$. We have
\[
  \gamma(U_k)\gamma(U_k^*) 
  = \gamma(E) + \gamma(F) \quadand \gamma'(U_k)\gamma'(U_k^*) 
  = \gamma'(E) + \gamma'(F),
\]
and $f(\gamma(U_k^*))=\gamma'(U_k^*)$ provided $f(\gamma(U))=\gamma'(U)$. 
Therefore $f(\gamma(M))=\gamma'(M)$ for all $M$ which are reachable from $T$, by induction.
\end{proof}

\begin{remark}
\Cref{lem:clus-alg} only requires that $\gamma$ satisfies
\[
  \gamma(U_i)\gamma(U_i^*)=\gamma(E) + \gamma(F)
\]
on mutation sequences, which is a special case of \Cref{def:clucha}\itmref{itm:cc3}.
Indeed, the result holds more generally for Frobenius 2-CY categories, and even other
types of categories, provided there is a suitable cluster character and cluster structure. 
\end{remark}

\begin{proposition}\label{Prop:algind}
If $T\moreq\bigoplus_{i=1}^{m} T_i$ is a cluster tilting object in $\GP B$, 
then the partition functions $\ptfn^T_{T_i}$, as defined in \eqref{eq:part-fun}, 
are algebraically independent.
\end{proposition}

\begin{proof}
Suppose that there is a polynomial $f(x)=\sum_d c_d x^d$ in $\CC[x_1,\ldots,x_m]$ such that 
\[
f\bigl( \ptfn^T_{T_1}, \dots, \ptfn^T_{T_m}\bigr) =  \sum_{d} c_d \prod_{i=1}^{m} \bigl(\ptfn^T_{T_i}\bigr)^{d_i}=0.
\]
Since the partial order $\leq$ is additive (i.e.~if $a\leq b$ and $c\leq d$, then $a+c\leq b+d$),
the leading exponent of the $d$-summand is $\sum_i d_i[T, T_i]$, by \Cref{Thm:char-parfunB}\itmref{itm:pfun4}.
Since the $[T, T_i]$ form a basis of $\Grot(\CM\algA)\isom\Grot(\proj\algA)$, by \Cref{rem:lat-ranks},
these leading exponents are all distinct 
and so we can deduce inductively that $c_d=0$ for all $d$. 
Thus $f=0$, as required.
\end{proof}

%Since $\GP\algB$ admits a cluster structure, by \Cref{rem:cluster-structure},
We can use \Cref{lem:clus-alg} and \Cref{Prop:algind} to define a cluster algebra as follows. 

\begin{definition}\label{def:clualgA}
Let $\clualgA_T=\clusalg{T}{\ptfn^T}\subset \CC[\Grot(\CM\algA)]$ be the cluster algebra 
generated by cluster variables obtained by iterated mutation from the initial variables,
that is, generated by all $\ptfn^T_M$ for reachable rigid $M$.
\end{definition}

\begin{remark}\label{rem:total-order}
Recall the partial order $\leq$ on $\Grot(\CM\algA)\isom\ZZ\oplus \Nstar$ from \Cref{def:part-ord}. 
We can refine this order to a lexicographic total order
by choosing an order on the vertices $i\in Q_0^\vstar := Q_0\setminus\{\vstar\}$
(cf.~\cite[Def.~8.1]{RW}),
as $\{[S_i]\st i\in Q_0^\vstar\}$ is a basis of $\Nstar$.
\end{remark}

We can then, using the total order in \Cref{rem:total-order}, define
\begin{equation}\label{eq:valT}
 \Val{T}\colon \CC[\Grot(\CM\algA)]\setminus 0 \,\lra\, \Grot(\CM\algA),
\end{equation}
where $\Val{T}(f)$ is the minimal exponent of the non-zero terms in $f$.
The restriction of $\Val{T}$ to $\clualgA_T\setminus 0$ 
can be regarded as %is 
an example of a `$g$-vector valuation', in the general sense of \cite[\S5]{BCMN}.

\begin{definition}\label{Def:noc-etc} 
We define two monoids in $\Grot(\CM\algA)$, 
and their respective cones in $\Grot(\CM\algA)\otimes_{\ZZ}\RR$, 
associated to any cluster tilting object $T$ in $\GP\algB$ as follows.

The \emph{Newton--Okounkov monoid/cone} is
\begin{align*}
  \nom{T} &= \{ \Val{T}(f) \st f \in \clualgA_T \setminus 0 \}, \\
  \noc{T} &= \overline{\Rspan}\, \nom{T}.
\end{align*}
Traditionally, $\nom{T}$ is also known as the \emph{value semigroup} for $\Val{T}$ on $\clualgA_T$.

The \emph{$g$-vector monoid/cone} is
\begin{align*}
   \gvm{T} &=\{ [T,M] \st M\in \CM \algB \}, \\
   \gvc{T} &= \overline{\Rspan}\, \gvm{T}. 
\end{align*}
\end{definition}  

Note that $\nom{T}$ and $\gvm{T}$ are semigroups, because 
\[
  \Val{T}(f_1f_2)=\Val{T}(f_1)+\Val{T}(f_2)
  \quadand
 [T,M_1\oplus M_2]=[T,M_1]+[T,M_2]
\]
and indeed they are monoids, because $\Val{T}(1)=0=[T,0]$.

\newcommand{\genmon}{\mathsf{M}}
\newcommand{\gencon}{\mathsf{C}}
\newcommand{\genlat}{\mathsf{L}}

\begin{remark} \label{rem:MonCon}
A general monoid $\genmon$ in a lattice $\genlat$ such as $\Grot(\CM\algA)$ can be quite complicated.
It may not be finitely-generated and it may not be \emph{saturated},
that is, equal to the set of all integral points in $\Rspan(\genmon)$.

On the other hand, there are `nice' monoids $\genmon$ which are defined by a finite number
of integral linear inequalities, that is, they are the integral points of a rational polyhedral cone
$\gencon\subset \genlat\tensR$.
In that case, $\genmon$ is both finitely-generated, by Gordan's Lemma, and saturated.
Furthermore $\gencon=\Rspan(\genmon)=\overline{\Rspan}(\genmon)$.
(See e.g.~\cite[\S1.2--3]{Ful} for more details.)

One goal of later sections (\S\ref{Sec:10}--\ref{Sec:11}), culminating in \Cref{rem:rat-pol-cone2},
is to show that $\gvm{T}$ is nice in this way.
\end{remark}

\begin{remark} \label{Rem:HomogCone}
Since the first term in the lexicographic order on $\Grot(\CM\algA)$ is $\rkk$,
the leading exponents of all functions $f\in\clualgA$ will coincide with the leading exponents of 
homogeneous functions, that is, whose exponents have fixed $\rkk$.
Thus, if we denote the subset of (non-zero) homogeneous functions by $\clualgA_\bullet$, 
then we can also write
\[ \nom{T} = \{ \Val{T}(f) \st f \in \clualgA_\bullet \}. \]
\end{remark}

\begin{remark} \label{Rem:SameCone}
Recall, from \Cref{Thm:char-parfunB}\itmref{itm:pfun4}, that $\ptfn_M$ has  
leading exponent $[T,M]$ for any refinement to a total order as in \Cref{rem:total-order}
and for any choice of vertex~$\vstar$.
Hence the two monoids intersect in at least
the $[T,M]$ for reachable rigid~$M$.

However, on one hand, there may be some $M$ for which $\ptfn_M\not\in\clualgA$, 
so it could be that $[T,M]\not\in\nom{T}$.
On the other hand, there may be some $f\in\clualgA$ whose leading exponent is not of the form $[T,M]$
for any $M$.

Note also that $\nom{T}$ may depend on the choice of total order, while $\gvm{T}$ plainly doesn't.
Despite all this, in the case where $B=C$, we will see, in \Cref{Thm:NOcone}, 
that $\nom{T}=\gvm{T}$, and so their cones also coincide.
\end{remark}

%%%% =========================================================== %%%%
\section{Cluster characters under isomorphisms of cluster algebras}\label{sec:twistchar}
%%%% =========================================================== %%%%

Since mutation in cluster categories is compatible with mutation in cluster algebras,
cluster characters of reachable rigid indecomposables are necessarily cluster variables
and hence will be identified under any (cluster) isomorphism of cluster algebras.
On the other hand, there is no reason \emph{a priori} why such isomorphisms
should identify cluster characters of arbitrary modules.
However, we will now proceed to show that this does happen for the cluster characters considered in this paper.

To do this, we currently have to restrict attention to the algebra $\algC$, 
as a special case of the algebras $\algB$ studied since \Cref{Sec:3}. 
For general $\algB$, we have no proof of \Cref{Thm:TandSigmaT}
and thus of \Cref{rem:monom-cluster}.
%\comment{R:say more about why this restriction is needed}
We will maintain this restriction for the rest of the paper,
unless otherwise stated.
Since $\GP\algC=\CM\algC$ in this case, the results proved so far for $\GP\algB$ all apply to $\CM\algC$.

\begin{remark}\label{rem:reachable}
Recall from \cite[Remark 5.7]{JKS1} that cluster tilting objects in $\CM \algC$ 
with all summands of rank 1 are in bijection with 
maximal non-crossing collections $\maxNC\subset\labsubset{n}{k}$.
% of $k$-sets in $\labset{n}=\{1, \dots, n\}$. 
More precisely, as in \eqref{eq:CTO-maxNC}, the cluster tilting object corresponding to $\maxNC$ is
\begin{equation}\label{eq:TmaxNC}
  T_{\maxNC}=\textstyle\bigoplus_{I\in \maxNC} M_I.
\end{equation}
By \cite[Theorem 1.4]{OPS}, any two such $T_{\maxNC}$ can be mutated into each other
(see e.g.~\cite[\S8]{JKS2} for more explanation).
Hence we can define a \emph{reachable} cluster tilting object in $\CM \algC$
to be one that is reachable from any (and hence all) $T_{\maxNC}$.
\end{remark}

% ======================================================================
\subsection{The rectangles cluster tilting object} \label{subsec:RTO}
% ======================================================================
Recall from \eqref{eq:PtoJ} that, for $M\in\CM\algC$, we have canonical embeddings 
\[
\funP M\subspc M \subspc \funJ M 
\quad\text{with}\quad 
e_\vstar\funP M=e_\vstar M =e_\vstar \funJ M.
\]
We can further define
\begin{equation}\label{eq:pifun-omfun}
\pifun M = M/\funP M
\quadand
\omfun  M = \funJ M/ M.
\end{equation}
When $M$ is a rank one module, 
$\funJ M/\funP M=\modJ/\modP$ can be depicted as a $k\times (n-k)$ rectangle
with edges in north-west and north-east directions respectively,
so that $\pifun M$ is depicted as a Young diagram contained in the rectangle 
and based at the bottom corner, 
while $\omfun M$ can be depicted as the complementary Young diagram, based at the top corner. 
See Figure~\ref{fig:pi-omega49} for an example:
the lower part depicts $\pifun M_{1457}$ and the upper part depicts $\omfun  M_{1457}$.
The dividing line between the two Young diagrams is 
the `profile' of the module $M_I$ in the sense of \cite[\S6]{JKS1}.
The elements of $I$ give the south-east edges or `down-steps' of the profile
(read from left to right).
The profile here matches the NE-SW boundary of the indexing Young diagrams in \cite{RW},
as we are just rotating the diagram by $3\pi/4$ 
(but see \Cref{rem:compare-RW} for differences of convention).
 
\begin{figure}[h]
\begin{tikzpicture} [scale=0.4,
upbdry/.style={thick, gray},
lowbdry/.style={thick, gray},
bboxes/.style={thick, blue, densely dotted},
rboxes/.style={thick, red, densely dotted},
ridge/.style={very thick, purple},
>={Stealth[inset=2.5pt,length=4.5pt,angle'=40,round]},
outarr/.style={->, gray},
midarr/.style={->,blue},
rdgarr/.style={->,red},
keyarr/.style={->,black},
outdot/.style={gray, fill= gray},
middot/.style={blue, fill=blue}]
\newcommand{\abit}{0.133}
\begin{scope} [scale=1.2, shift={(-10,-5)},rotate=135]
\draw [bboxes] (0,-1)--(3,-1) (0,-2)--(1,-2)--(1,0) (2,0)--(2,-2);
\draw [rboxes] (4,-1)--(3,-1) (4,-2)--(3,-2)--(3,-5) (4,-3)--(1,-3)--(1,-5) (4,-4)--(0,-4) (2,-2)--(2,-5);
\draw [upbdry] (0,-5)--(4,-5)--(4,0);
\draw [lowbdry] (0,-5)--(0,0)--(4,0);
\draw [ridge] (4,0)--(3,0)--++(0,-2)--++(-2,0)--++(0,-1)--++(-1,0)--(0,-5);
\end{scope}
\end{tikzpicture}
\caption{A depiction of $\pifun M_{I}$ and $\omfun  M_{I}$, for $I=1457$, $(k,n)=(4,9)$.}
\label{fig:pi-omega49}
\end{figure}

Note that  $M\in \CM\algC$ is indecomposable projective if and only if $M$ has a simple top.
Let $P_i=\algC e_i$ and $S_i=\top P_i$, for $i=1, \dots, n$
(note $\vstar=n$).
We will call a module~$M$ \emph{rectangular} if it has rank~1 and the Young diagram for $\pifun M$ is a rectangle.
We will call a module~$M$ \emph{co-rectangular} if it has rank~1 and the Young diagram for $\omfun M$ is a rectangle.
Note that these depend on the choice of $\vstar$.

\begin{lemma} \label{Lem:rect-too} 
An $M\in\CM\algC$ is rectangular if and only if either (1) $M$ is projective 
or (2) $\top M=S_\vstar\oplus S_i$, for some $i\neq \vstar$. 
In case (2), the minimal syzygy $\Syz M$ has rank~$1$. 
Furthermore, if $M$ and $N$ are both rectangular,
then $\Ext^1(M,N)=0$.
\end{lemma}

\begin{proof}
The first part follows immediately from the definition of $\pifun M$ 
and how the diagram depicts the module.
In case (2), the minimal projective cover is $P_\vstar\oplus P_i\to M$
and so the kernel $\Syz M$ has rank~$1$, because rank is additive. 
Indeed, any rank 1 module~$M$ with two tops has minimal syzygy $\Syz M$ of rank~$1$.

If $M$ and $N$ are both rectangular, then, 
after removing initial common down-steps in their profiles,
at least one of the residual profiles is `projective', in the sense that
the up and down-steps divide the circle into two cyclic intervals.
Such a profile is always non-crossing with any other profile, 
so $\Ext^1(M,N)=0$, by \cite[Prop.~5.6]{JKS1}.
\end{proof}

\begin{definition}\label{def:rect-clus}
Let $\grid=\grid(k,n)$ be the $k\times (n-k)$ grid
\begin{equation}\label{eq:grid}
  \grid(k,n)=\{ (i,j)\in \ZZ^2\st 1\leq i \leq k,\, 1\leq j \leq n-k \}.
\end{equation}
The \emph{rectangles cluster tilting object} is
\begin{equation}\label{eq:rectCTO}
  \rectCTO=T_\vempty\oplus \bigoplus_{ij\in\grid} T_{ij},
\end{equation}
where $T_\vempty=\modP$ and $T_{ij}$ is the rank~1 module for which $\pifun T_{ij}$
is depicted, as in \Cref{fig:pi-omega49}, by an $i\times j$ rectangle.
In other words, $T_\vempty = M_{[1,k]}$ and $T_{ij}=M_{K_{ij}}$ for 
\[ 
  K_{ij} = [1, k-i]\cup [k-i+j+1, k+j]. 
\]
Note that $\rectCTO$ is rigid, by \Cref{Lem:rect-too}, and has the maximal number of summands,
namely $k(n-k)+1$, so it is indeed a cluster tilting object, by \cite[Rem~5.7]{JKS1}.
For an explicit plabic graph whose face labels are $[1,k]$ and the $K_{ij}$, see \cite[\S4]{RW}.
\end{definition}

By \Cref{rem:reachable}, any reachable cluster tilting object in $\CM\algC$ 
can be obtained by mutation from $\rectCTO$.

Given any cluster tilting object $T=\bigoplus_i T_i$, 
define $\CTOsyz{T}$ to be the direct sum of the indecomposable projectives 
and the minimal syzygies $\Syz T_i$, when $T_i$ is not projective.
Note that $\CTOcosyz{T}$ is already defined in a similar way in \Cref{rem:SigmaT}.

\begin{remark}\label{rem:rk1syz}
If $\rectCTO_i$ is not projective,
then the minimal syzygy $\Syz {\rectCTO_i}$
has rank~$1$, by \Cref{Lem:rect-too},
and also has two tops.
In particular, all summands of $\CTOsyz{\rectCTO}$ have rank~$1$.
For similar reasons as above, there is also a `co-rectangles cluster tilting object' $T''$,
whose summands are all the co-rectangular (rank~$1$) modules, 
and all summands of~$\CTOsyz{T''}$ also have rank~$1$.
In fact, $\CTOsyz{\rectCTO}$ is a cyclic shift of $T''$,
that is, a co-rectangles cluster tilting object
defined relative to a different choice of $\vstar$.
\end{remark}

\begin{proposition}\label{Thm:TandSigmaT}
Let $T$ be a reachable cluster tilting object. 
Then both $\CTOsyz{T}$ and  $\CTOcosyz{T}$ are reachable 
cluster tilting objects. 
\end{proposition}

\begin{proof} Note that $\Syz$ is an equivalence on the stable category of $\CM\algC$. 
So there are exactly $m$ indecomposable pairwise nonisomorphic summands in $\CTOsyz{T}$ and 
\[
  \Ext^1(\CTOsyz{T} , \CTOsyz{T})=\Ext^1(T, T)=0.
\]
Therefore $\CTOsyz{T}$ is a maximal rigid module and so it is a cluster tilting object in $\CM\algC$.

Again, as $\Syz$ is an equivalence on the stable category of $\CM\algC$, the sequence 
\[ 
  \ShExSeq{T_i}{F_i}{T^*_i} 
\]
is a mutation sequence if and only if the following is a mutation sequence
\[
  \ShExSeq{\Syz T_i}{\Syz F_i}{\Syz T^*_i}, 
\]
where $\Syz T_i$ and $\Syz T^*_i$ are both minimal, although $\Syz F_i$ is, in general, not minimal. 
Thus $\CTOsyz{T}$ can be mutated to $\CTOsyz{T'}$ if and only if $T$ can be mutated to $T'$.   
A similar fact is proved for $\CTOcosyz{T}$ as part of \Cref{thm:PF-zeta-Phi}
(see in particular \eqref{eq:two-squares}).

By \Cref{rem:reachable}, since $T$ is reachable, 
it can be mutated into the rectangles cluster tilting object $\rectCTO$
and so $\CTOsyz{T}$ can be mutated to $\CTOsyz{\rectCTO}$.
But, by \Cref{rem:rk1syz}, there is a maximal non-crossing collection $\maxNC\subset\labsubset{n}{k}$ 
such that $\CTOsyz{\rectCTO}=T_\maxNC$.
Hence $\CTOsyz{\rectCTO}$ is reachable, by \cite[Thm~1.4]{OPS},
and so $\CTOsyz{T}$ is reachable, as required.

For the case of $\CTOcosyz{T}$, we can mutate $T$ to $T_\maxNC$
and thus $\CTOcosyz{T}$ to $\CTOcosyz{T_\maxNC}$.
But $\CTOcosyz{T_\maxNC}=\rectCTO$, which is reachable,
and hence $\CTOcosyz{T}$ is reachable,
completing the proof.
\end{proof}

\begin{remark}\label{rem:monom-cluster}
\Cref{Thm:TandSigmaT} implies that the cluster algebra $\clualgA$ from \Cref{def:clualgA}
contains a (reachable) cluster of monomials, namely
\[
  \{ \ptfn_{\Sigma {T_i}}\st \text{$T_i$ is a mutable summand of $T$} \}
  \cup \{ \ptfn_{P}\st \text{$P$ is indecomposable projective} \},
\]
since \Cref{Thm:char-parfunB}\itmref{itm:pfun5} implies that these are all monomials.
\end{remark}

% ======================================================================
\subsection{Cluster characters $\cluschar{}$ and $\Phi$}\label{subsec:PsiPhi}
% ======================================================================
\newcommand{\olbeta}{\overline{\beta}}
\newcommand{\GNmap}{q}
\newcommand{\AAmap}{\pi_*}

Recall, from \cite[\S9]{JKS1}, that the cluster character 
$\cluschar{}\colon \CM\algC\to \CC[\Gr(k, n)]$
is obtained by homogenisation of Geiss--Leclerc--Schr\"{o}er's character 
$\psi\colon \sub Q_{k}\to \CC[\unirad]$, from \cite{GLS08}, 
where $\unirad$ is an open Schubert cell in $\Gr(k, n)$. 
More precisely, $\unirad$ is the open set where the Pl\"ucker coordinate $\Psi_{\modP}\neq 0$,
so that the affine coordinate ring 
$\CC[\unirad] = \CC[\Gr(k, n)] / (\Psi_{\modP}-1)$.

Furthermore, by \cite[Prop~4.3]{JKS1}, 
there is an exact functor $\pifun\colon \CM\algC\to \sub Q_{k}$,
defined as in \eqref{eq:pifun-omfun},
which kills just the subcategory $\add\modP$ and
induces an equivalence 
\[ \CM\algC/\add\modP \to \sub Q_{k}. \]
The cluster characters are then related by the following commutative square,
in which $\GNmap$ is the quotient map.
\begin{equation}\label{eq:cc-square}
\begin{tikzpicture}[xscale=3.0, yscale=1.7, baseline=(bb.base)]
\coordinate (bb) at (0,1.5);
 \draw (1, 1) node (b1) {$\sub Q_{k}$};
 \draw (2, 1) node (b3) {$\CC[\unirad]$.};
 \draw (1, 2) node (a1) {{$\CM\algC$}}; 
 \draw (2, 2) node (a3) {$\CC[\Gr(k, n)]$};
 \draw[cdarr] (a1) to node [auto] {$\Psi$} (a3);
 \draw[cdarr] (b1) to node [auto] {$\psi$} (b3);
 \draw[cdarr] (a1) to node [auto] {$\pifun$} (b1);
 \draw[cdarr] (a3) to node [auto] {$\GNmap$} (b3);
\end{tikzpicture}
\end{equation}
Note that $\GNmap$ is not injective, but $\Psi_M$ is the unique lift of 
$\psi_{\pifun M}$ of degree $\rnk{\algC} M$.

If $T=\bigoplus_i T_i$ is a cluster tilting object in $\CM\algC$,
then $\olT=\pifun\, T$ is a cluster tilting object in $\sub Q_{k}$,
with one less summand.
Furthermore, $\olA=\End(\olT)\op$ is a quotient of $\algA=\End(T)\op$
by the ideal $(e_\vstar)$.

In this case, using $\olT$,
Fu--Keller's cluster character (cf. \S\ref{sec:FKclucha}) is 
\begin{equation}\label{eq:phi-Tbar}
\begin{aligned}
\phi\colon &\sub Q_k \to \CC[\Grot(\proj\olA)]\colon M\mapsto \phi_M, \\
\phi_M &=\xvar^{[\olT, M]}\sum_{d} \euler (\qvGr{d}\Ext^1(\olT, M)) \xvar^{-\olbeta(d)},
\end{aligned}
\end{equation}
where $\olbeta\colon \Grot(\fd\olA) \to \Grot(\proj\olA)$
is the isomorphism given by projective resolution.

Note that the formula \eqref{eq:phi-Tbar} can be interpreted as an explicit monomial in the 
variables $\xvar_i=\xvar^{[\olT, \olT_i]}$, as $[\olT, \olT_i]$ form a basis of $\Grot(\proj\olA)$.
Geiss--Leclerc--Schr\"{o}er showed that $\phi_M$ is the expression for $\psi_M$ 
in the initial variables corresponding to $\olT$.
In other words, $\phi_M$ is the expression for $\psi_M$ in the cluster chart associated to $\olT$.

\begin{theorem}\label{Thm:glsgeneric}  
\cite[Thm~4]{GLS12}  
For any $M\in \sub Q_{k}$, under the substitution $x_i=\psi_{\olT_i}$,
one has
\[
  \phi_M=\psi_M.
\]
So, in particular, $\phi_M\in \CC[\unirad]$.
\end{theorem}

In this subsection, we will lift this result to $\CM \algC$.
First we prove a useful lemma.

\begin{lemma} \label{Lem:upsilon}
Let $\Gamma\colon \CM\algC \to H$ be a cluster character
such that, for some reachable cluster tilting object $T=\bigoplus_i T_i$, 
the characters $\Gamma_{T_i}$ are algebraically independent.
Then $\Gamma$ induces an injective algebra homomorphism such that
\[
\Upsilon_{\Gamma}\colon \CC[\Gr(k, n)] \to H
 \colon \Psi_M\mapsto \Gamma_{M},
\]
for any reachable rigid $M$. % (i.e.~$\Psi_M$ is a cluster variable). 
%Note:No need to have the bracket part. M may be decomposable
In particular,  $\Upsilon_{\Gamma}$ is determined by 
$\Upsilon_{\Gamma}(\minor{I})= \Gamma_{M_I}$.
\end{lemma}

\begin{proof}
By \cite[\S9]{JKS1}, $\CC[\Gr(k, n)]=\clusalg{T}{\Psi}$, in the notation of \Cref{lem:clus-alg},
because $T$ is reachable.
The result then follows from \Cref{lem:clus-alg}, 
and the fact that $\Psi_{M_I}=\minor{I}$, by \cite[(9.4)]{JKS1}.
More precisely, $\Upsilon_{\Gamma}$ is the induced isomorphism 
$\CC[\Gr(k, n)]\to \clusalg{T}{\Gamma}$.
\end{proof}

\begin{remark}\label{rem:plucker-rels}
It is an immediate consequence of \Cref{Lem:upsilon} that, under the assumptions of the lemma,
the characters $\Gamma_{M_I}$ satisfy all Pl\"ucker relations.
\end{remark}

\begin{remark}\label{rem:FK-applic}
If we apply \Cref{Lem:upsilon} to the Fu-Keller cluster character $\Phi^T$ in \eqref{eq:FK-CC},
then we obtain the well-known fact (e.g.~\cite[Thm~5.4]{FK}, \cite[(1.3)]{GLS12}) 
that one can obtain an expression for any
cluster variable $\Psi_M$, for $M$ reachable rigid, in terms of an initial cluster $\Psi_{T_i}$,
on substituting $\Psi_{T_i}$ for $\xvar^{[T,T_i]}$ in the formula for $\Phi^T$.
\end{remark}

\begin{theorem} \label{Thm:PsiandTheta} 
Let $T$ be a reachable cluster tilting object in $\CM C$. 
Let $\Upsilon^T=\Upsilon_{\Phi^T}$ in the notation of \Cref{Lem:upsilon}.
Then $\Phi^T=\Upsilon^T\compo \Psi$,
that is, for all $M\in \CM\algC$,
\[
\Phi^T_M=\Upsilon^T \bigl( \cluschar{M} \bigr).
\]
%(Note that $\Upsilon^T$ is defined by this equation just for $M$ of rank one.)
\end{theorem}
 
\begin{proof}
Define a map $\AAmap\colon \Grot(\CM\algA) \to \Grot(\proj\olA)$
induced by $\pifun\colon \CM \algC \to \sub Q_{k}$,
in the sense that $\AAmap([T,T_i])=[\pifun T,\pifun T_i]$,
which includes $\AAmap([T,\modP])=0$.

Using $\olT=\pifun\, T$ to define $\phi$ as in \eqref{eq:phi-Tbar},
we have the following diagram,  
which we want to show is fully commutative.
\begin{equation}\label{diag:prism}
\begin{tikzpicture}[xscale=2.7, yscale=2.5, baseline=(bb.base)]
\coordinate (bb) at (0,1.5);
 \draw (1, 1) node (b1) {$\sub Q_{k}$};  
 \draw (2, 0.6) node (b2) {$\CC[\unirad]$};
 \draw (3, 1) node(b3) {$\CC[\Grot(\proj\olA)]$};
 \draw (1, 2) node (a1) {$\CM\algC$};
 \draw (2, 1.6) node (a2) {$\CC[\Gr(k, n)]$};
 \draw (3, 2) node(a3) {$\CC[\Grot(\CM\algA)]$};
 \draw[cdarr] (b1) to node [below left=-1pt] {$\psi$} (b2);
 \draw[cdarr] (b1) to node [above left] {$\phi$} (b3);
 \draw[cdarr] (b2) to node [below right=-1pt] {$\upsilon$} (b3);
 \draw[cdarr] (a1) to node [below left=-1pt] {$\Psi$} (a2);
 \draw[cdarr] (a1) to node [above] {$\Phi^T$} (a3);
 \draw[cdarr] (a2) to node [below right=-1pt] {$\Upsilon^T$} (a3);
 \draw[cdarr] (a1) to  node [left] {$\pifun$} (b1);
 \draw[cdarr] (a2) to  node [above right] {$\GNmap$} (b2);
 \draw[cdarr] (a3) to  node [auto] {$\CC[\AAmap]$} (b3);
\end{tikzpicture}
\end{equation}
Here $\upsilon$ is the algebra homomorphism $\psi_{\pifun T_i}\mapsto \phi_{\pifun T_i}$,
so \Cref{Thm:glsgeneric} says 
\[
  \phi=\upsilon\compo \psi,
\]
i.e.~the bottom triangle in \eqref{diag:prism} commutes.
It will then suffice to show that the three vertical squares commute, because we then get 
$\CC[\AAmap] \bigl(\Upsilon^T(\cluschar{M})\bigr)=\CC[\AAmap]\bigl(\Phi^T_M\bigr)$.
However, note that both $\Upsilon^T(\cluschar{M})$ and $\Phi^T_M$ are in the linear subspace
of $\CC[\Grot(\CM \algA)]$ spanned by monomials of degree $\rnk{\algC} M$.
Furthermore, $\AAmap$ is an isomorphism when restricted to the affine subspace of $\Grot(\CM \algA)$
of classes of any fixed rank.
Therefore $\Upsilon^T(\cluschar{M})=\Phi^T_M$, as required.

To prove the claim that the three vertical rectangles commute,
note first that the front left commutation, i.e.~$\psi_{\pifun N}=\GNmap\Psi_N$, is just \eqref{eq:cc-square}.

For the back commutation, i.e.~$\phi_{\pifun N}= \CC[\AAmap] \Phi^T_N$,
we must compare the two Fu-Keller formulae \eqref{eq:FK-CC} and \eqref{eq:phi-Tbar}.
First note that $\Ext^1(T, N)\isom \Ext^1(\pifun T, \pifun N)$ as 
modules for $\sEnd(T)\op=\sEnd(\olT)\op$ 
and so the sums match up term-by-term.
Thus it remains to show that the exponents match, i.e.~that 
\begin{equation}\label{eq:piTN}
\AAmap [T,N] = [\pifun T,\pifun N]
\quadand
\AAmap \beta(d)= \olbeta(d).
\end{equation}
This follows because $\pifun$ is an exact functor and so transforms
one calculation, using $\add T$-approximation and \Cref{Prop:projres}\itmref{itm:pr1} respectively,
into the other.

For the front right commutation, it suffices to know that the square commutes when applied to  
$\minor{I}=\Psi_{M_I}$ for all rank one modules $M_I$.
But $\Upsilon^T$ is defined precisely so that $\Upsilon^T \Psi_{M_I} = \Phi^T_{M_I}$,
and so the commutation of the other squares and the bottom triangle yields the result.
 \end{proof}
 
\begin{theorem}\label{cor:partfun-map}
Let $T$ be a reachable cluster tilting object in $\CM C$,
with associated partition function $\ptfn^T$,
as in \eqref{eq:part-fun}.
Then there is an injective homomorphism, 
which induces an isomorphism on function fields,
\[ 
  \Xi^T\colon \CC[\Gr(k, n)]\to \CC[\Grot(\CM\algA)]\colon \cluschar{M}\mapsto \ptfn^T_M,
\]
for all $M\in\CM C$.
Furthermore, $\Xi^T$ induces a cluster isomorphism $\CC[\Gr(k, n)] \isom \clualgA_T$,
the cluster algebra of \Cref{def:clualgA}.
In particular, $\ptfn^T_M\in\clualgA_T$, for all $M\in\CM C$.
\end{theorem}

\begin{proof} 
By \Cref{Prop:algind}, we can define $\Xi^T=\Upsilon_{\ptfn^T}$,
in the notation of \Cref{Lem:upsilon},
so that $\Xi^T(\Psi_M)=\ptfn^T_M$ for any reachable rigid object $M$. 

By \Cref{thm:PF-zeta-Phi} and \Cref{Thm:PsiandTheta}, we have
\begin{equation}\label{eq:PF-Phi-Psi}
\ptfn^T = \CC[\zeta] \compo \Phi^{\CTOcosyz{T}}
\quadand
\Phi^{\CTOcosyz{T}} = \Upsilon^{\CTOcosyz{T}} \compo \Psi,
\end{equation}
where $\Upsilon^{\CTOcosyz{T}} =\Upsilon_{\Phi^{\CTOcosyz{T}}}$
in the notation of \Cref{Lem:upsilon}.
By \Cref{rem:monom-cluster},
the image of $\Xi^T$ contains a set of generators for the function field of $\CC[\Grot(\CM\algA)]$,
so the induced map on function fields is an isomorphism.

Furthermore, $\Xi^T=\CC[\zeta] \compo \Upsilon^{\CTOcosyz{T}}$,
because both maps take $\cluschar{M_I}$ to $\ptfn^T_{M_I}$, for all rank one modules $M_I$.
Hence $\ptfn^T = \Xi^T\compo \Psi$, that is, 
\[
  \Xi^T(\Psi_M)=\ptfn^T_M,\; \text{for all $M\in \CM\algC$,}
\]
as required.
The cluster isomorphism is a special case of \Cref{lem:clus-alg}.
\end{proof}

\begin{remark}\label{rem:net-chart}
Suppose $T=T_\maxNC$, as in \eqref{eq:TmaxNC},
for a maximal non-crossing collection~$\maxNC$
of face labels of a plabic graph $G$ of type $\perm_{k,n}$. 
By \Cref{cor:partfun-map} and \Cref{prop:partitionchar}, 
\begin{equation}\label{eq:def-nethat}
 \pbfun{\nu}{\Xi^T}\colon \CC[\Gr(k, n)]\to \CC[\Mlat]\colon \minor{I}\mapsto \partfun_I,
\end{equation}
so it coincides with the `homogeneous network chart' $\nethat_G$ (or boundary measurement map) 
associated to $G$, as in \cite[\S3.2]{MS}.
Then the `inhomogeneous network chart' $\net_G$ 
(cf. \cite[Thm.~6.8]{RW}, \cite[Thm.~1.1]{Tal})
coincides with
\begin{equation}\label{eq:def-net}
 \pbfun{\wt}{\Xi^T}\colon \CC[\Gr(k, n)]/(\minor{I_\vstar}-1)\to \CC[\Nstar]\colon \minor{I}\mapsto \flowp_I.
\end{equation}
Note that this map vanishes on the ideal $(\minor{I_\vstar}-1)$, 
because $\flowp_{I_\vstar}=\fp_{\modJ}=1$, 
by \eqref{eq:class-flow-poly} and/or \Cref{Thm:char-flopol}\itmref{itm:flo6}.

These two charts are familiar coordinate systems, but \Cref{cor:partfun-map}
gives a new proof that \eqref{eq:def-nethat} and \eqref{eq:def-net} give well-defined maps,
i.e.~that the $\partfun_I$ and $\flowp_I$ satisfy the appropriate Pl\"ucker relations.
It also provides the new results that, for any $M\in\CM\algC$,
\begin{equation}\label{eq:net-nethat}
  \nethat_G(\Psi_M)=\pbfun{\nu}{\ptfn^T_M}\quadand 
 \net_G(\Psi_M) =\pbfun{\wt}{\ptfn^T_M} = \fp^T_M.
\end{equation}
\end{remark}

% ======================================================================
\subsection{Twist of characters of arbitrary modules}
% ======================================================================

Let $\Gro(k, n)$ be the complement of the `boundary divisors' in Grassmannian $\Gr(k, n)$,
and let $\CC [\Gro(k, n)]$ be its homogeneous coordinate ring,
which is obtained from $\CC[\Gr(k, n)]$ by inverting the frozen cluster variables,
i.e. the boundary minors.

For open positroid varieties, Muller--Speyer \cite{MS} defined a twist map 
which we here denote by `$\twi$'.
In the case of $\Gro(k, n)$, 
\c{C}anak\c{c}i--King--Pressland \cite[Theorem 12.2]{CKP} showed that 

\begin{equation}\label{eq:CKP-twi}
\twi\colon \CC[\Gro(k, n)]\to \CC [\Gro(k, n)]
 \colon \Psi_{M_I} \mapsto \frac{\Psi_{\Syz M_I}}{\Psi_{\procov{M_I}}}.
\end{equation}
Note that $\Psi_M\in \CC[\Gr(k, n)]$, for any $M\in \CM\algC$, and 
$\Psi_{\procov{M}}$ is a product of frozen variables. 
A natural question is whether 
\[
\twi(\Psi_M)=\frac{\Psi_{\Syz M}}{\Psi_{\procov{M}}} 
\]
for any $M\in \CM\algC$. 
In this subsection we will prove that this is indeed true.  

To this end, we define the map
\begin{equation}\label{eq:psi-dag}
  \Psitil \colon  \CM\algC\to \CC[\Gro(k, n)]
  \colon M\mapsto  \frac{ \Psi_{\Syz M} }{ \Psi_{\procov{M}}}.
\end{equation}
Let $T$ be a cluster tilting object in $\CM\algC$ and $\algA=\End (T)\op$.
Then, from \Cref{Thm:PsiandTheta} and \Cref{cor:partfun-map}, we have maps 
\begin{align} 
\label{eq:UpsT}
\Upsilon^{T} &\colon \CC[\Gr(k, n)]\to \CC[\Grot(\CM \algA)]
\colon \Psi_{M} \mapsto \Phi^{T}_{M}, \\
\label{eq:XiT}
\widetilde{\Xi}^T &\colon \CC[\Gr(k, n)]\to \CC[\Grot(\CM\algA)]
\colon \cluschar{M}\mapsto \ptbar^T_M.
\end{align}
By \Cref{lem:Phi-monom} and \Cref{Thm:char-parfunB}\itmref{itm:pfun5}, the frozen variables, 
i.e.~$\Psi_{P}$ for $P$ indecomposable projective,
are mapped by $\Upsilon^{T}$ and $\widetilde{\Xi}^T$ 
to monomials in $\CC[\Grot(\CM \algA)]$, which are invertible.
Hence $\Upsilon^{T}$ and $\widetilde{\Xi}^T$ extend  to $\CC[\Gro(k, n)]$, 
without change of codomain,
and we will also denote these extensions by $\Upsilon^{T}$ and $\widetilde{\Xi}^T$.

\begin{proposition} \label{Thm:twistchar}
For any cluster tilting object $T$ in $\CM\algC$, with $\algA=\End (T)\op$,
we have the following commutative diagram.
\begin{equation}\label{diag:twistchar}
 \begin{tikzpicture}[xscale=2.7, yscale=1.5, baseline=(bb.base)]
\coordinate (bb) at (0,2);
]
 \draw (2,2) node (a) {$\CM\algC$};
 \draw (2,0.7) node (b3) {$\CC[\Grot(\CM\algA)]$};
 \draw (1,3) node (c2) {$\CC[\Gro(k, n)]$};
 \draw (3,3) node (c3) {$\CC[\Gro(k, n)]$};
 \draw[cdarr] (c2) to node[above] {$\twi$} (c3);
 \draw[cdarr] (a) to node[right] {$\ptbar^T$} (b3); 
 \draw[cdarr] (a) to node[above right=-2pt] {$\Psi$} (c2); 
 \draw[cdarr] (a) to node[above left=-2pt] {$\Psitil$} (c3); 
 \draw[cdarr] (c3) to [bend left=10] node[below right] {$\Upsilon^{T}$} (b3);
 \draw[cdarr] (c2) to [bend right=10] node[below left] {$\widetilde{\Xi}^T$}  (b3);
\end{tikzpicture}
\end{equation}
In particular, for any $M\in \CM \algC$, 
\[
  \twi(\Psi_M) =\Psitil_M =\frac{\Psi_{\Syz M}}{\Psi_{\procov{M}}}. 
\]
\end{proposition}

\begin{proof}
%Then, from \Cref{Thm:PsiandTheta} and \Cref{cor:partfun-map}, we have maps 
We know that the left-hand triangle commutes by \eqref{eq:XiT}, 
i.e.~\Cref{cor:partfun-map},
and that the right-hand triangle commutes by \eqref{eq:UpsT}, 
i.e.~\Cref{Thm:PsiandTheta,},
since \eqref{eq:pt-Phi} from \Cref{Thm:char-parfunA} then gives
\[
  \Upsilon^{T} \Psitil_M 
 = \Upsilon^{T} \frac{\Psi_{\Syz M}}{\Psi_{\procov{M}}}
 = \frac{\Phi^T_{\Syz M}}{\Phi^T_{\procov{M}}} 
 = \ptbar^{T}_{M}.
\]
On the other hand, we know that the top triangle commutes when applied 
to any rank~1 module $M_I$, by \eqref{eq:CKP-twi}. 

Thus all inner triangles commute when applied to any $M_I$,
which is sufficient to show that the outer triangle commutes,
as the $\Psi_{M_I}$ (and ${\Psi_{P}}^{-1}$, 
for $P$ indecomposable projective) 
generate $\CC[\Gro(k, n)]$.
That is, $\Upsilon^T \compo \twi=\widetilde{\Xi}^T$, so
\[
  \Upsilon^T \twi (\Psi_M)= \widetilde{\Xi}^T(\Psi_M) = \ptbar^{T}_{M} = \Upsilon^T\Psitil_M,
\]
since the left and right triangles commute.
But $\Upsilon^T$ is injective, by \Cref{Lem:upsilon}, 
and remains so on extension to $\CC[\Gro(k, n)]$,
so we can conclude that $\twi (\Psi_M)=\Psitil_M$, 
for all $M\in \CM\algC$, completing the proof.
\end{proof} 

We can combine the proofs of \Cref{cor:partfun-map} and \Cref{Thm:twistchar} as follows.

\begin{corollary}\label{rem:flag-thm}
The following diagram commutes,
\begin{equation}\label{diag:flag-thm}
\begin{tikzpicture}[xscale=2.7, yscale=1.6,baseline=(bb.base)]
\coordinate (bb) at (0,2);
 \draw (2,2) node (a) {$\CM\algC$};
 \draw (1,1) node (b2) {$\CC[\Grot(\CM\algA')]$};
 \draw (3,1) node (b3) {$\CC[\Grot(\CM\algA)]$};
 \draw (1,3) node (c2) {$\CC[\Gro(k, n)]$};
 \draw (3,3) node (c3) {$\CC[\Gro(k, n)]$};
 \draw[cdarr] (b2) to node[above] {$\CC[-\zeta]$}  (b3);
 \draw[cdarr] (c2) to node[above] {$\twi$} (c3);
 \draw[cdarr] (c3) to node[right] {$\Upsilon^{T}$} (b3);
 \draw[cdarr] (c2) to node[left] {$\Upsilon^{\CTOcosyz{T}}$}  (b2);
 \draw[cdarr] (a) to node[above] {$\Phi^{\CTOcosyz{T}}$} (b2); 
 \draw[cdarr] (a) to node[above right=-1pt] {$\ptbar^T$} (b3); 
 \draw[cdarr] (a) to node[above right=-2pt] {$\Psi$} (c2); 
 \draw[cdarr] (a) to node[above left=-2pt] {$\Psitil$} (c3); 
\end{tikzpicture}
\end{equation}
where $\algA'=\End (\CTOcosyz{T})\op$ and $\zeta$ is as in \eqref{eq:zeta-def-0}.
\end{corollary}

\begin{proof}
The left and bottom triangles commute by \eqref{eq:PF-Phi-Psi},
while the top and left triangles are from \eqref{diag:twistchar}.
\end{proof}

\begin{remark}
The twist automorphism (and the Chamber Ansatz) originates in work of 
Berenstein--Fomin--Zelevinsky \cite{BFZ} 
(and more generally in \cite{BZ}) where, among other things, it is used to 
invert Lusztig's parametrisation of the totally non-negative part of a unipotent group. 
The Muller--Speyer twist \cite{MS} plays a similar role, being used to compute the inverse of the boundary measurement map for positroids. 
Both these constructions are combinatorial in nature and it is known that the MS twist induces the BFZ twist in the Grassmannian case \cite[App.~4]{MS}. 
Geiss--Leclerc--Schroer \cite[Thm 6]{GLS12} gave a categorical description of 
the BFZ twist using cluster characters and partial projective presentations, 
similar to the definition of $\Psitil$ in \eqref{eq:psi-dag}.
\end{remark}

%%%% =========================================================== %%%%
\section{Bases and Newton--Okounkov cones for Grassmannians}\label{Sec:7} 
%%%% =========================================================== %%%%. 
\newcommand{\basis}{\mathcal{B}}
\newcommand{\GrotAbar}{\Grot(\fd\olA)}
\newcommand{\gen}{\mathrm{gen}}
\newcommand{\gminvec}[2]{\overline{g}_{\gen}(#1)}
\newcommand{\gminmap}[1]{\overline{g}_{\gen}}
\newcommand{\ghminvec}[2]{g_{\gen}(#1)}
\newcommand{\ghminmap}[1]{g_{\gen}}
\newcommand{\hCmin}{\hC_{\gen}}
\newcommand{\uc}{\mathbf{u}_{\hC}}
\newcommand{\gvecloc}[1]{\mathbf{v}_{#1}}
\newcommand{\dvec}[1]{d_{#1}}
% ======================================================================

% ==============================================
\subsection{Bases and their leading exponents}\label{subsec:bases-leading}
% ==============================================
As before, let $T$ be a cluster tilting object in $\CM\algC$ and $A=\End(T)\op$.
Fix a refinement of the order $\leq$ from \Cref{def:part-ord}
to a total order on $\Grot(\CM A)$, as in \Cref{rem:total-order}.

If $f\in \CC[\Gr(k,n)]$ is expressed in a chart associated to $T$,
that is, either the cluster chart $\Upsilon^{T}$ or the network chart $\Xi^T$, 
then its leading exponent is an element in $\Grot(\CM A)$. 
Thus, for any basis $\basis$ of $\CC[\Gr(k,n)]$, there is a map 
\[ \basis \to \Grot(\CM A) \]
sending a basis element to its leading exponent in the chart. 
We are interested in bases where this map is injective, 
that is, the leading exponents are distinct. 
The main reason for this is the following.

\begin{lemma} \label{Lem:uniqueleading}
Let $V\subspc \CC[\Grot(\CM\algA)]$ be a linear subspace with a basis whose leading exponents are distinct. 
Then the leading exponents of the basis elements are all the possible leading exponents for elements of $V$.
\end{lemma}

\begin{proof}
Write any $f\in V$ as a linear combination in the basis.
Among the basis elements with non-zero coefficient,
there is (precisely) one whose leading term has minimal exponent. 
This leading term cannot cancel due to the minimality, so it gives the leading term of $f$.  
\end{proof}

In addition, having distinct leading exponents is useful in identifying bases.

\begin{lemma} \label{Lem:uleading-indep}
If $f_1,\dots,f_k\in \CC[\Grot(\CM\algA)]$ have distinct leading exponents,
then they are linearly independent.
\end{lemma}

\begin{proof}
If there were a linear dependence between the $f_i$, then
among those with a non-zero coefficient,
there is (precisely) one whose leading term has minimal exponent. 
But this term cannot cancel due to the minimality, so there could not have been such a dependence.  
\end{proof}

In this section, we will see how to construct bases with this leading term property
and will see that such bases are far from unique. 
The dual semicanonical basis of $\CC[\Gr(k,n)]$ is an example of such a basis. 
The theta basis \cite[Thm~16.15]{RW} is another example. 
Even the standard monomial basis (\cite{Ho43} or \cite[Ch.XIV \S9]{HP52}) 
has this property for some clusters.

In our case, the basis elements will be cluster characters, like $\Psi_M$, 
so the leading exponents are $g$-vectors, like $[T, M]$.
This will enable us to relate the $g$-vector cone to the Newton-Okounkov cone (see \Cref{Def:noc-etc}).
This also helps us to show that the leading exponents are independent of the 
choices made when defining the order. 

% ======================================================================
\subsection{Generic $g$-vectors}\label{subsec:min-class}
% ======================================================================
%\comment{QO: explain relation to \cite{DF}, \cite{GLabS}, \cite{Plam}}

Let $\preproj$ be the preprojective algebra of type $\dyn{A}_{n-1}$,
with vertices explicitly labelled $1,\dots,n-1$,
and let $\preprojQ_k$ be the indecomposable injective $\preproj$-module at the vertex $k$
(which is also the projective at vertex $n-k$).

Let $\sub\preprojQ_k$ be the subcategory of $\mmod \preproj$ consisting of  
submodules of modules in $\add \preprojQ_k$. 
Note that the socle of $\preprojQ_k$ is the simple $S_k$ at vertex $k$,
and $\sub\preprojQ_k$ also consists of those $\preproj$-modules with socle only at $k$.

Let $\rep(\preproj, d)$ be the representation variety of 
$\preproj$ with dimension vector $d\in \NN^{n-1}$,
that is, the set of module structures on $\CC^d=\bigoplus_{i=1}^{n-1} \CC^{d_i}$,
equipped with the natural structure of an algebraic variety.
This variety carry a natural conjugation action of $\GL{d}=\prod_{i=1}^{n-1} \GL{d_i}$.

Note that, for any $d\in \NN^{n-1}$ %component $\hC\in \IrrSub $ 
and any $\preproj$-module $X$, the map
\begin{equation}\label{eq:usc-dim-hom}
  \rep(\preproj, d) \to \NN \colon M \mapsto \dim \Hom(X, M)
\end{equation}
is upper semicontinuous, 
that is, $\{ M\st \dim \Hom(X, M) \leq m\}$ is open, for any $m\in\NN$.
Let
\[
  \sub(\preprojQ_k, d)\subset \rep(\preproj, d),
\] 
be the $\GL{d}$-invariant subset of $\preproj$-modules in $\sub\preprojQ_k$.
Furthermore, let 
\[
  \sub(\preprojQ_k, d,r)\subset \sub(\preprojQ_k, d),
\] 
be the $\GL{d}$-invariant subset of modules $M$ with $\dim\soc(M)\leq r$.
Note that this is a special case of varieties considered in \cite[\S1.5]{Lus03}.

\begin{lemma}\label{Lem:quotrep}
The subsets $\sub(\preprojQ_k, d)$ and $\sub(\preprojQ_k, d,r)$ are open in $\rep(\preproj, d)$. 
\end{lemma}

\begin{proof}
Modules $M$ in $\sub(\preprojQ_k, d)$ are characterised by  
$\Hom(S, M)=0$, for any simple module $S\not\isom S_k$. 
By the semicontinuity of $\dim \Hom(S, M)$, 
this is an intersection of open subsets, hence open in $\rep(\preproj, d)$.
Modules $M$ in $\sub(\preprojQ_k, d,r)$ are additionally characterised by 
$\dim\Hom(S_k, M)\leq r$, which is again open by semicontinuity.
\end{proof}

Let $\olT=\bigoplus_i \olT_i$ be a cluster tilting object in $\sub\preprojQ_k$, 
let $\olA =\End (\olT)\op$ and let $\olQ$ be the Gabriel quiver of $\olA$.
For $M\in \sub\preprojQ_k$, denote by $[\olT, M]$ the class of the $\olA$-module $\Hom(\olT, M)$ 
in the Grothendieck group $\Grot(\proj\olA) = \Grot(\fd\olA) = \ZZ^{\olQ_0}$.
Thus $[\olT, M]=\dimv\Hom(\olT, M)$.

Denote by $\IrrSub(d)$ the set of irreducible components in $\sub(\preprojQ_k, d)$ and let 
\[
  \IrrSub=\bigcup_d \IrrSub(d).
\] 
\begin{definition}\label{def:gmin-cmin}
By semicontinuity of \eqref{eq:usc-dim-hom},
for each $\hC\in \IrrSub$, 
there is a class 
\[ \gminvec{\hC}{T}\in\GrotAbar, \] 
which is minimal (as a dimension vector) in $\{[\olT, N]\st N\in \hC\}$. 
Furthermore, the set  
\[
  \hCmin=\{N\in \hC\st [\olT, N]=\gminvec{\hC}{T} \}
\] 
is an open dense subset of $\hC$. 
Consequently, we call $\gminvec{\hC}{T}$
the \emph{generic $g$-vector} of~$\hC$.
\end{definition}

\begin{remark}\label{rem:sem-can-basis}
Following Lusztig~\cite{Lus00} and Geiss--Leclerc--Schr\"{o}er~\cite{GLS05,GLS08, Lec},
the dual semicanonical basis of $\CC[\unirad]$ is $\{ \psi_{M_\hC} \st \hC\in \IrrSub\}$,
where $M_\hC$ is any module in a certain dense open subset $\uc\subset \hC$ 
on which $\psi_{M}$ (as in \S\ref{subsec:PsiPhi}) is constant.
The leading exponent of~$\psi_{M}$ in a cluster chart,
that is, the leading exponent of~$\phi_{M}$ from \eqref{eq:phi-Tbar},
is $[\olT,M]$.
Hence $\uc\subset \hCmin$, that is, $[\olT,M_\hC]=\gminvec{\hC}{T}$.

An alternative (more general) approach to finding a basis of generic cluster characters 
indexed by $g$-vectors is given by Plamondon \cite[Thm~1.1]{Plam};
see also \cite[\S5]{Fei1}.
However, in general, this approach only finds a linearly independent set in
the upper cluster algebra.
The indexing sets of $g$-vectors in \cite{Plam} and \cite{Fei1} are also described differently.
\end{remark}

Let $T', T''\in \add \olT$ with $T''$ a submodule of $T'$.
In other words, suppose that the set
\[ \Inj(T'', T')=\{f\in \Hom(T'', T')\st \text{$f$ is injective}\} \]
is non-empty and hence a dense open subset of $\Hom(T'', T')$.
Note that $\Hom(T'', T')$ is a finite-dimensional vector space over $\CC$ 
and we consider the Zariski topology on it,
so any non-empty open subset is dense.

Let $d=\dimv (T'/T'')\in\NN^{n-1}$ and define 
\begin{equation} \label{eq:defV}
 V = \{(f, g, X) \st 
 \text{$T''\stackrel{f}{\to} T'\stackrel{g}{\to} X$ is a short exact sequence, 
 $X\in \rep(\preproj, d)$} \}. 
\end{equation}

The following lemma is standard, but we include a brief proof for completeness.

\begin{lemma}\label{Lem:prinbundle}
The projection $\alpha\colon V\to \Inj(T'', T')\colon (f, g, X)\mapsto f$ 
is a principal $\GL{d}$-bundle.
Hence $V$ is irreducible.
\end{lemma}

\begin{proof}
If we fix $f\in \Inj(T'', T')$, then the fibre $\alpha^{-1}(f)$ % set of compatible $(g,X)$ 
can be identified with the set $\Iso(T'/\img f,\CC^d)$ of isomorphisms of graded vector spaces 
$\ol{g}\colon T'/\img f \to \CC^d$, 
since any such isomorphism induces a unique module structure on $\CC^d$.
This set is a single free orbit for the action of $\GL{d}$ on the pairs $(g,X)$.

Now write $T'=\img f\oplus W$, as graded vector spaces, and let
\[
  U_f = \{ f'\in \Inj(T'', T')\st \img f' \cap W=0 \},
\] 
which is an open neighbourhood of $f$. Then 
\[
  \alpha^{-1}(U_f) = \{ (f', g, X) \st 
   \text{$g|_{W}$ is an isomorphism} \} = U_f\times \Iso(W,\CC^d),
\]
Hence $\alpha$ is a principal $\GL{d}$-bundle.
As $\Inj(T'', T')$ is irreducible, $V$ is also.
\end{proof}

\begin{lemma}\label{Lem:largeapprox}
Given a dimension vector $d\in\NN^{n-1}$, there is some $T'\in \add \olT$ that gives
the first term of an $\add \olT$-presentation of any $X\in \sub(\preprojQ_k, d)$. 
\end{lemma}

\begin{proof}
For each $i\in\olQ_0$, we can necessarily find $c_i$ such that 
$\dim \Hom(\olT_i, X)\leq c_i$, for all $X\in \sub(\preprojQ_k, d)$.
As a very crude bound, we could take $c_i= \sum_j d_j t_{ij}$,
where $(t_{ij})=\dimv\top \olT_i$.
Then $T'=\bigoplus_i (\olT_i)^{c_i}$ gives a (right) $\add \olT$-approximation to any 
$X\in \sub(\preprojQ_k, d)$, 
by choosing $c_i$ maps that span $\Hom(\olT_i, X)$ to make a map $g\colon T'\to X$.

Since $\add \olT$ includes the projective $\olA$-modules, $g$
will be surjective.
Furthermore, $T'':=\ker g\in\sub \preprojQ_k$ and $\Ext^1(\olT,T'')=0$, 
so $T''\in\add \olT$, as required.
\end{proof}

As $\olT$ contains the projectives in $\sub \preprojQ_k$,
there is a map $\ZZ^{\olQ_0}\to\ZZ^{n-1}\colon x\mapsto \dvec{x}$
such that, for any $M\in \sub\preprojQ_k$, if $[\olT,M]=x$, then $\dimv M=\dvec{x}$.
For $x\in \NN^{\olQ_0}$, define
\begin{equation}\label{eq:gvecloc}
  \gvecloc{x}=\{M\in \sub(\preprojQ_k, \dvec{x})\colon [\olT,M]=x\},
\end{equation}
which is a locally closed subvariety of $\sub(\preprojQ_k, \dvec{x})$.

\begin{proposition}\label{lem:cx} 
If $\gvecloc{x}$ is non-empty, then it is irreducible.
\end{proposition}

\begin{proof} 
By \Cref{Lem:largeapprox}, we can find a fixed $T'\in \add \olT$ that gives an 
$\add\olT$-presentation
\[ \ShExSeq{T''} {T'} {M}\]
of any $M\in \sub(\preprojQ_k, \dvec{x})$.
If $[\olT, M]=x$, then  $[\olT, T''] = [\olT, T']-x$ and so, 
since the classes $[\olT, \olT_i]$ form a basis of $\GrotAbar$, 
we can also fix $T''$ and suppose that $T''\leq T'$.

Then $T'$, $T''$ and $d=\dvec{x}$ determine $V$, as in \eqref{eq:defV}, 
and we can consider
\begin{equation}\label{eq:Wset}
  W = \{(f, g, X) \in V  \st X\in \sub(\preprojQ_k, \dvec{x}) \},
\end{equation}
which is a nonempty open subset of $V$.
By \Cref{Lem:prinbundle}, $W$ is irreducible 
and so its projection onto the last component, which is $\gvecloc{x}$,
is irreducible. 
\end{proof}

\begin{corollary}\label{cor:ggen-inj}
The map 
$\gminmap{T}\colon\IrrSub \to \GrotAbar\colon \hC\mapsto \gminvec{\hC}{T}$,
from \Cref{def:gmin-cmin},
is injective.
Moreover, $\hCmin = \gvecloc{x}$, for $x=\gminvec{\hC}{T}$.
\end{corollary}

\begin{proof}
If $\gminvec{\hC}{T}=x$, then $\hCmin \subset \gvecloc{x}$.
But $\gvecloc{x}$ is irreducible, by \Cref{lem:cx}, 
so $\hCmin = \gvecloc{x}$, proving the second part.
But then $\hC=\overline{\gvecloc{x}}$, so $\gminvec{\hC}{T}$ determines $\hC$,
that is, the map $\gminmap{T}$ is injective.
\end{proof}

\begin{remark}\label{rem:rigid-gvec}
By Voigt's Lemma \cite[II.3.5]{Vo}, rigid modules are generic,
in the sense that the corresponding orbit is a dense open subset of $\sub(\preprojQ_k,d)$.
In particular, a rigid module $M$ is contained in only one component $\hC$ 
and $[\olT,M]=\gminvec{\hC}{T}$.
Hence \Cref{cor:ggen-inj} includes the already known fact that
different rigid modules have different $g$-vectors 
(cf.~\cite[Thm~2.3]{DK}, \cite[Thm~5.5]{FK}).
\end{remark}

Note that $\CC[\unirad]=\bigoplus_d\CC[\unirad]_d$,
where $d\in \NN^{n-1}$ is a weight
for a maximal torus of $\SL{n}$, but is also a dimension vector for $\preproj$.
For example, if $\dimv M=d$, then $\psi_M\in \CC[\unirad]_d$.
It is known (ultimately by \cite{Lus91}; see also \cite[\S9]{GLS08}, \cite[\S5]{Lec}) 
that $\dim \CC[\unirad]_d$ is equal to the number of components of $\sub(\preprojQ_k,d)$.
We include a proof for completeness, using the following elementary fact, 
which seems special to the case in hand. 

\begin{lemma}\label{lem:SubQk-cond}
Let $\Lambda$ be the path algebra of a quiver of type $\dyn{A}_{n-1}$ 
with all arrows directed towards vertex $k$ and let $M\in\rep\preproj$.
Then $M\in\sub\preprojQ_k$ if and only if 
all indecomposable summands of $M_{|\Lambda}$ are supported at $k$.
\end{lemma}
\begin{proof}
If all summands of $M_{|\Lambda}$ are supported at $k$, then, due to the choice of arrow orientations,
$\soc_\Lambda M$ is only at $k$ and hence $\soc_\preproj M$ is only at $k$, that is, $M\in\sub\preprojQ_k$.
Conversely, if $M\in\sub\preprojQ_k$, that is, $M$ is a submodule of a sum of copies of $\preprojQ_k$,
then all the arrows pointing towards $k$ act injectively, because $\preprojQ_k$ has this property.
Thus every indecomposable summand of $M_{|\Lambda}$ is supported at $k$.
\end{proof}

\begin{corollary}\label{cor:dim=numcomp}
$\dim \CC[\unirad]_d=|\IrrSub (d)|$.
\end{corollary}

\begin{proof}
First note that $\CC[\unirad]$ is a polynomial ring in variables $x_\alpha$ of degree $\alpha$, 
for those positive roots $\alpha$ supported at vertex $k$.
Thus $\CC[\unirad]_d$ has a basis of monomimials $\prod_\alpha x_\alpha^{d_\alpha}$, 
for  $\sum_{\alpha} d_\alpha \alpha = d$.

On the other hand, by considering restrictions such as $M\mapsto M_{|\Lambda}$,
Lusztig \cite[Prop 14.2]{Lus91} showed that,
for any $\preproj$ of Dynkin type $\dyn{A}$, $\dyn{D}$ or $\dyn{E}$, 
the irreducible components of $\rep(\Pi, d)$ 
are closures of conormal bundles of the $\GL{d}$-orbits in $\rep(\Lambda, d)$.
By Gabriel's Theorem, such orbits correspond to expressions $d=\sum_{\alpha} d_\alpha \alpha$,
for any positive roots $\alpha$.
Hence, the result follows by \Cref{lem:SubQk-cond}.
\end{proof}

\begin{proposition}\label{thm:gen-basis}  
If $N_\hC\in\hCmin$, for each $\hC\in\IrrSub (d)$,
then $\{\psi_{N_\hC}\st \hC\in\IrrSub (d)\}$ is a basis of $\CC[\unirad]_d$.
\end{proposition}

\begin{proof}
In the cluster chart associated to $\olT$, the leading exponent of $\psi_{N_\hC}$, 
that is, the leading exponent of~$\phi_{N_\hC}$ from \eqref{eq:phi-Tbar},
is $[\olT,N_\hC]$.
These exponents are distinct, by \Cref{cor:ggen-inj}.
Hence, by \Cref{Lem:uleading-indep}, the $\phi_{N_\hC}$, and thus the $\psi_{N_\hC}$,
are linearly independent.
Since $\dim \CC[\unirad]_d=|\IrrSub (d)|$,
by \Cref{cor:dim=numcomp},
we have a basis as claimed.
\end{proof}

By \Cref{rem:sem-can-basis}, 
the dual semicanonical basis is an example of such a basis. 

\begin{proposition}\label{thm:subQComp}  
Every $g$-vector is generic.
More precisely, for any $M\in\sub(\preprojQ_k,d)$,
there is some (unique) $\hC\in \IrrSub (d)$ such that $[\olT,M]=\gminvec{\hC}{T}$
and, furthermore, $M\in\hCmin$.
In other words,
\begin{equation}\label{eq:subQComp} 
\sub(\preprojQ_k,d) = \bigcup_{\hC\in\IrrSub(d)} \hCmin.
\end{equation}
\end{proposition}

\begin{proof} 
For any basis $\{ \psi_{N_{\hC}} \st \hC\in  \IrrSub (d)\}$,
as in \Cref{thm:gen-basis}, the leading exponents are $[\olT,N_{\hC}]=\gminvec{\hC}{T}$,
which are distinct, by \Cref{cor:ggen-inj}.
Hence, by \Cref{Lem:uniqueleading}, the leading exponent $[\olT,M]$ of any $\psi_{M}$,
must be equal to $\gminvec{\hC}{T}$ for some $\hC$,
that is, every $g$-vector is generic.

The last part was already seen in the proof of \Cref{cor:ggen-inj},
that is, if $x=\gminvec{\hC}{T}$, then $\gvecloc{x}=\hCmin$,
because $\gvecloc{x}$ is irreducible, by \Cref{lem:cx}.
\end{proof}

\begin{corollary}\label{cor:comp} 
For $M\in \sub(\preprojQ_k, d)$, let $\hC_{1}, \dots, \hC_{t}\in\IrrSub(d)$
be the components containing~$M$. 
Then 
\begin{equation}\label{eq:class}
  [\olT, M]=\gminvec{\hC_j}{T},
\end{equation}
for some (unique) $j$ with $\gminvec{\hC_j}{T}>\gminvec{\hC_\ell}{T}$, for $\ell\neq j$.

In particular, if $M\in \hC$ for just one component $\hC$,
then $[\olT, M]=\gminvec{\hC}{T}$.
In other words, $\hC_\gen$ always contains the 
complement (in $\hC$) of all the other components.
\end{corollary}

\begin{proof}
\Cref{thm:subQComp} implies that $[\olT, M]=\gminvec{\hC_j}{T}$, for some $j$. 
The uniqueness follows from \Cref{cor:ggen-inj},
while the inequality follows from the upper semicontinuity of $\dimv \Hom(\olT, -)$.
The rest of the statement follows immediately.
\end{proof}

% =================================================================
\newcommand{\grotiso}{\theta}
\newcommand{\grotisotoo}{\widehat{\theta}}
% =================================================================
\subsection{Bases of $\CC[\Gr(k, n)]$ with distinct leading exponents.} %and of the cluster algebra $\clualgA$}
\label{subsec:bases}
% =================================================================
We now explain how to lift results of \S\ref{subsec:min-class} from $\sub\preprojQ_k$ to $\CM\algC$.
Recall, from \eqref{eq:pifun-omfun} and \S\ref{subsec:PsiPhi}, 
that the choice of boundary vertex $\vstar$ determines an exact functor
\[ \pifun\colon \CM\algC\to \sub\preprojQ_k \colon M\mapsto M/\funP{M} \]
such that $\pifun \modP=0$.
If $T$ is a cluster tilting object in $\CM\algC$, 
then $\olT=\pi T$ is a cluster tilting object in $\sub\preprojQ_k$,
with one fewer summand.

Any $M\in\CM\algC$ is determined (up to isomorphism) 
by the pair $(\rnk{\algC} M, \pifun M)$. 
Conversely, a pair $(r, X)$ comes from some $M$ provided $\dim\soc X\leq r$
(see \cite[Thm~4.5]{JKS1}).
At the level of Grothendieck groups, there is an isomorphism
\begin{equation}\label{eq:Grot-isom}
 \grotiso\colon \Grot(\CM\algC)\to \ZZ\times \Grot(\sub\preprojQ_k)\colon [M]\mapsto (\rkk [M],[\pifun M]),
\end{equation}
where $\rkk [M]=\rnk{\algC} M$.
For $\lambda\in\Grot(\CM\algC)$, 
let $\CM(\lambda)$ be the set of isomorphism classes in $\CM\algC$ of modules of class $\lambda$
and let
\[
  \CM(r)=\bigcup_{\rkk\lambda=r} \CM(\lambda),
\]
that is, the set of isomorphism classes modules of rank $r$.

If $\grotiso(\lambda)=(r,d)$, 
then $\pifun$ gives a bijection between $\CM(\lambda)$ and the set of $\GL{d}$-orbits in 
$\sub(\preprojQ_k,d,r)$,
which is an open subset of $\sub(\preprojQ_k,d)$, by \Cref{Lem:quotrep}.
Slightly abusing notation, 
if $\hC\in\IrrSub(d)$ intersects $\sub(\preprojQ_k,d,r)$, then we write 
\begin{equation}\label{eq:Crc}
  \compC(r, \hC)=\{ M\in \CM(\lambda)\st \pifun M \in \hC \} 
\end{equation}
and call this set an irreducible component of $\CM(\lambda)$.
Let $\IrrCM(\lambda)$ be the collection of all such components and further write
\[ 
 \IrrCM(r) = \bigcup_{\rkk\lambda=r} \IrrCM(\lambda),\quad 
 \IrrCM = \bigcup_r \IrrCM(r).
\]
Note, in particular, that $\IrrSub(0)$ has precisely one element $\mathbf{0}$ 
which is the component consisting of just the zero module in $\sub(\preprojQ_k, 0)$.
Then $\compC(r,\mathbf{0})= \bigl\{ (\modP)^r \bigr\}$.

\begin{proposition}\label{lem:num-components}
For any $\lambda\in \Grot(\CM\algC)$, 
\begin{equation}\label{eq:dimIrr}
|\IrrCM(\lambda)| = \dim \CC[\Gr(k, n)]_\lambda.
\end{equation}
On the right-hand side, $\lambda$ is interpreted as a $\GL{n}$ weight %in $\ZZ^n(k)\subset \ZZ^n$.
(as in \cite[\S8]{JKS1}).
\end{proposition}

\begin{proof}
By definition, the components of $\CM(\lambda)$ are in bijection with the components of $\sub(\preprojQ_k, d,r)$,
while $\CC[\Gr(k, n)]_\lambda$ is the $\lambda$ weight space of $V_{r\omega_k}$, 
as in \eqref{eq:irrep-decomp}.
Hence the result is a dual version of \cite[Prop 1.6 (a)]{Lus03} (see also \cite[\S 5.1]{GLS06}).
\end{proof}

\begin{lemma} \label{lem:g-vec-projection}
At the level of Grothendieck groups, there is an isomorphism
\begin{equation}\label{eq:Grot-isomA}
 \grotisotoo\colon \Grot(\CM\algA)\to \ZZ\times \GrotAbar
 \colon [T,M]\mapsto (\rnk{\algC} M,[\pifun T,\pifun M]).
\end{equation}
Furthermore, if $\rnk{\algC} M=\rnk{\algC} N$, then 
\[
[\pifun T,\pifun M] \leq [\pifun T,\pifun N] \iff \kapvec(T,M) \geq \kapvec(T,N),
\] 
noting that the right-hand side is also the partial order in \Cref{def:part-ord}.
\end{lemma}

\begin{proof} 
The first component of the map $\grotisotoo$ is $\rkk$, since $\rkk [T,M]=\rnk{\algC} M$.
The map is uniquely determined by its values on the basis $[T,T_i]$ of $\Grot(\CM\algA)$.
Thus $[T,\modP]$ maps to~$(1,0)$, but otherwise the second components of $\grotisotoo[T,T_i]$
are a basis of $\GrotAbar$.
Hence $\grotisotoo$ is an isomorphism.
The fact that $\grotisotoo$ has the claimed effect on $[T,M]$ follows from~\eqref{eq:piTN},
because the second component of $\grotisotoo$ is $\AAmap$.

For the second part, recall from \eqref{eq:wtmod-diag} that there is a short exact sequence
\[
\ShExSeq {\Hom(T, M)} {\Hom(T, \funJ M)} {\funK(T, M)}.
\]
Applying $\Hom(T, -)$ to the defining sequences (cf. \eqref{eq:pifun-omfun})
of $\pifun M$ and $\pifun\funJ M$ gives
\[
\ShExSeq {\Hom(\pifun T, \pifun M)} {\Hom(\pifun T, \pifun\funJ M)} {\funK(T, M)}.
\]
The dimension vectors of the outer terms are $[\pifun T, \pifun M]$ and 
$\kapvec(T,M)=\wt[T,M]$ (see \eqref{eq:wtTM}).
The middle term is fixed by $\rnk{\algC} M$, so these dimension vectors are oppositely ordered,
as required.
\end{proof}

As in \Cref{def:gmin-cmin}, every component $\compC\in \IrrCM$
contains a dense open set~$\compC_{\gen}$ on which the class $[T, M]$ is constant,
with value $\ghminvec{\compC}{T}\in \Grot(\CM \algA)$.
More precisely, if $\compC=\compC(r, \hC)$, as in \eqref{eq:Crc}, then 
\begin{equation}\label{eq:Crc-too}
  \compC_{\gen}=\{ M\in \CM(\lambda)\st \pifun M \in \hC_{\gen} \} 
\end{equation}
and furthermore $\grotisotoo\bigl( \ghminvec{\compC}{T} \bigr)=(r,\gminvec{\hC}{T})$.

\begin{remark}\label{rem:Cgen}
The bijection of $\CM(\lambda)$ with the set of $\GL{d}$-orbits in 
$\sub(\preprojQ_k,d,r)$ induces a (quotient) topology on $\CM(\lambda)$,
for which the components $\compC$ are the irreducible components.
Although the definition of this topology, and hence the subsets~$\compC$, 
depends on the choice of $\vstar$ which determines $\pi$, 
it seems reasonable to conjecture that, in fact, it does not depend on $\vstar$.
Then $\compC_{\gen}$ and $\ghminvec{\compC}{T}$ would just depend on~$T$.
\end{remark}

\begin{proposition} \label{prop:ghgen-inj}
The map 
$\ghminmap{T}\colon \IrrCM\to \Grot(\CM \algA)\colon \compC\mapsto \ghminvec{\compC}{T}$
is injective.
\end{proposition}

\begin{proof}  
By \Cref{lem:g-vec-projection}, the injectivity of $\ghminmap{T}$
follows from the injectivity of $\gminmap{T}$, %\colon \IrrSub\to \Grot(\proj \ol\algA)$, 
from \Cref{cor:ggen-inj}.
\end{proof}

We can now find the analogue of \Cref{thm:gen-basis}.

\begin{proposition}\label{Thm:GrassGeneric}
If $N_\compC\in\compC_{\gen}$, for each $\compC\in\IrrCM(\lambda)$,
then $\{\Psi_{N_\compC}\colon \compC\in \IrrCM(\lambda)\}$ is a basis of $\CC[\Gr(k, n)]_{\lambda}$. 
\end{proposition}

\begin{proof}
Following \cite[Prop 9.4]{JKS1}, if $M\in \CM C$ has $[M]= \lambda$, 
then $\Psi_M\in\CC[\Gr(k, n)]_\lambda$.
In the cluster or network chart associated to $T$, the leading exponent of $\Psi_{N_\compC}$, 
is $[T,N_\hC]$.
These exponents are distinct, by \Cref{prop:ghgen-inj}.
Hence, by \Cref{Lem:uleading-indep}, the $\Psi_{N_\hC}$ are linearly independent.
The number of components $|\IrrCM(\lambda)|$ equals $\dim \CC[\Gr(k, n)]_\lambda$,
by \Cref{lem:num-components},
so we have a basis as claimed.
\end{proof}

Since
\[\CC[\Gr(k, n)]_r=\bigoplus_{\rkk\lambda=r} \CC[\Gr(k, n)]_{\lambda}
\quadand \CC[\Gr(k, n)]=\bigoplus_{r\geq 0} \CC[\Gr(k, n)]_{r},
\]
we obtain bases 
$\{\Psi_{N_\compC}\colon \compC\in \IrrCM(r)\}$ for $\CC[\Gr(k, n)]_{r}$
and $\{\Psi_{N_\compC}\colon \compC\in \IrrCM\}$ for $\CC[\Gr(k, n)]$,
simply by maintaining the condition $N_\compC\in\compC_{\gen}$.

\begin{remark}\label{rem:semican-too}
If we choose $N_\compC$ to be the lift of $M_\hC$ from \Cref{rem:sem-can-basis}
to $\CM(r)$, for any $\hC$ with $\dim \soc M_\hC\leq r$, 
then the basis $\{\Psi_{N_\compC}\colon  \compC\in \IrrCM(r)\}$
is the dual semicanonical basis of $\CC[\Gr(k, n)]_r$ 
and is the lift of the appropriate part of the dual semicanonical  basis
of $\CC[\unirad]$
(cf. \cite[\S9]{GLS08}).
\end{remark}

We also have the analogue of \Cref{thm:subQComp}.

\begin{theorem}\label{thm:subQComp-too}  
Every $g$-vector is generic, 
that is, 
\begin{equation}\label{eq:gen-g-vec} 
 \gvm{T} = \{ \ghminvec{\compC}{T} \st \compC\in \IrrCM \}.
\end{equation}
Consequently, 

\begin{equation}\label{eq:subQComp-too} 
\CM(\lambda) = \bigcup_{\compC\in \IrrCM(\lambda)} \compC_{\gen} .
\end{equation}
If $\compC_{1}, \dots, \compC_{t}\in\IrrCM(\lambda)$
are the components containing~$M\in \CM(\lambda)$, then 
\begin{equation}\label{eq:class-too}
  [T, M]=\ghminvec{\compC_j}{T},
\end{equation}
for some (unique) $j$ with $\ghminvec{\compC_j}{T}<\ghminvec{\compC_\ell}{T}$, for $\ell\neq j$,
where the partial order here is as in \Cref{def:part-ord}.
\end{theorem}

\begin{proof} 
Using \Cref{lem:g-vec-projection}, the result follows directly from \Cref{thm:subQComp}  
and \Cref{cor:comp}.
Alternatively, it follows in the same way as these do, 
from \Cref{Thm:GrassGeneric}, \Cref{prop:ghgen-inj} and  \Cref{Lem:uniqueleading}.
\end{proof}

\begin{corollary}\label{cor:bottom-line}
If $M\in \compC$, for just one component $\compC$,
then $[T, M]=\ghminvec{\compC}{T}$.
In other words, $\compC_\gen$ always contains the 
complement (in $\compC$) of all the other components.
Furthermore, simply choosing, for each $\gvec\in \gvm{T}$, a module~$M$ with $[T,M]=\gvec$, 
ensures that the corresponding $\cluschar{M}$ form a basis of $\CC[\Gr(k, n)]$.
\end{corollary}

\begin{proof}
This follows immediately from \Cref{Thm:GrassGeneric} and \Cref{thm:subQComp-too}. 
\end{proof}

\begin{remark}
If $N\in\CM(\lambda)$ is rigid, then
$\pifun N$ is rigid for any choice of $\vstar$.
Hence by \Cref{rem:rigid-gvec},
$\pifun N$ is generic in the component $\hC$ with $\gminvec{\hC}{T}=[\pifun T, \pifun N]$,
for any~$T$.
In other words, $N$ is generic in the component $\compC$ 
with $\ghminvec{\compC}{T}=[T, N]$, for any~$T$, independent of the choice of~$\vstar$.
In particular, $N$ is in $\compC$ and no other component of $\CM(\lambda)$,
independent of~$\vstar$.

Even if the conjecture from \Cref{rem:Cgen} (that the components $\compC$ 
are independent of~$\vstar$) is not true,
then one would still conjecture that every component contains a set of modules
that are in no other component for any ~$\vstar$.
Furthermore, this set should include some $M$ 
for which $\Psi_M$ is the dual semicanonical basis element for that component.

This would prove that the dual semicanonical basis is invariant 
under cycling the indices of Pl\"ucker coordinates.
Note that, by \cite[Thm 1(iv)]{Lam}, the dual canonical basis has this property.
\end{remark}

% ======================================================================
\subsection{Newton--Okounkov and $g$-vector monoids}\label{subsec:NO-gvec-mons}
% ======================================================================

Recall, from \Cref{cor:partfun-map}, that there is an isomorphism of cluster algebras
\[
  \Xi^T\colon \CC[\Gr(k,n)]\to \clualgA_T \colon \Psi_M\mapsto \ptfn^T_{M}
\] 
where $\clualgA_T\subset \CC[\Grot(\CM\algA)]$ is as in \Cref{def:clualgA}.
Consequently, \Cref{Thm:GrassGeneric} immediately gives the following. 

\begin{proposition}\label{cor:GrassGeneric}
If $N_\compC\in\compC_{\gen}$, for each $\compC\in\IrrCM$,
then $\{\ptfn^T_{N_\compC} \st  \compC\in \IrrCM\}$ 
is a basis for $\clualgA_T$.
\end{proposition}

Recall also, from \Cref{Def:noc-etc}, that we have two monoids in $\Grot(\CM\algA)$, namely
\begin{align*}
 \nom{T} &= \{ \Val{T}(f) \st f \in \clualgA_T \setminus 0 \}, \\
 \gvm{T} &= \{ [T,M] \st M\in \CM\algC \},
\end{align*}
together with their cones in $\Grot(\CM\algA)\tensR$.
Note that $\nom{T}$ is essentially the NO-monoid of $\CC[\Gr(k,n)]$,
studied in \cite[\S17.3]{RW}
(see \Cref{rem:noc=gvc} for a more precise statement).

\begin{theorem} \label{Thm:NOcone} 
The Newton--Okounkov and $g$-vector monoids coincide. More precisely,
\begin{equation}\label{eq:gvecs-r-generic}
  \nom{T} = \{\ghminvec{\compC}{T} \st \compC\in \IrrCM \} = \gvm{T}.
\end{equation}
Consequently, we also have $\noc{T}= \gvc{T}$.
\end{theorem}

\begin{proof}
The second equality is \eqref{eq:gen-g-vec} in \Cref{thm:subQComp-too}.
For the first, we reuse the main argument of that proof (cf.~\Cref{thm:subQComp}).
That is, by \Cref{cor:GrassGeneric}, $\clualgA_T$ has a basis consisting of 
$\ptfn^T_{N_\compC}$, for $\compC\in \IrrCM$,
with leading exponents $[T, N_\compC] = \ghminvec{\compC}{T}$, 
which are distinct by \Cref{prop:ghgen-inj}. %\Cref{Thm:genericb}.
By \Cref{Lem:uniqueleading}, the leading exponent $\Val{T}(f)$ 
of any $f\in\clualgA_T$ is the leading exponent of a basis element, 
which gives the second equality.
The final equality follows by applying $\overline{\Rspan}$
to the monoids to get the cones.
\end{proof}

\begin{remark}
As the proof shows, \Cref{Lem:uniqueleading} enables the NO-cone to be understood using any 
basis with distinct leading exponents.
For example, in the next section we can use the standard monomial basis 
in the case of the rectangles cluster.
\end{remark}

%%%% =========================================================== %%%%
\section{$\kapvec(T, M)$ and Gelfand--Tsetlin polytopes} \label{Sec:10}
%%%% =========================================================== %%%%
\newcommand{\profile}[1]{\operatorname{pr}_{#1}}
\newcommand{\level}{level}

Let $\grid=\grid(k,n)$ be the $k\times (n-k)$ grid, as in \eqref{eq:grid}.
A \emph{Gelfand--Tsetlin pattern} (cf.~\cite[Def.~16.1]{RW})
for the $\GL{n}$ representation $\CC[\Gr(k, n)]_r$ can be defined as 
(or simplified to) an integer vector $u=(u_{ij})\in \ZZ^\grid$ satisfying
\begin{equation}\label{eq:GT-pattern}
\begin{aligned}
 u_{1 1} & \geq 0, \quad
 u_{k (n-k)} \leq r , \\
 u_{i j} &\geq u_{(i-1) j} \quad\text{for $2\leq i\leq k$ and $1\leq j\leq n-k$,} \\
 u_{i j} &\geq u_{i (j-1)} \quad\text{for $1\leq i \leq k$ and $2 \leq j\leq n-k$.}
\end{aligned}
\end{equation}
The \emph{GT-polytope} at \level~$r$ is the collection of 
all solutions to \eqref{eq:GT-pattern} in $\RR^\grid$.
Thus GT-patterns are the integer points of the GT-polytope.
Moreover, this polytope is integral, that is, the convex hull of its integer points
(see e.g.~\cite{Ax} for a modern account).
Clearly the GT-polytope is the level $r$ slice of a cone in $\RR\oplus\RR^\grid$,
where $r$ is the first coordinate.

Following \cite[Rem.~10.12, Lem.~16.2]{RW}, we can add $u_{ij}$ along diagonals
to get new coordinates $v_{ij}$ related by a unimodular change of variables
\begin{equation}\label{eq:u-v-cov}
u_{ij} = v_{ij}-v_{(i-1)(j-1)},
\end{equation} 
with the convention that $v_{ij}=0$ if $i=0$ or $j=0$.

\begin{definition}\label{Def:GTpattern}
A \emph{cumulative GT-pattern at \level~$r$} is an integer vector $v\in \ZZ^\grid$
that satisfies the following.
\begin{equation}\label{eq:CGT-pattern}
\begin{aligned}
 v_{11} &\geq 0, \qquad 
 v_{k(n-k)}-v_{(k-1)(n-k-1)}\leq r, \\
 v_{i j}-v_{(i-1)(j-1)} &\geq v_{(i-1) j}-v_{(i-2)(j-1)} \quad\text{for $2\leq i\leq k$ and $1\leq j\leq n-k$,}\\
 v_{i j}-v_{(i-1)(j-1)} &\geq v_{i (j-1)}-v_{(i-1)(j-2)} \quad\text{for $1\leq i \leq k$ and $2 \leq j\leq n-k$.}
\end{aligned}
\end{equation}
Note that here $k$ and $n-k$ are swapped relative to \cite[Rem.~10.12]{RW}
(cf. \Cref{rem:compare-RW}).

The (cumulative) \emph{GT-monoid} and (cumulative) \emph{GT-cone} are 
\begin{align*}
 \gtm &= \{ (r,v)\in \ZZ\oplus\ZZ^\grid \st \text{$v$ satisfies \eqref{eq:CGT-pattern}} \}, \\
 \gtc &= \{ (r,v)\in \RR\oplus\RR^\grid \st \text{$v$ satisfies \eqref{eq:CGT-pattern}} \}.
\end{align*} 
Thus cumulative GT-patterns are the integral points of the rational polyhedral cone $\gtc$.
\end{definition}

Recall also, from \Cref{def:rect-clus}, the rectangles cluster tilting object 
\[
T=\rectCTO=T_\vempty\oplus \bigoplus_{ij\in\grid} T_{ij}.
\]
Note that $T_\vempty=\modP$, so $\kappa(T_\vempty, M)=0$, for any $M\in\CM\algC$, by \eqref{eq:KP*M}.
Hence we can consider that $\kapvec(T, M)\in \ZZ^\grid$, with
\[
\kapvec(T, M)_{ij}=\kappa(T_{ij}, M).
\]

\begin{lemma}\label{Lem:kappacump}
If $M$ is a rank 1 module, then $\kapvec(T, M)$ is a cumulative GT-pattern
at \level~1.
Furthermore, every such pattern occurs uniquely in this way.
\end{lemma}

\begin{proof}
Every rank~1 module has the (unique) form $M=M_I$ for some $k$-set $I\subset\labset{n}$.
On the other hand, there is a clear bijection $I\mapsto u(I)$ between such $k$-sets and GT-patterns
at \level~1, determined by the boundary `profile' $\profile{I}$ between $0$'s and $1$'s in the pattern, 
as illustrated in \Cref{fig:russian}.

Note that here we orient $(i,j)$-coordinates with the $i$-axis north-west and the $j$-axis north-east.
We then write $u_{ij}$ or $v_{ij}$ in the unit box whose top corner is at $(i,j)$, 
so that the profile runs from $(k,0)$ to $(0,n-k)$, with $I$ labelling the south-east steps.

If $v(I)$ is the cumulative GT-pattern corresponding to $u(I)$ under \eqref{eq:u-v-cov}, 
then, to complete the proof, 
we observe that $\kappa(T_{ij}, M_I)=v(I)_{ij}$.
More precisely, either
\begin{itemize}
\item[(a)] 
the $ij$-box is below the profile $\profile{I}$ and $\kappa(T_{ij}, M_I)=0=v(I)_{ij}$, or
\item[(b)] 
the $ij$-box is above $\profile{I}$ and $\kappa(T_{ij}, M_I)$ 
is the vertical distance from $\profile{I}$ to $(i, j)$, 
which is precisely $v(I)_{ij}$. \qedhere
\end{itemize}
\end{proof}

\begin{figure}[h]
\begin{tikzpicture} [scale=0.4,
upbdry/.style={thick, gray},
lowbdry/.style={thick, gray},
boxes/.style={thick, blue},
ridge/.style={very thick, purple}]

\begin{scope} [shift={(-9,0.5)},rotate=135]
\draw [boxes] (0,-1)--(3,-1) (0,-2)--(1,-2) (1,0)--(1,-2) (2,0)--(2,-2);
\draw [upbdry] (0,-5)--(4,-5)--(4,0);
\draw [lowbdry] (0,-5)--(0,0)--(4,0);
\draw [ridge] (4,0)--(3,0)--++(0,-2)--++(-2,0)--++(0,-1)--++(-1,0)--(0,-5);
\draw (3.5,0) node [above right=-3pt] {\tiny 1};
\draw (2.5,-2) node [above right=-3pt] {\tiny 4};
\draw (1.5,-2) node [above right=-3pt] {\tiny 5};
\draw (0.5,-3) node [above right=-3pt] {\tiny 7};
\end{scope}

\begin{scope} [shift={(0,0)}, rotate=135]
\foreach \a/\b in {1/1, 1/2, 1/3, 2/1, 2/2, 2/3, 3/1}
\draw (\b,-\a) node {\tiny $0$}; 
\foreach \a/\b in {1/4, 2/4, 3/2, 3/3, 3/4, 4/1, 4/2,  4/3, 4/4, 5/1, 5/2, 5/3, 5/4}
\draw (\b,-\a) node {\tiny $1$}; 
\draw [upbdry] (0.5,-5.5)--++(4,0)--++(0,5);
\draw [lowbdry] (0.5,-5.5)--++(0,5)--++(4,0);
\draw [ridge] (0.5,-5.5)--++(0,2)--++(1,0)--++(0,1)--++(2,0)--++(0,2)--++(1,0);
\draw [-latex] (2,0) to node [below left] {\small $i$} (3.5,0);
\draw [-latex] (0,-2) to node [below right] {\small $j$} (0,-3.5);
\end{scope}

\begin{scope} [shift={(9,0)}, rotate=135]
\foreach \a/\b in {1/1, 1/2, 1/3, 2/1, 2/2, 2/3, 3/1}
\draw (\b,-\a) node {\tiny $0$}; 
\foreach \a/\b/\g in {1/4/1, 2/4/1, 3/2/1, 3/3/1, 3/4/1, 
   4/1/1, 4/2/1,  4/3/2, 4/4/2, 5/1/1, 5/2/2, 5/3/2, 5/4/3}
\draw (\b,-\a) node {\tiny $\g$}; 
\draw [upbdry] (0.5,-5.5)--++(4,0)--++(0,5);
\draw [lowbdry] (0.5,-5.5)--++(0,5)--++(4,0);
\draw [ridge] (0.5,-5.5)--++(0,2)--++(1,0)--++(0,1)--++(2,0)--++(0,2)--++(1,0);
\end{scope}

\end{tikzpicture}
\caption{The profile $\profile{I}$, Young diagram $\lambda_{I}$,
GT-pattern $u(I)$ and cumulative GT-pattern $v(I)$;
for $I=1457$, $(k,n)=(4,9)$.}
\label{fig:russian}
\end{figure}

\begin{remark}\label{rem:max-diag}
The profile $\profile{I}$ is also the `profile' of $M_I$,
in the sense of \cite[\S6]{JKS1},
while the Young diagram $\lambda_{I}$ below it is a picture of $\pifun M_I$, 
as in \S\ref{subsec:RTO}.
In particular, $\pifun T_{ij}$ corresponds to the rectangular Young diagram $\lambda_{ij}$, 
whose top corner is at~$(i,j)$.
The calculation at the end of the proof of \Cref{Lem:kappacump} 
amounts to the elementary observation that 
\begin{equation}\label{eq:maxdiag-cump}
  \maxdiag(\lambda_{ij}\setminus \lambda_I)=v(I)_{ij}
\end{equation}
given that we know more generally, from \cite[Lem.~7.1]{JKS2}, that
\begin{equation}\label{eq:kappa-maxdiag}
   \kappa(M_J, M_I)=\maxdiag(\lambda_J\setminus \lambda_I).
\end{equation}
Here $\maxdiag(\lambda_J\setminus \lambda_I)$, as in \cite[Def.~14.3]{RW}, is 
the maximum height of the Young diagram $\lambda_J$ above $\lambda_I$,
in Russian orientation as in \Cref{fig:russian}.
In fact, \eqref{eq:maxdiag-cump} is already known 
(indirectly, at least) from \cite[Lem.~14.2, Prop.~14.4]{RW}.
\end{remark}

Recall, from \Cref{Def:noc-etc}, the $g$-vector monoid
\[
 \gvm{T}= \{[T,M] \st M\in\CM\algC\} \subset \Grot(\CM\algA).
\]
Recall also, from \eqref{eq:wt-hat}, the isomorphism
\[
 \wthat\colon \Grot(\CM\algA) \to \ZZ\oplus \Nstar\colon [Z]\mapsto (\rkk[Z],\wt[Z])
\]
and, from \eqref{eq:wthatTM}, that $\wthat[T,M]=(\rnk{C} M,\kapvec(T,M))$.

\goodbreak\medskip
When $T=\rectCTO$, we can identify $\Nstar=\ZZ^\grid$ and we have the following.

\begin{theorem} \label{Prop:kappaGT}
$\wthat\gvm{\rectCTO} = \gtm$. %and $\wthat\gvc{\rectCTO} = \gtc$.
Consequently, $\gvm{\rectCTO}$ consists of the integral points of $\gvc{\rectCTO}$,
which is a rational polyhedral cone.
%in particular, it is saturated.
\end{theorem}

\begin{proof}
The result follows from \Cref{Lem:kappacump}, provided we can show that both monoids
are generated by their slices at rank/\level~1.

For $\gtm$, this is the (known) integer decomposition property of GT-patterns \cite{Ax},
which is preserved by the unimodular change of variables \eqref{eq:u-v-cov},
since that doesn't change the \level~$r$.
The proof is straightforward: if $u$ is a GT-pattern and $u'$ is the GT-pattern at \level~1
with the same support (i.e.~non-zero entries), then $u-u'$ is a GT-pattern
of smaller \level~and so, inductively, $u$ is a sum of \level~1 GT-patterns.

The canonical sums $u=u(I_1)+\cdots+u(I_r)$ that arise in this way correspond 
precisely to semistandard Young tableau or to standard monomials $\minor{I_1}\cdots\minor{I_r}$.
This is one way to understand the well-known fact (cf.~\cite[\S7.10]{Stan})
that GT-patterns index a basis of $\CC[\Gr(k, n)]$.

For $\gvm{T} $, note that $[T,M\oplus  N]=[T,M]+[T,N]$, so it suffices to prove that
for any module $M$, we have $[T, M]=[T, N]$,
for a \emph{standard module} $N= M_{I_1}\oplus\cdots\oplus M_{I_r}$, 
whose cluster character $\Psi_N$ is a standard monomial $\minor{I_1}\cdots\minor{I_r}$.

Thus the cluster characters $\Psi_N$ of standard modules form a basis of $\CC[\Gr(k, n)]$,
whose leading exponents 
(in the cluster chart $\Upsilon^{T}$ or the network chart $\Xi^T$) 
are $[T, N]$.
When $T=\rectCTO$, these leading exponents correspond to cumulative GT patterns,
by \Cref{Lem:kappacump}, that is, 
\[
  \wthat[T,N]=(\rnk{C} N,\kapvec(T,N)) = 
  (r, v(I_1)+\cdots+v(I_r)).
\]
These exponents are distinct, because the sum $v=v(I_1)+\cdots+v(I_r)$ is canonical, 
in the sense, explained above, that $I_1,\ldots,I_r$ (and thus $N$) can be recovered from $v$.

Hence, by \Cref{Lem:uniqueleading}, the leading exponent $[T, M]$ of $\Psi_M$,
for any $M$, must coincide with the leading exponent $[T, N]$ of $\Psi_N$, 
for some standard module $N$, as required.

For the second part, observe that $\gtm$ is the set of integral points of $\gtc$,
which is defined by finitely many integral linear inequalities \eqref{eq:CGT-pattern}.
Hence this is preserved by the linear isomorphism $\wthat$.
\end{proof}

In the next section (\Cref{rem:rat-pol-cone2}), we will use \Cref{Prop:kappaGT} 
as the base of an inductive proof that $\gvm{T}$ and $\gvc{T}$
always have these properties.

%%%% =========================================================== %%%%
\section{Mutation of $\kapvec(T, M)$}\label{Sec:11}
%%%% =========================================================== %%%%
\newcommand{\presmap}{\rho}
% =========================================================== 

For just this section, we return to considering modules in $\CM\algB$, 
for a general algebra~$\algB$ as in \Cref{Sec:3},
and the algebra $\algA=\End(T)\op$, for a 
cluster tilting object $T\moreq\bigoplus_{i\in Q_0} T_i$ in $\GP\algB$.
Recall, from \Cref{def:KTM} \emph{et seq}, the defining short exact sequence
for $\funK(N,M)$, namely
\[
  0\lra \Hom(N,M) \lraa{\epsemb{M}_*} \Hom(N,\funJ M) \lra \funK(N,M) \lra 0,
\]
for any $M,N\in\CM\algB$. 
Furthermore $\kappa(N,M)=\dim \funK(N,M)$, so that 
the $\algA$-module $\funK(T,M)$ has class/dimension vector
in $\Grot(\fd\algA)\isom \ZZ^{Q_0}$, given by 
\[
  \kapvec(T,M) = \bigl( \kappa(T_i,M) \bigr)_{i\in Q_0}.
\]

\begin{lemma} \label{Lem:factsonkappa} 
Let $0\to X\to Y\to Z\to 0$ be a short exact sequence in $\CM\algB$ and $M\in \CM\algB$.
\begin{enumerate}
\item\label{itm:fackap1}
  If $\Ext^1_\algB(M, Y)=0$ or $\Ext^1_\algB(M, X)=0$, then
\[
  \kappa(M, Y)=\kappa(M, X)+\kappa(M, Z)+\dim \Ext^1_\algB(M, X).
\]
\item\label{itm:fackap2}
  If $\Ext^1_\algB(Y, M)=0$ or $\Ext^1_\algB(Z, M)=0$, then
\[
  \kappa(Y, M)=\kappa(X, M)+\kappa(Z, M)+\dim \Ext^1_\algB(Z, M).
\]
\end{enumerate}
\end{lemma}

\begin{proof}
\itmref{itm:fackap1}
Applying $\Hom(M, -)$ to the following commutative diagram
\[ \begin{tikzpicture}[xscale=1.6, yscale=1.4]
\draw (0.2,2) node (b0) {$0$};
\draw (1,2) node (b1) {$X$};
\draw (2,2) node (b2) {$Y$};
\draw (3,2) node (b3) {$ Z$};
\draw (3.8,2) node (b4) {$0$};
\draw (0.2,1) node (a0) {$0$};
\draw (1,1) node (a1) {$\funJ X$};
\draw (2,1) node (a2) {$\funJ Y$};
\draw (3,1) node (a3) {$\funJ Z$};
\draw (3.8,1) node (a4) {$0$};
\foreach \t/\h in {b0/b1, b1/b2, b2/b3, b3/b4, a0/a1, a1/a2, a2/a3, a3/a4} 
  \draw[cdarr] (\t) to (\h);
\foreach \t/\h/\lab in {b1/a1/$\epsemb{X}$, b2/a2/$\epsemb{Y}$, b3/a3/$\epsemb{Z}$} 
  \draw[cdarr] (\t) to node [right] {\small \lab}(\h);
\end{tikzpicture} \]
and noting that $\funJ X$ is injective in $\CM\algB$, 
we obtain a commutative diagram in which the first two rows and first three columns
are exact sequences,
\[ 
\begin{tikzpicture}[xscale=3.2, yscale=1.25]
\draw (0.3,2) node (b0) {$0$};
\draw (1,2) node (b1) {$\Hom(M, X)$};
\draw (2,2) node (b2) {$\Hom(M, Y)$};
\draw (3,2) node (b3) {$ \Hom(M, Z)$};
\draw (4,2) node (b4) {$\Ext^1_\algB(M, X)$};
\draw (4.7,2) node (b5) {$0$};
\draw (0.3,1) node (a0) {$0$};
\draw (1,1) node (a1) {$\Hom(M, \funJ X)$};
\draw (2,1) node (a2) {$\Hom(M, \funJ Y)$};
\draw (3,1) node (a3) {$\Hom(M, \funJ Z)$};
\draw (4,1) node (a4) {$0$};
\draw (1,0) node (c1) {$\funK(M, X)$};
\draw (2,0) node (c2) {$\funK(M, Y)$};
\draw (3,0) node (c3) {$\funK(M, Z)$};
\draw (1,2.8) node (x1) {$0$};
\draw (2,2.8) node (x2) {$0$};
\draw (3,2.8) node (x3) {$0$};
\draw (1,-0.8) node (y1) {$0$};
\draw (2,-0.8) node (y2) {$0$};
\draw (3,-0.8) node (y3) {$0$};
\foreach \t/\h in {b0/b1, b1/b2, b3/b4, b4/b5, a0/a1, a1/a2, a2/a3, a3/a4, c1/c2, c2/c3, 
  b1/a1, a1/c1, b2/a2, a2/c2, b3/a3, b4/a4, a3/c3, x1/b1, x2/b2, x3/b3, c1/y1, c2/y2, c3/y3} 
  \draw[cdarr] (\t) to (\h);
\draw[cdarr] (b2) to node [above] {\small $f$}  (b3);
\end{tikzpicture}
\]
provided either $\Ext^1_\algB(M, Y)=0$ or $\Ext^1_\algB(M, X)=0$.
Note that the $\Hom$ spaces here are infinite dimensional,
but we can still use a spectral sequence argument to
deduce that the third row and fourth column have the same cohomology,
so, in particular,
\[
  \kappa(M, Y) - \kappa(M, X) - \kappa(M, Z) = \dim \Ext^1_\algB(M, X).
\]
More explicitly, applying the Snake Lemma to the diagram with $\Hom(M, Z)$ replaced by $\img f$,
we obtain a short exact sequence of finite dimensional spaces
\[
  0\lra \funK(M, X) \lra \funK(M, Y) \lra  \frac{\Hom(M, \funJ Z)}{\img f} \lra 0
\]
and can make another one as follows
\[
  0\lra \frac{\Hom(M, Z)}{\img f} \lra \frac{\Hom(M, \funJ Z)}{\img f} \lra  \funK(M, Z) \lra 0.
\]
The result then follows because $\Ext^1_\algB(M, X)\isom \Hom(M, Z) /\img f$.

\itmref{itm:fackap2}
We can apply $\Hom(-, M)$ and $\Hom(-, \funJ M)$ to the original short exact sequence to obtain
a commutative diagram very similar to the one in \itmref{itm:fackap1}, 
\[
\begin{tikzpicture}[xscale=3.2, yscale=1.4]
\draw (0.3,2) node (b0) {$0$};
\draw (1,2) node (b1) {$\Hom(Z, M)$};
\draw (2,2) node (b2) {$\Hom(Y, M)$};
\draw (3,2) node (b3) {$ \Hom(X, M)$};
\draw (4,2) node (b4) {$\Ext^1_\algB(Z, M)$};
\draw (4.7,2) node (b5) {$0$};
\draw (0.3,1) node (a0) {$0$};
\draw (1,1) node (a1) {$\Hom(Z, \funJ M)$};
\draw (2,1) node (a2) {$\Hom(Y, \funJ M)$};
\draw (3,1) node (a3) {$\Hom(X, \funJ M)$};
\draw (4,1) node (a4) {$0$};
\draw (1,0) node (c1) {$\funK(Z, M)$};
\draw (2,0) node (c2) {$\funK(Y, M)$};
\draw (3,0) node (c3) {$\funK(X, M)$};
\draw (1,2.8) node (x1) {$0$};
\draw (2,2.8) node (x2) {$0$};
\draw (3,2.8) node (x3) {$0$};
\draw (1,-0.8) node (y1) {$0$};
\draw (2,-0.8) node (y2) {$0$};
\draw (3,-0.8) node (y3) {$0$};
\foreach \t/\h in {b0/b1, b1/b2, b3/b4, b4/b5, a0/a1, a1/a2, a2/a3, a3/a4, c1/c2, c2/c3, 
  b1/a1, a1/c1, b2/a2, a2/c2, b3/a3, b4/a4, a3/c3, x1/b1, x2/b2, x3/b3, c1/y1, c2/y2, c3/y3} 
  \draw[cdarr] (\t) to (\h);
\draw[cdarr] (b2) to (b3); %node [above]{$f$}  (b3);
\end{tikzpicture}
\]
if either $\Ext^1_\algB(Y, M)=0$ or $\Ext^1_\algB(Z, M)=0$. 
The result then follows as in \itmref{itm:fackap1}.
\end{proof}

Given a cluster tilting object $T\moreq\bigoplus_j T_j$ in $\GP\algB$,
recall the two sequences \eqref{eq:mut-two}
for a mutable summand $T_i$
\begin{equation}\label{eq:mut-too}
\begin{aligned}
 &\ShExSeq {T_i^*} {E_i} {T_i^{\phantom*}}, \\
 &\ShExSeq {T_i^{\phantom*}} {F_i } {T_i^*}.
\end{aligned}
\end{equation}
%\begin{gather*}
%  0 \lra T_i^* \lra E_i \lra T_i^{\phantom*} \lra 0 \\ %\label{eq:mmut1}
%  0 \lra T_i^{\phantom*} \lra F_i \lra T_i^* \lra 0   %\label{eq:mmut2}
%\end{gather*}
In particular, the mutation of $T$ at $i$ is given by
\begin{equation*}\label{eq:mutT}
  \mu_i(T)=(T/T_i)\oplus T_i^*.
\end{equation*}

\begin{lemma}\label{Prop:kappafacts}
We have
\begin{enumerate}
\item\label{itm:kapfac2}
$\kappa(E_i, T_j) = \kappa(T^*_i, T_j) + \kappa(T_i, T_j) = \kappa(F_i, T_j)$, for $j\neq i$.
\item\label{itm:kapfac3}
$\kappa (E_i, T_i) = \kappa(T^*_i, T_i) + \kappa(T_i, T_i) = \kappa(F_i, T_i) - 1$.
\end{enumerate}
\end{lemma}

\begin{proof}
\itmref{itm:kapfac2} follows by applying \Cref{Lem:factsonkappa}\itmref{itm:fackap2}
to the two mutation sequences with $M=T_j$, for $j\neq i$, and using that
$\Ext^1(T_i^*, T_j) = \Ext^1(T_i, T_j) = 0$.

\itmref{itm:kapfac3} follow similarly with $M=T_i$, using that
$\Ext^1(T_i, T_i) =  \Ext^1(F_i, T_i) = 0$ and $\dim\Ext^1(T_i^*, T_i)=1$.
\end{proof}

\begin{lemma} \label{Prop:kappadv}
Writing $T'=\mu_i(T)$, we have 
\begin{enumerate}
\item\label{itm:kapdv1} 
$\kapvec (T, E_i) - [S_i] = \kapvec(T, T_i)+\kapvec(T, T_i^*) =\kapvec(T, F_i) $.
\item\label{itm:kapdv2}
$\kapvec (T', E_i) = \kapvec(T', T_i)+\kapvec(T', T_i^*) =\kapvec(T', F_i) - [S'_i]$.
\end{enumerate}
Here $S_i$ and $S'_i$ are the simple modules at $i$ %$i\in Q_0$ 
for $\End(T)\op$ and $\End(T')\op$, respectively.
\end{lemma}

\begin{proof}
\itmref{itm:kapdv1} follows by applying \Cref{Lem:factsonkappa}\itmref{itm:fackap1} to the two mutation sequences with $M=T_k$, for all $k$, and using that, for $j\neq i$,
\[
  \Ext^1(T_j ,T_i^*) = \Ext^1(T_j, T_i) = 0,
\]
while $\Ext^1(T_i, T_i) =  \Ext^1(T_i, E_i) = 0$ 
and $\dim\Ext^1(T_i, T_i^*)=1$. 

\itmref{itm:kapdv2} follows similarly, but now using that 
$\Ext^1(T_i^*, T_i^*) =  \Ext^1(T_i^*, F_i) = 0$ 
and $\dim\Ext^1(T_i^*, T_i)=1$,
for the case $M= T_i^*$.
\end{proof}

Recall, from \Cref{Lem:appseq}, that every $M\in \CM\algB$ 
has an exact $\add T$-presentation 
\begin{equation}\label{eq:addTappM}
  \ShExSeq{X}{Y}{M},
\end{equation}
that is, with $X,Y\in\add T$.

\begin{definition}\label{def:gen-pres}
An $\add T$-presentation \eqref{eq:addTappM} is called \emph{generic} if
no summand of $T$ is a summand of both $X$ and $Y$.
\end{definition}

The terminology is justified by the next result.
Recall, from \Cref{thm:subQComp-too}, 
that every $g$-vector $\gvec\in\gvm{T}$ is generic,
that is, $\gamma=\ghminvec{\compC}{T}$ for some (unique) component $\compC$ of $\CM\algC$.

\begin{lemma}\label{lem:gen=gen}
%Hence we can choose such an $M$ with $[T,M]=\ghminvec{\compC}{T}$
For any $\gvec=\ghminvec{\compC}{T}\in\gvm{T}$, 
the set of $M$ in $\compC$ with $[T,M]=\gvec$ and with a generic $\add T$-presentation
contains a non-empty open subset of $\compC$.
\end{lemma}

\begin{proof}
%Note first that, for any $T$, those $M$ in $\compC$ with a generic $\add T$-presentation,
%as in \Cref{def:gen-pres}, form an open subset of $\compC$.
Since an $\add T$-presentation can be chosen without projective-injectives in the left-hand term,
this amounts to the fact that, if $\compC=\compC(r, \hC)$ as in \eqref{eq:Crc},
then there is an open set of $\pifun M$ in $\hC$ with a 
generic $\add\olT$-presentation.
This is true because, as in the proof of \Cref{lem:cx}, all $\pifun M\in\hCmin$
are cokernels of injective maps $\presmap\colon T''\to T'$ for fixed $T'',T'\in\add\olT$.
But then those $\presmap$ which restrict to an isomorphism on the common summands, 
which thus split off, form an open subset,
giving an open subset of $W$, as in \eqref{eq:Wset},
whose projection onto $\hCmin$ must then contain an open subset of $\hCmin$,
as required.
\end{proof}

\begin{proposition}\label{Thm:kappamain} 
Let $T$ be a cluster tilting object in $\GP\algB$.
%and write 
%\begin{equation}\label{eq:mutT}
%  \mu_i(T)=(T/T_i)\oplus T_i^*.
%\end{equation}
If $M\in \CM\algB$ has a generic $\add T$-presentation, then, 
in the notation of \eqref{eq:mut-too},
\begin{equation}\label{eq:kap-mut}
  \kappa (T_i^*, M) = \min \{\kappa (E_i, M), \;\kappa (F_i, M)\} - \kappa(T_i, M).
\end{equation}
\end{proposition}

\begin{proof}
Consider a generic $\add T$-presentation of $M$ as in \eqref{eq:addTappM}.
For each summand $T_i$, we can suppose that either 
\begin{itemize}
\item[(a)] $T_i$ is not a summand of $X$, so $\Ext^1(T^*_i, X)=0$, or 
\item[(b)] $T_i$ is not a summand of $Y$, so $\Ext^1(T^*_i, Y)=0$. 
\end{itemize}
In both cases, since $T_i,E_i,F_i,X \in\add T$, we have
\[
  \Ext^1(T_i, X) =\Ext^1(E_i, X)=\Ext^1(F_i, X)=0. 
\]
In case (a), \Cref{Lem:factsonkappa}\itmref{itm:fackap1} then implies that
\begin{align}
\kappa(T_i, M)+\kappa(T_i^*, M)
&= \kappa(T_i, Y)+\kappa(T^*_i, Y)-\kappa(T_i, X)-\kappa(T^*_i, X),\label{eqn:0} \\
\kappa(E_i, M)
&= \kappa(E_i, Y)-\kappa(E_i, X),  \label{eqn:1} \\
\kappa(F_i, M)
&= \kappa(F_i, Y)-\kappa(F_i, X). \label{eqn:2}
\end{align}
By \Cref{Prop:kappafacts}, we have
\begin{align}
\kappa(T_i, X) +\kappa(T_i^*, X)
 &= \kappa(E_i, X)
 = \kappa(F_i, X) \label{Eqn:1} \\
\kappa(T_i, Y) +\kappa(T_i^*, Y)
 &= \kappa(E_i, Y)
 = \kappa(F_i, Y) -s, \label{Eqn:2}
\end{align}
where $s$ is the multiplicity of $T_i$ as a summand of $Y$.
Substituting \eqref{Eqn:1} and \eqref{Eqn:2} into \eqref{eqn:0}, 
then using \eqref{eqn:1} and \eqref{eqn:2}, we obtain the required result, that is,
\begin{align*}
  \kappa(T_i, M)+\kappa(T_i^*, M) &=  \kappa(E_i, M) = \kappa(F_i, M)-s \\
  &= \min\{\kappa(E_i, M) ,\, \kappa(F_i, M)\}.
\end{align*}
In case (b), \Cref{Lem:factsonkappa}\itmref{itm:fackap1} now implies that
\[
\kappa(T_i, M)+\kappa(T_i^*, M)
=  \kappa(T_i, Y)+ \kappa(T^*_i, Y) -\kappa(T_i, X) -\kappa(T_i^*, X)-t
\]
where $t=\dim\Ext^1(T^*_i, X)$, which is also the multiplicity of $T_i$ as a summand of $X$.
The other equations are then the same, except that
\begin{align*}
\kappa(T_i, X) +\kappa(T_i^*, X)
 &= \kappa(F_i, X)-t \\
\kappa(T_i, Y) +\kappa(T_i^*, Y)
 &= \kappa(F_i, Y).
\end{align*}
Thus a similar calculation to (a) gives
\begin{align*}
  \kappa(T_i, M)+\kappa(T_i^*, M)
  &= \kappa(E_i, M)-t = \kappa(F_i, M) \\
  &= \min\{\kappa(E_i, M),\, \kappa(F_i, M)\}
\end{align*}
as required.
\end{proof}

\begin{remark}\label{rem:non-generic}
If $T_i$ is a summand of both $X$ and $Y$ in \eqref{eq:addTappM},
then neither $\Ext^1(T_i^*,X)$ nor $\Ext^1(T_i^*,Y)$ vanish and 
\Cref{Lem:factsonkappa}\itmref{itm:fackap1} can't be used to get a formula for $\kappa(T_i^*, M)$.
A more careful analysis of the proof shows that, 
for any presentation $\eqref{eq:addTappM}$, we have
\begin{align*}
  \kappa(T_i, M)+\kappa(T_i^*, M)
  &= \kappa(E_i, M)-\dim\ker\presmap_* \\
  &= \kappa(F_i, M)-\dim\cok\presmap_*,
\end{align*}
where $\presmap_*\colon \Ext^1(T_i^*,X)\to\Ext^1(T_i^*,Y)$ is induced by the presentation map
$\presmap\colon X\to Y$.
Thus \eqref{eq:kap-mut} holds for a particular $M$,
precisely when $\presmap_*$ has maximal rank, i.e.~$\ker\presmap_*$ or $\cok\presmap_*$ is trivial.
In fact, $\ker\presmap_*$ can be identified with the cokernel of the composition map
$\sHom(T_i^*,T)\otimes \sHom(T,M)\to \sHom(T_i^*,M)$,
showing directly that it depends only on $M$ and not on the choice of $\presmap$.
Similarly, $\cok\presmap_*$ can be identified with the image of composition 
$\sHom(T_i^*,\CTOcosyz T)\otimes \sHom(\CTOcosyz T,\coSyz M)
  \to \sHom(T_i^*,\coSyz M) \isom \Ext^1(T_i^*,M)$.
\end{remark}

\goodbreak
We can interpret \Cref{Thm:kappamain} in the language of tropical mutation, as follows.

%\begin{definition}\label{def:tropAmut}
%(Tropical $\cluA$-mutation, cf.~\cite[Def.~11.8]{RW})
%For any quiver $Q$ and any (mutable) vertex $i\in Q_0$,
%define the (bijective) map 
%\begin{equation}\label{eq:tropAmut}
%\tropAmut{Q}{i}\colon \ZZ^{Q_0} \to \ZZ^{Q_0} \colon (v_j) \mapsto (v'_j)
%\end{equation}
%by $v'_j=v_j$, if $j\neq i$, while $v'_i = \min \{\sum_{j\from i} v_j ,\, \sum_{j\to i} v_j\}-v_i$. 
%\end{definition}

\begin{definition}\label{def:tropAmut}
(Tropical $\cluA$-mutation, cf.~\cite[Def.~11.8]{RW})
Let $T$ be a cluster tiling object in $\GP\algB$, 
with $\algA= \End(T)\op$ and $Q$ the Gabriel quiver of $\algA$
(see \Cref{rem:gab-quiv}).
When $T'=\mu_i(T)$, with $\algA'= \End(T')\op$,
%For any quiver $Q$ and any (mutable) vertex $i\in Q_0$,
define the (bijective) map 
\begin{equation}\label{eq:tropAmut}
\tropAmut{Q}{i}\colon \Grot(\fd\algA) \to \Grot(\fd\algA') \colon (v_j) \mapsto (v'_j)
\end{equation}
by setting $v'_j=v_j$, for each $j\in Q_0\setminus\{i\}$, while 
%$v'_i = \min \{\sum_{j\from i} v_j ,\, \sum_{j\to i} v_j\}-v_i$,
\[ v'_i = \min \Bigl\{\sum_{j\from i} v_j ,\, \sum_{j\to i} v_j \Bigr\} - v_i, \]
where the arrows in the sums are taken from $Q_1$.
\end{definition}
Note that, as $T$ and $T'$ are related by a single mutation,
both $\Grot(\fd\algA)$ and $\Grot(\fd\algA')$ are identified
with $\ZZ^{Q_0}$, even though $Q$ is not the Gabriel quiver of $\algA'$.

%Recall that the summands of a cluster tilting object $T$
%are canonically identified with the vertices $Q_0$ of the Gabriel quiver of $\algA= \End(T)\op$
%and then $K(\fd\algA)$ is identified with $\ZZ^{Q_0}$.
%Furthermore, if $T'=\mu_i(T)$ and $\algA'= \End(T')\op$,
%then there is an obvious bijection between the summands of $T$ and $T'$, so
%$K(\fd\algA')$ is also identified with $\ZZ^{Q_0}$, 
%even though the Gabriel quiver of $\algA'$ is different.

\begin{theorem}\label{Cor:mutkappa}
Let $T$ be a cluster tilting object in $\GP\algB$. 
If $M\in \CM\algB$ has a generic $\add T$-presentation, %(\Cref{def:gen-pres}), 
then 
%Let $T$ and $M$ be as in \Cref{Thm:kappamain}
%and $Q$ be the Gabriel quiver of $A= \End(T)\op$.
%Then 
\begin{equation}\label{eq:mutkapvec}
\kapvec (\mu_i(T), M)=\tropAmut{Q}{i}\, \kapvec (T, M).
\end{equation}
%once we make the identifications $\Grot(\fd\algA)=\ZZ^{Q_0}=\Grot(\fd\algA')$,
%as described above.
\end{theorem}

\begin{proof}
Use \Cref{Thm:kappamain} and observe from \eqref{eq:EiFi} that
$\kappa (E_i, M) = \sum_{j\from i} \kappa (T_j, M)$
and $\kappa (F_i, M) = \sum_{j\to i} \kappa (T_j, M)$.
\end{proof}

Note that $\tropAmut{Q}{i}$ always preserves $\Nstar\subset \ZZ^{Q_0}$,
as $\vstar$ is not mutable,
and hence induces a piecewise linear map 
\begin{equation}\label{eq:tropAmuthat}
  \tropAmuthat{Q}{i}\colon \Grot(\CM\algA)\to \Grot(\CM\algA')
\end{equation}
via the isomorphism $\wthat\colon \Grot(\CM\algA') \to \ZZ\oplus \Nstar$
(and the same for $A$), 
defined in \eqref{eq:wt-hat}. 
More precisely,
\[
\wthat \tropAmuthat{Q}{i} [Z] = (\rkk[Z], \tropAmut{Q}{i} \wt[Z]).
\]
%\danger{
%that is, 
%\begin{gather*}
% \tropAmuthat{Q}{i} [Z] 
% = \betahat (\rkk[Z], \tropAmut{Q}{i} \wt[Z])
% = \rkk[Z]\cdot [T,\modJ] - \beta\bigl( \tropAmut{Q}{i} \wt[Z] \bigr) \\
% = [Z] + \beta \bigl( \wt[Z] - \tropAmut{Q}{i} \wt[Z] \bigr)
% =  [Z] + \lambda_i \beta [S_i],
%\end{gather*}
%where $\lambda_i = 2w_i - \min \Bigl\{\sum_{j\from i} w_j ,\, \sum_{j\to i} w_j \Bigr\}$,
%with $w_i=\wt_i [Z]$.
%Cf. \cite[Conj.~7.12]{FZ4}.
%}
Then \Cref{Cor:mutkappa} immediately gives 
\begin{equation}\label{eq:trop-hat-on-gvec}
  \tropAmuthat{Q}{i} ([T, M]) =[\mu_i(T), M], 
\end{equation}
when $M$ has a generic $\add T$-presentation.
This leads to the following.

\begin{corollary}\label{rem:rat-pol-cone2}
For all cluster tilting objects $T$ in $\CM\algC$ and all $i\in Q_0$,
\begin{equation}\label{eq:mut-mon}
  \gvm{\mu_i(T)}=\tropAmuthat{Q}{i}\, \gvm{T}.
\end{equation}
Consequently, if $T$ is reachable, then $\gvm{T}$ consists of the integral points of $\gvc{T}$,
which is a rational polyhedral cone.
In particular, $\gvm{T}$ is saturated and finitely generated.
\end{corollary}

\begin{proof}
For any $\gvec\in \gvm{T}$,
choose $M$ with a generic $\add T$-presentation and $[T,M]=\gvec$, by \Cref{lem:gen=gen}.
Then \eqref{eq:trop-hat-on-gvec} implies that $\tropAmuthat{Q}{i}\, \gvm{T} \subset \gvm{\mu_i(T)}$.
For the opposite containment, since mutation is an involution, we can replace $T$ by $\mu_i(T)$
and use the correspondiing $\tropAmuthat{Q}{i}$ map.

Now, when $T=\rectCTO$, we know from \Cref{Prop:kappaGT} that $\gvm{T}$ and $\gvc{T}$
have the claimed properties and these properties are preserved by $\tropAmuthat{Q}{i}$,
because it is a piecewise linear map of lattices.
Hence they hold for all reachable $T$ by induction.
The final statement follows by \Cref{rem:MonCon}.
\end{proof}

In \S\ref{Sec:13} (Theorems~\ref{thm:trop-GT}, \ref{thm:coneGV-SP}), 
we will find explicit inequalities for $\gvc{T}$,
using the mirror symmetry viewpoint of  \cite{RW}.

\begin{remark}\label{rem:RWanalogue}
\Cref{Cor:mutkappa} and \Cref{rem:rat-pol-cone2} play a role here which is similar to
\cite[Lemma~16.17]{RW} for studying $\nom{G}$
(or equivalently its slices $\val{G}(L_r)$) as in \S\ref{subsec:12.3}.
Their lemma shows that the leading exponents of theta functions transform 
under tropical $\cluA$-mutation, although this is not true for all functions.
\end{remark}

\begin{remark}\label{rem:moretropAmut}
By \Cref{thm:subQComp-too} (last part),
if $M\in \compC$ for just one component~$\compC$,
then $[T, M]=\ghminvec{\compC}{T}$, for any~$T$ 
(recall that $\ghminvec{\compC}{T}$ depends implicitly on~$T$ also).
But then \Cref{Cor:mutkappa} implies that %for $T'=\mu_i(T)$,
\[
 \kapvec (\mu_i(T), M)=\tropAmut{Q}{i}\, \kapvec (T, M),
 \quad\text{or}\quad
 [\mu_i(T),M] = \tropAmuthat{Q}{i} [T, M].
\]
\end{remark}

\begin{remark}\label{rem:tropAmuthat}
One can check that, under the identification $\Grot(\CM\algA)=\Grot(\add T)$, 
the map $\tropAmuthat{Q}{i}$ of \eqref{eq:tropAmuthat} is identified with 
the map $\Grot(\add T)\to\Grot(\add T')$ implicitly defined in \cite[Thm~3.1]{DK}.
As noted in \cite[Thm~5.5(c)]{FK}, this is also the same as the transformation rule for $g$-vectors 
conjectured in \cite[Conj.~7.12]{FZ4},
consistent with \eqref{eq:trop-hat-on-gvec}.
\end{remark}

%%%% =========================================================== %%%%
\section{Mutation of simple modules and $\cluX$-mutation}\label{Sec:new10} 
%%%% =========================================================== %%%%. 
\newcommand{\Fint}[1]{F_{#1, \mathrm{int}}}
\newcommand{\Ebar}{\overline{E}}
\newcommand{\Kbar}{\overline{K}}
\newcommand{\fhat}{\widehat{f}}

We return to the context where $\algB=\algC$.
Let $T=\bigoplus_{i\in Q_0} T_i$ be a cluster tilting object in $\CM\algC$
and $T'=\mu_j(T)=(T/T_j)\oplus T_j^*$.
Let $Q$ and $Q'$ be the Gabriel quivers of $A=\End(T)\op$ and $A'=\End(T')\op$, respectively. 
Note that $Q$ and $Q'$ have no loops or two cycles and $Q'=\mu_j(Q)$.

% ===========================================================
\subsection{More on projective resolutions}\label{subsec:projres2}
% ===========================================================

Recall, from \eqref{Eq:projres3}, the minimal $\add T$-approximation of $\rrad T_i$
when $T_i$ is projective, that is, $i\in Q_0$ is a boundary vertex,
\begin{equation*}
  0\lra K_i\lra E_i \lraa{f} \rrad T_i\lra 0.
\end{equation*}
Define
\begin{equation}\label{eq:Fint}
  \Fint{i}=\bigoplus_{\substack{\ell\to i\\ \mathrm{int}}} T_\ell,
\end{equation}
summing over all the \emph{interior} arrows in $Q$ with head at $i$ (cf. \eqref{eq:EiFi}).
Since $j$ is an interior vertex, the multiplicity of $T_j$ in $\Fint{i}$ is 
the number of arrows $b_{ji}$ from $j$ to $i$.

\begin{proposition} \label{prop:minappradP}
Suppose that $\rrad T_i$ has the following minimal $\add T$-approximation, 
\begin{equation}\label{eq:minaddTapp}
  0\lra \Fint{i} \lraa{g} E_i \lraa{f} \rrad T_i\lra 0.
\end{equation}
Then the same  holds for its minimal $\add T'$-approximation, that is, 
\begin{equation}\label{eq:minaddT'app}
  0 \lra \Fint{i}' \lraa{g'} E'_i \lraa{f'} \rrad T_i\lra 0,
\end{equation}
where $\Fint{i}'$ is defined with respect to  $Q'$.
\end{proposition}

\begin{proof} 
If $b_{ij}=b_{ji}=0$, then \eqref{eq:minaddTapp} is also a minimal $\add T'$-approximation
and is equal to \eqref{eq:minaddT'app}.
The proof is now divided into two cases.

First, suppose that $b_{ji}>0$. 
By assumption, $K_i=\Kbar_i\oplus T^{b_{ji}}_j$,
where $\Kbar_i$ has no summand  isomorphic to $T_j$. 
Let 
\[
  h = \idmap_{\Kbar_i} \oplus \; l_j^{\oplus b_{ji}}
  \colon \Kbar_i\oplus T^{b_{ji}}_j\to \Kbar_i\oplus F^{b_{ji}}_j,
\]
where $l_j\colon T_j\to F_j$ is the left map in the mutation sequence (see \eqref{eq:mut-two}) 
starting from~$T_j$.
Constructing the pushout of $g$ and $h$ gives the following commutative diagram 
with exact rows and columns. 
\[
\begin{tikzpicture}[xscale=2.2, yscale=1.5, baseline=(bb.base)] 
\coordinate (bb) at (2,2);
\pgfmathsetmacro{\xeps}{0.2}
\pgfmathsetmacro{\yeps}{0.2}
\draw (1,3-\yeps) node (e1) {$0$};
\draw (2,3-\yeps) node (e2) {$0$};
\draw (0+\xeps, 2) node (a0) {$0$};
\draw (1,2) node (a1) {$\Kbar_i\oplus T^{b_{ji}}_j$};
\draw (2,2) node (a2) {$E_i$};
\draw (3,2) node (a3) {$\rrad T_i$};
\draw (4-\xeps, 2) node (a4) {$0$};
\draw (0+\xeps,1) node (b0) {$0$};
\draw (1,1) node (b1) {$\Kbar_i\oplus F^{b_{ji}}_j$};
\draw (2,1) node (b2) {$X$};
\draw (3,1) node (b3) {$ \rrad T_i$};
\draw (4-\xeps,1) node (b4) {$0$};
\draw (1,0) node (c1) {$(T_j^*)^{b_{ji}}$};
\draw (2,0) node (c2) {$(T_j^*)^{b_{ji}}$};
\draw (1,-1+\yeps) node (d1) {$0$};
\draw (2,-1+\yeps) node (d2) {$0$};
\foreach \t/\h in {a0/a1, a3/a4, b0/b1, b1/b2, b3/b4, 
  a2/b2, b1/c1, b2/c2, c1/d1, c2/d2, e1/a1, e2/a2}
  \draw[cdarr] (\t) to (\h);
\foreach \t/\h in {c1/c2, a3/b3}
  \draw[equals] (\t) to (\h);
\draw[cdarr] (b2) to node [auto] {\small $\fhat$} (b3);
\draw[cdarr] (a2) to node [auto] {\small $f$} (a3);
\draw[cdarr] (a1) to node [auto] {\small $g$} (a2);
\draw[cdarr] (a1) to node [auto] {\small $h$} (b1);
\end{tikzpicture}
\]
By assumption, $T_j$ is not a summand of $E_i$, 
so $\Ext^1(T_j^*, E_i)=0$. 
Hence
\[
  X=E_i\oplus (T_j^*)^{b_{ji}}\in \add T'.
\]
Observe also that 
\[
   \Ext^1(T', \Kbar_i\oplus F_j)=0.
\]
Therefore the map $\fhat\colon X\to \rrad T_i$ is  an $\add T'$-approximation, 
which is usually not minimal. 
The common summand of $\Kbar_i\oplus F^{b_{ji}}_j$ and $X$ is the 
common summand of $F^{b_{ji}}_j$ and $E_i$, 
which is $M=\bigoplus_p T_p$, summing over all the paths $i\to p\to j$. 
Each such  $T_p$ corresponds to a 2-cycle deleted in the process of mutating $Q$ at $j$.
Splitting off $M$ gives the following minimal $\add T'$-approximation of $\rrad T_i$, 
\begin{equation*}\label{eq:minapp}
  0\lra K_i' \lra E_i' \lra \rrad T_ i\lra  0.
\end{equation*}
In particular, $K_i'=\Fint{i}'$.

As the second case, 
suppose that $b_{ij}>0$. 
Then  $T^{b_{ij}}_j$ is a summand of $E_i$. 
Write $E_i=\Ebar_i\oplus T_j^{b_{ij}}$ and let
\[
  h = \idmap_{\Ebar_i}\oplus r_j^{\oplus b_{ij}} 
  \colon \Ebar_i\oplus  E^{b_{ij}}_j\to \Ebar_i\oplus  T^{b_{ij}}_j ,
\] 
where $r_j$ is the right map in the mutation sequence (see \eqref{eq:mut-two}) ending at $T_j$.
Constructing the pullback of $g$ and $h$ gives the following commutative diagram 
with exact rows and columns. 
\[
\begin{tikzpicture}[xscale=2.2, yscale=1.5, baseline=(bb.base)] 
\coordinate (bb) at (2,2);
\pgfmathsetmacro{\xeps}{0.3}
\pgfmathsetmacro{\yeps}{0.2}
\draw (1,3-\yeps) node (e1) {$0$};
\draw (2,3-\yeps) node (e2) {$0$};
\draw (1,2) node (a1) {$(T^*_j)^{b_{ij}}$};
\draw (2,2) node (a2) {$(T^*_j)^{b_{ij}}$};
\draw (0+\xeps,1) node (b0) {$0$};
\draw (1,1) node (b1) {$X$};
\draw (2,1) node (b2) {$\Ebar_i\oplus E_j^{b_{ij}}$};
\draw (3,1) node (b3) {$\rrad T_i$};
\draw (4-\xeps,1) node (b4) {$0$};
\draw (0+\xeps,0) node (c0) {$0$};
\draw (1,0) node (c1) {$K_i$};
\draw (2,0) node (c2) {$\Ebar_i\oplus T^{b_{ij}}_j$};
\draw (3,0) node (c3) {$\rrad T_i$};
\draw (4-\xeps,0) node (c4) {$0$};
\draw (1,-1+\yeps) node (d1) {$0$};
\draw (2,-1+\yeps) node (d2) {$0$};
\foreach \t/\h in {b0/b1, b1/b2, b3/b4, c0/c1, c3/c4,
  a1/b1, a2/b2, b1/c1, c1/d1, c2/d2, e1/a1, e2/a2}
  \draw[cdarr] (\t) to (\h);
\foreach \t/\h in {a1/a2, b3/c3}
  \draw[equals] (\t) to (\h);
\draw[cdarr] (b2) to node [auto]{\small $\widehat{f}$} (b3);
\draw[cdarr] (c2) to node [auto]{\small $f$} (c3);
\draw[cdarr] (c1) to node [auto]{\small $g$} (c2);
\draw[cdarr] (b2) to node [auto]{\small $h$} (c2);
\end{tikzpicture}
\]
Similar to the first case,  $X=(T_j^*)^{b_{ij}}\oplus K_i$ and $\fhat$ is an $\add T'$-approximation. 
The common summand of $X$ and $\Ebar_i\oplus E_j^{b_{ij}}$
is $M=\bigoplus_p T_p$, summing over all the paths $j\to p\to i$, and each such $T_p$ corresponds to 
a 2-cycle deleted in the process of mutating $Q$ at $j$. 
Splitting off $M$ gives the minimal $\add T'$-approximation of $\rrad T_i$ as claimed,
\begin{equation*}\label{eq:minapp2}
0\lra K_i' \lra E_i' \lra \rrad T_i \lra  0.
\end{equation*}
This completes the proof.
\end{proof}

\begin{proposition}\label{thm:projresolb}
Let $T$ be a reachable cluster tilting object in $\CM\algC$. 
Then any boundary simple $\algA$-module $S_i$ has the following minimal projective resolution,
\begin{equation}\label{eq:projres5}
  0 \lra \Hom(T,\Fint{i}) \lra\Hom (T, E_i) \lra \Hom(T,T_i) \lra S_i \lra 0.
\end{equation}
\end{proposition}

\begin{proof}
We show that $\rrad T_i$ has a minimal $\add T$-approximation as in \eqref{eq:minaddTapp}
and so, by applying $\Hom(T, -)$, we obtain the given projective resolution, 
as in \Cref{Prop:projres}\itmref{itm:pr2}.

By \Cref{prop:minappradP}, it suffices to show that this is true for the rectangles cluster tilting object. 
Indeed, we have 
\[ 
  T_{i}=M_{[i+1, i+k]}, \quad
  \rrad T_i=M_{\{i\}\cup [i+2, i+k]},
\]
and  
\[
  E_i
  = \left\{ \begin{array}{ll}
  M_{[i, i+k-1]}\oplus M_{\{1\}\cup [i+2, i+k]},  & \text{when $i\leq n-k$}, \medskip \\
  M_{[1, i-n+k] \cup [i, n-1]} \oplus M_{[i+2, i+k]}, & \text{otherwise}.
\end{array} \right.
\]
Therefore, 
\[
 K_i
  = \left\{ \begin{array}{ll}
  M_{\{1\}\cup [i+1, i+k-1]},  & \text{when $i\leq n-k$}, \medskip \\
  M_{[1, i-n+k+1] \cup [i+1, n-1]}, & \text{otherwise}.
  \end{array} \right.
\]
So, in either case, we have $K_i = \Fint{i}$, as required. 
\end{proof}

% ======================================================================
\subsection{$\cluX$-mutation and flow polynomials}
% ======================================================================

Recall, from \Cref{cor:partfun-map}, that, for any cluster tilting object~$T$ in $\CM\algC$, we have a map 
\[ 
  \Xi^T\colon \CC[\Gr(k, n)]\to \CC[\Grot(\CM\algA)]\colon \cluschar{M}\mapsto \ptfn^T_M.
\]
We wish to regard  $\Xi^T$ as a generalised `homogeneous network chart' (cf. \Cref{rem:net-chart}),
so that the generalised `inhomogeneous network chart' is then $\pbfun{\wt}{\Xi^T}=\CC[\wt]\compo \Xi^T$,
in the notation of \Cref{rem:new-clucha}.
To justify this, we will show that $\pbfun{\wt}{\Xi^T}$ transforms under cluster $\cluX$-mutation,
when $T$ mutates to $T'=\mu_j(T)$.

To this end, we take $\cluX$-mutation to be the birational coordinate transformation 
\[ 
  \xmut{j}\colon \CC[\Grot(\fd\algA')]\to \CC(\Grot(\fd\algA)),
\] 
where $\CC(\Grot(\fd\algA))$ is the field of fractions of $\CC[\Grot(\fd\algA)]$,
given by the following formula
\[
\xmut{j}(\yvar^{s'_i}) =\left\{ \begin{tabular}{ll} $\yvar^{-s_i}$, & {if } $i=j$; \\
$\yvar^{s_i} (1+ \yvar^{s_j})^{b_{ij}}$, & if  $b_{ij}>0$;\\
$\yvar^{s_i} (1+ \yvar^{-s_j})^{b_{ij}}$, & if  $b_{ij}<0$;\\
$\yvar^{s_i}$, & otherwise,      \end{tabular}\right. 
\]
where, as in \S\ref{subsec:projres2}, $b_{ji}$ is the number of arrows from $j$ to $i$ in $Q$
(cf. \cite[Def~6.14]{RW}).
%This is the formula for $\cluX$-mutation (cf. \cite[Def~6.14]{RW}).
Here $\{s_i=[S_i]\st i\in Q_0\}$ is the standard basis of $\Grot(\fd\algA)$
and $\{s'_i=[S'_i]\st i\in Q_0\}$ is the standard basis of $\Grot(\fd\algA')$.

\begin{theorem}\label{thm:Xflow} 
Let $T$ be a cluster tilting object in $\CM\algC$ and $T'=\mu_j(T)$.
Then
\[ \xmut{j}\compo \pbfun{\wt}{\Xi^{T'}} = \pbfun{\wt}{\Xi^T} \]
%(see \Cref{rem:new-clucha} and \Cref{rem:net-chart})
and thus,  for any $M\in \CM C$, 
\[\xmut{j}\bigl( \fp^{T'}_M \bigr) = \fp^{T}_M,\]
that is, (generalised) flow polynomials are related by $\cluX$-mutation.
\end{theorem}

\begin{proof}
The second part follows from the first by \Cref{cor:partfun-map},
which implies that
\[ \pbfun{\wt}{\Xi^T}(\Psi_M)=\pbfun{\wt}\ptfn^{T}_M=\fp^{T}_M, \]
for any $M\in\CM\algC$.

The first part follows from the existence of the following commutative diagram
of algebra homomorphisms
\begin{equation}\label{eq:diag}
\begin{tikzpicture}[xscale=2.2, yscale=1.2, baseline=(bb.base)]
\coordinate (bb) at (0,2);
 \draw (-0.5,2) node (a) {$\CC[\Gr(k, n)]$};
 \draw (1,1) node (b2) {${\CC(\Grot(\CM\algA))}$};
 \draw (3,1) node (b3) {${\CC(\Grot(\fd\algA))},$};
 \draw (1,3) node (c2) {$\CC[\Grot(\CM\algA')]$};
 \draw (3,3) node (c3) {${\CC[\Grot(\fd\algA')]}$};
 \draw[cdarr] (b2) to node[above] {$\CC[\wt]$}  (b3);
 \draw[cdarr] (c2) to node[above] {$\CC[\wt]$} (c3);
 \draw[cdarr] (c3) to node[left] {$\xmut{j}$} (b3);
 \draw[cdarr] (c2) to node[right] {$\amut{j}$} (b2);
 \draw[cdarr] (a) to node[below left=-2pt] {$\Xi^{T}$} (b2); 
 \draw[cdarr] (a) to node[above left=-3pt] {$\Xi^{T'}$} (c2); 
\end{tikzpicture}
\end{equation}
where $\CC(\Grot(\CM\algA))$ is the field of fractions of $\CC[\Grot(\CM\algA)]$
and $\amut{j}$ is defined, initially, so that 
the left-hand triangle in \eqref{eq:diag} commutes. 

This can be achieved, and in a unique way, 
because $\CC[\Grot(\CM\algA')]$ has algebraically independent `generators'
in the image of~$\Xi^{T'}$,
namely $\xvar^{[T', \CTOcosyz T'_i]}=\ptfn^{T'}_{\CTOcosyz T'_i}$,
in the notation of \Cref{rem:SigmaT}.
But \Cref{cor:partfun-map} means that the commuting of the triangle requires
that $\amut{j}\bigl( \ptfn^{T'}_M \bigr) = \ptfn^{T}_M$, for all $M\in\CM\algC$.
Hence we can and must define $\amut{j}$ as follows.
\begin{equation}\label{eq:mapalp}
  \amut{j} \bigl( \xvar^{[T', \CTOcosyz T'_i]} \bigr) = \ptfn^{T}_{\CTOcosyz T'_i}
  = \begin{cases} %\left\{ \begin{array}{ll} 
  \xvar^{-[T, \Sigma T_j]} \bigl( \xvar^{[T, \Sigma E_j]} + \xvar^{[T, \Sigma F_j]} \bigr), 
  & \text{if $i=j$}, \medskip \\
  \xvar^{[T, \CTOcosyz T_i]}, 
  & \text{if $i\neq j$,}
  \end{cases} %\end{array} \right.
\end{equation}
where $\Sigma E_j$ and $\Sigma F_j$ are as in \eqref{eq:two-squares},
and noting that $\CTOcosyz T'_j= (\Sigma T_j)^*$, by  \eqref{eq:Sigma-mut}.
Note also that the $\xvar^{[T', \CTOcosyz T'_i]}$ are strictly not all the generators 
of $\CC[\Grot(\CM\algA')]$ as an algebra.
Implicitly, we also set 
\[ 
  \amut{j} \bigl( \xvar^{-[T', \CTOcosyz T'_i]} \bigr) 
= \amut{j} \bigl( \xvar^{[T', \CTOcosyz T'_i]} \bigr)^{-1}.
\]
This explains why the codomain of $\amut{j}$ must be the field of fractions.

Thus, to prove the theorem, we must show that the right-hand square in \eqref{eq:diag} commutes,
which we will do below in \Cref{lem:mapalp2} and \Cref{prop:sq-commutes}.
\end{proof}

First we give an equivalent definition of $\alpha_j$ using
the monomials of the standard basis elements $[T', T'_i]$ of $\Grot(\CM\algA')$.

\begin{lemma}\label{lem:mapalp2}
The map $\alpha_j$ in \eqref{eq:mapalp} can also be defined as follows.
\begin{equation}\label{eq:mapalp2}
  \amut{j} \bigl( \xvar^{[T', T'_i]} \bigr)
  = \begin{cases} 
  \xvar^{-[T,  T_j]} \bigl( \xvar^{-[T, E_j]} + \xvar^{-[T, F_j]} \bigr)^{-1}, & \text{if $i=j$}, \medskip \\
  \xvar^{[T,  T_i]}, & \text{if $i\neq j$, }
  \end{cases}
\end{equation}
where $E_j$ and $F_j$ are as in \eqref{eq:mut-two}.
\end{lemma}

\begin{proof}
By \eqref{eq:TSigmaT} for boundary $i$ (necessarily $\neq j$), 
we have $[T, T_i]=[T, \CTOcosyz T_i]$ (and similar for $T'$), 
so \eqref{eq:mapalp} gives
\[
  \amut{j}(\xvar^{[T', T'_i]})
  = \amut{j}(\xvar^{[T', \CTOcosyz T'_i]}) 
  = \xvar^{[T, \CTOcosyz T_i]}
  = \xvar^{[T, T_i]}.
\]

On the other hand, by \eqref{eq:TSigmaT} for interior $i$, 
we have $[T, T_i]=[T, \procov{\Sigma T_i} ]-[T, \Sigma T_i]$.
Hence, when $i\neq j$, so $T'_i=T_i$, \eqref{eq:mapalp} gives
\[
  \amut{j}(\xvar^{[T', T'_i]})
  = \amut{j}(\xvar^{[T', \procov{\Sigma T'_i}]}) \amut{j}(\xvar^{-[T', \Sigma T'_i]})
  = \xvar^{[T, \procov{\Sigma T_i}]} \xvar^{-[T, \Sigma T_i]}
  =\xvar^{[T, T_i]},
\]
because all summands of $\procov{\Sigma T'_i}=\procov{\Sigma T_i}$ are projective, 
that is, boundary summands.

Now consider the case $i=j$. 
Then \eqref{eq:mapalp} and \eqref{eq:two-squares} give
\begin{align*}
\amut{j} \bigl( \xvar^{[T', T'_j]} \bigr)
  &= \amut{j} \bigl( \xvar^{[T', \procov{\Sigma T'_j}]} \bigr) \amut{j} \bigl( \xvar^{-[T', \Sigma T'_j]} \bigr) \\
  &= \xvar^{ [T, \procov{\Sigma T'_j}]} \xvar^{[T, \Sigma T_j]} 
       \bigl( \xvar^{[T, \Sigma E_j]} + \xvar^{[T, \Sigma F_j]} \bigr)^{-1} \\
  &= \xvar^{ [T, \procov{\Sigma T'_j}]} \xvar^{[T, \Sigma T_j]} 
       \bigl( \xvar^{[T, \procov{\Sigma E_j]}-[T, E_j]} + \xvar^{[T,\procov{\Sigma F_j]} -[T, F_j]} \bigr)^{-1}.
\end{align*}
Recall, from \eqref{eq:procovSigEF}, that
$ %\[
\procov{\Sigma F_j}=\procov{\Sigma T_j}\oplus \procov{\Sigma T'_j}=\procov{\Sigma E_j}
$. %\]
Therefore, 
\begin{align*}
\amut{j} \bigl( \xvar^{[T', T'_j]} \bigr) 
  &= \xvar^{[T, \Sigma T_j]} \xvar^{-[T, \procov{\Sigma T_j}]} 
       \bigl( \xvar^{-[T, E_j]} + \xvar^{-[T, F_j]} \bigr)^{-1} \\
  &=   \xvar^{-[T, T_j]} 
       \bigl( \xvar^{-[T, E_j]} + \xvar^{-[T, F_j]} \bigr)^{-1}.          
\end{align*}
This completes the proof of the lemma.
\end{proof}

\goodbreak
\begin{proposition}\label{prop:sq-commutes}
For $\alpha_j$ as in \eqref{eq:mapalp2}, the right-hand square in \eqref{eq:diag} commutes.
\end{proposition}

\begin{proof}
It suffices to show that the square commutes when applied to the standard basis monomials, that is,
for all $i\in Q_0$,
\begin{equation}\label{eq:sqr-commutes}
  \xmut{j}(\CC[\wt] (\xvar^{[T', T'_i]})) =\CC[\wt](\amut{j}(\xvar^{[T', T'_i]})).
\end{equation}
In other words, we must show that 
\begin{equation}\label{eq:diag1} 
  \xmut{j}\bigl( \yvar^{\kapvec(T', T'_i)} \bigr)
  = \left\{ \begin{array}{ll} 
  \yvar^{ -\kapvec(T, T_j)}  
 \bigl( \yvar^{-\kapvec(T, E_j)}+\yvar^{-\kapvec(T, F_j)} \bigr)^{-1}, 
 & \text{if $i=j$,} \medskip \\
  \yvar^{\kapvec(T, T_i)}, 
 & \text{if $i\neq j$.}
  \end{array} \right.
\end{equation}
Write 
\begin{equation*}%\label{eq:kappaT} 
\kapvec(T', T'_i)= \kappa(T'_j, T'_i) s'_j +
\sum_{\substack{b_{tj}=0\\ t\neq j}} \kappa(T'_t, T'_i) s'_t +
\sum_{b_{tj}>0} \kappa(T'_t, T'_i) s'_t  + 
\sum_{b_{tj}<0} \kappa(T'_t, T'_i) s'_t , 
\end{equation*}
which is also true without the dashes.
From the definition of $\xmut{j}$, we have
\begin{equation}\label{eq0}
  \xmut{j}\bigl( \yvar^{\kappa(T'_j, T'_i) s'_j} \bigr) 
  = \yvar^{- \kappa(T'_j, T'_i) s_j}
\end{equation}
\begin{equation}\label{eq1}
  \xmut{j}\bigl( \prod_{\substack{b_{tj}=0\\ t\neq j}} \yvar^{\kappa(T'_t, T'_i) s'_t} \bigr)
  = \prod_{\substack{b_{tj}=0\\ t\neq j}}\yvar^{\kappa(T'_t, T'_i) s_t}
\end{equation}
\begin{equation}\label{eq2}
\begin{aligned}
  \xmut{j}\bigl( \prod_{b_{tj}>0} \yvar^{\kappa(T'_t, T'_i) s'_t} \bigr)
 &= \bigl( \prod_{b_{tj}>0}\yvar^ {\kappa(T'_t, T'_i) s_t} \bigr) 
   (1+\yvar^{s_j})^{\sum_{b_{tj}>0} b_{tj}\kappa(T'_t, T'_i)} \\ 
 &= \bigl( \prod_{b_{tj}>0}\yvar^{\kappa(T'_t, T'_i) s_t} \bigr) 
   (1+\yvar^{s_j})^{\kappa(F_j, T'_i)} 
\end{aligned}
\end{equation}
\begin{equation}\label{eq3}
\begin{aligned}
  \xmut{j}\bigl( \prod_{b_{tj}<0} \yvar^{\kappa(T'_t, T'_i) s'_t} \bigr)
 &= \bigl( \prod_{b_{tj}<0}\yvar^{\kappa(T'_t, T'_i) s_t} \bigr) 
   (1+\yvar^{-s_j})^{-\kappa(E_j, T'_i)} \\
 &= \bigl( \prod_{b_{tj}<0}\yvar^{\kappa(T'_t, T'_i) s_t} \bigr) 
   (1+\yvar^{s_j})^{-\kappa(E_j, T'_i)} \yvar^{\kappa(E_j, T'_i)s_j}.
\end{aligned}
\end{equation}

Hence, in the case $i\neq j$, noting that $T'_t=T_t$ when $t\neq j$, we get
\begin{align*}
  \xmut{j}\bigl( \yvar^{\kapvec(T', T'_i)} \bigr)
 &= \yvar^{- \kappa(T'_j, T_i) s_j} 
   \bigl( \prod_{t\neq j}\yvar^{\kappa(T_t, T_i) s_t} \bigr) (1+\yvar^{s_j})^{\kappa(F_j, T_i)}
   (1+\yvar^{s_j})^{-\kappa(E_j, T_i)} \yvar^{\kappa(E_j, T_i)s_j} \\
 &=  \yvar^{\kapvec(T, T_i)},
\end{align*}
since $\kappa(F_j, T_i)=\kappa(E_j, T_i)$ and 
$\kappa(E_j, T_i)-\kappa(T'_j, T_i) =\kappa(T_j, T_i)$,
by \Cref{Prop:kappafacts}\itmref{itm:kapfac2}.
So \eqref{eq:diag1} holds in this case.

Now consider the case $i=j$. 
By \Cref{Prop:kappafacts}\itmref{itm:kapfac3} for $T'$
(so $E_j'=F_j$ and $F_j'=E_j$), we have 
\begin{equation} \label{eq:kappaEF} 
  \kappa (F_j, T'_j) =  \kappa(T_j, T'_j) +  \kappa(T'_j, T'_j)= \kappa(E_j, T'_j) - 1.
\end{equation}
Using \eqref{eq0}--\eqref{eq3} again, we get
\begin{align*}
  \xmut{j} \bigl( \yvar^{ \kapvec(T', T'_j) } \bigr)
 &= \yvar^{-\kappa(T'_j, T'_j)s_j}
   \bigl( \prod_{t\neq j}\yvar^{\kappa(T'_t, T'_j) s_t} \bigr) 
  (1+\yvar^{s_j})^{\kappa(F_j, T'_j) -\kappa(E_j, T'_j)} \yvar^{\kappa(E_j, T'_j) s_j} \\
 &= 
  \bigl( \prod_{t\neq j}\yvar^{\kappa(T_t, T'_j) s_t} \bigr) 
  (1+\yvar^{s_j})^{-1} \yvar^{\kappa(T_j, T'_j) s_j} \yvar^{s_j } ,
 \quad\text{by \eqref{eq:kappaEF},} \\ 
  &= \yvar^{\kapvec(T, T'_j) } (\yvar^{-s_j}+1)^{-1} \\
  &= \yvar^{ -\kapvec(T, T_j)} 
    \bigl( \yvar^{-\kapvec(T, E_j)}+\yvar^{-\kapvec(T, F_j)} \bigr)^{-1}, 
\end{align*}
since $\kapvec(T, T'_j) = \kapvec(T, F_j)-\kapvec(T, T_j)$
and $\kapvec(T, F_j)+s_j = \kapvec(T, E_j)$,
by \Cref{Prop:kappadv}\itmref{itm:kapdv1}.
This completes the proof of the proposition, and thus the proof of \Cref{thm:Xflow}. 
\end{proof}

\begin{remark}\label{rem:gen-B}
The definitions of $\alpha_j$ in \eqref{eq:mapalp} and \eqref{eq:mapalp2}
and the proofs of \Cref{lem:mapalp2} and \Cref{prop:sq-commutes}
work more generally for cluster tilting objects in $\GP\algB$.
This is because they only depend on results from \Cref{Sec:6} and \Cref{Sec:11}.
However, in this general case, it is hard to know what is the right replacement for $\Xi^{T}$.
\end{remark}

\begin{remark}\label{rem:beta-sqr}
In the $\CM\algC$ case,
by using the explicit projective resolutions in \Cref{Prop:projres} and \Cref{thm:projresolb},
we can show that the square in \eqref{eq:diag} also commutes
with the two horizontal maps replaced by $\CC[\beta]$ in the opposite direction,
that is, 
\[ 
  \CC[\beta]\compo\xmut{j}=\amut{j}\compo \CC[\beta].
\]
\end{remark}

%%%% =========================================================== %%%%
\section{Newton--Okounkov cones and bodies for Grassmannians}\label{Sec:12}
%%%% =========================================================== %%%%

% ======================================================================
\subsection{Network charts and valuations}\label{subsec:plab-vals}
% ======================================================================

Let $G$ be a plabic graph of type $\perm_{k,n}$ 
and $\maxNC$ be the collection of face labels of $G$.
Let $T=T_\maxNC=\bigoplus_{J\in \maxNC} M_J$ be 
the corresponding cluster tilting object in $\CM\algC$.
By \cite[Thm.~6.8]{RW}, there is a network torus $\torus_G\subset\Gr(k, n)$,
whose character lattice can be identified with $\Nstar$,
and a corresponding network chart
\begin{equation}\label{eq:net-chart-rec}
  \net_G\colon \CC[\Gr(k, n)]/(\minor{I_\vstar}-1)\to \CC[\Nstar] %\colon \minor{I}\mapsto \flowp_I
\end{equation}
where $\net_G(\minor{I})= \flowp_I$, the flow polynomial as in \eqref{eq:class-flow-poly}.

If $G$ and $G'$ are plabic graphs related by a square move, so that their dual quivers $Q$ and $Q'$ are related by mutation,
then the network charts $\net_G$ and $\net_{G'}$ are related by $\cluX$-mutation
(see \cite[Lemma 6.15]{RW}).
In other words, $\net_G$ is an $\cluX$-cluster chart, or $\torus_G$ is an $\cluX$-cluster torus.

Hence, as in \cite[Remark 6.17]{RW}, one can write any $\cluX$-cluster chart,
as a map $\net_G$, as in \eqref{eq:net-chart-rec},
where now $G$ is just a symbol for an $\cluX$-seed $\bigl(Q, \{\yvar_i \st i\in Q_0\} \bigr)$,
related to an initial plabic seed by some sequence of mutations.
Note that $N_\vstar\subspc\ZZ^{Q_0}$ still consists of vectors vanishing at the boundary vertex $\vstar\in Q_0$, 
which, being `frozen', can be followed consistently through any sequence of mutations.

On the other hand, since we can mutate cluster tilting objects arbitrarily,
we can still associate some $T$ to such a generalised $G$, 
by following the same sequence of mutations from a plabic $G^0$ and corresponding $T^0=T_\maxNC$.
The main difference is that summands of $T$ may now have arbitrary rank.

By ordering the basis $\{[S_i]\st i\in Q_0^\vstar \}$ to define a lexicographic order on $\Nstar$
(as in \Cref{rem:total-order}),
Rietsch--Williams \cite[Def.~8.1]{RW} define %(the restriction of) 
a valuation
\begin{equation}\label{eq:def-valG}
  \val{G}\colon \CC[\Gr(k, n)]\setminus 0 \,\lra\, \Nstar, 
\end{equation}
by setting $\val{G}(f)$ to be minimal exponent of $\net_G(f)$.

\begin{theorem}\label{Thm:KVal}
Let $G$ be any $\cluX$-seed and $T$ be the corresponding (reachable) cluster tilting object in $\CM\algC$.
Then, for any $M\in \CM\algC$,
\begin{equation}\label{eq:KVal1}
  \val{G}(\cluschar{M}) = \kapvec (T, M).
\end{equation}
In particular, 
\begin{equation}\label{eq:KVal2}
  \val{G}(\minor{I}) = \kapvec(T, M_I).
\end{equation}
\end{theorem}

\begin{proof}
We start by observing that
\begin{equation}\label{eq:net=wtXi}
\net_G=\pbfun{\wt}{\Xi^T},
\end{equation} 
where $\Xi^T$ is as in \Cref{cor:partfun-map}.
When $G$ is a plabic graph and $T= T_\maxNC$, this is \Cref{rem:net-chart}.
But \Cref{thm:Xflow} shows that the right-hand side undergoes $\cluX$-mutation,
when $T$ mutates, while the left-hand side does also, by construction.
Hence \eqref{eq:net=wtXi} is true for all $G$ and the corresponding $T$.

Now, by \eqref{eq:net=wtXi} and \Cref{cor:partfun-map}, and recalling \S\ref{subsec:gen-flo-pol}, we have
\[
 \net_G(\cluschar{M})=\pbfun{\wt}{\Xi^T}(\cluschar{M}) = \pbfun{\wt}{\ptfn^T_M}=\fp^T_M.
\]
But we know by \Cref{Thm:char-flopol}\itmref{itm:flo3}
(see also \Cref{rem:fp-F-poly}),
that the leading exponent of $\fp^T_M$ is $\kapvec(T, M)$, as required.
Note that this is, after all, independent of the choice of total ordering.
The last part follows immediately, because $\cluschar{M_I}=\minor{I}$.
\end{proof}

This theorem gives a new way of proving the following result \cite[Theorem 15.1]{RW}.
\begin{corollary}
  $\val{G}(\minor{I})  = (\maxdiag(\partn{J} \setminus \partn{I}))_{J\in\maxNC}$.
\end{corollary}

\begin{proof}
Follows immediately from \eqref{eq:KVal2} and \eqref{eq:kappa-maxdiag}, 
since $T_\maxNC=\bigoplus_{J\in\maxNC} M_J$.
\end{proof}
 
In fact, \cite[Theorem 15.1]{RW} together with the elementary formula \eqref{eq:kappa-maxdiag}
was how we originally knew \eqref{eq:KVal2} and thus conjectured \eqref{eq:KVal1} 
(see \cite[Remark~7.3]{JKS2}).

\begin{remark}\label{rem:BCMN}
The formula \eqref{eq:KVal2} is closely related to \cite[Thm~7.9 \& Rem~7.10]{BCMN}.
%and its generalisation in \cite[Rem~7.10]{BCMN}.
To see this, first identify their map $\psi\colon N\to M$ with
$\beta\colon \Grot(\fd\algA)\to \Grot(\CM\algA)$ from \eqref{eq:def-beta},
written in the basis of simples and projectives as in \Cref{rem:beta-stuff}.
Note that the rectangles cluster quiver in their Fig~1 is 
oriented opposite to ours (cf.~\Cref{fig:gab-quiv}).
Their hook formula for the $g$-vector $\mathrm{g}_{\square\op}(p_J)$ \cite[Cor~7.6]{BCMN}
is a computation of $[\rectCTO, M_J]$ from an $\add \rectCTO$-presentation of $M_J$,
while 
\[ 
  \overline{\mathrm{g}}_{\square\op}(p_J)
 = [\rectCTO, M_J]-[\rectCTO,\modJ] 
 = -\beta\, \kapvec(\rectCTO, M_J),
\]
by \eqref{eq:wtTM} and \eqref{eq:beta-wt}.
%Finally $\kapvec(T, M)=\wt [T, M]$, as in \eqref{eq:wtTM}.
%while $\beta \wt [Z] = \rkk [Z] [T,\modJ] - [Z]$, as in \eqref{eq:beta-wt}.
Thus \eqref{eq:KVal2} amounts to the generalisation of \cite[(7.6)]{BCMN}
to an arbitrary seed, since $\beta$ is injective on $\Nstar$, 
where $\kapvec$ and $\val{G}$ take their values.
\end{remark}

% ======================================================================
\subsection{Newton--Okounkov bodies}\label{subsec:12.3}
% ======================================================================

Let $L_r=\CC[\Gr(k, n)]_r$ be the subspace of homogeneous functions of degree $r$
and $\CC[\Gr(k, n)]_\bullet$ be the subset of all non-zero homogeneous functions.
Define 
\begin{equation}\label{eq:valGLr}
  \val{G}(L_r) = \bigl\{ \val{G}(f)  \st  f\in L_r \setminus 0 \bigr\} \subset \Nstar.
\end{equation}
Rietsch--Williams \cite[(8.2)]{RW} defined the \emph{Newton--Okounkov body} 
in $\Nstar\tensR$ by
\begin{equation}\label{eq:nob-def}
  \nob{G} = \ConvHullBar\, \bigcup_r\frac{1}{r} \val{G}(L_r).
\end{equation}
In \cite[\S17.3]{RW}, they also considered the NO-monoid/cone, that is, 
\begin{align*}
  \nom{G} &= \bigl\{ \bigl(\deg(f), \val{G}(f)\bigr) \st f\in\CC[\Gr(k, n)]_\bullet \bigr\} \subset \ZZ\oplus\Nstar, \\
  \noc{G} &= \overline{\Rspan}\, \nom{G} \subset (\ZZ\oplus\Nstar)\tensR.
\end{align*}
%\[
%  \noc{G} = \overline\Rspan \bigl\{ \bigl(\deg(f), \val{G}(f)\bigr) \st f\in\CC[\Gr(k, n)]_\bullet \bigr\}
%\] 
In particular, $\noc{G}$ is the cone whose \level~1 slice is $\bigl(1,\nob{G}\bigr)$.
In \cite[Lem.~17.6]{RW}, they explained why $\nom{G}$ consists of the integral points of
$\noc{G}$, which is a rational polyhedral cone.

For comparison, we can rewrite $\gvc{T}$ from \Cref{Def:noc-etc} using the isomorphism $\wthat$ from 
\eqref{eq:wt-hat}, or strictly $\wthat \tensR$, as 
\[
  \wthat\gvc{T} = \overline\Rspan \bigl\{ \bigl(\rnk{\algC} M, \kapvec(T,M)\bigr) \st M\in\CM\algC \bigr\}
\]
whose \level~1 slice is $\bigl(1,\kab{T}\bigr)$, where
\begin{equation}\label{eq:kab-def}
  \kab{T} = \ConvHullBar\, \bigcup_r\frac{1}{r} 
  \bigl\{ \kapvec(T,M) \st M\in\CM\algC,\, \rnk{\algC} M = r \bigr\}.
\end{equation}

\begin{remark}\label{rem:GVbody}
We can also define a canonical \level~1 body in $\Grot(\CM\algA)\tensR$
\[
  \gvb{T} = \ConvHullBar\, \bigcup_r\frac{1}{r} 
  \bigl\{ [T,M]  \st M\in\CM\algC,\, \rnk{\algC} M = r \bigr\},
\]
which is the \level~1 slice of $\gvc{T}$ and is identified with $\kab{T}$ under $\wt$,
so that, explicitly, $\gvb{T}= [T,\modJ]-\beta\bigl( \kab{T} \bigr)$.
\end{remark}

\begin{theorem} \label{prop:nob=kab}
Let $G$ be any $\cluX$-seed and $T$ be the corresponding (reachable) cluster tilting object in $\CM\algC$.
Then 
\begin{equation}\label{eq:nob=kab}
  \nob{G}=\kab{T}.
\end{equation}  
\end{theorem}

\begin{proof}
By \Cref{Thm:GrassGeneric}, we know that, if we choose
$N_\compC\in\compC_{\gen}$, for $\compC\in \IrrCM(r)$,
then the $\Psi_{N_\compC}$ are a basis of $L_r$. 
Moreover, from \Cref{Thm:KVal} and its proof, we have 
\begin{equation}\label{eq:netPsi=fp}
  \net_G(\Psi_{M}) = \fp^T_{M},
\end{equation}
with leading exponent $\val{G}(\Psi_{M})=\kapvec(T,M)$.
Thus the $\net_G(\Psi_{N_\compC})$ have leading exponents 
$\kapvec(T,N_\compC)$
and, since $(r, \kapvec(T,N_\compC))=\wthat[T,N_\compC]$,
these are distinct, by \Cref{prop:ghgen-inj}.
Hence, by \Cref{Lem:uniqueleading},
\begin{equation}\label{eq:nob-generic}
 \val{G}(L_r) = \bigl\{ \kapvec(T,N_\compC)\st \compC\in \Irr_{\CM}(r) \bigr\}. 
\end{equation}
But we also know from \Cref{thm:subQComp-too} %\Cref{Thm:NOcone} 
that
\[
 \bigl\{ \kapvec(T,N_\compC)\st \compC\in \Irr_{\CM}(r) \bigr\}
 = \bigl\{ \kapvec(T,M) \st M\in\CM\algC,\, \rnk{\algC} M = r \bigr\},
\]
which completes the proof by applying $\ConvHullBar\, \bigcup_r\frac{1}{r} (-)$.
\end{proof}

\begin{remark}\label{rem:noc=gvc}
\Cref{prop:nob=kab} is equivalent to the fact that 
\begin{equation}\label{eq:noc=gvc}
  \noc{G}=\wthat\gvc{T}.
\end{equation}
We can also prove this via \Cref{Thm:NOcone}, by showing that
\begin{equation}\label{eq:noc=noc}
\begin{aligned}
  \nom{G} &= \wthat\nom{T},\\ 
  \noc{G} &= \wthat\noc{T}. 
\end{aligned}
\end{equation}
Indeed, by \Cref{cor:partfun-map}, we know that $\Xi^T$ gives an isomorphism 
$\CC[\Gr(k, n)]\isom\clualgA_T$, the cluster algebra that determines $\noc{T}$ as in \Cref{Def:noc-etc}.
Hence, by \Cref{Rem:HomogCone}, what we need to show is that, for 
all $f\in \CC[\Gr(k, n)]_\bullet$,
\[
\bigl(\deg(f), \val{G}(f)\bigr) = \wthat \Val{T}\bigl( \Xi^T(f) \bigr).
\]
But we know that $\deg(f)=\rkk \Xi^T(f)$ and also that $\val{G}(f)= \wt \Val{T}\bigl( \Xi^T(f) \bigr)$,
since $\net_G=\pbfun{\wt}{\Xi^T}$, by \eqref{eq:net=wtXi}, and we can choose the same order on $\Nstar$
to define the leading exponents. 
This completes the proof of \eqref{eq:noc=noc}.
\end{remark}

Note that \Cref{prop:nob=kab} and/or \Cref{rem:noc=gvc} imply that $\nob{G}$ and $\noc{G}$ 
do not depend on the choice of total order implicit in their definition (cf. \cite[Rem.~8.7]{RW}),
just as \Cref{Thm:NOcone} implies this for $\noc{T}$.

% ======================================================================
\subsection{Alternative proof}\label{subsec:alt-proof}
% ======================================================================
We now give a different proof of \Cref{prop:nob=kab}, that is $\nob{G}=\kab{T}$,
that doesn't use \Cref{thm:Xflow} to identify a general network chart, 
i.e.~prove \eqref{eq:net=wtXi} for all associated $(G,T)$,
but uses \cite[Lem.~16.17]{RW} and \Cref{Cor:mutkappa} to track how the two bodies mutate.

\begin{proof}[Another proof of \Cref{prop:nob=kab}]
When $G$ is a plabic graph, with $T= T_\maxNC$, 
then \eqref{eq:net=wtXi} %, that is, $\net_G=\pbfun{\wt}{\Xi^T}$,
is just \Cref{rem:net-chart}.
This gives \eqref{eq:netPsi=fp} and then \eqref{eq:nob-generic}
as in the proofs of \Cref{Thm:KVal} and \Cref{prop:nob=kab}.
That is, for plabic $G$, we have, for any $N_\compC\in\compC_{\gen}$,
\begin{equation}\label{eq:nob-generic2}
 \val{G}(L_r) = \bigl\{ \kapvec(T,N_\compC)\st \compC\in \Irr_{\CM}(r) \bigr\},  
\end{equation}
without the use of \Cref{thm:Xflow}.
This is the key to proving $\nob{G}=\kab{T}$, 
so we need to prove \eqref{eq:nob-generic2} for a general $\cluX$-seed $G$.

For a general $G$, with associated $T$ and $Q$, 
if $G'=\mu_i G$, then, by \cite[Lemma 16.17]{RW},
\[ 
  \val{G'} (L_r) = \tropAmut{Q}{i}\, \val{G}(L_r),
\]
where $\tropAmut{Q}{i}$ is tropical $\cluA$-mutation, as in \Cref{def:tropAmut}.
Similarly, if $T'=\mu_i T$, so $T'$ is associated to $G'$, then, 
by \Cref{Cor:mutkappa}, we have 
\[ 
  \kapvec (T', N_\compC)=\tropAmut{Q}{i}\, \kapvec (T, N_\compC), 
\]
because, as in the proof of \Cref{rem:rat-pol-cone2}, we can choose $N_\compC$ 
to have a generic $\add T$-presentation.
Hence \eqref{eq:nob-generic2} holds for $(G',T')$ provided it holds for $(G,T)$,
and thus it holds by induction for all associated $(G,T)$, 
starting from any plabic case.
\end{proof}

Note that \cite[Lemma 16.17]{RW} uses a theta basis for $L_r$,
which is in bijection with $\val{G} (L_r)$ because the basis has distinct leading exponents.
The fact that the basis is parameterised by tropical points of a cluster $\cluA$-variety
is behind why $\val{G} (L_r)$ undergoes tropical $\cluA$-mutation.

\begin{remark}\label{rem:finite-type}
When $\Gr(k, n)$ is of finite cluster type, there are only finitely many indecomposable modules $M$
in $\CM\algC$, all of which are rigid (and hence generic),
so that their cluster characters $\Psi_M$ are all the cluster variables.
Since any module in $\CM\algC$ is a direct sum of indecomposable modules,
we can write
\[
  \kab{T}
  = \ConvHull \Bigl\{ \frac{ \kappa(T, M) }{ \rnk{\algC} M }
   \st \text{$M$ is indecomposable} \Bigr\},
\]
that is, the convex hull of a finite set of rational points,
without taking any closure.
Hence we have
\[
  \nob{G}
  = \ConvHull \Bigl\{ \frac{ \val{G}(\Psi) }{ \deg\Psi }
   \st \text{$\Psi$ is a cluster variable} \Bigr\}.
\]
which would also follow from the fact that cluster monomials are a basis of $\CC[\Gr(k,n)]$
in this case, since they have distinct leading exponents,
by \cite[Thm~5.5]{FK} (cf.~\Cref{rem:rigid-gvec}).
\end{remark}

%%%% =========================================================== %%%%
\section{Superpotential and tropicalisation}\label{Sec:13}
%%%% =========================================================== %%%%

% ======================================================================
\subsection{Superpotential}\label{subsec:rect-sp}
% ======================================================================

Marsh--Rietsch \cite[Def.~6.1]{MR} gave the following formula for a superpotential
on the Grassmannian $\Gr(k, n)$
\begin{equation}\label{eq:superpot}
\superpot = q  \frac{ \minor{\MRnum{n-k}} }{ \minor{\MRden{n-k}} } 
+ \sum_{i\neq n-k} \frac{ \minor{\MRnum{i}} }{ \minor{\MRden{i}} },
\end{equation}
where
\[
  \MRden{i}=[i+1,  i+k] 
  \quadand
  \MRnum{i}=[i+1, i+k-1]\cup \{i+k+1\}
\]
and indices are taken mod $n$.

In terms of representations, $M_{\MRden{i}}$ is the projective module at (boundary) vertex $i$
and $M_{\MRnum{i}}$ is the unique nontrivial extension of $M_{\MRden{i}}$  
by the simple $\algC$-module $S_{i+k}$.
In terms of Young diagrams, this means that, for $i\neq n-k$, 
the diagram for $\MRnum{i}$ is obtained by
adding a box to the diagram for $\MRden{i}$.
However, for $\MRnum{n-k}$ one instead removes a hook,
because the cyclic structure is hidden from the Young diagram point of view.

Marsh--Rietsch \cite[Prop.~6.1]{MR} also gave an expression 
for the superpotential $\superpot$ in terms of rectangle cluster variables. 
In this section, we will give a new derivation of this expression 
using the rectangles cluster tilting object $T=\rectCTO$ 
and the Fu-Keller cluster character $\Phi^T$ from \eqref{eq:FK-CC}.

Recall, from \Cref{def:rect-clus}, that the rectangles cluster tilting object is defined as
\[
T=\rectCTO=T_\vempty\oplus \bigoplus_{ij\in\grid} T_{ij},
\]
where $T_\vempty = M_{[1,k]}$ and $T_{ij}=M_{K_{ij}}$ for $K_{ij}=[1, k-i]\cup [k-i+j+1, k+j]$.
The Gabriel quiver of $A=\End(T)\op$ has a quite uniform structure, 
illustrated in \Cref{fig:gab-quiv}.
This quiver is dual to a particular plabic graph, as in \cite[\S4, Fig.~5]{RW},
although some labelling conventions may differ 
(see \S\ref{subsec:nota-conv} and \Cref{rem:compare-RW}).

\begin{figure} [h]
\begin{tikzpicture}[xscale=1.6, yscale=1.6]
\draw (0.25, 0.25) node (d) {$T_\vempty$};
\draw  (1, 1) node (a1) {$T_{11}$};
\draw  (2, 1) node (a2) {$T_{12}$};
\draw  (3, 1) node (a3) {$T_{13}$};
\draw  (4, 1) node (a4) {$T_{14}$};
\draw  (1, 2) node (b1) {$T_{21}$};
\draw  (2, 2) node (b2) {$T_{22}$};
\draw  (3, 2) node (b3) {$T_{23}$};
\draw  (4, 2) node (b4) {$T_{24}$};
\draw  (1, 3) node (c1) {$T_{31}$};
\draw  (2, 3) node (c2) {$T_{32}$};
\draw  (3, 3) node (c3) {$T_{33}$};
\draw  (4, 3) node (c4) {$T_{34}$};
\foreach \t/\h in {a1/a2, a2/a3, a3/a4, b1/b2, b2/b3, b3/b4, c1/c2, c2/c3, c3/c4, 
 a1/b1, a2/b2, a3/b3, a4/b4, b1/c1, b2/c2, b3/c3, b4/c4, b2/a1, b3/a2, b4/a3, c2/b1, c3/b2, c4/b3, d/a1}
  \draw[quivarr]  (\h) to (\t);
\draw[quivarr] (d) to [out=90, in=-132] (c1);
\draw[quivarr] (d) to [out=0, in=-142] (a4);
\end{tikzpicture} 
\caption{The Gabriel quiver $Q$ of $\End\bigl(\rectCTO\bigr)\op$, for $(k,n)=(3,7)$.} \label{fig:gab-quiv}
\end{figure}

Observe that $M_{\MRden{s}}=P_s=\algC e_s$ is projective, so always in $\add T$.
In particular, $M_{\MRden{s}}=T_{ks}$, if $0\leq s\leq n-k$, 
and $M_{\MRden{s}}=T_{(n-s)(n-k)}$, if $n-k\leq s\leq n$,
where, by convention, $T_{k0}=T_\vempty=T_{0(n-k)}$.
On the other hand, $M_{\MRnum{i}}\in \add T$ if and only if $i=0$ or $i=n-k$,
in which case
\begin{equation}\label{eq:M-in-addT}
  M_{\MRnum{0}}=T_{11} 
  \quadand
  M_{\MRnum{n-k}}=T_{(k-1) (n-k-1)}.
\end{equation}

\begin{lemma} \label{Lem:topMIhat}
$\top M_{\MRnum{s}}=S_s\oplus S_{s+k}$,
where $S_i$ are the simple $\algC$-modules.
\end{lemma}

\begin{proof}
Note that $M_{\MRnum{s}}$ is an extension by $S_{s+k}$ of 
$M_{\MRden{s}}=P_s$, which has simple top $S_s$.
\end{proof}

Denote the minimal syzygy of $T_{ij}$ by $\syzT_{ij}$, for all $ij\in\grid$. 
If $i=k$ or $j=n-k$, then $T_{ij}$ is projective and so $\syzT_{ij}=0$.
The following properties of the non-projective syzygies are straightforward.

\begin{lemma}\label{lem14.2}
If $1\leq i<k$ and $1\leq j<n-k$, then
\begin{enumerate}
\item\label{itm:14.2.2}
$\syzT_{ij}=M_{L_{ij}}$ with $L_{ij}=[k-i+1, k]\cup [k+j+1, 2k+j-i ]$. 
\item\label{itm:14.2.3}
$\top \syzT_{ij}=S_{k-i}\oplus S_{k+j}$. 
\item\label{itm:14.2.4}
The irreducible map $T_{ij}\to T_{i(j+1)}$ induces a map
$\iota_{ij}\colon \syzT_{ij} \to \syzT_{i(j+1)}$ 
which is an isomorphism on the common summands $S_{k-i}$ of their tops.
\item\label{itm:14.2.5}
The irreducible map $T_{ij}\to T_{(i+1)j}$ induces a map
$\epsilon_{ij}\colon \syzT_{ij} \to \syzT_{(i+1)j}$ 
which is an isomorphism on the common summands $S_{k+j}$ of their tops.
\end{enumerate}
\end{lemma} 

\begin{proof}
These follow by considering the profiles of the relevant modules. 
More precisely, if the two projectives in the minimal cover are taken to be
submodules of $T_{ij}$, then $\syzT_{ij}$ is their intersection
(see \Cref{fig:syzTij} for examples).
\end{proof}

\begin{figure}[h]
\begin{tikzpicture} [scale=0.4,
lowbdry/.style={thick, gray},
boxes/.style={thick, blue},
ridge/.style={very thick, purple}]
\pgfmathsetmacro{\epsstep}{0.5}
\begin{scope} [shift={(-8+\epsstep,0)},rotate=135]
\draw [boxes] (0,-1)--(2,-1) (1,0)--(1,-2);
\draw [lowbdry] (0,-5)--(0,0)--(4,0);
\draw [ridge] (4,0)--++(-2,0)--++(0,-2)--++(-2,0)--++(0,-3);
\draw (3.5,0) node [above right=-3pt] {\tiny 1};
\draw (2.5,0) node [above right=-3pt] {\tiny 2};
\draw (1.5,-2) node [above right=-3pt] {\tiny 5};
\draw (0.5,-2) node [above right=-3pt] {\tiny 6};
\end{scope}
\begin{scope} [shift={(0,0)},rotate=135]
\draw [ridge, densely dotted] (4,0)--++(-2,0)--++(0,-2)--++(-2,0)--++(0,-3);
\draw [ridge] (2,2)--++(0,-2)--++(-2,0)--++(0,-2)--++(-2,0)--++(0,-1);
\draw (1.5,0) node [above right=-3pt] {\tiny 3};
\draw (0.5,0) node [above right=-3pt] {\tiny 4};
\draw (-0.5,-2) node [above right=-3pt] {\tiny 7};
\draw (-1.5,-2) node [above right=-3pt] {\tiny 8};
\end{scope}
\begin{scope} [shift={(9,0)},rotate=135]
\draw [boxes] (0,-1)--(3,-1) (2,0)--(2,-2) (1,0)--(1,-2);
\draw [lowbdry] (0,-5)--(0,0)--(4,0);
\draw [ridge] (4,0)--++(-1,0)--++(0,-2)--++(-3,0)--++(0,-3);
\draw (3.5,0) node [above right=-3pt] {\tiny 1};
\draw (2.5,-2) node [above right=-3pt] {\tiny 4};
\draw (1.5,-2) node [above right=-3pt] {\tiny 5};
\draw (0.5,-2) node [above right=-3pt] {\tiny 6};
\end{scope}
\begin{scope} [shift={(17-\epsstep,0)},rotate=135]
\draw [ridge, densely dotted] (4,0)--++(-1,0)--++(0,-2)--++(-3,0)--++(0,-3);
\draw [ridge] (3,1)--++(0,-1)--++(-3,0)--++(0,-2)--++(-1,0)--++(0,-2);
\draw (2.5,0) node [above right=-3pt] {\tiny 2};
\draw (1.5,0) node [above right=-3pt] {\tiny 3};
\draw (0.5,0) node [above right=-3pt] {\tiny 4};
\draw (-0.5,-2) node [above right=-3pt] {\tiny 7};
\end{scope}
\end{tikzpicture}
\caption{The modules $T_{22}=M_{1256}$, with syzygy $\syzT_{22}=M_{3478}$,
and $T_{32}=M_{1456}$, with syzygy $\syzT_{32}=M_{2347}$; for $(k,n)=(4,9)$.}
\label{fig:syzTij}
\end{figure}

Next we compute $\Ext^1(T_{ij}, M_{\MRnum{s}})$, noting that this automatically vanishes
when $i=k$ or $j=n-k$, as $T_{ij}$ is projective, and when $s=0$ or $n-k$, as $M_{\MRnum{s}}\in\add T$.

\begin{lemma}\label{lem14.4}  
If $1\leq i< k$ and  $1\leq j< n-k$. 
\begin{enumerate}
\item\label{itm14.4.1}
If $0< s< n-k$, then
$
\dim\Ext^1(T_{ij}, M_{\MRnum{s}})=
\left\{ \begin{tabular}{ll} 
$1$ & if $j=s$, \\ 
$0$ & otherwise.
\end{tabular} \right.
$
\medskip
\item\label{itm14.4.2}
If $n-k<s<n$, then
$
\dim\Ext^1(T_{ij}, M_{\MRnum{s}})=
\left\{ \begin{tabular}{ll} 
$1$ & if $i=n-s$, \\ 
$0$ & otherwise.
\end{tabular} \right.
$
\end{enumerate}
\end{lemma}

\begin{proof}
\itmref{itm14.4.1}
Note that $\Ext^1(N, M)\isom \sHom(\Syz N, M)$. 
Comparing profiles of the modules $\syzT_{ij}$ and $M_{\MRnum{s}}$ we can see that,
if $j<s$, then any map $\syzT_{ij}\to M_{\MRnum{s}}$ factors through $P_s$, 
while, if $j>s$, then any such map factors through $P_0$.
Thus $\sHom(\Syz T_{ij}, M_{\MRnum{s}})=0$.

If $j=s$, then $\top \syzT_{ij}$ and $\top M_{\MRnum{s}}$ 
share a common summand $S_{s+k}$,
by \Cref{Lem:topMIhat} and \Cref{lem14.2}
and there is a map $g_{ij}\colon \syzT_{ij}\to M_{\MRnum{s}}$ 
which is an isomorphism on these summands.
On the other hand, any map with image in $tM_{\MRnum{s}}$ factors through $P_0$. 
Thus $\sHom(\Syz T_{ij}, M_{\MRnum{s}})$ is 1-dimensional, spanned by (the class of) $g_{ij}$.

\itmref{itm14.4.2}
Similar to \itmref{itm14.4.1}.
If $i\neq n-s$, then any map $\syzT_{ij}\to M_{\MRnum{s}}$ factors through $P_s$ or~$P_0$.
If $i=n-s$, then $\top \syzT_{ij}$ and $\top M_{\MRnum{s}}$ 
share a common summand $S_{k-i}=S_{s+k}$ 
and there is a map $f_{ij}\colon \syzT_{ij}\to M_{\MRnum{s}}$ 
which is an isomorphism on these summands.
Any map with image in $tM_{\MRnum{s}}$ factors through $P_0$. 
\end{proof}

We denote the simple $\algA$-module at vertex $ij\in\grid$ by $S_{ij}$. 

\begin{proposition} \label{Prop:ExtTM} 
Let $X=\Ext^1(T, M_{\MRnum{s}})$.
\begin{enumerate}
\item\label{itm:prETM1}
If $0< s<n-k$, then $X$ is a uniserial  $A$-module, 
supported at all the vertices in column $s$ but the top-most one. 
In particular, 
\[ 
  \dim X=k-1,\quad  \top X=S_{(k-1)s} \quadand \soc X= S_{1s}.
\]
\item\label{itm:prETM2}
If $n-k< s< n$, then $X$ is a uniserial  $A$-module, 
supported at all the vertices in row $s'=n-s$ but the right-most one. 
In particular,
\[ 
  \dim X=n-k-1,\quad \top X=S_{s'(n-k-1)} \quadand \soc X= S_{s'1}.
\]
\end{enumerate}
\end{proposition}

\begin{proof}
The support and dimension are determined by \Cref{lem14.4}. 
Thus $X$ is supported on a linear sub-quiver of the 
Gabriel quiver $Q$ (cf. \Cref{fig:gab-quiv}).

We know \emph{a priori} the $X$ must be indecomposable, 
because $M_{\MRnum{s}}$ is indecomposable
and $\Ext^1(T, -)$ is an equivalence $\underline\CM\algC / \add T \to \mmod \sEnd(T)\op$,
by \cite[Prop.~2.1(c)]{KeRe}.
Thus $X$ is uniserial, with the stated top and socle.

We can also see this directly by observing that all arrows in the sub-quiver act as isomorphisms,
because $f_{(n-s)j}=f_{(n-s)(j+1)}\iota_{(n-s)j}$ and $g_{is}=g_{(i+1)s}\,\epsilon_{is}$,
where $\iota,\epsilon$ are as in \Cref{lem14.2} and $f,g$ are as in \Cref{lem14.4}.
\end{proof}

The cluster variables for $\rectCTO$ are
$\Psi_{T_{ij}}= \minor{K_{ij}}$, for $ij\in\grid$,
and $\Psi_{T_\vempty}=\minor{\MRden{0}}$.
We call them \emph{rectangle cluster variables}
and the goal is to express the superpotential $\superpot$
in terms of them, as in \cite[Prop.~6.10]{MR} and \cite[Prop.~10.5]{RW}.

To simplify notation, as in \emph{loc. cit.}, we write the rectangle cluster variables as
\[
 p_{ij} = 
  \begin{cases}
    \minor{K_{ij}} & \text{for $ij\in\grid$,} \\
    \minor{\MRden{0}} & \text{if $i=0$ or $j=0$.}
  \end{cases}
\]  
Thus the cluster variable $p_\vempty=\minor{\MRden{0}}$
can also appear as $p_{00}$, $p_{10}$, $p_{01}$, \emph{etc}.

%\goodbreak
\begin{proposition}\label{Prop:superpot}
The fractions in \eqref{eq:superpot} can be written in rectangle cluster variables as follows.
\begin{enumerate}
\item\label{itm:prSP1}
If $0< s< n-k$, then
$\displaystyle
  \frac{ \minor{\MRnum{s}} }{ \minor{\MRden{s}} }
  = \sum_{t=0}^{k-1}  \frac{ p_{(t+1) (s+1)} p_{t (s-1)}}{p_{t s}\, p_{(t+1) s}}.
$
\medskip
\item\label{itm:prSP2}
If $n-k< s< n$ and $s'=n-s$, then
$\displaystyle
  \frac{ \minor{\MRnum{s}} }{ \minor{\MRden{s}} }
  = \sum_{t=0}^{n-k-1}  \frac{ p_{(s'+1) (t+1)} p_{(s'-1) t}}{p_{s' t}\, p_{s' (t+1)}}.
$
\item\label{itm:prSP3}
$\displaystyle
  \frac{ \minor{\MRnum{0}} }{ \minor{\MRden{0}} }=\frac{p_{11}}{p_{\vempty}} 
$
and
$\displaystyle
  \frac{ \minor{\MRnum{n-k}} }{ \minor{\MRden{n-k}} }=\frac{p_{(k-1)(n-k-1)}}{p_{k(n-k)}}.
$
\end{enumerate}
Thus the Marsh--Rietsch superpotential from \eqref{eq:superpot} can be written
\begin{equation}\label{eq:MRsp-rect}
\superpot = 
  \frac{p_{11}}{p_{\vempty}}
 + \sum_{i=1}^{k} \sum_{j=2}^{n-k} \frac{ p_{ij}\, p_{(i-1)(j-2)}}{p_{(i-1)(j-1)} p_{i(j-1)}}
% + \sum_{i=1}^{n-k-1}\sum_{t=0}^{k-1}  \frac{ p_{(t+1) (i+1)} p_{t (i-1)}}{p_{(t+1) i}p_{t i}}
 + q\frac{p_{(k-1)(n-k-1)}}{p_{k(n-k)}}
 + \sum_{i=2}^{k}\sum_{j=1}^{n-k}  \frac{ p_{ij} \, p_{(i-2)(j-1)}}{p_{(i-1)(j-1)}p_{(i-1)j}}
\end{equation}
recovering \cite[Prop.~10.5]{RW}, i.e.~\cite[Prop.~6.10]{MR}.
\end{proposition}
\begin{proof}
By \Cref{rem:FK-applic}, we can use \eqref{eq:FK-CC}, to write
\begin{equation}\label{eq:coeff}
\minor{\MRnum{s}}=\Psi_{M_{\MRnum{s}}}
 =\xvar^{[T, M_{\MRnum{s}}]} \sum_d \euler \bigl( \Gr_d \Ext^1(T, M_{\MRnum{s}}) \bigr) \xvar^{-\beta(d)},
\end{equation}
after making the identification $\xvar^{[T, T_{ij}]}=p_{ij}$.

\itmref{itm:prSP1} 
By \Cref{Prop:ExtTM}\itmref{itm:prETM1}, $X_s=\Ext^1(T, M_{\MRnum{s}})$ is 
a uniserial module with top $S_{(k-1)s}$ and socle $S_{1s}$. 
For each $0\leq t< k$, $X_s$ has a unique submodule $X_{ts}$ of dimension~$t$. 
So the non-zero coefficients $\euler( \Gr_d X_s)$ 
in \eqref{eq:coeff} are all~1.
Thus  
\[
  \minor{\MRnum{s}}=\xvar^{[T, M_{\MRnum{s}}]} \sum_{t=0}^{k-1} \xvar^{-\beta[X_{ts}]}.
\]
Thus $\minor{\MRnum{s}}/\minor{\MRden{s}}$ is a sum of terms
$\xvar^{[T, M_{\MRnum{s}}] - \beta[X_{ts}]} / \minor{\MRden{s}}
= \xvar^{ [T, M_{\MRnum{s}}] - [T, T_{ks}] - \beta[X_{ts}] }$.  

Observe that the following sequence is exact.
\[
0\to T_{1s}\to T_{ks}\oplus T_{1(s+1)}\to M_{\MRnum{s}}\to 0. 
\]
Hence
$
[T, M_{\MRnum{s}}] -[T, T_{ks}] = [T, T_{1(s+1)}] - [T, T_{1s}]
$ 
and the $t=0$ term of $\minor{\MRnum{s}} / \minor{\MRden{s}}$ is 
\begin{equation}\label{eq:t0term}
  \xvar^{ [T, M_{\MRnum{s}}] - [T, T_{ks}] } 
  = \frac{p_{1(s+1)}}{p_{1s}}
  = \frac{ p_{1 (s+1)} p_{0 (s-1)}}{p_{0 s}\, p_{1 s}},
\end{equation}
as required, because $p_{0 (s-1)}=p_\vempty= p_{0 s}$.

Observe that $X_{ts}/X_{(t-1)s}\isom S_{ts}$, so $[X_{ts}]=[X_{(t-1)s}]+[S_{ts}]$.
Hence we can try to prove the formula for the general term
by induction on $t$ starting from \eqref{eq:t0term}.
Assuming the case $t-1$ is proven, we would have
\begin{align*}
 \xvar^{[T, M_{\MRnum{s}}] - [T, T_{ks}] - \beta[X_{ts}] }
 &=\xvar^{[T, M_{\MRnum{s}}] - [T, T_{ks}] - \beta[X_{(t-1)s}] - \beta[S_{ts}]} \\
 &=  \frac{ p_{t (s+1)} p_{(t-1) (s-1)}}{p_{(t-1) s}\, p_{t s}} \xvar^{- \beta[S_{ts}]}
 =  \frac{ p_{(t+1) (s+1)} p_{t (s-1)}}{p_{t s}\, p_{(t+1) s}},
 \end{align*}
provided we can prove that 
\begin{equation}\label{eq:req-for-beta}
  \xvar^{\beta[S_{ts}]} = \frac
  { p_{t (s+1)} p_{(t-1) (s-1)} p_{(t+1) s} }
  { p_{(t-1) s}\, p_{(t+1) (s+1)} p_{t (s-1)} }. 
\end{equation}
Now, by \Cref{Prop:projres}, 
\[
  \beta[S_{ts}]=[T, F_{ts}]-[T, E_{ts}],
\]
where $E_{ts}$ and $F_{ts}$ are the middle terms of the mutation sequences for $T_{ij}$
from \eqref{eq:mut-two}.
At a general vertex $ts$ in the Gabriel quiver, there are three incoming arrows from
$t (s+1)$,  $(t-1) (s-1)$ and $(t+1) s$ and three outgoing arrows to
$(t-1) s$, $(t+1) (s+1)$ and $t (s-1)$, giving
\[
 F_{ts} = T_{t (s+1)}\oplus T_{(t-1) (s-1)} \oplus T_{(t+1) s}
 \quadand
 E_{ts} = T_{(t-1) s}\oplus T_{(t+1) (s+1)}\oplus T_{t (s-1)}
\]
and thus, by the substitution $p_{ij}=\xvar^{[T,T_{ij}]}$, 
we obtain the required expression \eqref{eq:req-for-beta} to complete the induction.

In the border cases, where $s=1$ or $t=1$, there is a cancellation in \eqref{eq:req-for-beta},
due to the term $p_\vempty$ in both the numerator and denominator.
In the quiver, there is one incoming and one outgoing arrow missing, 
which precisely compensates for this cancellation 
and so (a slight variant of) the above argument is still valid.

\itmref{itm:prSP2} is similar to \itmref{itm:prSP1}, so we omit the details.
\itmref{itm:prSP3} is immediate because $\minor{\MRnum{0}}$ and $\minor{\MRnum{n-k}}$
are rectangle cluster variables (cf.~\eqref{eq:M-in-addT}).
To write the sums as in \eqref{eq:MRsp-rect} make the change of variables 
$(i,j)=(t+1,s+1)$ for \itmref{itm:prSP1} and $(i,j)=(s'+1,t+1)$ for \itmref{itm:prSP2}.
\end{proof}

\begin{remark}\label{rem:trop-GT}
As noted in \cite[Lem.~16.2]{RW},
tropicalising the superpotential $\superpot$,
written in rectangle cluster variables as in \Cref{Prop:superpot} (and with $p_\vempty=1$), 
%that is, \eqref{eq:MRsp-rect} %\Cref{Prop:superpot} 
yields the inequalities \eqref{eq:CGT-pattern}
that define the cone $\gtc$ of cumulative GT-patterns.
%as noted in \cite[Lem.~16.2]{RW}.
Specifically, $p_{ij}$ tropicalises to $v_{ij}$ and $q$ to $r$.
In other words,
\begin{equation}\label{eq:GT=SP}
  \gtc  = \cone_W(\rectCTO)
\end{equation}
where the \emph{superpotential cone} associated to any $T$ by $\superpot$ is,
as in \cite[Def.~10.10]{RW}, 
\[
  \cone_W(T) = \bigl\{ (r,v)\in (\ZZ\oplus\Nstar)\tensR \st \trop_T(\superpot)(r,v)\geq 0 \bigr\}.
\]
Here we write $\trop_T$, as $T$ indexes a cluster for us, while \cite{RW} writes $\trop_G$.
Moreover, we define the cone itself, while they define its level $r$ slices.
Note that, since it is given by integral inequalities, $\cone_W(T)$ is always a rational polyhedral cone.
\end{remark}

This leads to the following categorical incarnation of Grassmannian mirror symmetry,
in the sense of \cite{RW}.

\begin{theorem}\label{thm:trop-GT}
For any reachable cluster tilting object $T$ in $\CM\algC$, we have
\begin{equation}\label{eq:wtGV=SPcone}
  \wthat\gvc{T} = \cone_W(T).
\end{equation}
\end{theorem}

\begin{proof}
In the case $T=\rectCTO$, we know that $\wthat\gvc{\rectCTO}=\gtc$, by \Cref{Prop:kappaGT},
so this is just \eqref{eq:GT=SP}.

We can then follow the inductive strategy of \cite{RW} 
to see that \eqref{eq:wtGV=SPcone} holds for all~$T$,
by considering what happens when $T$ is mutated to $T'$.
On one hand, $\cone_W(T)$, and hence its set of integer points, 
undergoes tropical $\cluA$-mutation (e.g.~\cite[Cor.11.16]{RW}).
On the other hand, $\wthat\gvc{T}$ also undergoes tropical $\cluA$-mutation, 
by \Cref{Cor:mutkappa}. % by \Cref{rem:rat-pol-cone2}.
Hence $\wthat\gvc{T'}$ is the set of integral points of $\cone_W(T')$,
whenever the same holds for $T$, as required for the induction.
\end{proof}

Note that Rietsch--Williams use this inductive strategy to show (in effect) that 
\begin{equation}\label{eq:RW-coneSP}
 \noc{G}=\cone_W(G), 
\end{equation}
for all $\cluX$-seeds $G$. In particular, they have to prove that $\val{G}(L_r)$, 
as in \eqref{eq:valGLr}, undergoes tropical $\cluA$-mutation \cite[Lemma 16.17]{RW}
and, for the rectangles cluster, coincides with the cumulative GT-polytope at \level~$r$ 
\cite[Lemma 16.6]{RW}.

We know, from \eqref{eq:noc=noc}, that $\noc{G}=\wthat\noc{T}$,
when $G$ and $T$ are associated, so $\cone_W(G)$ is just another name for $\cone_W(T)$.
Thus \eqref{eq:wtGV=SPcone} is equivalent to \eqref{eq:RW-coneSP}.

% =========================
\newcommand{\spb}[1]{\Gamma_{#1}}
% =========================
\begin{remark}\label{rem:body-vers}
In \cite{RW}, \eqref{eq:RW-coneSP} is actually stated as an equality of polytopes 
\begin{equation}\label{eq:RW-bodSP}
 \nob{G}=\spb{W}(G)
\end{equation}
and we can write \eqref{eq:wtGV=SPcone} similarly as
\begin{equation}\label{eq:kap=SPbod}
 \kab{T}=\spb{W}(T)
\end{equation}
where $\nob{G}$ and $\kab{T}$ are as in \eqref{eq:nob-def} and \eqref{eq:kab-def},
%where $\nob{G}$ is as in \eqref{eq:nob-def}, $\kab{T}$ is as in \eqref{eq:kab-def} 
and $\spb{W}$ is the level 1 slice of $\cone_W$, that is,
\[
 \spb{W}(T) = \bigl\{ v\in \Nstar\tensR \st \trop_T(\superpot)(1,v)\geq 0 \bigr\}.
\]
Then $\spb{W}(G)$ is just another name for this, as $\trop_G$ is another name for $\trop_T$.
Thus \eqref{eq:RW-bodSP} is equivalent to \eqref{eq:kap=SPbod} by \Cref{prop:nob=kab}.
% that is, \eqref{eq:nob=kab}, that is $\nob{G}=\kab{T}$.
\end{remark}

% ======================================================================
\subsection{Categorical interpretation of the superpotential}\label{subsec:cat-superpot}
% ======================================================================

To describe $\gvc{T}$ itself in $\Grot(\CM\algA)$ in a similar way to \eqref{eq:wtGV=SPcone},
we want an expression for the superpotential with exponents in the dual lattice $\Grot(\fd\algA)$.
We will find such an expression in what follows.

We can write the Marsh--Rietsch superpotential $\superpot=\sum_{i=1}^n \superpot_i$
in terms of cluster characters by
\begin{equation}\label{eq:superpot-term}
  \superpot_i = q^{m_i} \cluschar{\extd{P_i}} / \cluschar{P_i}
\end{equation}
where~$m_i=1$, if $i=n-k$, and $m_i=0$, otherwise, 
while $\extd{P_i}$ is an extension of~$P_i$ by a simple $\algC$-module~$\extd{S_i}=S_{i+k}$, 
that is, we have a short exact sequence
\begin{equation}\label{eq:PPS}
0\lra P_i \lraa{f_i} \extd{P_i} \lra \extd{S_i} \lra 0.
%\ShExSeq{P_i}{\extd{P_i}}{\extd{S_i}}
\end{equation}
In the notation of \S\ref{subsec:rect-sp}, we have $P_i=M_{\MRden{i}}$ and $\extd{P_i}=M_{\MRnum{i}}$.

% ================================
%\newcommand{\Ex}[1]{\hlite{\mathsf{E}}(#1)} % temporary notation
\newcommand{\Ex}[1]{\Ext^1(T,#1)} % drop notation
% ================================

A key step in \cite[\S18]{RW} is to write the superpotential $\superpot$ in cluster coordinates 
(and then tropicalise, i.e.~look at the exponents).
In the notation of \Cref{Lem:upsilon} and \Cref{Thm:PsiandTheta},
this means computing $\Upsilon^{T}\superpot$ by substituting $\Phi^T$ for $\cluschar{}$ in \eqref{eq:superpot-term}.
More precisely, we will use the co-$g$-vector variant of the CC formula
\eqref{eq:FK-CC}, namely
\begin{equation}\label{eq:cog-CC}
 \Phi^T_{N} = \xvar^{\cogvec{N}} \Fpolyd{\Ex{N}} (\xvar^\beta)
\end{equation}
where %$\Ex{N}=\Ext^1(T, N)$, while 
the co-$g$-vector $\cogvec{N}=[T,N]-\beta[\Ex{N}]$ and 
\[
  \Fpolyd{E} (\yvar)= \sum_{d} \euler \bigl( \quotGr{d} E \bigr)\yvar^{d}
\]
is the quotient F-polynomial (cf. \Cref{rem:fp-F-poly}),
that is, $\quotGr{d} E$ is the Grassmannian of $d$-dimensional quotients of $E$.  
%\begin{remark}\label{rem:ch-vble}
Note also that writing $\Fpolyd{\Ex{N}} (\xvar^\beta)$ is notation 
for making a monomial change of variables,
equivalent to writing $\CC[\beta] \Fpolyd{\Ex{N}} (\yvar)$.
%\end{remark}

Thus the substitution yields
\begin{equation}\label{eq:UpsW-A}
  \Upsilon^{T}\superpot_i = q^{m_i}\Phi^T_{\extd{P_i}} / \Phi^T_{P_i} 
  = q^{m_i} \xvar^{\cogvec{\extd{P_i}}-[T,P_i]} \Fpolyd{\Ex{\extd{P_i}}} (\xvar^\beta).
\end{equation}
The right-hand side can be interpreted as a formal Laurent polynomial in $\CC[\ZZ\oplus \Mzero]$,
with the exponent of $q$ in the $\ZZ$ summand.
We can then rewrite \eqref{eq:UpsW-A} as follows.
%We can interpret the exponent $m_i$ of $a$ as $m_i=[T, \modJ][S_i]$, making use 
%of the duality between $\Grot(\CM\algA)=\Grot(\proj\algA)$ and $\Grot(\fd\algA)$ and rewrite 
%\eqref{eq:UpsW-A} as follows.

\begin{proposition}\label{thm:Wformula}
\begin{equation}\label{eq:UpsW-B}
  \Upsilon^{T}\superpot_i 
%  = \CC \bigl[ \betahatdual \bigr] \Bigl( \yvar^{[S_i]} \Fpolyd{\Ex{\extd{P_i}}} (\yvar) \Bigr) 
  = \CC \bigl[ \betahatdual \bigr] \Bigl( \yvar^{[S_i]} \Fpolyd{\Ext^1(T,\extd{P_i})} (\yvar) \Bigr), 
\end{equation}
where $S_i$ is the simple $\algA$-module at the boundary vertex $i\in Q_0$ and 
\begin{equation}\label{eq:wt-hat-inv-dual}
  \betahatdual \colon \Grot(\fd\algA) \to \ZZ\oplus \Mzero \colon [X] \mapsto ([T,\modJ][X], -\beta\dual[X])
\end{equation}
is dual (as a lattice map) to the isomorphism 
$\betahat\colon \ZZ\oplus \Nstar \to \Grot(\CM\algA)$ in \eqref{eq:wt-hat-inv}. 
%with the exponent of~$q$ for the map $ \CC \bigl[ \betahatdual \bigr]$  
%taken in the $\ZZ$ summand of $\ZZ\oplus \Mzero$.
\end{proposition}

\begin{proof}
Note that we can effectively treat $\beta\dual\colon \Grot(\fd\algA)\to \Grot(\CM\algA)$ just as
the dual of $\beta\colon \Grot(\fd\algA)\to \Grot(\CM\algA)$ and then restrict the codomain
of $\beta\dual$ to $\Mzero$, into which it maps.
Since $\beta$ appears in $\betahat$ restricted to $\Nstar$,
we are thus identifying $\Nstar\dual=\Mzero$, as in \Cref{rem:dual}.

If we write $\beta$ as a matrix using the dual bases of simples and projectives,
then $\beta\dual$ is given by the transpose matrix.
Thus $\beta=-\beta\dual$ on $Q_0^{int}$, by \Cref{Prop:projres}\itmref{itm:pr1}.

Hence
\[
  \CC \bigl[ \betahatdual \bigr] \Fpolyd{\Ex{\extd{P_i}}} (\yvar) = \Fpolyd{\Ex{\extd{P_i}}} (\xvar^\beta),
\]
which lies in $\CC[\Mzero]$ because $[T,\modJ]\colon \Grot(\fd\algA) \to \ZZ$ vanishes on $Q_0^{int}$.
This is because $\modJ=P_{n-k}$, so $[T,\modJ]$ is the projective $\algA$-module
at a boundary vertex of $Q_0$.
This also means that the $q^{m_i}$ factor is correctly encoded, because $m_i=[T,\modJ][S_i]$.

Finally, we need to see that
\[
  \beta\dual [S_i] = [T,P_i]-\cogvec{\extd{P_i}}.
\]
Note that we can compute $\cogvec{\extd{P_i}}=[T,T']-[T,T'']$, where 
\begin{equation}\label{eq:addTappEP}
  0\to \extd{P_i}\to T' \to T''\to 0,
\end{equation}
is an $\add T$-approximation, as in  \eqref{eq:addT->M-other}.

Applying $\Hom(-,T)$ to \eqref{eq:PPS} yields an exact sequence of right $\algA$-modules
\begin{equation}\label{eq:diam-seq}
0\lra \Hom(\extd{P_i},T)\lraa{f_i^*} \Hom(P_i,T)
\lra \Ext^1(\extd{S_i},T) \lra \Ext^1(\extd{P_i},T)\lra 0.
\end{equation}
Observe that $\cok f_i^*=S_i\opmod$, the simple right $\algA$-module at $i$,
because $\idmap\colon P_i\to P_i$ is the only map $P_i\to T$ that doesn't lift to $\extd{P_i}$.

Applying $\Hom(-,T)$ to \eqref{eq:addTappEP} yields the short exact sequence
\[
\ShExSeq{\Hom(T'',T)}{\Hom(T',T)}{\Hom(\extd{P_i},T)}
\]
which can be spliced into the first part of \eqref{eq:diam-seq} 
to give a projective resolution of $S_i\opmod$
\begin{equation}\label{eq:right-proj-res}
  0\lra \Hom(T'',T) \lra \Hom(T',T) \lra \Hom(P_i,T) \lra S_i\opmod \lra 0.
\end{equation}
%By general homological algebra (e.g.~\cite[Prop IV.11.1]{HS}), 
By \Cref{rem:beta-stuff},
we can identify $\beta\dual$ with $\beta\op\colon \Grot(\fd\algA\op)\to \Grot(\CM\algA\op)$.
%In our case, this also follows from \Cref{thm:projresolb}.
Thus
\[
  \beta\dual [S_i] = [T,P_i]-[T,T']+[T,T''],
\]
as required to complete the proof.
\end{proof}

We can now adapt Rietsch--Williams mirror symmetry result \eqref{eq:RW-coneSP},
describing the NO-cone as a superpotential cone, to our context.

\begin{theorem}\label{thm:coneGV-SP}
The equations defining $\gvc{T}$ in $\Grot(\CM\algA)$ can be obtained by 
tropicalising the superpotential
\begin{equation}\label{eq:WinX}
  W_\cluX(T)= \sum_{i=1}^n \yvar^{[S_i]} 
  \Fpolyd{\Ext^1(T,\extd{P_i})} (\yvar) %  \Fpolyd{\Ex{\extd{P_i}}} (\yvar) 
  \in \CC[\Grot(\fd\algA)].
\end{equation}
Explicitly, these equations are
\begin{equation}\label{eq:gv-faces1}
  ([S_i]+[V])(x) \geq 0, \quad\text{for all quotients $V$ of $\Ext^1(T,\extd{P_i})$, for $i=1,\ldots,n$.} 
\end{equation}
%for all $i=1,\ldots,n$. 
Alternatively, we can write these equations as
\begin{equation}\label{eq:gv-faces2}
  [U](x) \geq 0, \quad\text{for all $U\subspc\Ext^1(\extd{S_i},T)$, for $i=1,\ldots,n$,} 
\end{equation}
where $[U]\in\Grot(\fd\algA\op)$, which we can also view as dual to $\Grot(\CM\algA)$.
%for all $i=1,\ldots,n$.
\end{theorem}

\begin{proof}
By \cite[Thm.~16.18]{RW}, the tropicalisation of $\Upsilon^{T}\superpot$ yields the equations defining the cone
\[ 
  \noc{G}\subset (\ZZ\oplus\Nstar)\tensR,
\]
as in \S\ref{subsec:12.3}.
We also know from \Cref{prop:nob=kab}, specifically \eqref{eq:noc=gvc}, 
since $\betahat$ is the inverse of  $\wthat$, that
\[ 
  \betahat \noc{G}=\gvc{T}\subset \Grot(\CM\algA)\tensR.
\]
But \Cref{thm:Wformula} shows precisely that the superpotential in \eqref{eq:WinX}
transforms into $\Upsilon^{T}\superpot$ under the dual isomorphism, as required.

The equations \eqref{eq:gv-faces1} follow immediately, because the exponents in 
$\Fpolyd{\Ext^1(T,\extd{P_i})} (\yvar)$ % $\Fpolyd{\Ex{\extd{P_i}}} (\yvar)$
are precisely $[V]$, for the quotients $V$ of $\Ext^1(T,\extd{P_i})$.

The equations \eqref{eq:gv-faces2} come from the observation that $\Ext^1(T,\extd{P_i})$ is the (linear) dual of 
$\Ext^1(\extd{P_i},T)$, while the last half of \eqref{eq:diam-seq} shows that 
submodules of $\Ext^1(\extd{P_i},T)$ can be extended by $S_i\opmod$
to give submodules of $\Ext^1(\extd{S_i},T)$.
\end{proof}

It would be very interesting to find a categorical proof that $g$-vectors satisfy these inequalities.
One easy case is that $[S_i] [T,M]\geq 0$ for all $i$ and $M$,
because all the projectives in an $\add T$-presentation of $M$ appear in the first term.

\begin{remark}\label{rem:right-vers}
The argument deriving \eqref{eq:gv-faces2} from \eqref{eq:gv-faces1}
also shows that adding~$1$ to the $i$-summand of \eqref{eq:WinX} gives the 
usual (sub-module) F-polynomial for $\Ext^1(\extd{S_i},T)$.
Note that adding $1$ does not affect its use as a superpotential.
This also provides an alternative approach to calculating
$\Ext^1(T, \extd{P_i})$, compared to the one in \S\ref{subsec:rect-sp}.
\end{remark}

\begin{remark}\label{rem:LFS}
Similar expressions to \eqref{eq:WinX} and \Cref{rem:right-vers} 
for the potential in terms of F-polynomials are derived in \cite[Thm.~6.5]{LFS}.
The description there holds in a more general context and the relevant modules 
are expressed just in terms of the quiver (with potential).
It would be interesting to understand why our description of these modules as
$\Ext^1(T,\extd{P_i})$ or (the dual of) $\Ext^1(\extd{S_i},T)$ gives the same answer.
\end{remark}

\begin{remark}\label{rem:Fei}
Conditions defining the $g$-vector cone similar to \eqref{eq:gv-faces2} are found by 
Fei \cite[Thm 5.16, Rem 5.17]{Fei3}.
The context is a little different, in that $g$-vectors are defined intrinsically 
and for a general quiver with potential, but which is required to be `frozen-Jacobi finite',
which ours are not.
The conditions come from modules that are essentially the same as those from \cite{LFS}.
It seems likely that our $g$-vector cone could be obtained by omitting the frozen arrows from our quiver
(with potential) and using Fei's conditions.
\end{remark}

% ===================================================================

% ===================================================================
\end{document}